\documentclass[letterpaper,11pt]{amsart}
\usepackage{amssymb}
\usepackage{amsthm}
\usepackage{mathrsfs} 
\usepackage[utf8]{inputenc}
\usepackage[margin=1.0in]{geometry}
\usepackage{pst-node}
\usepackage{multirow}
\usepackage{tikz-cd} 
\usepackage{hyperref}
\usepackage{xypic}
\newtheorem{theorem}{Theorem}[section]
\newtheorem{lemma}[theorem]{Lemma}
\newtheorem{proposition}[theorem]{Proposition}
\newtheorem{corollary}[theorem]{Corollary}
\newtheorem{conjecture}[theorem]{Conjecture}
 
\newtheorem{condition/definition}[theorem]{Condition/Definition} 

\newtheorem{assumption}[theorem]{Assumption}

\usepackage[T1]{fontenc}
\usepackage{appendix}
\usepackage[english]{babel}
\usepackage[utf8]{inputenc}
\usepackage[backend=biber,style=alphabetic]{biblatex}
\theoremstyle{definition}
\newtheorem{definition}[theorem]{Definition}
\theoremstyle{remark}
\newtheorem{remark}[theorem]{Remark}
\theoremstyle{definition}

\usepackage{amsmath}
\usepackage{amsfonts}
\usepackage{tikz-cd}
\def\ul{\underline}
\def\ULA{\mathrm{ULA}}
\def\Spf{\mathrm{Spf}}
\def\ol{\overline}

\def\SL{\mathrm{SL}}
\def\U{\mathrm{U}}
\def\mf{\mathfrak}
\def\ra{\rightarrow}
\def\la{\leftarrow}
\def\Perv{\mathrm{Perv}}
\def\lim{\mathop{\rm lim}\nolimits}
\def\colim{\mathop{\rm colim}\nolimits}
\def\Spec{\mathop{\rm Spec}}
\def\Spa{\mathop{\rm Spa}}
\def\Spd{\mathop{\rm Spd}}
\def\Hom{\mathop{\rm Hom}\nolimits}

\def\Sh{\mathop{\textit{Sh}}\nolimits}

\def\Im{\text{Im}}

\def\Bun{\mathrm{Bun}}

\def\nmEis{\mathrm{nEis}}

\def\Hck{\mathrm{Hck}}
\def\HT{\mathrm{HT}}
\def\Gr{\mathrm{Gr}}
\def\Ind{\mathrm{Ind}}
\def\Perf{\mathrm{Perf}}
\def\det{\mathrm{det}}
\def\Ad{\mathrm{Ad}}
\def\Lie{\mathrm{Lie}}
\def\dim{\mathrm{dim}}

\def\Red{\mathrm{Red}}
\def\Sht{\mathrm{Sht}}
\def\Rep{\mathrm{Rep}}
\def\D{\mathrm{D}}
\def\DULA{\mathrm{D}^{\mathrm{ULA}}}
\def\Dadm{\mathrm{D}^{\mathrm{adm}}}
\def\pD{\phantom{}^{\mathrm{p}}\mathrm{D}}

\def\RHom{R\mathrm{Hom}}

\def\GL{\mathrm{GL}}

\def\Tilt{\mathrm{Tilt}}

\def\bb{\mathbb}

\def\LLC{\mathrm{LLC}}
\def\GU{\mathrm{GU}}

\def\Sh{\mathrm{Sh}}
\newcommand{\Dlis}{\mathrm{D}_{\mathrm{lis}}}

\newcommand{\domgamorb}{\mathbb{X}_{*}(T_{\ol{\mathbb{Q}}_{p}})^{+}/\Gamma}

\newcommand{\mc}{\mathcal}

\newcommand{\Res}{\mathrm{Res}}
\newcommand{\GSp}{\mathrm{GSp}}

\newcommand{\Stor}{\mathcal{S}^{\mathrm{tor}}}

\newcommand{\id}{\mathrm{id}}
\def\Aut{\mathrm{Aut}}

\newcommand{\Ig}{\mathrm{Ig}}
\newcommand{\mIg}{\mathfrak{Ig}}
\newcommand{\cochar}{\mathbb{X}_{*}(T_{\ol{\mathbb{Q}}_{p}})}
\newcommand{\domcochar}{\mathbb{X}_{*}(T_{\ol{\mathbb{Q}}_{p}})^{+}}
\newcommand{\tor}{\mathrm{tor}}
\newcommand{\der}{\mathrm{der}}

\title{Torsion Vanishing for some Shimura Varieties}

\author{Linus Hamann and Si Ying Lee, with an Appendix by David Hansen}
\begin{document}
\maketitle
\begin{abstract} We generalize the torsion vanishing results of \cite{CS1,CS2,Ko,San}. Our results apply to the cohomology of general Shimura varieties $(\mathbf{G},X)$ of PEL type AC, localized at a suitable maximal ideal $\mf{m}$ in the spherical Hecke algebra  at primes $p$ such that $\mathbf{G}_{\mathbb{Q}_{p}}$ is a group for which we know the Fargues-Scholze local Langlands correspondence is the semi-simplification of a suitably nice local Langlands correspondence, as shown in \cite{FS,Ham1,HKW,BMNH}. This is accomplished by combining Koshikawa's technique \cite{Ko}, the theory of geometric Eisenstein series over the Fargues-Fontaine curve \cite{Ham2}, the work of Santos \cite{San} describing the structure of the fibers of the minimally and toroidally compactified Hodge-Tate period morphism for general PEL type Shimura varieties of type AC, and ideas developed by Zhang \cite{Zha} on comparing Hecke correspondences on the moduli stack of $G$-bundles with the cohomology of Shimura varieties. In the process, we also establish a description of the generic part of the cohomology that bears resemblance to the work of Xiao-Zhu \cite{XZ}. Moreover, we also construct a filtration on the compactly supported cohomology that differs from Mantovan's filtration in the case that the Shimura variety is non-compact, allowing us to circumvent some of the circumlocutions taken in \cite{CS2,Ko}. Our method showcases a very general strategy for proving such torsion vanishing results, and should bear even more fruit once the inputs are generalized. Motivated by this, we formulate an even more general torsion vanishing conjecture (Conjecture \ref{conj: generaltorsionvanish}).
\end{abstract}
\tableofcontents
\section{Introduction}
\subsubsection{The Main Result}
Let $\mathbf{G}$ be a connected reductive group over $\mathbb{Q}$ admitting a Shimura datum $(\mathbf{G},X)$ which we fix from now on. Fix a prime number $p$ and let $G := \mathbf{G}_{\mathbb{Q}_{p}}$ be the base-change to $\mathbb{Q}_{p}$. We will assume that $G$ is unramified so that there exists a hyperspecial subgroup $K^{\mathrm{hs}}_{p} \subset G(\mathbb{Q}_{p})$ and a Borel $B \subset G$ surjecting onto a maximal torus $T \subset G$ which we now fix. We consider the open compact subgroup $K := K^{p}K_{p}^{\mathrm{hs}} \subset \mathbf{G}(\mathbb{A}_{f})$, where $K^{p} \subset \mathbf{G}(\mathbb{A}_{f}^{p})$ denotes a sufficiently small level away from $p$. Let $\mathrm{Sh}(\mathbf{G},X)_{K}$ denote the corresponding Shimura variety defined over the reflex field $E$. Given a prime $p \neq \ell$, we will be interested in understanding the $\ell$-torsion cohomology groups 
\[ R\Gamma_{c}(\mathrm{Sh}(\mathbf{G},X)_{K,\ol{E}},\ol{\mathbb{F}}_{\ell}) \]
and 
\[ R\Gamma(\mathrm{Sh}(\mathbf{G},X)_{K,\ol{E}},\ol{\mathbb{F}}_{\ell}). \]
In particular, since the level at $p$ is hyperspecial, the unramified Hecke algebra 
\[ H_{K_{p}^{\mathrm{hs}}} := \ol{\mathbb{F}}_{\ell}[K_{p}^{\mathrm{hs}} \backslash G(\mathbb{Q}_{p})/K_{p}^{\mathrm{hs}}] \]
will act on these complexes via the right action. Given a maximal ideal $\mf{m} \subset H_{K_{p}^{\mathrm{hs}}}$, we can localize both of these cohomology groups at $\mf{m}$. We will be interested in describing this localization. To do this, we recall that, given such a maximal ideal $\mf{m} \subset H_{K_{p}^{\mathrm{hs}}}$, this defines an unramified $L$-parameter 
\[ \phi_{\mf{m}}: W_{\mathbb{Q}_{p}} \ra \phantom{}^{L}G(\ol{\mathbb{F}}_{\ell}) \]
up to $\hat{G}$-conjugacy specified by a semisimple element $\phi_{\mf{m}}(\mathrm{Frob}_{\mathbb{Q}_{p}})$. In particular, the $L$-parameter $\phi_{\mf{m}}$ up to conjugacy is induced from a parameter $\phi_{\mf{m}}^{T}: W_{\mathbb{Q}_{p}} \ra \phantom{}^{L}T(\ol{\mathbb{F}}_{\ell}) \subset \phantom{}^{L}G(\ol{\mathbb{F}}_{\ell})$ factoring through the $L$-group of the maximal torus. Now, for $\alpha \in \bb{X}_{*}(T_{\ol{\bb{Q}}_{p}})/\Gamma$, a Galois orbit of coroots of $T$. There is a canonical representation $\alpha_{\Ad}$ of $\phantom{}^{L}T$ that we can associate to it (see \cite[Section~3.2]{Ham2} or \S \ref{section:verification}). We have the following definition. 
\begin{definition}{\cite[Definition~1.2]{Ham2}}{\label{def: generic}}
Given a toral $L$-parameter $\phi_{T}: W_{\mathbb{Q}_{p}} \ra \phantom{}^{L}T(\ol{\mathbb{F}}_{\ell})$, we say that $\phi_{T}$ is generic if, for all $\alpha \in \mathbb{X}_{*}(T_{\ol{\mathbb{Q}}_{p}})/\Gamma$ corresponding to a $\Gamma$-orbit of coroots, we have that the complex $R\Gamma(W_{\mathbb{Q}_{p}},\alpha_{\Ad} \circ \phi_{T})$ is trivial. Similarly, we say that a maximal ideal $\mf{m} \subset H_{K_{p}^{\mathrm{hs}}}$ is generic if the $L$-parameter $\phi_{\mf{m}}^{T}$ is a generic toral parameter, where we note that this only depends on $\phi_{\mf{m}}$ up to $\hat{G}$-conjugacy.
\end{definition}
For unramified parameters and $G = \GL_{n}$, this coincides with the notion of decomposed generic considered in \cite[Definition~1.9]{CS1}. We set $d = \dim(\mathrm{Sh}(\mathbf{G},X)_{K})$. Motivated by \cite[Theorem~1.1]{CS1} and \cite[Theorem~1.1]{CS2}, we make the following conjecture.
\begin{conjecture}{\label{conj: torsionvanishing}}
Let $(\mathbf{G},X)$ be a Shimura datum such that $G = \mathbf{G}_{\mathbb{Q}_{p}}$ is unramified and $K = K_{p}K^{p}$ is a sufficiently small level with $K_{p} = K_{p}^{\mathrm{hs}}$ hyperspecial. If $\mf{m} \subset H_{K_{p}^{\mathrm{hs}}}$ is a generic maximal ideal then the cohomology of  $R\Gamma(\mathrm{Sh}(\mathbf{G},X)_{K,\ol{E}},\ol{\mathbb{F}}_{\ell})_{\mf{m}}$ (resp. $R\Gamma_{c}(\mathrm{Sh}(\mathbf{G},X)_{K,\ol{E}},\ol{\mathbb{F}}_{\ell})_{\mf{m}}$) is concentrated in degrees $d \leq i \leq 2d$ (resp. $0 \leq i \leq d$). 
\end{conjecture}
We first recall the motivating situation of Caraiani-Scholze \cite{CS1,CS2}. Let $F/\mathbb{Q}$ be a CM field, and let $(B,*,V,\langle\cdot,\cdot\rangle)$ be a PEL datum with $B$ a central simple $F$-algebra and $V$ a non-zero finite type left $B$-module. Let $(\mathbf{G},X)$ denote the Shimura datum attached to it, where $\mathbf{G}$ is the connected reductive group over $\mathbb{Q}$ defined by the automorphisms of $V$ preserving the choice of $B$-linear pairing $\langle\cdot,\cdot\rangle$ up to similitude. We have the following result.
\begin{theorem}{\label{thm: cstorsionvanishing}}{\cite{CS1,CS2,Ko,San}}
Assume that $(\mathbf{G},X)$ is a PEL type Shimura datum of type $A$. If the prime $p$ splits completely in $F$ then Conjecture \ref{conj: torsionvanishing} is true.
\end{theorem}
\begin{remark}
Koshikawa proved this under the assumption that $B = F$ and $V = F^{2n}$, and the global unitary group $\mathbf{G}$ is quasi-split, as well as in the case when $p$ is split in $F$ and the Shimura variety is compact. These additional assumptions were removed in the PhD thesis of Santos \cite{San}. 
\end{remark}
\begin{remark}
Caraiani--Scholze actually proved a slightly different result.
Work under the respective global hypotheses of \cite{CS1,CS2},
and assume that $p$ splits completely in $F$.
Let $S$ be a finite set of rational primes not containing $p$
such that $\mathbf{G}$ is unramified and $K^{p}$ is hyperspecial
away from $S$. Consider a maximal ideal $\mf{m} \subset \mathbb{T}^{S}$
in the prime-to-$S$ Hecke algebra such that
\[
H^{i}(\mathrm{Sh}(\mathbf{G},X)_{K,\ol{E}},
      \ol{\mathbb{F}}_{\ell})_{\mf{m}} \neq 0
\]
for some $i$.
There is then an associated global Galois representation
$\rho_{\mf{m}}$ of $\mathrm{Gal}(\bar F/F)$.
Assume that, for every place $w \mid p$ of $F$, the restriction
\[
\rho_{\mf{m}}\vert_{\mathrm{Gal}(\bar F_w/F_w)}
\]
is unramified and its associated toral $L$-parameter is generic
in the sense of Definition~\ref{def: generic}.
Then Caraiani--Scholze show that localization at
$\mf{m}^{p}:=\mf{m}\cap\mathbb{T}^{S\cup\{p\}}$
is concentrated in the same degrees as in
Conjecture~\ref{conj: torsionvanishing}.
\end{remark}
\begin{remark}
In the case of Harris-Taylor Shimura varieties, there is also work of Boyer \cite{Boy}, which describes the localization of the torsion cohomology at non-generic maximal ideals. 
\end{remark}
\begin{remark}
\label{rmk:weakgeneric}
We believe that Conjecture \ref{conj: torsionvanishing} is true under the weaker hypothesis that $H^2(W_{\mathbb Q_p},
    \alpha_{\mathrm{Ad}}\circ\phi^T)$ is trivial for all $\Gamma$-orbits of coroots $\alpha$, as is shown in \cite{CS2,San,Ko} for the particular Shimura varieties they consider. However, the theory of geometric Eisenstein series which we will invoke in this paper becomes more complicated in this case (see the discussion around \cite[Conjecture~1.18]{Ham2}), and so a proof of this conjecture using our methods would require a deeper understanding of geometric Eisenstein series when this assumption is dropped (cf. Remark \ref{rem: weaklyLanglandsShahidicase}). In certain cases however the theory of geometric Eisenstein series becomes unnecessary, and in this case an argument using compatibility of Fargues-Scholze with some suitably nice form of local Langlands is sufficient (see Assumption \ref{compatibility}), as in the style of the arguments of \cite{Ko,San} (see Remark \ref{rem: perversetexactnessintheSplitCase} for details). 
\end{remark}
Caraiani-Scholze \cite{CS1,CS2} proved their results under some technical restrictions of a global nature, which Koshikawa \cite{Ko} was able to remove by using compatibility of the Fargues-Scholze local Langlands correspondence with the semi-simplification of the Harris-Taylor correspondence for $\GL_n$. In the process, Koshikawa exhibited a much more flexible method for proving Theorem \ref{thm: cstorsionvanishing}. The goal of the current paper is to expand the scope of Koshikawa's technique, motivated by work of the first author on geometric Eisenstein series in the Fargues-Fontaine setting \cite{Ham2}. We then carry the strategy out in some particular cases using work on local-global compatibility of the Fargues-Scholze local Langlands correspondence beyond the case of $\GL_{n}$, as studied in \cite{Ham2,BMNH}.

In order to state our results cleanly, we first note that when $(\mathbf{G},X)$ is a PEL datum of type AC, we can write
\[\mathbf{G}_{\mathbb{Q}_p}\simeq (\prod_{j \in J} G_j) \times_{\prod \Res_{L_j/\mathbb{Q}_p}\mathbb{G}_m} \mathbb{G}_m,\]
where $G_j=\Res_{L_j/\mathbb{Q}_p} H_j$ for $L_j/\mathbb{Q}_p$ an unramified extension (if the PEL datum is unramified at $p$) with a similitude character $\nu_j:G_j\rightarrow \Res_{L_j/\mathbb{Q}_p}\mathbb{G}_m$. $\mathbb{G}_m$ in the isomorphism above is the diagonally embedded copy of $\mathbb{G}_{m,\mathbb{Q}_p}$ in $\prod_j \Res_{L_j/\mathbb{Q}_p}\mathbb{G}_m$. We say that $\mathbf{G}_{\bb{Q}_{p}}$ is a similitude fiber product of the $G_{j}$ in this case.

We assume that $G_{j}$ is one of the following groups with $\ell$ and $p$ satisfying the corresponding conditions. 
{
\renewcommand{\arraystretch}{1.15}
\begin{equation}{\label{constrainttable}}
\begin{array}{|c|c|c|c|}
\hline
G_j & \text{Constraint on }G_j & \ell & p \\[1.5em]
\hline
\operatorname{Res}_{L/\mathbb Q_p}
(\mathbb G_{m,L}\times\mathrm{GL}_{n,L})
&
L/\mathbb Q_p\text{ unramified}
&
(\ell,[L:\mathbb Q_p])=1
&
\\
\hline
\mathrm{GU}_{n,\mathbb Q_p}
&
\text{$n$ odd}
&
\ell\neq 2
&
\\
\hline
\operatorname{Res}_{L/\mathbb Q_p}(\mathrm{GU}_{2,L})
&
L/\mathbb Q_p\text{ unramified}
&
(\ell,2[L:\mathbb Q_p])=1
&
\\
\hline
\operatorname{Res}_{L/\mathbb Q_p}(\mathrm{GSp}_{4,L})
&
\begin{array}{c}
L=\mathbb Q_p\\
L/\mathbb Q_p\text{ unramified}
\end{array}
&
\begin{array}{c}
\mbox{}\\
(\ell,[L:\mathbb Q_p])=1
\end{array}
&
\begin{array}{c}
\mbox{}\\
p\neq 2
\end{array}
\\
\hline
\operatorname{Res}_{L/\mathbb Q_p}(\mathrm{GL}_{2,L})
&
L/\mathbb Q_p\text{ unramified}
&
(\ell,[L:\mathbb Q_p])=1
&
\\
\hline
\end{array}
\end{equation}
}
These groups and their similitude maps are defined in Section \ref{section:verification}. For a group $G$ which is a similitude fiber product of the groups appearing in Table (\ref{constrainttable}), we consider the quantity
\begin{equation}{\label{eqn: adhocinvariant}}
s(G):=
\sum_{\substack{j\in J:\ G_j\text{ is of type }\\
\operatorname{Res}_{L_{j}/\bb{Q}_{p}}(\mathrm{GL}_{2,L_j}) \text{ or }
\operatorname{Res}_{L_{j}/\bb{Q}_{p}}(\mathrm{GSp}_{4,L_j})}}
[L_j:\mathbb Q_p].
\end{equation}
We now precisely impose the following condition using this table.
\begin{assumption}{\label{assumption: keygroupassumption}}
We assume the following conditions on
$G$, $p$, and $\ell$.
\begin{enumerate}
\item The group $G$ is a common-similitude
fiber product
\[
G\simeq
\left(\prod_{j\in J}G_j\right)
\times_{\prod_{j\in J}\Res_{L_j/\mathbb{Q}_p}\mathbb{G}_m}\mathbb G_m,
\]
where the $G_j$ are factors appearing
in Table~(\ref{constrainttable}), with their indicated
similitude characters.
\item
We impose the following numerical
conditions:
\begin{enumerate}
\item
The primes $p$ and $\ell$ satisfy
the conditions of Table~(1)
for each factor $G_j$.
\item
If $s(G)\geq2$, then $\ell\neq2$,
where $s(G)$ is defined in (2).
\end{enumerate}
\end{enumerate}
\end{assumption}
\begin{remark}
We emphasize that a common-similitude
fiber product may itself be isomorphic
to a direct product. For example, if all the factors $G_j$
are of the form
\[
\operatorname{Res}_{L_j/\mathbb Q_p}
(\mathbb G_m\times\mathrm{GL}_{n_j}),
\]
with similitude given by projection
onto the first factor, then
\[
G\simeq
\mathbb G_m\times
\prod_{j\in J}
\operatorname{Res}_{L_j/\mathbb Q_p}
\mathrm{GL}_{n_j}.
\]
This happens for example in the split Harris-Taylor case.
\end{remark}

Our main result is the following.
\begin{theorem}{\label{theorem: mainthm}}{(Theorem \ref{theorem: mainthmbody})}
Suppose $(\mathbf{G},X)$ is a PEL datum of type AC satisfying Assumption \ref{assump: codim} such that the following $G$, $p$, and $\ell$ satisfy Assumption \ref{assumption: keygroupassumption}. Let $\mf{m} \subset H_{K_{p}^{\mathrm{hs}}}$ be a generic maximal ideal. Then, for a level $K = K^{p}K_{p}^{\mathrm{hs}} \subset \mathbf{G}(\bb{A}_{f})$, the cohomology $R\Gamma(\mathrm{Sh}(\mathbf{G},X)_{K,\ol{E}},\ol{\mathbb{F}}_{\ell})_{\mf{m}}$ (resp. $R\Gamma_{c}(\mathrm{Sh}(\mathbf{G},X)_{K,\ol{E}},\ol{\mathbb{F}}_{\ell})_{\mf{m}}$) is concentrated in degrees $d \leq i \leq 2d$ (resp. $0 \leq i \leq d$). 
\end{theorem}
\begin{remark}
This notably allows one to relax the assumption in \cite{CS1,CS2,Ko,San} that the prime $p$ splits in $F$ in the setting of Theorem \ref{thm: cstorsionvanishing}, answering a question of Caraiani. We also formulate a more axiomatic/less explicit version of this and the rest of our main Theorems as Theorem \ref{thm: axiomatictheorem}.
\end{remark}
We can also deduce a result for some abelian type Shimura varieties, such as Hilbert modular varieties, from the above result, which recovers results similar to that of Caraiani-Tamiozzo \cite[Theorem~B]{CT} (see Corollary \ref{cor: CarTamcomp}). 
\begin{corollary}\label{cor:abeliantype}{(Corollary \ref{cor:abeliantypebody})}
Suppose $(\mathbf{G}_1,X_1)$ is an
abelian-type Shimura datum
admitting a PEL-type covering
datum $(\mathbf{G},X)$ of type AC
satisfying Assumption~\ref{assumption: keygroupassumption} and Assumption \ref{assump: codim},
together with a central isogeny
\[
\mathbf{G}^{\mathrm{der}}
\longrightarrow \mathbf{G}_1^{\mathrm{der}}
\]
inducing an isomorphism of the
associated adjoint Shimura data. Assume that both  groups are unramified
at $p$. Let $K_{1,p}^{\mathrm{hs}}\subset \mathbf{G}_1(\mathbb Q_p)$ be a hyperspecial subgroup, and let $\mathfrak m_1\subset
 H_{K_{1,p}^{\mathrm{hs}}}$ be a generic maximal ideal.

Then, for any sufficiently small level
$K_1=K_1^pK_{1p}^{\mathrm{hs}}\subset\mathbf{G}_1(\mathbb{A}_f)$,
the complexes
$R\Gamma_c(\operatorname{Sh}(\mathbf{G}_1,X_1)_{K_1,\ol{E}},
\ol{\mathbb{F}}_{\ell})_{\mathfrak{m}_1}$
(resp. $R\Gamma(\operatorname{Sh}(\mathbf{G}_1,X_1)_{K_1,\ol{E}},
\ol{\mathbb{F}}_{\ell})_{\mathfrak{m}_1}$)
are concentrated in degrees $0\leq i\leq d$ (resp.
$d\leq i\leq 2d$), where
$d=\dim\operatorname{Sh}(\mathbf{G}_1,X_1)_{K_1}$.
\end{corollary}

We now explain how we are able to establish these results, as well as several other interesting results on the perverse $t$-exactness of Hecke operators on $\Bun_{G}$ (Theorem \ref{thm: appliedperversetexactnessintro}) and the splitting of Mantovan's filtration on the cohomology of the Shimura variety (Theorem \ref{thm: appliedmantprodform}) along the way.
\subsubsection{Proof Sketch of the Main Theorem}

One of the basic ingredients used in verifying Conjecture \ref{conj: torsionvanishing} in all known instances is the perspective on Mantovan's product formula provided by the Hodge-Tate period morphism. To explain this, we let $\mu \in \mathbb{X}_{*}(T_{\ol{\mathbb{Q}}_{p}})^{+}$ denote the minuscule geometric dominant cocharacter of $G$ determined by the inverse of the Hodge cocharacter of $X$ and an isomorphism $j: \ol{\mathbb{Q}}_{p} \simeq \mathbb{C}$ which we fix from now on. We consider the Kottwitz set $B(G)$, and with it the subset $B(G,\mu) \subset B(G)$ of $\mu$-admissible elements. Let $\mf{p}|p$ be the prime dividing $p$ in the reflex field $E$ induced by the embedding $\ol{\mathbb{Q}} \ra \ol{\mathbb{Q}}_{p}$ given by $j^{-1}$. We let $E_{\mf{p}}$ be the completion at $\mf{p}$, $C := \hat{\ol{E}}_{\mf{p}}$ be the completion of the algebraic closure, and $\Breve{E}_{\mf{p}}$ be the compositum of $E_{\mf{p}}$ with $\Breve{\bb{Q}}_{p}$. We recall that, attached to each element $b \in B(G,\mu)$, there exists a diamond
\[ \Sht(G,b,\mu)_{\infty} \ra  \Spd(\Breve{E}_{\mf{p}}) \]
parametrizing modifications 
\[ \mathcal{E}_{b} \dashrightarrow \mathcal{E}_{0} \]
of meromorphy $\mu$ between the $G$-bundle $\mathcal{E}_{b}$ corresponding to $b$ and the trivial $G$-bundle on the Fargues-Fontaine curve. This space has an action by $\underline{G(\mathbb{Q}_{p})} = \mathrm{Aut}(\mathcal{E}_{0})$ and $\underline{J_{b}(\mathbb{Q}_{p})} \subset \mathrm{Aut}(\mathcal{E}_{b})$, where $J_{b}$ is the $\sigma$-centralizer of $b$. This allows us to consider the quotients 
\[ \Sht(G,b,\mu)_{\infty}/\ul{K_{p}} \ra \Spd(\Breve{E}_{\mf{p}})  \]
for varying compact open subgroups $K_{p} \subset G(\mathbb{Q}_{p})$. In certain cases, these quotients are representable by rigid analytic varieties called local Shimura varieties, but they are always representable as diamonds. We can consider the compactly supported cohomology
\[ R\Gamma_{c}(\Sht(G,b,\mu)_{\infty,C}/\ul{K_{p}^{\mathrm{hs}}},\ol{\mathbb{F}}_{\ell}) \]
at hyperspecial level with torsion coefficients. This has an action of $W_{E_{\mf{p}}} \times J_{b}(\mathbb{Q}_{p}) \times H_{K_{p}^{\mathrm{hs}}}$\footnote{It is clear that the cohomology of $\Sht(G,b,\mu)_{\infty,C}$ has an action of the inertia group $I_{E_{\mf{p}}}$ of $E_{\mf{p}}$ from the definition; however, the space $\Sht(G,b,\mu)_{\infty} \ra \Spd(\Breve{E}_{\mf{p}})$ also has a canonical (non-effective) Frobenius descent datum giving a full action of $W_{E_{\mf{p}}}$ on the cohomology (see \cite[Section~IX.3]{FS}).}. Now, the Mantovan product formula tells us that if we look at $R\Gamma(\Sh(\mathbf{G},X)_{K,\ol{E}},\ol{\mathbb{F}}_{\ell})$ then this should always admit a filtration in the derived category whose graded pieces are
\[ R\Gamma_{c}(\Sht(G,b,\mu)_{\infty,C}/\ul{K_{p}^{\mathrm{hs}}},\ol{\mathbb{F}}_{\ell}(d_{b}))[2d_{b}] \otimes^{\bb{L}}_{\mathcal{H}(J_{b})} R\Gamma(\Ig^{b},\ol{\mathbb{F}}_{\ell}) \]
for varying $b \in B(G,\mu)$, where the objects are as follows.
\begin{enumerate}
    \item $\Ig^{b}$ is the perfect Igusa variety attached to an element $b \in B(G,\mu)$ in the $\mu$-admissible locus inside $B(G)$ and $d_{b} := \dim(\Ig^{b}) = \langle 2\rho_{G},\nu_{b} \rangle$, where $\rho_{G}$ is the half sum of all positive roots and $\nu_{b}$ is the slope cocharacter of $b$.
    \item $\mathcal{H}(J_{b}) := C^{\infty}_{c}(J_{b}(\mathbb{Q}_{p}),\ol{\mathbb{F}}_{\ell})$ is the usual smooth Hecke algebra.
    \item $\ol{\mathbb{F}}_{\ell}(d_{b})$ is the sheaf on $\Sht(G,b,\mu)_{\infty,C}/\ul{K_{p}^{\mathrm{hs}}}$ with trivial Weil group action and $J_{b}(\mathbb{Q}_{p})$ action as defined in \cite[Lemma~7.4]{Ko}.
\end{enumerate}
Such a filtration should always exist, but is not currently proven in general. In the case that the Shimura datum $(\mathbf{G},X)$ is PEL of type AC, a modern proof of this result can be found in \cite[Theorem~7.1]{Ko}.

This filtration on the complex $R\Gamma(\mathrm{Sh}(\mathbf{G},X)_{K,\ol{E}},\ol{\mathbb{F}}_{\ell})$ allows us to roughly split the verification of Conjecture \ref{conj: torsionvanishing} into two parts.
\begin{enumerate}
    \item Controlling the cohomology of the shtuka spaces $R\Gamma_{c}(\Sht(G,b,\mu)_{\infty,C}/\ul{K_{p}^{\mathrm{hs}}},\ol{\mathbb{F}}_{\ell}(d_{b}))_{\mf{m}}$.
    \item Controlling the cohomology of the Igusa varieties $R\Gamma(\mathrm{Ig}^{b},\ol{\mathbb{F}}_{\ell})$.
\end{enumerate}
We first discuss point (1). One of the key observations underlying Koshikawa's method was that the cohomology of the space $\Sht(G,b,\mu)_{\infty}$ computes the action of a Hecke operator $T_{\mu}$ corresponding to $\mu$ on $\Bun_{G}$ the moduli $v$-stack of $G$-bundles of the Fargues-Fontaine curve over $\ol{\bb{F}}_{p}$. The Hecke operators commute with the action of the excursion algebra on $\Bun_{G}$, and the action of the excursion algebra on a smooth irreducible representation $\rho$, viewed as a sheaf on $\Bun_{G}$, determines the Fargues-Scholze parameter of $\rho$. It follows that $R\Gamma_{c}(\Sht(G,b,\mu)_{\infty,C}/\ul{K_{p}^{\mathrm{hs}}},\ol{\mathbb{F}}_{\ell}(d_{b}))_{\mf{m}}$ as a complex of smooth $J_{b}(\mathbb{Q}_{p})$-modules will have irreducible constituents $\rho$ with Fargues-Scholze parameter $\phi_{\rho}^{\mathrm{FS}}$ equal to $\phi_{\mf{m}}$ as conjugacy classes of parameters. When $\mathbf{G}_{\mathbb{Q}_{p}} = G$ is a product of $\GL_{n}$s, as in Theorem \ref{thm: cstorsionvanishing} (by the assumption that $p$ splits in $F$), it follows from the work of Hansen-Kaletha-Weinstein \cite[Theorem~1.0.3]{HKW} that the Fargues-Scholze correspondence for $J_{b}(\mathbb{Q}_{p})$ with rational coefficients agrees with the semi-simplification of the Harris-Taylor correspondence precomposed with Baduluescu's Jacquet-Langlands map, where we recall that $J_{b}$ is a product of inner forms of $\GL_{n}$s in this case. In particular, using that $\mf{m}$ is generic, it follows that $\phi_{\rho}^{\mathrm{FS}} = \phi_{\mf{m}}$ must lift to a $\ol{\mathbb{Z}}_{\ell}$ parameter which is also generic in the analogous sense, and the condition of genericity implies that the lift cannot come from the semi-simplification of a parameter with non-trivial monodromy. Using this, one can deduce that such a $\rho$ only exists if the group $J_{b}$ is quasi-split. In this particular case ($G$ is a product of $\GL_{n}$s), this can only happen if $b \in B(G,\mu)$ corresponds to the ordinary element (cf. Remark \ref{rem: perversetexactnessintheSplitCase}). 

This argument of Koshikawa was formalized and generalized further in work of the first author \cite{Ham2}. In particular, it was noted that, for a general quasi-split $G$ and $\mf{m}$ generic, the cohomology $R\Gamma_{c}(\Sht(G,b,\mu)_{\infty,C}/\ul{K_{p}^{\mathrm{hs}}},\ol{\mathbb{F}}_{\ell}(d_{b}))_{\mf{m}}$ will only be non-trivial if $b \in B(G,\mu)_{\mathrm{un}} := B(G)_{\mathrm{un}} \cap B(G,\mu)$, where $B(G)_{\mathrm{un}}$ is the set of elements lying in the image of the map $B(T) \ra B(G)$, assuming that the Fargues-Scholze local Langlands correspondence has certain expected properties (Assumption \ref{compatibility}). These unramified elements will be precisely the elements for which $J_{b}$ is quasi-split.
The set $B(G,\mu)_{\mathrm{un}}$ corresponds to Weyl group orbits of weights in the representation $V_{\mu}$ of $\hat{G}$ restricted to the fixed points $\hat{G}^{\Gamma}$. In particular, if $G$ is split then, since $\mu$ is minuscule, $B(G,\mu)_{\mathrm{un}}$ consists of one element, corresponding to the unique Weyl group orbit of the highest weight. This is the situation occurring in the previous paragraph. Moreover, the contribution of the cohomology of this shtuka space is easily understood, and the problem completely reduces to controlling the cohomology of $\Ig^{b}$ when $b \in B(G,\mu)_{\mathrm{un}}$ is the $\mu$-ordinary element. However, if $G$ is not split then the restriction of $V_{\mu}$ to $\hat{G}^\Gamma$ may have multiple Weyl group orbits of weights. In particular, one needs to control the cohomology groups  
\[ R\Gamma_{c}(\Sht(G,b,\mu)_{\infty,C}/\ul{K^{\mathrm{hs}}_{p}},\ol{\mathbb{F}}_{\ell}(d_{b}))_{\mf{m}} \]
for all possible $b \in B(G,\mu)_{\mathrm{un}}$. This makes the situation much more complicated; in fact, for non-split $G$, the basic element could be unramified, and in this case the Igusa variety $\Ig^{b}$ is just a profinite set. Hence, the problem of torsion vanishing for the contribution of the basic locus is completely reduced to controlling the generic part of the torsion cohomology of the local shtuka space attached to the basic element. 

Such control of the cohomology of shtuka spaces with torsion coefficients for these more general situations was attained in \cite{Ham2}. In order to understand this, it is helpful to move away from the language of isotypic parts of shtuka spaces and consider the action of Hecke operators on $\D(\Bun_{G},\ol{\mathbb{F}}_{\ell})$, the derived category of \'etale $\ol{\mathbb{F}}_{\ell}$-sheaves on $\Bun_{G}$. Since we are interested in cohomology localized at a generic maximal ideal $\mf{m}$, Hansen constructs in Appendix \ref{append: spectralactionproperties} a full subcategory $\D(\Bun_{G},\ol{\mathbb{F}}_{\ell})_{\phi_{\mf{m}}} \subset \D(\Bun_{G},\ol{\mathbb{F}}_{\ell})$ together with an idempotent localization map $(-)_{\phi_{\mf{m}}}: \D(\Bun_{G},\ol{\mathbb{F}}_{\ell}) \ra \D(\Bun_{G},\ol{\mathbb{F}}_{\ell})_{\phi_{\mf{m}}}$ such that, on smooth irreducible representations, the localization map is either an isomorphism or $0$ depending on whether the representation has Fargues-Scholze parameter conjugate to $\phi_{\mf{m}}$ or not (Lemma \ref{lemma: localization map properties} (1)). We let $\DULA(\Bun_{G},\ol{\mathbb{F}}_{\ell})$ denote the full subcategory of ULA objects, where we recall by \cite[Theorem~V.7.1]{FS}, that this is equivalent to insisting that the restrictions to all the HN-strata indexed by $b \in B(G)$ are valued in the full subcategories $\Dadm(J_{b}(\mathbb{Q}_{p}),\ol{\mathbb{F}}_{\ell})$ of perfect admissible complexes (i.e the invariants under all open pro-$p$ $K \subset J_{b}(\mathbb{Q}_{p})$ is a perfect complex). Using the results of \cite{Ham2}, we show (Corollary \ref{cor: appliedsplitofsemiorthog}) that, under some regularity and genericity hypothesis on $\mf{m}$, one has a direct product decomposition: 
\begin{equation}{\label{eqn: DirectSumDecomposition}}
 \DULA(\Bun_{G},\ol{\mathbb{F}}_{\ell})_{\phi_{\mf{m}}}\simeq \prod_{b \in B(G)_{\mathrm{un}}} \Dadm(J_{b}(\mathbb{Q}_{p}),\ol{\mathbb{F}}_{\ell})_{\phi_{\mf{m}}}. 
\end{equation}
More precisely, we show that the $!$ and $*$ push-forwards with respect to the inclusion of HN-strata agree on this subcategory, and so the excision semi-orthogonal decomposition splits on $\DULA(\Bun_{G},\ol{\mathbb{F}}_{\ell})_{\phi_{\mf{m}}}$. This decomposition is a refinement of the fact mentioned above that only the shtuka spaces corresponding to the unramified elements $b \in B(G,\mu)_{\mathrm{un}}$ can contribute to the generic localization of the cohomology of the Shimura variety. The desired control of the shtuka spaces is now in turn encoded in understanding how Hecke operators interact with a perverse $t$-structure on $\Bun_{G}$ after restricting to the localized category $\DULA(\Bun_{G},\ol{\mathbb{F}}_{\ell})_{\phi_{\mf{m}}}$. 

We recall $\D(\Bun_{G},\ol{\mathbb{F}}_{\ell})$ has an action by Hecke operators. In particular, for each geometric dominant cocharacter $\mu$, we have a correspondence 
\[ \begin{tikzcd}
& & \arrow[dl,"h_{\mu}^{\leftarrow}"] \Hck_{G,\leq \mu} \arrow[dr,"h_{\mu}^{\rightarrow}"] & & \\
& \Bun_{G} & & \Bun_{G} \times \Spd(C),  & 
\end{tikzcd} \]
where $\Hck_{G, \leq \mu}$ is the stack parametrizing modifications $\mathcal{E}_{1} \dashrightarrow \mathcal{E}_{2}$ of a pair of $G$-bundles with meromorphy bounded by $\mu$ at the closed Cartier divisor defined by the fixed untilt given by $C$, and $h_{\mu}^{\ra}$ (resp. $h_{\mu}^{\la}$) remembers $\mathcal{E}_{2}$ (resp. $\mathcal{E}_{1}$). We define
\[ T_{\mu}: \D(\Bun_{G},\ol{\mathbb{F}}_{\ell}) \ra \D(\Bun_{G},\ol{\mathbb{F}}_{\ell})^{BW_{E_{\mu}}} \]
\[ A \mapsto h_{\mu*}^{\ra}(h_{\mu}^{\la*}(A) \otimes^{\mathbb{L}} \mathcal{S}_{\mu}), \]
where $E_{\mu}$ denotes the reflex field of $\mu$ and $\mathcal{S}_{\mu}$ is the sheaf on $\Hck_{G,\leq \mu}$ attached to the highest weight tilting module $\mathcal{T}_{\mu} \in \Rep_{\ol{\mathbb{F}}_{\ell}}(\hat{G})$ of highest weight $\mu$ by geometric Satake (see Remark \ref{rem: HeckeOperatorsWellDefined} for an explanation of why this correspondence gives a functor of the claimed shape). The action of Hecke operators commutes with the action of excursion operators and therefore the action of the spectral Bernstein center. Moreover, it preserves the subcategory of ULA objects. It follows that we have an induced map
\[ T_{\mu}: \DULA(\Bun_{G},\ol{\mathbb{F}}_{\ell})_{\phi_{\mf{m}}} \ra \DULA(\Bun_{G},\ol{\mathbb{F}}_{\ell})^{BW_{E_{\mu}}}_{\phi_{\mf{m}}} \] 
on the localized category (see Lemma \ref{lemma: localization map properties} (2)).  

We are almost ready to state the result on Hecke operators we will need. To do this, we recall that $\D(\Bun_{G},\ol{\mathbb{F}}_{\ell})$ has a natural perverse $t$-structure, which can be defined as follows. The $v$-stack $\Bun_{G}$ is cohomologically smooth of $\ell$-dimension $0$. Moreover, it has a locally closed HN-stratification $j_{b}: \Bun_{G}^{b} \hookrightarrow \Bun_{G}$, where each of the strata $\Bun_{G}^{b}$ is cohomologically smooth of $\ell$-dimension $-d_{b}$, where $d_b:=\langle 2\rho_G,\nu_b\rangle$, and, for all $b \in B(G)$, we have an equivalence $\D(\Bun_{G}^{b},\ol{\bb{F}}_{\ell}) \simeq \D(J_{b}(\bb{Q}_{p}),\ol{\bb{F}}_{\ell})$ of sheaves on the strata with the derived category of smooth $J_{b}(\bb{Q}_{p})$-representations. Therefore, it is natural to define a perverse $t$-structure $\pD^{\geq 0}(\Bun_{G},\ol{\mathbb{F}}_{\ell})$ (resp. $\pD^{\leq 0}(\Bun_{G},\ol{\mathbb{F}}_{\ell})$) on $\D(\Bun_{G},\ol{\mathbb{F}}_{\ell})$ given by insisting that the $!$ (resp. $*$) restriction to $\Bun_{G}^{b}$ is a complex of $J_{b}(\bb{Q}_{p})$-representations concentrated in degrees $\geq \langle 2\rho_{G},\nu_{b} \rangle$ (resp. $\leq \langle 2\rho_{G}, \nu_{b} \rangle$). The key result that follows from the work of \cite{Ham2} and various compatibility results for the Fargues-Scholze correspondence is as follows.
\begin{theorem}{(Corollary \ref{cor: appliedperversetexactness})}{\label{thm: appliedperversetexactnessintro}}
Let $\mu$ be a geometric dominant cocharacter and suppose that $G$, $p$, and $\ell$ satisfy Assumption \ref{assumption: keygroupassumption}. Let $\mf{m}$ be a  generic maximal ideal with associated semisimple $L$-parameter $\phi_{\mf{m}}$. The restriction of the Hecke operator
\[ j_{1}^{*}T_{\mu}: \DULA(\Bun_{G},\ol{\mathbb{F}}_{\ell})_{\phi_{\mf{m}}} \ra \Dadm(G(\mathbb{Q}_{p}),\ol{\mathbb{F}}_{\ell})^{BW_{E_{\mu}}}_{\phi_{\mf{m}}} \]
to the neutral strata $j_{1}: [\ast/\underline{G(\bb{Q}_{p})}] \simeq \Bun_{G}^{1} \hookrightarrow \Bun_{G}$ is perverse $t$-exact (in fact much more is true (see Corollary \ref{cor: appliedperversetexactness})), where the perverse $t$-structure on $\D(\Bun_{G}^{1},\ol{\mathbb{F}}_{\ell}) \simeq \D(G(\mathbb{Q}_{p}),\ol{\mathbb{F}}_{\ell})$ coincides with the usual $t$-structure. 
\end{theorem}
\begin{remark}{\label{rem: AdditionalconditionsDiscussed}}
Assuming the Fargues-Scholze correspondence for $G$ behaves as expected with rational coefficients (Assumption \ref{compatibility}), the analysis in \cite{Ham2} allows one to verify this for \emph{any} $\phi$ induced from a generic toral parameter $\phi_{T}$ which is moreover regular (see Theorem \ref{thm: generalperverseexact}). In fact, this holds even if $\phi$ is not necessarily unramified or the groups in table (\ref{constrainttable}) are unramified (However, the precise assumptions on $p$ and $\ell$ will be more difficult to state if $G$ is ramified). We show that, under Assumption~\ref{assumption: unramifiedness}, if $\phi$ is unramified this additional regularity assumption is actually not necessary for this perverse $t$-exactness statement (see Theorem \ref{thm:unramified-perverse-without-regularity}) giving the above result. In the body of the text, we also discuss our results in cases where $\phi$ is allowed be to be ramified; however, here the regularity can not always be removed (see Remark \ref{rem: regularityishownecessary} and Assumption \ref{assumption: generalparameters}).

We also formulate a more axiomatic approach to our global applications to Shimura varieties in Theorem \ref{thm: axiomatictheorem} under the required Assumption \ref{compatibility} and generic regularity, as well as discuss analogous global applications the cohomology of Shimura varieties at parahoric level (Theorem \ref{thm:general-p-level}). 
\end{remark}
These local torsion vanishing results would allow us to prove Conjecture \ref{conj: torsionvanishing} in several new cases if one could get control over the Igusa varieties $\Ig^{b}$. In Koshikawa's argument, this is done by using a semi-perversity result proven by Caraiani-Scholze \cite[Theorem~4.6.1]{CS2}, which was further generalized in work of Santos \cite{San}. Roughly speaking, we want to show that $R\Gamma(\Ig^{b},\ol{\mathbb{F}}_{\ell})$ is concentrated in degrees $\geq d_{b}$ (however note that this is not true), so that the complex of $J_{b}(\mathbb{Q}_{p})$-representations $R\Gamma(\Ig^{b},\ol{\mathbb{F}}_{\ell})$ defines the stalk of a semi-perverse sheaf on $\Bun_{G}$ at $b \in B(G)$, to which we can apply the previous result. In the case that the Shimura varieties $\mathrm{Sh}(\mathbf{G},X)_{K}$ are compact, there is a simpler way of seeing this. In particular, $\Ig^{b}$ is known to be a perfect affine scheme in this case, and so the desired semi-perversity just follows by applying Artin vanishing and then using Poincar\'e duality on the global Shimura variety. It turns out that this style of argument can be made to work even in the non-compact case. In \cite{CS1,CS2,Ko,San}, the non-compactly supported cohomology $R\Gamma(\mathrm{Sh}(\mathbf{G},X)_{K},\ol{\mathbb{F}}_{\ell})_{\mf{m}}$ is studied together with its filtration involving $R\Gamma(\Ig^{b},\ol{\mathbb{F}}_{\ell})$ coming from Mantovan's formula, and shown to be concentrated in degrees $\geq d$. However, one could also study the compactly supported cohomology $R\Gamma_{c}(\mathrm{Sh}(\mathbf{G},X)_{K},\ol{\mathbb{F}}_{\ell})_{\mf{m}}$ and show that it is concentrated in degrees $\leq d$, \`a la Poincar\'e duality. To do this, we recall \cite[Section~3.3]{CS2} that, in the non-compact case, the perfect scheme $\Ig^{b}$ is not affine, but it admits a partial minimal compactification $g_{b}: \Ig^{b} \hookrightarrow \Ig^{b,*}$ which is affine, as proven in this more general setting of PEL type AC by Santos \cite{San}. We define 
\[ V_{b} := R\Gamma_{c-\partial}(\Ig^{b},\ol{\mathbb{F}}_{\ell}) := R\Gamma(\Ig^{b,*},g_{b!}(\ol{\mathbb{F}}_{\ell})) \in \Dadm(J_{b}(\bb{Q}_{p}),\ol{\bb{F}}_{\ell})  \]
the partially compactly supported cohomology, which is supported in degrees $\leq d_{b}$ by Artin-vanishing (Proposition \ref{prop: artvanish}). Now, for $K \subset \mathbf{G}(\mathbb{A}_{f})$ a sufficiently small open compact, we define $\mathcal{S}(\mathbf{G},X)_{K} := (\Sh(\mathbf{G},X)_{K} \otimes_{E} E_{\mf{p}})^{\mathrm{ad}}$ to be the adic space over $\Spa(E_{\mf{p}})$ attached to the Shimura variety. We can define the infinite level perfectoid Shimura varieties $\mathcal{S}(\mathbf{G},X)_{K^{p}}$ by taking the inverse limit of $\mathcal{S}(\mathbf{G},X)_{K^{p}K_{p}}$ as $K_{p} \ra \{1\}$ in the category of diamonds. The base-change $\mathcal{S}(\mathbf{G},X)_{K^{p},C}$ is representable by a perfectoid space if $(\mathbf{G},X)$ is of pre-abelian type, and in general it is a diamond. By the results of \cite{Schtor,Han1,PR23}, we have a Hodge-Tate period map 
\[ \pi_{\mathrm{HT}}: [\mathcal{S}(\mathbf{G},X)_{K^{p},C}/\ul{G(\mathbb{Q}_{p})}] \ra [\mathcal{F}\ell_{G,\mu^{-1}}^{\Diamond}/\ul{G(\mathbb{Q}_{p})}] \]
recording the Hodge-Tate filtration on the abelian varieties with additional structure that $\mathcal{S}(\mathbf{G},X)_{K^{p},C}$ parametrizes. Here $\mathcal{F}\ell_{G,\mu^{-1}} := (G_{C}/P_{\mu^{-1}})^{\mathrm{ad}}$ is the adic flag variety attached to the parabolic in $G_{C}$ given by a dominant inverse of $\mu$ and the dynamical method (i.e., we look at the set $g \in G_{C}$ such that $\lim_{t \ra 0} \mathrm{ad}(\mu^{-1}(t))g$ exists, see also Remark \ref{rem: SignConvention} for a discussion of the sign conventions we use), $\mathcal{F}\ell_{G,\mu^{-1}}^{\Diamond}$ denotes the associated diamond, and we have passed to the associated $v$-stack quotients by the $G(\bb{Q}_{p})$-action. We recall that the flag variety $[\mathcal{F}\ell_{G,\mu^{-1}}^{\Diamond}/\ul{G(\mathbb{Q}_{p})}]$ admits a locally closed stratification $i_{b}: [\mathcal{F}\ell_{G,\mu^{-1}}^{b,\Diamond}/\ul{G(\mathbb{Q}_{p})}] \hookrightarrow [\mathcal{F}\ell_{G,\mu^{-1}}^{\Diamond}/\ul{G(\mathbb{Q}_{p})}]$ indexed by $b \in B(G,\mu)$, given by pulling the HN-stratification along a natural map $h^{\la}: [\mathcal{F}\ell_{G,\mu^{-1}}^{\Diamond}/\ul{G(\mathbb{Q}_{p})}] \ra \Bun_{G}$. We will now impose the following very mild assumption in what follows.
\begin{assumption}{\label{assump: codim}}
Write $\partial\Ig^{b,*} \subset \Ig^{b,*}$ for the closed complement of $\Ig^{b}$ in $\Ig^{b,*}$. We assume that $(\mathbf{G},X)$ is a PEL datum of type AC such that, for all $b \in B(G,\mu)$, the perfect scheme $\partial\Ig^{b,*}$ is empty or has codimension in $\Ig^{b,*}$ greater than or equal to $2$.  
\end{assumption}
\begin{remark}
If $\mathbf{G}^{\mathrm{ad}}$ is simple then it is easy to show that this assumption will be satisfied if $\dim(\mathcal{S}(\mathbf{G},X)_{K^{p},C}) \geq 2$, by using that the boundary of the partially minimally compactified Igusa varieties can be expressed in terms of the Igusa varieties of Shimura varieties attached to Levis of $\mathbb{Q}$-rational parabolics of $\mathbf{G}$, as we will explain in \S \ref{sss: compactif}. Moreover, if $\mathcal{S}(\mathbf{G},X)_{K^{p},C}$ is compact then it is automatic that $\partial\Ig^{b,*}$ is empty. Therefore, if $\mathbf{G}^{\mathrm{ad}}$ is simple, this is excluding the cases where  $\dim(\mathcal{S}(\mathbf{G},X)_{K^{p},C}) = 1$ and $\mathcal{S}(\mathbf{G},X)_{K^{p},C}$ is non-compact. There are two possibilities: either $(\mathbf{G},X)$ is the Shimura datum attached to the modular curve, or it is the Shimura datum attached to the unitary Shimura curve (see \cite[Proposition~1.9]{Zha}). In the latter case, we have that the connected components are given by modular curves. In these cases, the results of \cite{Ko} are sufficient to prove Conjecture \ref{conj: torsionvanishing}.
\end{remark}
Now, assuming this, one can show that the stalk of $R\pi_{\mathrm{HT}!}(\ol{\mathbb{F}}_{\ell})$ at a geometric point $x: \Spa(C,C^{+}) \ra \mathcal{F}\ell_{G,\mu^{-1}}^{\Diamond}$ which lies in the adic Newton strata $\mathcal{F}\ell_{G,\mu^{-1}}^{b,\Diamond}$ is given by the complex $V_{b}$. Moreover, if we write $h_{b}^{\la}: [\mathcal{F}\ell_{G,\mu^{-1}}^{b,\Diamond}/\ul{G(\mathbb{Q}_{p})}] \ra \Bun_{G}^{b}$ for the pullback of $h^{\la}$ to the strata $\Bun_{G}^{b}$ then one can deduce that the complex $i_{b!}i_{b}^{*}R\pi_{\mathrm{HT}!}(\ol{\mathbb{F}}_{\ell})$ is isomorphic to $h^{\la*}j_{b!}(V_{b})$ under the natural equivalence $\D(\Bun_{G}^{b}) \simeq \D(J_{b}(\bb{Q}_{p}),\ol{\bb{F}}_{\ell})$. Therefore, by excision, we deduce that the complex of $G(\mathbb{Q}_{p}) \times W_{E_{\mf{p}}}$-representations
\begin{align*}
h_{*}^{\ra}R\pi_{\mathrm{HT}!}(\ol{\mathbb{F}}_{\ell}) \simeq R\Gamma_{c}(\mathcal{S}(\mathbf{G},X)_{K^{p},C},\ol{\mathbb{F}}_{\ell}) & \simeq \colim_{K_{p} \ra \{1\}} R\Gamma_{c}(\mathcal{S}(\mathbf{G},X)_{K^{p}K_{p},C},\ol{\mathbb{F}}_{\ell}) \\
& \simeq \colim_{K_{p} \ra \{1\}} R\Gamma_{c}(\mathrm{Sh}(\mathbf{G},X)_{K^{p}K_{p},C},\ol{\bb{F}}_{\ell}) 
\end{align*}
has a filtration with graded pieces isomorphic to $Rh^{\ra}_{*}h^{\la*}j_{b!}(V_{b})$ for varying $b \in B(G,\mu)$, where $h^{\ra}: [\mathcal{F}\ell_{G,\mu^{-1}}^{\Diamond}/\ul{G(\mathbb{Q}_{p})}] \ra [\Spd(C)/\ul{G(\mathbb{Q}_{p})}]$ is the structure map quotiented by $G(\mathbb{Q}_{p})$, which we note is proper. Here the second isomorphism follows since taking compactly supported cohomology respects taking limits of spaces, and the third isomorphism is a standard comparison result due to Huber \cite[Theorem~3.5.13]{Hub}.

Now, via the Bialynicki-Birula isomorphism, the flag variety $[\mathcal{F}\ell_{G,\mu^{-1}}^{\Diamond}/\ul{G(\mathbb{Q}_{p})}]$ identifies with an open substack of $\Hck_{G,\leq \mu}$ for the fixed minuscule $\mu$. In particular, under this relationship the maps $h^{\ra}_{\mu}$ and $h^{\la}_{\mu}$ identify with $h^{\ra}$ and $h^{\la}$, and therefore we can relate the graded pieces of the excision filtration to Hecke operators. We write 
\[ R\Gamma_{c}(G,b,\mu) := \colim_{K_{p} \ra \{1\}} R\Gamma_{c}(\Sht(G,b,\mu)/\ul{K_{p}},\ol{\mathbb{F}}_{\ell}(d_{b})) \]
for the complex of $G(\mathbb{Q}_{p}) \times J_{b}(\mathbb{Q}_{p}) \times W_{E_{\mf{p}}}$-modules defined by the compactly supported cohomology of this tower. Here $\ol{\mathbb{F}}_{\ell}(d_{b})$ is the sheaf with $J_{b}(\mathbb{Q}_{p})$-action defined as in \cite[Lemma~7.4]{Ko}.

We deduce the following variant of the Mantovan product formula for the compactly supported cohomology.
\begin{theorem}{\label{thm: mantprodform}}
Suppose that $(\mathbf{G},X)$ is a Shimura datum of PEL type AC satisfying Assumption \ref{assump: codim}. Then the complex $R\Gamma_{c}(\mathcal{S}(\mathbf{G},X)_{K^{p},C},\ol{\mathbb{F}}_{\ell})$ has a filtration as a complex of $G(\mathbb{Q}_{p})$-representations with graded pieces isomorphic to $j_{1}^{*}T_{\mu}j_{b!}(V_{b})[-d]$. More specifically, the graded pieces are isomorphic to 
\[ (R\Gamma_{c}(G,b,\mu) \otimes^{\mathbb{L}}_{\mathcal{H}(J_{b})} V_{b})[2d_{b}]. \]
as $G(\mathbb{Q}_{p})$-modules. 
\end{theorem}
\begin{remark}
When the Shimura variety is compact, so that $\Ig^{b,*} = \Ig^{b}$ and correspondingly that $R\Gamma_{c-\partial}(\Ig^{b},\ol{\mathbb{F}}_{\ell}) \simeq R\Gamma(\Ig^{b},\ol{\mathbb{F}}_{\ell})$, and this recovers precisely \cite[Theorem~7.1]{Ko}. 
\end{remark}
We now apply our localization functor $(-)_{\phi_{\mf{m}}}: \D(\Bun_{G}) \ra \D(\Bun_{G})_{\phi_{\mf{m}}}$ for a generic maximal ideal $\mf{m}$ to get a complex $R\Gamma_{c}(\mathcal{S}(\mathbf{G},X)_{K^{p},C},\ol{\mathbb{F}}_{\ell})_{\phi_{\mf{m}}} \in \D^{\mathrm{adm}}(G(\mathbb{Q}_{p}),\ol{\mathbb{F}}_{\ell})_{\phi_{\mf{m}}}$, which we view as a sheaf on $\Bun_{G}$ by $!$ extending along the neutral strata. Now we know, by the direct product decomposition of $\DULA(\Bun_{G},\ol{\mathbb{F}}_{\ell})_{\phi_{\mf{m}}}$ described in (\ref{eqn: DirectSumDecomposition}), that the natural map $j_{b!}(V_{b}) \ra Rj_{b*}(V_{b})$ is an isomorphism after applying $(-)_{\phi_{\mf{m}}}$. Moreover, one only has interesting contributions coming from the unramified elements $B(G,\mu)_{\mathrm{un}}$. In particular, we can deduce the following corollary of Theorem \ref{thm: mantprodform}. 
\begin{theorem}{\label{thm: appliedmantprodform}}{(Theorem \ref{thm: appliedmantprodformbody})}
Suppose $(\mathbf{G},X)$ is a PEL datum of type AC satisfying Assumption \ref{assump: codim} such that $\mathbf{G}_{\mathbb{Q}_{p}}$, $p$, and $\ell$ are as in Assumption \ref{assumption: keygroupassumption}. Let $\mf{m} \subset H_{K_{p}^{\mathrm{hs}}}$ be a generic maximal ideal. Then the complex $R\Gamma_{c}(\mathcal{S}(\mathbf{G},X)_{K^{p},C},\ol{\mathbb{F}}_{\ell})_{\phi_{\mf{m}}}$ breaks up as a direct sum 
\[ \bigoplus_{b \in B(G,\mu)_{\mathrm{un}}} (R\Gamma_{c}(\Sht(G,b,\mu)_{\infty,C},\ol{\mathbb{F}}_{\ell}(d_{b})) \otimes^{\mathbb{L}}_{\mathcal{H}(J_{b})} R\Gamma_{c-\partial}(\Ig^{b},\ol{\mathbb{F}}_{\ell}))_{\phi_{\mathfrak m}}[2d_{b}] \]
of $G(\mathbb{Q}_{p})$-modules.
\end{theorem}
\begin{remark}
As we will explain more in \S \ref{sec: comaprisonwithXiaoZhu}, in the case that the unique basic element $b \in B(G,\mu)_{\mathrm{un}}$ is unramified, the contribution of the corresponding summand to middle degree cohomology serves as a generic fiber analogue of the description of the middle degree cohomology on the special fiber of the integral model at hyperspecial level, as provided in \cite[Theorem~1.1.4]{XZ}.
\end{remark}
By Artin vanishing applied to the inverse limit defining the partial minimal compactification $\Ig^{b,*}$, we have that $R\Gamma_{c-\partial}(\Ig^{b},\ol{\mathbb{F}}_{\ell}) \in \D^{\leq d_{b}}(J_{b}(\mathbb{Q}_{p}),\ol{\mathbb{F}}_{\ell})$. Combining this with Theorem \ref{thm: appliedperversetexactnessintro} and a Hochschild-Serre style argument, it is easy to see we obtain Theorem \ref{theorem: mainthm}.
\section*{Acknowledgements}
We would like to thank Ana Caraiani, Jean-Fran\c{c}ois Dat, David Hansen, Naoki Imai, Teruhisa Koshikawa, and Chris  Skinner for helpful discussions pertaining to this work. Special thanks go to Mafalda Santos for sharing with us the results of her thesis, Peter Scholze for encouraging us to avoid working with the good reduction locus by directly describing the fibers of the Hodge-Tate period morphism, Matteo Tamiozzo for comments and corrections on an earlier draft, as well as suggestions for the arguments in \S5.2, and Mingjia Zhang for very helpful discussions and filling in several gaps in the arguments used in \S 3, as well as several comments and corrections. We would also like to thank the referees for carefully reading the paper and supplying several useful comments and corrections. This project was carried out while the second author was at the Max Planck Institute for Mathematics in Bonn and she thanks them for their hospitality and financial support.
\section*{AI Acknowledgment}
At some point in the referee process, AI tools were used to assist with checking the manuscript. They were also used in reorganizing some of the LaTeX in Section~\ref{section:verification} and \ref{s: abeliantype} and finding references for the proofs of Lemma \ref{lemma: localization map properties} (3) and Proposition \ref{prop: constituent proposition}. It was also used in testing the extent to which Lemma~\ref{lem:unramified-regular-lift} is true and partially supplied the idea of the proof. As for the corrections, it's most substantial contributions are as follows. It identified the false claim in a previous version that Hansen's generosity condition is equivalent to Langlands-Shahidi type leading to the updated Remark \ref{rem: Generous?}, as well as identified in Proposition \ref{prop: constituent proposition} that one required that $G$ is split for the initially supplied argument of the authors to go through.
\section*{Notation}
\begin{itemize}
\item Fix distinct primes $\ell \neq p$.
\item We write $\mathbb{Q}_{p}$ for the $p$-adic numbers, and $\Breve{\mathbb{Q}}_{p}$ for the completion of the maximal unramified extension with Frobenius $\sigma$.
\item We let $\ol{\mathbb{F}}_{\ell}$ denote the algebraic closure of the finite field $\mathbb{F}_{\ell}$. We fix a choice of square root of $p$ in $\ol{\mathbb{F}}_{\ell}$ and normalize all half Tate twists and square roots of the norm character with respect to this choice. In particular, this is how we normalize the geometric Satake correspondence of \cite[Chapter~VI]{FS} and in turn the Fargues-Scholze correspondence. 
\item For $L/\mathbb{Q}_{p}$ a finite extension, we write $\Breve{L} := L\Breve{\mathbb{Q}}_{p}$ for the compositum of $L$ with the maximal unramified extension and $W_{L}$ for the Weil group of $L$. 
\item We let $\mathbb{A}$ (resp. $\mathbb{A}_{f}$) denote the adeles (resp. finite adeles) over $\mathbb{Q}$.
\item A pair $(\mathbf{G},X)$ will denote a Shimura datum. We will use $E$ to denote the reflex field. For $K \subset \mathbf{G}(\mathbb{A}_{f})$ a sufficiently small open compact, we write $\mathrm{Sh}(\mathbf{G},X)_{K} \ra \Spec{E}$ for the attached Shimura variety of level $K$.
\item We fix an isomorphism $j: \ol{\mathbb{Q}}_{p} \xrightarrow{\simeq} \mathbb{C}$. Consider the embedding $\ol{\mathbb{Q}} \ra \ol{\mathbb{Q}}_{p}$ induced by $j^{-1}$; this gives a finite place $\mf{p}$ of $E$. We write $E_{\mf{p}}$ for the completion at $\mf{p}$.
\item We let $C := \hat{\overline{E}}_{\mf{p}}$ be the completed algebraic closure of $E_{\mf{p}}$. 
\item We use the symbol $G$ to always denote a connected reductive group over $\mathbb{Q}_{p}$, usually taken to be $\mathbf{G}_{\mathbb{Q}_{p}}$. We will always assume that $G$ is quasi-split with a fixed choice $T \subset B \subset G$ of maximal torus and Borel, respectively. 
\item If $G$ is unramified then we let $K_{p}^{\mathrm{hs}} \subset G(\mathbb{Q}_{p})$ be a choice of hyperspecial subgroup. We set $H_{K_{p}^{\mathrm{hs}}} := \ol{\mathbb{F}}_{\ell}[K_{p}^{\mathrm{hs}}\backslash G(\mathbb{Q}_{p})/K_{p}^{\mathrm{hs}}]$ to be the unramified Hecke algebra with $\ol{\mathbb{F}}_{\ell}$-coefficients. 
\item We let $\domcochar$ denote the set of geometric dominant cocharacters of $G$ and let $\domgamorb$ denote the set of Galois orbits, where $\Gamma := \mathrm{Gal}(\ol{\mathbb{Q}}_{p}/\mathbb{Q}_{p})$. 
\item Let $B(G) := G(\Breve{\mathbb{Q}}_{p})/(g \sim hg\sigma(h)^{-1})$ denote the Kottwitz set of $G$.
\item For $b \in B(G)$, we write $J_{b}$ for the $\sigma$-centralizer of $b$. We note that this depends on a choice of representative in the $\sigma$-conjugacy class attached to $b$, but the isomorphism class does not; as a result, we will suppress this technicality.
\item For $\mu \in \domcochar$, we let $B(G,\mu)$ be the $\mu$-admissible locus (as defined in \cite[Definition~2.3]{RV}).
\item Let $\Perf$ denote the category of affinoid perfectoid spaces in characteristic $p$ over $\ast := \Spd(\ol{\mathbb{F}}_{p})$ endowed with the $v$-topology. For a perfectoid space $S$, let $\Perf_{S}$ denote the category of affinoid perfectoid spaces over the tilt $S^{\flat}$. 
\item For $S \in \Perf$, let $X_{S}$ denote the relative schematic Fargues-Fontaine curve over $S$, as defined in the discussion proceeding \cite[Definition~3.3.2]{CS1}.

\item For $b \in B(G)$, we write $\mathcal{E}_{b}$ for the associated $G$-bundle on $X_{(F,\mathcal{O}_{F})}$ for $\Spa(F,\mathcal{O}_{F}) \in \Perf$ an algebraically closed field, as defined in \cite[Definition~1.1]{GTorseursFargues}. As with the $\sigma$-centralizer $J_{b}$, this depends on a choice of representative in $G(\Breve{\mathbb{Q}}_{p})$ of the $\sigma$-conjugacy class attached to $b$, but the isomorphism class does not, as a result we will suppress this technicality.
\item For $S \in \Perf$, we let $\mathcal{E}_{0,S}$ denote the trivial $G$-bundle on $X_{S}$ and simply write $\mathcal{E}_{0}$ when $S = \Spa(F,\mathcal{O}_{F})$ is a geometric point.
\item To a diamond or $v$-stack $X$ over $\ast$, we write $\D(X,\ol{\mathbb{F}}_{\ell})$ for the category of \'etale $\ol{\mathbb{F}}_{\ell}$-sheaves, as defined in \cite[Definition~14.13]{Ecod}. We let $\DULA(X,\ol{\mathbb{F}}_{\ell})$ denote the full subcategory of ULA sheaves over $\ast$.
\item For an Artin $v$-stack $X$ and $\Lambda \in \{\ol{\mathbb{F}}_{\ell},\ol{\mathbb{Z}}_{\ell},\ol{\mathbb{Q}}_{\ell}\}$, we write $\D_{\blacksquare}(X,\Lambda)$ for the condensed $\infty$-category of solid $\Lambda$-sheaves on $X$, and write $\Dlis(X,\Lambda) \subset \D_{\blacksquare}(X,\Lambda)$ for the full subcategory of $\Lambda$-lisse-\'etale sheaves, as defined in \cite[Chapter~VII]{FS}. 
\item If $X$ is an Artin $v$-stack (\cite[Definition~IV.1.1]{FS}) admitting a separated cohomologically smooth surjection $U \ra X$ from a locally spatial diamond $U$ such that the \'etale site has a basis with bounded $\ell$-cohomological dimension (which will always be the case for our applications) then we will regard $\D(X,\ol{\mathbb{F}}_{\ell})$ as a condensed $\infty$-category via the identification $\Dlis(X,\ol{\mathbb{F}}_{\ell}) \simeq \D(X,\ol{\mathbb{F}}_{\ell})$ when viewed as subcategories of $\D_{\blacksquare}(X,\ol{\mathbb{F}}_{\ell})$ \cite[Proposition~VII.6.6]{FS}.
\item We let $\hat{G}$ denote the Langlands dual group of $G$ with fixed splitting $(\hat{T},\hat{B},\{X_{\alpha}\})$.
\item If $F$ denotes the splitting field of $G$ then the action of $W_{\mathbb{Q}_{p}}$on $\hat{G}$ factors through $Q:= W_{\mathbb{Q}_{p}}/W_{F}$. We let $\phantom{}^{L}G := \hat{G} \rtimes Q$ denote the $L$-group.
\item For $I$ a finite index set, we let $\Rep_{\ol{\mathbb{F}}_{\ell}}(\phantom{}^{L}G^{I})$ (resp. $\Rep_{\ol{\mathbb{F}}_{\ell}}(\hat{G}^{I})$) denote the category of finite-dimensional algebraic representations of $\phantom{}^{L}G^{I}$ (resp. $\hat{G}^{I}$) over $\ol{\mathbb{F}}_{\ell}$.
\item For $\mu \in \domcochar$, we write $V_{\mu} \in \Rep_{\ol{\mathbb{F}}_{\ell}}(\hat{G})$ (resp. $\mathcal{T}_{\mu} \in \Rep_{\ol{\mathbb{F}}_{\ell}}(\hat{G})$) for the usual highest weight representation (resp. highest weight tilting module, as in \cite{Don}) of highest weight $\mu$. 
\item To any condensed $\infty$-category $\mathcal{C}$, we write $\mathcal{C}^{BW^{I}_{\mathbb{Q}_{p}}}$ for the category of objects with continuous $W_{\mathbb{Q}_{p}}^{I}$-action, as defined in \cite[Section~IX.1]{FS}.
\item For any separated $v$-stack, $X \ra \Spa(K,\mathcal{O}_{K})$ where $\Spa(K,\mathcal{O}_{K})$ is a non-archimedean field, we write $\ol{X}$ for the canonical compactification of $X$ with respect to the structure map (\cite[Proposition~18.6]{Ecod}, \cite[Theorem~5.15]{Hub}). 
\item For a reductive group $H/\mathbb{Q}_{p}$, we write $\D(H(\mathbb{Q}_{p}),\ol{\mathbb{F}}_{\ell})$ for the (left-completed) unbounded derived category of smooth $\ol{\mathbb{F}}_{\ell}$-representations. 
\end{itemize}
\section{Preliminaries on Shimura Varieties}
In this section, we will recall some facts about Shimura varieties and Igusa varieties which we will use to study the geometry of the Hodge-Tate period morphism in the next section. More specifically, in \S \ref{s: Shimuravarieties}, we start by reviewing the definition of a Shimura datum of PEL type AC that we will be interested in (\S \ref{sec: PELtypeAorC}), and then we introduce the minimal and toroidal compactifications of the Shimura varieties attached to this datum (\S \ref{sss:compactifications}). In the next section \S \ref{ss:def_igusa}, we carry out an analogous discussion for Igusa varieties, first reviewing their definition and moduli interpretation (\S \ref{ss: defIgusavarieties}), and then studying their toroidal and minimal compactifications (\S \ref{sss:Igusacusplabels}, \S \ref{sss:boundarycomponents}). This will be important in our discussion in \S \ref{subsec: HodgeTatePeriodMap} for describing the fibers of the Hodge-Tate period morphism on the open Shimura variety.
\subsection{Shimura Varieties}{\label{s: Shimuravarieties}}
We will mainly work with the following type of Shimura varieties. Our main reference here is \cite[\S1]{Lan}.
\subsubsection{PEL type AC}{\label{sec: PELtypeAorC}}
Let $(\mathcal{O}_{B} , \ast , L, \langle\cdot,\cdot\rangle,h)$ be an (integral) PEL datum, as defined in \cite[\S1.2.1]{Lan}. Note here that by tensoring $O_B$ and $L$ by $\mathbb{Q}$, we obtain the (rational) PEL datum as introduced by Kottwitz \cite[\S1-4]{Kottwitz92}. More precisely, we require that $B$ is a finite-dimensional semisimple $\mathbb{Q}$-algebra and $\ast$ is a positive $\mathbb{Q}$-linear involution of $B$. $\mathcal{O}_{B}$ is a $\ast$-stable $\mathbb{Z}$-order of $B$, $L$ is a lattice with $\mathcal{O}_{B}$-action, and $\langle \cdot , \cdot \rangle: L \times L \rightarrow \mathbb{Z}(1)$ is a non-degenerate alternating form such that $\langle bv, v'\rangle = \langle v, b^*v'\rangle$, for all $b \in \mathcal{O}_{B}$ and $v, v'\in L$. Moreover $h$ is an $\mathbb{R}$-algebra homomorphism
\[h:\mathbb{C}\rightarrow \mathrm{End}_{O_B\otimes_{\mathbb{Z}} \mathbb{R}}(L\otimes_{\mathbb{Z}}\mathbb{R}),\]
such that
\begin{enumerate}
    \item $\langle h(z)v, w\rangle = \langle v, h(\bar{z})w\rangle$ for all $v, w \in L\otimes_{\mathbb{Z}}\mathbb{R}$ and all $z\in \mathbb{C}$
    \item the symmetric bilinear form $(v, h(i)w)$ on $L\otimes_{\mathbb{Z}}\mathbb{R}$ is positive definite.
\end{enumerate}

To our integral PEL datum, we can associate the following group scheme $\mathbf{G}$ over $\mathbb{Z}$ whose $R$-points, for each $\mathbb{Z}$-algebra $R$, are given by
\begin{equation*}
    \mathbf{G}(R) :=\{(g, r) \in \mathrm{GL}_{\mathcal{O}_{B}\otimes_{\mathbb{Z}}R}(L\otimes_{\mathbb{Z}}R)\times R^\times\vert \langle gv, gw\rangle = r\langle v, w\rangle\text{ for all }v, w \in L\otimes_{\mathbb{Z}}R\}.
\end{equation*}

We now make the following assumptions. Firstly, using Morita equivalence we will restrict our discussion to the case where $B$ is a product of $\ast$-stable division algebras $B_i$ over $\mathbb{Q}$. 
Define $E_i=Z(B_i)$ and $F_i=Z(B_i)^{\ast=1}$, these are number fields, and $F_i$ is totally real. We will assume that $\mathcal{O}_B$ is the product of orders $\mathcal{O}_{B_i}$. 
The $\mathcal{O}_B$-lattice $L$ hence decomposes as $L=\oplus_i L_i$, where $L_i$ is an $\mathcal{O}_{B_i}$-lattice. Observe that over an algebraic closure $\ol{E}_i$ of $E_i$, $B_i\otimes \ol{E}_i\simeq M_{d_i}(\ol{E}_i)$, for some positive integers $d_i$.

Moreover, we assume from now on that $B$ only has simple factors of type AC. Here, we follow \cite[Definition 1.2.1.15]{Lan} classifying simple $\mathbb{Q}$-algebras with a positive involution. Observe that from \cite[Proposition 1.2.3.11]{Lan}, this implies $\mathbf{G}^{\mathrm{ad}}$ have simple factors only of type A and C. 

We also further assume the datum is unramified at $p$; namely, $E_i$ (and hence $F_i$) is unramified at $p$, we have an isomorphism
\[B_{\mathbb{Q}_p} = \prod_i\prod_{\mathfrak{p}|p}B_i\otimes_{F_i}F_{i\mathfrak{p}},\]
where $B_i\otimes_{F_i}F_{i\mathfrak{p}}\simeq \prod_{\mathfrak P|\mathfrak p}M_{d_i}(E_{i\mathfrak{P}})$, for $\mathfrak P$ a prime of $E_i$ over $\mathfrak p$.
We will moreover assume that $\mathcal{O}_{B} \otimes \mathbb{Z}_p$ is a maximal order in $B_{\mathbb{Q}_p}$, and $L$ is self-dual after localization at $p$. This can be arranged following \cite[Remark 1.3.4.8]{Lan}. Note that these conditions ensure that the group $\mathbf{G}_{\mathbb{Z}_p}$ is a smooth reductive group scheme. 

For each $\mathfrak{p}$ prime of $F_{i}$ over $p$, let $L_{i\mathfrak{p}}=L_i\otimes_{O_{F_i}}O_{F_{i\mathfrak{p}}}$, $\mathcal{O}_{B_{i\mathfrak{p}}}=\mathcal{O}_{B_{i}}\otimes_{\mathcal{O}_{F_i}} \mathcal{O}_{F_{i\mathfrak p}}$. Note that if we take
$H_{i\mathfrak{p}}$ such that for each $O_{F_{i\mathfrak p}}$-algebra $R$, 
\begin{equation*}
    H_{i\mathfrak{p}}(R) :=\{(g, r) \in \mathrm{GL}_{\mathcal{O}_{B_{i\mathfrak{p}}}\otimes_{O_{F_{i\mathfrak p}}}R}(L_{i\mathfrak{p}}\otimes_{O_{F_{i\mathfrak p}}}R)\times R^\times\vert \langle gv, gw\rangle_{O_{F_{i\mathfrak p}}} = r\langle v, w\rangle_{O_{F_{i\mathfrak p}}}\text{ for all }v, w \in L_{i\mathfrak{p}}\otimes_{O_{F_{i\mathfrak p}}}R\},
\end{equation*}
then projecting onto $r$ defines a similitude character $\nu_{i\mathfrak p}:\Res_{F_{i\mathfrak p}/\mathbb{Q}_p} H_{i\mathfrak p}\rightarrow \Res_{F_{i\mathfrak p}/\mathbb{Q}_p} \mathbb{G}_m$. Observe that
\[\mathbf{G}_{\mathbb{Q}_p}\simeq \left(\prod_{i,\mathfrak{p}}\Res_{F_{i\mathfrak p}/\mathbb{Q}_p} H_{i\mathfrak p}\right)\times_{\prod_{i,\mathfrak{p}}\Res_{F_{i\mathfrak p}/\mathbb{Q}_p}\mathbb{G}_m} \mathbb{G}_m,\]
is the fiber product with the diagonal $\mathbb{G}_{m,\mathbb{Q}_p}\subseteq \prod_{i,\mathfrak{p}}\Res_{F_{i\mathfrak p}/\mathbb{Q}_p}\mathbb{G}_m$. 

We will now briefly discuss what conditions we may need to impose on the prime $p$ so that the form of the local group $G:=\mathbf{G}_{\mathbb{Q}_p}$ satisfies the conditions in Table (\ref{constrainttable}), for each of the division algebra factors $B_i$. 

{\bf Type A:} In this case, the center $E_i$ is a CM quadratic extension of $F_i$. 
Let $r_i$ be the $\mathcal{O}_{B_i}$ rank of $L_i$, and define $n_i=r_id_i$. 
Observe that we will need to assume that the prime $p$ satisfies for every prime $\mathfrak{p}$ of $F_i$ above $p$,
\begin{enumerate}
\item $\mathfrak{p}$ is split in $E_i$; or
\item $\mathfrak p$ is inert in $E_i/F_i$, $F_{i\mathfrak{p}}=\mathbb{Q}_p$, and $n_i$ is odd; or
\item $\mathfrak p$ is inert in $E_i/F_i$,
and $n_i=2$.
\end{enumerate}
These conditions imply that $H_{i\mathfrak{p}}$ is, respectively,
\begin{enumerate}
    \item $(\GL_{n_i})\times \mathbb{G}_m$ over $F_{i\mathfrak p}$
    \item $\GU_{n_i}$ for $n_i$ odd, the quasi-split unitary similitude group for the extension $E_{i\mathfrak{p}}/\mathbb{Q}_p$.
    \item $\GU_{2,F_{i\mathfrak p}}$ the quasi-split unitary similitude group for the quadratic extension $E_{i\mathfrak{p}}/F_{i\mathfrak p}$.
\end{enumerate}
For instance, when $B$ is simple, $n_i=2$ and every prime
$\mathfrak p\mid p$ of $F_i$ is inert in $E_i$,
the local group is the similitude group
$G(\mathrm U_{2,\mathbf L})$ we will introduce
in Section~\ref{section:verification}, where
$\mathbf L=(F_{i\mathfrak p})_{\mathfrak p\mid p}$.

{\bf Type C:} Let $r_i$ be the $\mathcal{O}_{B_i}$ rank of $L_i$, and define $n_i=\frac{1}{2}r_id_i$. Observe that $H_{i\mathfrak{p}}$ is $\GSp_{2n_i}$ over $F_{i\mathfrak p}$. In this case, we will need that $n_i$ is either $1$ or $2$ to satisfy the conditions in Table (\ref{constrainttable}).

To each integral PEL datum we can associate a moduli space of abelian varieties with extra structures. Let $K^p\subset \mathbf{G}(\mathbb{A}_f^p)$ be an open compact subgroup, and let $K=K_p^{\mathrm{hs}}K^p$, where we recall $K_{p}^{\mathrm{hs}} \subset G(\bb{Q}_{p})$ denotes a hyperspecial level. To any PEL datum, we let $\mathcal{M}_K$ over $\mathcal{O}_{E_\mathfrak{p}}$ be the scheme which represents the functor that associates to each scheme $S$ over $\mathcal{O}_{E_\mathfrak{p}}$ the set of isomorphism classes of tuples $(A,\lambda,\iota,\eta^p)$ consisting of
\begin{enumerate}
    \item An abelian scheme $A/S$ of dimension $\sum_i n_id_i[F_i:\mathbb{Q}]$ up to prime to $p$-isogeny,
    \item A prime-to-$p$ polarization $\lambda:A\rightarrow A^\vee$,
    \item An embedding $ \iota: \mathcal{O}_{B}\otimes \mathbb{Z}_{(p)}\hookrightarrow \mathrm{End}(A)\otimes_{\mathbb{Z}}\mathbb{Z}_{(p)}$ of $\mathbb{Z}_{(p)}$-algebras such that
    \begin{equation*}
        \lambda\circ \iota(b^*)=\iota(b)^\vee\circ\lambda,
    \end{equation*}
    \item A section $\eta^p\in \Gamma(S,\mathrm{Isom}_B(\ul{L\otimes_{\mathbb{Z}}\mathbb{A}_f^p},\ul{H_1(A,\mathbb{A}_f^p})/\ul{K^p}))$, where $\mathrm{Isom}_B(\ul{L\otimes_{\mathbb{Z}}\mathbb{A}_f^p},\ul{H_1(A,\mathbb{A}_f^p}))$ is the $\mathbf{G}(\mathbb{A}_f^p)$ pro-\'{e}tale torsor of $O_B\otimes\mathbb{A}_f^p$-equivariant isomorphisms that map $\langle, \rangle$ to a $\ul{\mathbb{A}_f^{p\times}}$-multiple of the pairing on $H_1(A,\mathbb{A}_f^p)$ defined by the Weil pairing,
\end{enumerate}
such that this satisfies the Kottwitz determinant condition that $\det(b |\mathrm{Lie}(A)) = \det(b | V^{-1,0})$ as polynomial functions on $\mathcal{O}_{B}$, where $V=L\otimes\mathbb{Q}$ and $V_{\mathbb{C}}=V^{-1,0}\oplus V^{0,-1}$ is the Hodge decomposition. Here, the polynomial functions on $\mathcal{O}_B$ for the action of $\mathcal{O}_B$ on a finite locally free $\mathcal{O}_S$-module $M$ for a scheme $S$ are defined in \cite[\S5]{Kottwitz92}.

Following \cite[\S8]{Kottwitz92}, we know that the generic fiber of $\mathcal{M}_K$ is a finite union of Shimura varieties for groups $\mathbf{G}'$ with $\mathbf{G}(\mathbb{A})\simeq \mathbf{G}'(\mathbb{A})$. In particular, the canonical integral model of the Shimura variety $\mathscr{S}(\mathbf{G},X)_K$ attached to $\mathbf{G}$ and the Hermitian space $X$ consisting of $\mathbf{G}(\mathbb{R})$-conjugates of $h$ is a union of connected components of $\mathcal{M}_K$. As such, we can still associate to $S$-points of $\mathscr{S}(\mathbf{G},X)_K$ an abelian scheme over $S$ with extra structures.

\subsubsection{$p^m$-level} We will also consider normal integral models of Shimura varieties with deeper level at $p$, defined as follows. Assume from now on that the level structure $K$ is a principal congruence subgroup for some $N\geq 3$; namely, we have that
\begin{equation*}
    K = K(N ) = \{g\in \mathbf{G}(\hat{\mathbb{Z}}) | g \equiv 1 \pmod{N} \}.
\end{equation*}
We consider Shimura varieties of level $K'=K(p^nN)$ where $p\nmid N$. Write $K'=K'_pK'^p$. In this case, we consider $M_{K'}$ the functor which assigns to every scheme $S$ over $E$ the set of isomorphism classes of tuples $(A,\lambda,\iota,\eta)$, where
\begin{enumerate}
    \item An abelian scheme $A/S$ of dimension $\sum_i n_id_i[F_i:\mathbb{Q}]$ up to prime-to-$p$ isogeny
    \item A prime-to-$p$ polarization $\lambda:A\rightarrow A^\vee$, satisfying
    \[
\lambda\circ\iota(b^*)=\iota(b)^\vee\circ\lambda
\]
    \item An embedding $ \iota: \mathcal{O}_{B}\otimes \mathbb{Z}_{(p)}\hookrightarrow \mathrm{End}(A)\otimes_{\mathbb{Z}}\mathbb{Z}_{(p)}$ of $\mathbb{Z}_{(p)}$-algebras such that
    \begin{equation*}
        \lambda\circ \iota(b^*)=\iota(b)^\vee\circ\lambda,
    \end{equation*}
    \item A section $\eta^p\in \Gamma(S,\mathrm{Isom}_B(\ul{L\otimes_{\mathbb{Z}}\mathbb{A}_f^p},\ul{H_1(A,\mathbb{A}_f^p})/\ul{K^p}))$,
    \item A section $\eta_p\in\Gamma(
S,\operatorname{Isom}_{\mathcal O_B\otimes\mathbb Z_p}(\ul{L\otimes\mathbb{Z}_p},\ul{H_1(A,\mathbb{Z}_p)}/\ul{K'_p}
)$,
\end{enumerate}
such that this satisfies the Kottwitz determinant condition that $\det(b |\mathrm{Lie}(A)) = \det(b | V^{-1,0})$. As in the case of the integral model at hyperspecial level, this moduli problem is representable for $N>3$ by a scheme over $E$. We can thus consider in this situation the composition
\[M_{K'}\rightarrow M_K\rightarrow \mathcal{M}_{K},\]
and following \cite{Lan2016a} we let $\mathcal{M}_{K'}$ denote the normalization of $\mathcal{M}_K$ in $M_{K'}$, which by definition is a flat normal scheme over $\mathcal{O}_{E_\mathfrak{p}}$. Similar to before, $\mathcal{M}_{K'}$ is a normal integral model for a finite union of Shimura varieties $\Sh(\mathbf{G},X)_{K'}$ for groups $\mathbf{G}'$ with $\mathbf{G}(\mathbb{A})\simeq \mathbf{G}'(\mathbb{A})$, and we thus obtain a normal integral model $\mathscr{S}(\mathbf{G},X)_{K'}$ for $\Sh(\mathbf{G},X)_{K'}$. 

\subsubsection{Compactifications}
\label{sss:compactifications}
We will now recall some constructions from the theory of toroidal compactifications of PEL type Shimura varieties from \cite{Lan,Lan2016a}.

We first recall the definition of a split, symplectic and admissible filtration from \cite[\S5.2.1]{Lan}. Let $R$ be a commutative ring. 

\begin{definition}
    A split, symplectic and admissible filtration on $L\otimes_{\mathbb{Z}}R$ is a two-step filtration on $L\otimes_{\mathbb{Z}}R$  by $(\mathcal{O}_{B} \otimes_{\mathbb{Z}}R)$-submodules, i.e., there is a flag
\begin{equation*}
   0 = Z_{-3} \subset Z_{-2}\subset Z_{-1}\subset L\otimes_{\mathbb{Z}}R=Z_0,
\end{equation*}
such that if we put $\Gr^{Z}_{-i}= Z_{-i}/Z_{-i-1}$ for $0 \leq i \leq 2$, and $\Gr^Z = \oplus_{0\leq i\leq 2} \Gr^Z_{-i}$, we have
\begin{enumerate}
    \item $\Gr^Z_{-i}$ is isomorphic to $M \otimes_{\mathbb{Z}}R$ for some $\mathcal{O}_{B}$-lattice $M$
    \item There is some isomorphism of $\mathcal{O}_{B}\otimes R$-modules
    \begin{equation*}
        L\otimes_{\mathbb{Z}}R\simeq \Gr^Z
    \end{equation*}
    \item  $Z_{-2}$ and $Z_{-1}$ are annihilators of each other under the pairing $\langle,\rangle$ induced from $L$.
\end{enumerate}
\end{definition}

Let $R=\hat{\mathbb{Z}}$ and suppose that we have a split symplectic admissible filtration $Z=Z_{\bullet}$ as above. The following is \cite[Definition 5.4.1.3]{Lan}: 
\begin{definition}
    A torus argument $\Phi$ for $Z$ is a tuple $\Phi = (X, Y, \phi, \varphi_{-2}, \varphi_0)$, where 
\begin{enumerate}
    \item $X$ and $Y$ are $\mathcal{O}_{B}$-lattices of the same $B$-multi-rank, and $\phi: Y \hookrightarrow X$ is an $\mathcal{O}_{B}$-linear embedding
    \item We have isomorphisms $\varphi_{-2} : \Gr^Z_{-2}\simeq \mathrm{Hom}_R(X\otimes_\mathbb{Z}R, R(1))$ and $\varphi_0 : \Gr^Z_0\simeq Y \otimes_\mathbb{Z} R$ such that the pairing $\langle,\rangle_{20} : \Gr^{Z}_{-2} \times \Gr^{Z}_0 \rightarrow R(1)$ is the pullback under these isomorphisms of the pairing
    \begin{equation*}
        \langle\cdot,\cdot\rangle^\phi: \mathrm{Hom}_R(X\otimes R, R(1))\times (Y\otimes R)\xrightarrow{\mathrm{id}\times \phi} \mathrm{Hom}_R(X\otimes R, R(1))\times (X\otimes R)\rightarrow R(1),
    \end{equation*}
    where the last arrow is the tautological pairing. 
\end{enumerate}
\end{definition}
In the above definition, we use the notion of $B$-multi-rank as defined in \cite[Definition 1.2.1.20]{Lan}.
\begin{definition}
We say that two torus arguments $\Phi = (X, Y, \phi, \varphi_{-2}, \varphi_0)$ and $\Phi' = (X', Y', \phi', \varphi'_{-2}, \varphi'_0)$ for $Z$ are equivalent if and only if there exists a pair of isomorphisms $(\gamma_X : X' \xrightarrow{\sim} X, \gamma_Y : Y\xrightarrow{\sim} Y')$ such that:
\begin{enumerate}
    \item $\phi = \gamma_X\phi'\gamma_Y$
    \item $\varphi_{-2}=\varphi_{-2}'\phantom{}^{t}\gamma_X$
    \item $\varphi_0=\varphi_0'\gamma_Y$.
\end{enumerate}
\end{definition}
\begin{definition}
\label{defn:cusp_label}
    A Shimura cusp label is an equivalence class of pairs $(Z,\Phi)$, where $Z$ is a split symplectic admissible filtration on $L\otimes_{\mathbb{Z}}\hat{\mathbb{Z}}$, and $\Phi$ is a torus argument for $Z$. 
\end{definition}
There is an action of $\mathbf{G}(\mathbb{A}_f)$ on the set of possible Shimura cusp labels, defined as follows. We can consider rational Shimura cusp labels $(Z\otimes\mathbb{Q},\Phi\otimes\mathbb{Q})$, where $Z\otimes\mathbb{Q}$ is the filtration on $L\otimes \mathbb{A}_f$ obtained by tensoring with $\mathbb{Q}$ over $\mathbb{Z}$, and $\Phi\otimes\mathbb{Q}$ is similarly the tuple obtained by tensoring $X,Y,\phi,\varphi_{-2},\varphi_0$ with $\mathbb{Q}$. Every $g\in \mathbf{G}(\mathbb{A}_f)$ defines an automorphism of $L\otimes \mathbb{A}_f$, so given a filtration $Z\otimes\mathbb{Q}$, we let $g(Z\otimes\mathbb{Q})$ be the filtration obtained as the image under $g$, and $g(\Phi\otimes\mathbb{Q})$ the torus argument obtained by setting $g\cdot\varphi_{-2}=\nu(g)\varphi_{-2}g^{-1}$, $g\cdot\varphi_0=\varphi_0g^{-1}$, with the isomorphisms $\Gr_{-i}^{g(Z\otimes\mathbb{Q})}\simeq \Gr_{-i}^{Z\otimes\mathbb{Q}}$ induced by $g^{-1}$.  Now, intersecting $L\otimes_{\mathbb{Z}}\hat{\mathbb{Z}}\cap g(Z_j\otimes\mathbb{Q})$ gives us a split symplectic filtration denoted by $Z^g$. From \cite[Lemma 5.4.3.3]{Lan}, we have lattices $Y^g$ and $X^g$ in $Y\otimes\mathbb{Q}$ and $X\otimes\mathbb{Q}$, with an embedding $\phi_g=\phi\otimes\mathbb{Q}|_{Y^g}$ such that 
\[Y^g\otimes\widehat{\mathbb Z}=g\cdot\varphi_0\bigl(\operatorname{Gr}_0^{Z^g}\bigr)\]
\[\operatorname{Hom}_{\widehat{\mathbb Z}}
(X^g\otimes\widehat{\mathbb Z},\widehat{\mathbb Z}(1))=g\cdot\varphi_{-2}\bigl(\operatorname{Gr}_{-2}^{Z^g}\bigr).\] 
We thus define $g\cdot(Z,\Phi)$ to be $(Z^g,X^g,Y^g,\phi_g,g\cdot\varphi_{-2},g\cdot\varphi_0)$.

We define a Shimura cusp label of level $K(N)$ to be a $K(N)$-orbit of Shimura cusp labels. For the rest of this section, if there is no risk of confusion with the Igusa cusp labels we will define later, we will refer to these Shimura cusp labels as simply cusp labels. 
\begin{remark}
    \label{rem:cusplabels}
    Note that this is the generalization of the cusp labels $(Z,X)$ considered in \cite[\S2.5.2]{CS2}, as for the PEL type A Shimura data they considered, the assumption of principal polarization means we can set $X=Y$, and the torus argument $\Phi$ is determined by the $\mathcal{O}_F$-isomorphism in part (2) of loc. cit. This definition is not the same as \cite[Definition~5.4.1.9]{Lan} for level $N$ (Lan uses $n$ as notation for the level). However, if we take a $K(N)$-orbit of cusps labels as defined here, we can look at the reduction mod $N$ of $\varphi_{-2}$ and $\varphi_{0}$ for any element in the orbit, and similarly taking the filtration mod $N$ gives a filtration on $L/NL$, and this recovers Lan's definition.
\end{remark}

To each cusp label $(Z,\Phi)$, we can associate a split torus $E_\Phi$ over $\mathbb{Z}$, as constructed by Lan in \cite[\S6.2.2]{Lan}. Let $S_{\Phi}=\mathbb{X}^*(E_\Phi)$, and note that from the construction it is a quotient of $Y\otimes_{\mathbb{Z}}X$. Let $S_\Phi^\vee:= \mathrm{Hom}_\mathbb{Z}(S_\Phi, \mathbb{Z})$ be the $\mathbb{Z}$-dual of $S_\Phi$, and let $(S_\Phi)^\vee_\mathbb{R}:=S^\vee_{\Phi}\otimes_{\mathbb{Z}}\mathbb{R}$. The $\mathbb{R}$-vector space $(S_\Phi)^\vee_\mathbb{R}$ is isomorphic to the space of Hermitian pairings $| \cdot , \cdot|:(Y\otimes \mathbb{R}) \times (Y\otimes \mathbb{R}) \rightarrow \mathcal{O}_{B}\otimes \mathbb{R}$ with isomorphism constructed by sending a Hermitian pairing $| \cdot , \cdot |$ to the function $y \otimes \phi(y') \mapsto \mathrm{Tr}_{B/\mathbb{Q}}(|y, y'|)$ in $\mathrm{Hom}_\mathbb{Z}(S_\Phi, \mathbb{R})$ for any $y,y'\in Y$ (c.f. \cite[\S6.2.5]{Lan}).

Thus, we have an $\mathbb{R}$-vector space $(S_{\Phi})_{\mathbb{R}}^\vee$ of Hermitian pairings, and we define $P_{\Phi}$ to be the subset of $(S_\Phi)_{\mathbb{R}}^\vee$ corresponding to positive semi-definite Hermitian pairings with admissible radicals (see \cite[Definition~6.2.5.4]{Lan} and subsequent discussion for the precise definition of admissible radical). Let $P_{\Phi}^+\subset P_{\Phi}$ be the positive definition forms. We regard $P_\Phi$ as a cone in
$(S_\Phi)^\vee_{\mathbb R}$ and choose a
$\Gamma_Z$-admissible rational polyhedral cone
decomposition of $P_\Phi$. Moreover, to every cusp $(Z,\Phi)$ we can also associate a stabilizer group $\Gamma_{Z}$ defined as follows. $\Gamma_Z \subset \GL_{O_B}(X) \times \GL_{O_B}(Y)$ is the group of pairs of isomorphisms $\gamma_X$ and $\gamma_Y$ of $X$ and $Y$ respectively satisfying $\phi = \gamma_X\circ \phi\circ \gamma_Y^{-1}$, and mod $N$ we have $\varphi_{-2,N} = \gamma_{X,N}^\vee \circ\varphi_{-2,N}$ and $\varphi_{0,N} = \gamma_{Y,N}\circ \varphi_{0,N}$. Finally, let $\Sigma_{Z}$ be a $\Gamma_{Z}$-admissible rational polyhedral cone decomposition of $P_{\Phi}$, as in \cite[Definition~6.1.1.14]{Lan}. Let $\Sigma_{Z}^+\subset \Sigma_Z$ denote the subset of cones $\sigma$ whose interior is entirely contained in $P_{\Phi}^+$; this is a $\Gamma_Z$-stable subset which forms a decomposition of $P_{\Phi}^+$.

From now on, we will assume that we have fixed a compatible choice of admissible smooth rational polyhedral cone decomposition data (rpcd) $\Sigma$ for $K$; namely, we have
\begin{enumerate}
    \item A complete set of representatives $(Z,\Phi)$ of cusp labels at level $K$,
    \item A $\Gamma_Z$ -admissible smooth rational polyhedral cone decomposition $\Sigma_{\Phi}$ for each cusp $(Z,\Phi)$ so that the cone decompositions are pairwise compatible.
\end{enumerate}
The precise definition and proof of existence of such smooth admissible rpcd is \cite[\S6.3.3.2,\S6.6.3.3]{Lan}. Associated to this admissible smooth rpcd, we have a toroidal compactification of $\mathscr{S}(\mathbf{G},X)_K$ for $K=K(N)$ as constructed by Lan \cite[Theorem 6.4.1.1]{Lan} (when $p\nmid N$), and \cite[\S7-10]{Lan2016a} (when $p\mid N$).

The toroidal compactification may be described as follows. To a cusp label $(Z,\Phi)$ of level $K$, there is a smaller-dimensional Shimura variety whose precise definition may be found in \cite[\S6.2]{Lan} $(p\nmid N)$ and \cite[\S8]{Lan2016a} ($p\mid N$). Then, we have the following:

\begin{theorem}
\label{thm:shimuraboundary}
To each compatible choice $\Sigma = \{\Sigma_\Phi\}$ of admissible smooth rational polyhedral cone decomposition data, there is an associated flat proper normal algebraic scheme $\mathscr{S}(\mathbf{G},X)_K^\tor$ over $\mathcal{O}_{E,\mathfrak{p}}$ (which is even smooth when $p\nmid N$) containing $\mathscr{S}(\mathbf{G},X)_K$ as an open dense subscheme, together with a semiabelian family $\mathcal{A}$ over $\mathscr{S}(\mathbf{G},X)_K^\tor$. Moreover, we have the following
\begin{enumerate}
    \item We have a decomposition into reduced locally closed subvarieties
    \begin{equation*}
    \mathscr{S}(\mathbf{G},X)_K^\tor=\bigsqcup_{[Z]}\mathscr{S}(\mathbf{G},X)^\tor_{K,Z},
    \end{equation*}
    where $Z=(Z,\Phi)$ is a cusp label of level $K$, and where each $\mathscr{S}(\mathbf{G},X)^\tor_{K,Z}$ is flat over $O_{E_\mathfrak{p}}$.
    \item 
    There exists a scheme 
    \[C_Z \rightarrow \mathscr{S}_Z\]
    (which when $p\nmid N$ is an abelian scheme torsor). Moreover, there exists a torsor under the torus $E_\Phi$, $\Xi_Z \rightarrow C_Z$, and a relative toroidal embedding $\Xi_Z\hookrightarrow \Xi_{Z,\Sigma_Z}$ induced by the cone $\Sigma_Z$ (see \cite[Definition 6.1.2.3]{Lan} for the definition of relative toroidal embedding), as depicted in the following diagram:
    \begin{equation*}
    \begin{tikzcd}
    \Xi_Z \arrow[r] \arrow[d, hook] & C_Z \arrow[r] & \mathscr{S}_Z \\
    \Xi_{Z,\Sigma_Z} \arrow[ru]                &             &  
    \end{tikzcd}
    \end{equation*}
    and $\Xi_Z$ has an action of $\Gamma_Z$ which extends to $\Xi_{Z,\Sigma_Z}$. $\Xi_{Z,\Sigma_Z}$ has a decomposition indexed by $\Sigma_Z$. We denote the union $\cup_{\sigma\in\Sigma_Z^+} \Xi_{Z,\sigma}$ by $\partial^{\mathrm{dist}}\Xi_{Z,\Sigma_Z}$; this is the distinguished boundary component of $\Xi_{Z,\Sigma_Z}$.
    
    If we let 
    \[\left(\mathscr{S}(\mathbf{G},X)^{\tor}_K\right)^\wedge_{[Z]},\]
    denote the completion of $\mathscr{S}(\mathbf{G},X)_K^\tor$ along the locally closed component indexed by $[Z]$\footnote{By
convention, by the formal completion of a scheme $X$ along a locally closed subscheme $Z$ we mean the formal completion of the open subscheme $X-(\ol{Z}-Z)$ along its closed subscheme $Z$.}, then we have an isomorphism
    \[\left(\mathscr{S}(\mathbf{G},X)^{\tor}_K\right)^\wedge_{[Z]}\simeq\mathfrak{X}_{Z,\Sigma_Z}/\Gamma_Z,\]
    where $\mathfrak{X}_{Z,\Sigma_Z}$ is the completion of $\Xi_{Z,\Sigma_Z}$ along $\partial^{\mathrm{dist}}\Xi_{Z,\Sigma_Z}$.
    \end{enumerate}
    \end{theorem}

We note that $C_Z,\mathscr{S}_Z,\Xi_Z$ parameterize certain degeneration structures on semiabelian schemes, which we will briefly sketch.

We first recall some facts about degenerations of abelian schemes, from \cite[\S3.3]{Lan} and \cite[\S2.5.1]{CS2}. Let $R$ be a  complete noetherian normal local ring with fraction field $H$. Consider a polarized abelian variety $(A,\lambda)$ over $H$ with $\mathcal{O}_{B}$-structure, and $\mathcal{A}$ a semiabelian scheme over $R$ which is a degeneration of $A$ (that is, $\mathcal{A}_{H}\simeq A$). Then this uniquely determines a short exact sequence
\begin{equation*}
    0 \rightarrow  T \rightarrow \mathcal{G} \rightarrow \mathcal{B} \rightarrow 0
\end{equation*}
where $T$ is a torus, $\mathcal{B}$ is an abelian scheme over $R$, and $\mathcal{G}$ is the Raynaud extension (see \cite[Definition~3.3.3.9]{Lan} for definition). Let $X=\mathbb{X}^*(T)$ be its character lattice, an \'etale locally constant sheaf of finite free $\mathbb{Z}$-modules on $\Spec(R)$. The lattice $X$ has an action of $\mathcal{O}_{B}$, and so does $\mathcal{B}$. Then, $\mathcal{G}$ determines, and is uniquely determined by, an $\mathcal{O}_{B}$-linear map $c:X \rightarrow \mathcal{B}^\vee$. Similarly, we can consider a degeneration $\mathcal{A}^\vee$ over $R$ of the dual $A^\vee/H$, which gives us a short exact sequence
\begin{equation*}
    0\rightarrow T^\vee\rightarrow \mathcal{G}^\vee\rightarrow \mathcal{B}^\vee\rightarrow 0,
\end{equation*}
and if we similarly let $Y=\mathbb{X}^*(T^\vee)$, the extension $\mathcal{G}^\vee$ determines, and is uniquely determined by an $\mathcal{O}_{B}$-linear map $c^\vee:Y \rightarrow \mathcal{B}$. 

Moreover, note that, as shown in \cite[\S3.4]{Lan}, the polarization on $A$ extends uniquely to a homomorphism $\lambda:\mathcal{G}\rightarrow\mathcal{G}^\vee$. The homomorphism $\lambda$ determines and is uniquely determined by the data of the polarization on $\mathcal{B}$, and an injective $\mathcal{O}_B$-linear map $\phi:Y\rightarrow X$.

Then, given a cusp label $Z$ of level $K(N)$, we can define an (integral) smaller-dimensional Shimura variety $\mathscr{S}_{Z}$ which will be a union of connected components of an integral PEL moduli space which parameterizes the data of abelian schemes $\mathcal{B}$ as described above, over which we consider another moduli space $C_Z$, additionally parametrizing the maps $c,c^\vee$, as well as extensions of maps $f_0:\frac{1}{N}X\rightarrow \mathcal{B}^\vee$ and $f_0^\vee:\frac{1}{N}Y\rightarrow\mathcal{B}$. The space $\Xi_Z$ parameterizes liftings of the map $f^\vee:\frac{1}{N}Y\rightarrow\mathcal{B}$ to a map $f^\vee:\frac{1}{N}Y\rightarrow \mathcal{G}$.

Finally, we also want to consider the integral model over $\mathcal{O}_{E_\mathfrak{p}}$ of the minimal compactification which was constructed in \cite[\S7.2.3]{Lan}, and we denote this by $\mathscr{S}(\mathbf{G},X)_K^*$. We will denote the generic fibers of the toroidal and minimal compactifications by $\Sh(\mathbf{G},X)_{K}^{\tor}$ and $\Sh(\mathbf{G},X)^{*}_{K}$, respectively.
\subsection{Igusa Varieties}
\label{ss:def_igusa}
We now turn our attention to the study of Igusa varieties. Throughout this subsection the level $K=K(N)$ satisfies $p\nmid N$. We start with the basic definition. 
\subsubsection{The Basic Definition}{\label{ss: defIgusavarieties}}
\begin{definition}
    A $p$-divisible group $\mathbb{X}$ with $\mathbf{G}_{\mathbb{Z}_p}$-structure is a tuple $(\mathbb{X},\lambda,\iota)$ where $\mathbb{X}$ is a $p$-divisible group of dimension $\sum_i n_id_i[F_i:\mathbb{Q}]$, $\lambda$ is a quasi-polarization (i.e., a quasi-isogeny $\lambda:\mathbb{X}\dashrightarrow \mathbb{X}^\vee$ such that $\lambda^\vee=-\lambda$), and $\iota:\mathcal{O}_{B_{\mathbb{Q}_p}}\hookrightarrow \mathrm{End(\mathbb{X})}$, where $\mathcal{O}_{B_{\mathbb{Q}_p}}$ is the $p$-adic completion of $\mathcal{O}_B$ in $B_{\mathbb{Q}_p}$, such that $\lambda \circ\iota(b^*)=\iota(b)^\vee\circ\lambda$. 

    For notational simplicity, we will often drop $\lambda$ and $\iota$ from the notation.
\end{definition}
Let $\mu\in \domcochar$ be the dominant cocharacter induced by the inverse of the Hodge cocharacter associated with the conjugacy class $X$. We fix now some $b\in B(G,\mu)$, and consider a geometric point $x\in \mathscr{S}(\mathbf{G},X)_K(\ol{\mathbb{F}}_p)$ lying in the Newton strata for $b$. This defines an abelian variety $\mathcal{A}_x/\Spec(\ol{\bb{F}}_{p})$. Consider now the associated $p$-divisible group $\mathbb{X}:=\mathcal{A}_x[p^\infty]$. $\mathbb{X}$ has $\mathbf{G}_{\mathbb{Z}_p}$-structure acquired from the additional structures on $\mathcal{A}_{x}$. 

Up to replacing $x$ by another element in its isogeny class, we can assume $\mathbb{X}$ is completely slope divisible, following a result of Oort and Zink \cite[Theorem 2.1]{OortZink02}. Thus, we can write $\mathbb{X}=\oplus_{i=1}^r\mathbb{X}_i$, where the $\mathbb{X}_i$ are isoclinic $p$-divisible groups of strictly decreasing slopes.

We consider the following subset
\begin{equation*}
    \mathscr{C}_\mathbb{X} := \{x \in \mathscr{S}(\mathbf{G},X)_{K,\ol{\mathbb{F}}_p}:\exists \text{ isomorphism }\rho:\mathcal{A}_x[p^\infty] \times k(\ol{x}) \simeq \mathbb{X}\times k(\ol{x})\text{ preserving $\mathbf{G}_{\mathbb{Z}_p}$-structure} \},
\end{equation*}
known as the central leaf, where we denote by $\mathscr{S}(\mathbf{G},X)_{K,\ol{\mathbb{F}}_p}$ the (geometric) special fiber of $\mathscr{S}(\mathbf{G},X)_{K}$. It is shown in \cite[Proposition 1]{Mant} that this is a locally closed subset of $\mathscr{S}(\mathbf{G},X)_{K,\ol{\mathbb{F}}_p}$, and that we can give this subset the induced reduced scheme structure which makes the associated scheme $\mathscr{C}_\mathbb{X}$ smooth.

Let $\mathscr{G}$ be the $p$-divisible group of the restriction to $\mathscr{C}_\mathbb{X}$ of the universal abelian variety over $\mathscr{S}(\mathbf{G},X)_{K}$. We further define $\Ig^b$ to be the functor parametrizing, for any perfect $\mathscr{C}_
\mathbb{X}$-scheme $\mathscr{T}$, isomorphisms $\mathscr{G} \times_{\mathscr{C}_\mathbb{X}}\mathscr{T} \simeq \mathbb{X} \times_{\ol{\mathbb{F}}_p} \mathscr{T}$ which preserve $\mathbf{G}_{\mathbb{Z}_p}$-structure. Equivalently, we can define $\Ig^b$ as the functor sending an $\ol{\mathbb{F}}_{p}$-algebra $R$ to the set
\begin{equation}{\label{igmoduliinterp}}
\Ig^b(R)=\{(\rho,x):x\in \mathscr{S}(\mathbf{G},X)_{K}(R), \rho:\mathbb{X}_{R}\xrightarrow{\sim}\mathcal{A}_x[p^\infty]\text{ preserving $\mathbf{G}_{\mathbb{Z}_p}$-structure}\}.
\end{equation}
By \cite[Proposition~4.3.3]{CS1}, we know that this functor is representable, and by \cite[Corollary~4.3.5]{CS1} the scheme is perfect, and hence it lifts uniquely to a flat $p$-adic formal scheme, which we denote by $\mathfrak{Ig}^{b}$ over $\mathrm{Spf}(W(\ol{\mathbb{F}}_{p}))$. Moreover, from \cite[Corollary 2.3.2]{CS2}, we know that $\Ig^b\rightarrow\mathscr{C}_{\mathbb{X}}$ is a torsor for $\Aut_G(\mathbb{X})_{\ol{\bb{F}}_{p}}$, the formal group scheme of automorphisms of $\mathbb{X}$ compatible with extra structures viewed as an fpqc sheaf over $\ol{\bb{F}}_{p}$.

We write $\mIg_{C}^{b}$ for the perfectoid space attached to the adic generic fiber of $\mathfrak{Ig}^{b}$ over $C$. These spaces are supposed to model the fibers of the Hodge-Tate period morphism, a connection we will elaborate upon in \S \ref{subsec: HodgeTatePeriodMap}.

\subsubsection{Compactifications}{\label{sss: compactif}}
In order to understand (partial) minimal and toroidal compactifications of $\Ig^b$, we must first consider compactifications of the central leaf $\mathscr{C}_{\mathbb{X}}$. As constructed \cite[\S3.2.1]{San}, the central leaf $\mathscr{C}_{\mathbb{X}}$ admits partial toroidal and minimal compactifications, which we will denote by $\mathscr{C}_{\mathbb{X}}^\tor$ and $\mathscr{C}_\mathbb{X}^*$ respectively. Moreover, it is also shown in \cite[\S3.2.1]{San} that $\mathscr{C}_{\mathbb{X}}^\tor$ is a locally closed subset of $\mathscr{S}(\mathbf{G},X)_{K,\ol{\mathbb{F}}_p}^\tor$, and $\mathscr{C}_{\mathbb{X}}^*$ is a locally closed subset of $\mathscr{S}(\mathbf{G},X)_{K,\ol{\mathbb{F}}_p}^*$. Let $Z$ be a cusp label at level $K(N)$. This determines a locally closed boundary stratum $\mathscr{C}_{\mathbb{X},Z}:=\mathscr{C}_\mathbb{X}^*\times_{\mathscr{S}(\mathbf{G},X)_{K}^*} \mathscr{S}_Z$ which as explained in \cite[Proposition 3.1.4]{CS2} is a central leaf in $\mathscr{S}_Z$.

The Igusa variety $\Ig^b$ over $\mathscr{C}_{\mathbb{X}}$ extends to a perfect scheme $\Ig^{b,\tor}$ over $\mathscr{C}_{\mathbb{X}}^\tor$. More precisely, we can define $\Ig^{b,\tor}$ as follows. Let $\mathcal{A}$ denote the universal semi-abelian scheme over $\mathscr{C}_{\mathbb{X}}^\tor$. This is the restriction to $\mathscr{C}_{\mathbb{X}}^\tor$ of the universal semi-abelian scheme over $\mathscr{S}(\mathbf{G},X)^\tor_{K,\ol{\mathbb{F}}_{p}}$. Then, we know from \cite[Proposition~3.2.6]{San} (which is exactly the same argument as in \cite[Proposition~3.2.1]{CS2}) that both the connected part $\mathcal{A}[p^\infty]^\circ$ of $\mathcal{A}[p^\infty]$ and the multiplicative part $\mathcal{A}[p^\infty]^\mu$ are $p$-divisible groups. We thus let $\mathcal{A}[p^\infty]^{(0,1)} = \mathcal{A}[p^\infty]^\circ/\mathcal{A}[p^\infty]^\mu$ be the biconnected part. We can similarly define $\mathbb{X}^\circ,\mathbb{X}^{(0,1)}$ as the connected and biconnected parts of $\mathbb{X}$.

Following \cite[Definition 3.2.9]{CS2}, we can define $\Ig^{b,\tor}$ to be the functor which, for a perfect $\mathscr{C}_{\mathbb{X}}^\tor$-scheme $\mathscr{T}$, parametrizes 
\begin{enumerate}
    \item $\mathcal{O}_{B_{\mathbb{Q}_p}}$ -linear isomorphisms $\rho: \mathcal{A}[p^\infty]^\circ\times_{\mathscr{C}_\mathbb{X}^\tor}\mathscr{T}\xrightarrow{\sim} \mathbb{X}^\circ\times_{\bar{\mathbb{F}}_p}{\mathscr{T}}$
    \item A scalar in $\mathbb{Z}^\times_p(\mathscr{T})$ such that the induced isomorphism $\rho^{(0,1)}:\mathcal{A}[p^\infty]^{(0,1)}\times_{\mathscr{C}_\mathbb{X}^\tor}\mathscr{T}\xrightarrow{\sim} \mathbb{X}_{\mathscr{T}}^{(0,1)}$ obtained by quotienting by the multiplicative parts commutes with the polarizations up to the given element of $\mathbb{Z}_p^\times(\mathscr{T})$.
\end{enumerate}
Here, $\mathbb{Z}_p^\times(\mathscr{T})$ is the set of $\mathscr{T}$-points of the group scheme $\mathbb{Z}_p^\times\times \mathscr{C}_{\mathbb{X}}^\tor$. As explained in \cite[\S3.2.3-\S3.2.7]{CS2}, this is representable by a perfect scheme.

We define the partial minimal compactification $\Ig^{b,*}$ as the
normalization of $\mathscr{C}_\mathbb{X}^*$ in $\Ig^b$.  Since we have $\Ig^b\subset \Ig^{b,*},\Ig^{b,\tor}$, we will denote the boundaries by $\partial\Ig^{b,*}$ and $\partial\Ig^{b,\mathrm{tor}}$, respectively. These schemes are all perfect, and we can lift them to $p$-adic formal schemes $\mIg^{b,*}$, $\mIg^{b,\mathrm{tor}}$, $\partial\mIg^{b,*}$, and $\partial\mIg^{b,\mathrm{tor}}$ over $\mathrm{Spf}(W(\ol{\mathbb{F}}_{p}))$. We similarly denote by $\mIg_{C}^{b,*}$, $\mIg_{C}^{b,\mathrm{tor}}$, $\partial\mIg_{C}^{b,*}$, and $\partial\mIg_{C}^{b,\mathrm{tor}}$ the associated perfectoid spaces over $C$.

\subsubsection{Igusa Cusp Labels}
\label{sss:Igusacusplabels}
In order to understand the boundary components $\partial\Ig^{b,*}$ and $\partial\Ig^{b,\mathrm{tor}}$, we will recall the notion of Igusa cusp labels, as in \cite[Definition~3.2.19]{San} (which we have slightly modified to match the definition of cusp labels previously introduced, and which is similar to \cite[Definition 3.3.10]{CS2}). We let $\mathbb{X}_b:=\mathbb{X}$ be the completely slope divisible $p$-divisible group attached to $b$ defined above. We reintroduce $b$ in the notation to emphasize that all constructions here depend on $b$. Finally, observe that, since from the moduli problem the polarization on $\mathcal{A}_x$ is prime-to-$p$, the $p$-divisible group $\mathbb{X}_b=\mathcal{A}_x[p^\infty]$ is principally polarized.
\begin{definition}
Consider a tuple $(Z_b,Z^p,X,Y,\phi,\varphi_0,\varphi_{-2},\tilde{\varphi}_0)$ where
\begin{enumerate}
    \item $Z_b$ is an $\mathcal{O}_{B_{\mathbb{Q}_p}}$ -stable filtration of $\mathbb{X}_b$ by $p$-divisible subgroups of the form
    \begin{equation*}
        0 = Z_{b,-3} \subset Z_{b,-2} \subset Z_{b,-1} \subset \mathbb{X}_b,
    \end{equation*}
    with quotients which are also $p$-divisible groups, where $\Gr_{-2}^{Z_b}=Z_{b,-2}$ is multiplicative, and $\Gr_0^{Z_b}=\mathbb{X}_b/Z_{b,-1}$ is \'{e}tale, and $Z_{b,-2}$, $\mathbb{X}_b/Z_{b,-1}$ are Cartier dual to each other under the principal polarization on $\mathbb{X}_b$.
    \item $Z^p$ is an $\mathcal{O}_{B}$ -stable split, symplectic and admissible filtration
    \begin{equation*}
        0 = Z^p_{-3} \subset Z^p_{-2} \subset Z^p_{-1} \subset L\otimes_{\mathbb{Z}}\hat{\mathbb{Z}}^p
    \end{equation*}
    \item $X,Y$ are $\mathcal{O}_{B}$-lattices of the same $B$-multirank, together with an $\mathcal{O}_B$-linear embedding $\phi:Y\hookrightarrow X$, and we have isomorphisms
    \begin{equation*}
        \varphi_0:Y\otimes_{\mathbb{Z}}\hat{\mathbb{Z}}^p\simeq \Gr_0^{Z^p}
    \end{equation*}
    \begin{equation*}
        \varphi_{-2}:\Hom(X\otimes_{\mathbb{Z}}\hat{\mathbb{Z}}^p,\hat{\mathbb{Z}}^p(1))\simeq \Gr_{-2}^{Z^p}
    \end{equation*}
    \begin{equation*}
        \tilde{\varphi}_0:Y\otimes(\mathbb{Q}_p/\mathbb{Z}_p)\simeq \Gr_0^{Z_b}
    \end{equation*}    
    such that the pairing $\langle,\rangle_{20} : \Gr^{Z^p}_{-2} \times \Gr^{Z^p}_0 \rightarrow \hat{\mathbb{Z}}^p(1)$ induced from the one on $L$ pulls back via $\varphi_{-2},\varphi_0$ to the pairing $(h,y)\mapsto h(\phi(y))$.
\end{enumerate}
Two such tuples are equivalent if and only if there exists a pair of isomorphisms $(\gamma_X : X' \xrightarrow{\sim} X, \gamma_Y : Y\xrightarrow{\sim} Y')$ such that:
\begin{enumerate}
    \item $\phi = \gamma_X\phi'\gamma_Y$
    \item $\varphi_{-2}=\varphi_{-2}'\phantom{}^{t}\gamma_X$, 
    \item $\varphi_0=\varphi_0'\gamma_Y$, and $\tilde{\varphi}_0=\tilde{\varphi}'_0\gamma_Y$.
\end{enumerate}
We define an Igusa cusp label to be an equivalence class of such tuples.
\end{definition}
For later use, we will denote by $\delta_b:\Gr^{Z_b}\simeq \mathbb{X}_b$ a choice of $\mathcal{O}_{B_{\mathbb{Q}_p}}$-linear symplectic splitting. The resulting boundary components we obtain will not depend on the choice of splitting. 

There is an action of $J_b(\mathbb{Q}_p) \times \mathbf{G}(\mathbb{A}^p_f)$ on Igusa cusp labels. For the action of $\mathbf{G}(\mathbb{A}_f^p)$, this follows as in the discussion after Definition \ref{defn:cusp_label}, and $\tilde{\varphi}_0$ does not change. For the action of $J_b(\mathbb{Q}_p)$, we first observe that by Dieudonn\'{e} theory we may replace all $p$-divisible groups with their Dieudonn\'{e} modules, i.e., we ask for a filtration
\[Z_b:0\subset \mathbb{D}(Z_{b,-2})\subset\mathbb{D}(Z_{b,-1})\subset \mathbb{D}(\mathbb{X}_b),\]
by Dieudonn\'{e} submodules with $\mathcal{O}_{B_{\bb{Q}_p}}$-action such that the quotients are also free $W(\ol{\mathbb{F}}_p)$-modules with $\mathcal{O}_{B_{\bb{Q}_p}}$-action. Now, because $\Gr_0^{Z_b}$ is \'{e}tale, we can consider the Tate modules instead, and $\tilde{\varphi}_0$ induces an isomorphism, also denoted $\tilde{\varphi}_0$
\[\tilde{\varphi_0}:Y\otimes_{\mathbb{Z}}\mathbb{Z}_p\simeq T_p(\Gr^{Z_{b}}_0).\]
We can consider rational versions obtained by inverting $p$, i.e., 
\[Z_b\left[\frac{1}{p}\right]:0\subset \mathbb{D}(Z_{b,-2})\left[\frac{1}{p}\right]\subset\mathbb{D}(Z_{b,-1})\left[\frac{1}{p}\right]\subset \mathbb{D}(\mathbb{X}_b)\left[\frac{1}{p}\right],\]
and isomorphisms $\tilde{\varphi_0}\left[\frac{1}{p}\right]$. We now see that $J_b(\mathbb{Q}_p)$ clearly acts on rational data. Moreover, observe that $j$ induces an isomorphism $j\circ\tilde\varphi_0[1/p]:Y\otimes\mathbb Q_p\xrightarrow{\sim}V_p\Gr_0^{Z^j_b}$. Applying \cite[Lemma 5.4.3.3]{Lan} there is a unique lattice $Y^j\subset Y_{\mathbb Q}$ satisfying
\[
Y^j\otimes\mathbb Z_p=(j\circ\tilde\varphi_0[1/p])^{-1}(T_p\Gr_0^{Z^j_b}),\qquad
Y^j\otimes\mathbb Z_\ell=Y\otimes\mathbb{Z}_\ell,
\]
and restricting $j\circ\tilde\varphi_0[1/p]$ to $Y^j\otimes\mathbb Z_p$ gives the desired $j\cdot\tilde{\varphi}_0$. Observe that the principal polarization assumption implies $\phi$ induces an isomorphism
\[
\phi\otimes\mathbb Z_p:Y\otimes\mathbb Z_p
\xrightarrow{\sim}X\otimes\mathbb Z_p.
\]
and define $X^j$ to be the unique lattice in $X_{\mathbb{Q}}$ such that 
\[
X^j\otimes\mathbb Z_p=\phi\otimes \mathbb Q_p(Y^j\otimes\mathbb Z_p),
\qquad
X^j\otimes\mathbb Z_\ell=X\otimes\mathbb Z_\ell.
\]
and set $\phi^j=\phi\otimes{\mathbb Q}|_{Y^j}$.

If $H \subset J_b(\mathbb{Q}_p) \times \mathbf{G}(\mathbb{A}^p_f)$ is a compact open subgroup then an Igusa cusp label at level $H$ is a $H$-orbit of Igusa cusp labels. For a general closed subgroup $H' \subset  J_b(\mathbb{Q}_p) \times \mathbf{G}(\mathbb{A}^p_f)$, an Igusa cusp label at level $H'$ is a compatible family of Igusa cusp labels at level $H$ for all $H' \subset H$.

Additionally, we have the following relationship between Igusa cusp labels and Shimura cusp labels. Given a Shimura cusp label $(Z,\Phi)$, we say that an Igusa cusp label $(Z_b,Z^p,X,Y,\phi,\varphi_0,\varphi_{-2},\tilde{\varphi}_0)$ \emph{lives over} $(Z,\Phi)$ if the lattices $X,Y$ are the same, and the base change of $Z$ to $\hat{\mathbb{Z}}^p$ is $Z^p$, while $\varphi_0,\varphi_{-2}$ are the inverses of the corresponding base changes of the Shimura torus-argument isomorphisms.
\begin{remark}
    The definition in \cite[Definition 3.2.15]{San} of an Igusa cusp label is different from ours in that the definition in loc. cit. involves (using our terminology) first fixing a Shimura cusp label of level $K=K(N)$, and then considering Igusa cusp labels of level $\{1\}\times K^p$ living over this Shimura cusp label. The definition of Shimura cusp labels in loc. cit. is adapted from Lan's definition, c.f. the discussion in \ref{rem:cusplabels}. In particular, the Igusa cusp labels in \cite{San} contains extra data of lattices and isomorphisms in $L\otimes\mathbb{Z}_p$. This slight variation does not affect any of the results quoted here, since we will always be in a situation where we fix a Shimura cusp label $Z$ and are looking at Igusa cusp labels $\tilde{Z}$ living over $Z$.
\end{remark}
\subsubsection{Boundary components}
\label{sss:boundarycomponents}
We now explain how we can understand the boundary $\partial\Ig^{b,\mathrm{tor}}$ according to Igusa cusp labels following \cite{CS2} and \cite{San}. Note that since $K=K(N)$ admits a decomposition as $K_pK^p$, where $K^p\subseteq \mathbf{G}(\mathbb{A}_f^p)$, we henceforth use $K^p(N)$ to denote this prime-to-$p$ level.

Recall that for a Shimura cusp label $Z=(Z,\Phi)$, we have $\mathscr{C}_{\mathbb{X}_b,Z}$, a central leaf of $\mathscr{S}_Z$, and a locally closed toroidal boundary component $\mathscr{C}_{\mathbb{X}_b,Z}^{\mathrm{tor}}\subset \mathscr{C}^{\tor}_{\mathbb{X}_b}$. We denote by $\Ig^{b,\tor}_{Z}$ the preimage of $\mathscr{C}_{\mathbb{X}_b,Z}^{\tor,\mathrm{perf}}$ under the map $f:\Ig^{b,\tor}\rightarrow\mathscr{C}_{\mathbb{X}_b}^{\tor,\mathrm{perf}}$. We want to understand the decomposition of $\Ig^{b,\tor}_{Z}$ according to Igusa cusp labels.

We first recall the following definition in \cite{San} of the perfect boundary components corresponding to the Igusa cusp labels. Set $\mathscr{C}_{C_Z}:=\mathscr{C}_{\mathbb{X}_b,Z}\times_{\mathscr{S}_Z}C_Z$, and recall that over $\mathscr{C}_{C_Z}$ there is a universal Raynaud extension
\[0\rightarrow T\rightarrow\mathcal{G}\rightarrow\mathcal{B}\rightarrow0,\]
where $\mathcal{B}$ is an abelian scheme. Similar to above, for $\mathcal{B}[p^\infty]$ we let the multiplicative part be denoted by $\mathcal{B}[p^\infty]^{\mu}$ and the biconnected part be denoted by $\mathcal{B}[p^\infty]^{(0,1)}$. We denote $\mathcal{G}[p^\infty]^{\mu}$ and $\bigl(\mathrm{Gr}^{Z_b}_{-1}\bigr)^\mu$ similarly as the respective multiplicative parts of the $p$-divisible groups.

\begin{definition}{\cite[Definition 3.2.34]{San}}
\label{defn:functorIgCZ}
Let $\tilde{Z}$ be an Igusa cusp label of level $\{1\}\times K^p(N)$ living over a Shimura cusp label $Z$ of level $K(N)$. Let $\Ig_{\tilde{Z},C_Z}^{b}$ be the functor on perfect schemes $\mathscr{T}$ over $\mathscr{C}_{C_Z}$ parametrizing:
\begin{enumerate}
  \item An isomorphism
  $
    \rho:
    \mathcal{B}[p^{\infty}]^{\mu}
    \times_{\mathscr{C}_{C_Z}} \mathscr{T}
    \xrightarrow{\sim}
    \bigl(\mathrm{Gr}^{Z_b}_{-1}\bigr)^{\mu}
    \times_{\ol{\mathbb{F}}_p} \mathscr{T}
  $
  compatible with the $\mathcal{O}_{B_{\mathbb{Q}_p}}$-actions;
  
  \item A splitting
  $
    \delta:\;
    \mathcal{B}[p^{\infty}]^{\mu}
    \times_{\mathscr{C}_{C_Z}} \mathscr{T}
    \longrightarrow
    \mathcal{G}[p^\infty]^{\mu}
    \times_{\mathscr{C}_{C_Z}} \mathscr{T}
  $
  compatible with the $\mathcal{O}_{B_{\mathbb{Q}_p}}$-actions;
  
\item  A scalar in $\mathbb{Z}^\times_p(\mathscr{T})$ together with an $\mathcal{O}_{B_{\mathbb{Q}_p}}$ -linear isomorphism 
\[\rho^{(0,1)}: \mathcal{B}[p^\infty]^{(0,1)}\times_{\mathscr{C}_{C_Z}}\mathscr{T}\xrightarrow{\sim} \mathbb{X}_b^{(0,1)}\times_{\bar{\mathbb{F}}_p}{\mathscr{T}}\]
which commutes with the polarizations up to the given element of $\mathbb{Z}_p^\times(\mathscr{T})$.
\end{enumerate}
\end{definition}
This functor is representable by a perfect scheme, as implicitly seen from \cite{San}. More precisely, we can first consider a level $p^m$-variant $\Ig_{m,\tilde{Z}_m,C_Z}$ as in \cite[Definition 3.2.20]{San} for each $m\in\mathbb{N}$ (this is denoted $\Ig^{\mathbb{X}_b}_{\tilde{Z}^m_n,C_{\Phi_n},\delta_n}$ in loc. cit.). Using \cite[Theorem 2.2.4]{CS2} we see that each of these functors is representable by a scheme over $\mathscr{C}_Z$, where we further note that in light of (1) the data in (2) of the splitting is equivalent to asking for an isomorphism 
\[T[p^\infty]\oplus \mathcal{B}[p^\infty]^\mu\simeq \mathcal{G}[p^\infty]^\mu\]
compatible with $\mathcal{O}_{B_{\mathbb{Q}_p}}$ actions. Now, taking the perfection of the inverse limit over $m$ recovers the functor defined above, and hence $\Ig_{\tilde{Z},C_Z}^{b}$\footnote{This is denoted by $\mIg^{\mathbb{X}_b}_{Z_n,C_{\Phi_n},\delta_n}$ in \cite{San}, but in this paper we reserve the fraktur font for formal schemes.} is also representable by a perfect scheme.

At each finite $p^m$-level, the preimage of the boundary stratum
indexed by $Z$ admits a decomposition into open and closed formal subschemes
indexed by the Igusa cusp labels of level $\Gamma_b(p^m)K^p(N)$ living over a fixed Shimura cusp label $Z$, where
\[
\Gamma_b:=\Aut_G(\mathbb{X}_b)(\ol{\mathbb{F}}_p),
\qquad
\Gamma_b(p^m):=\{\rho\in\Gamma_b:
\rho|_{\mathbb{X}_b[p^m]}=\id\}.
\]
Here $\Aut_G$ denotes automorphisms preserving the additional
$\mathbf{G}_{\mathbb{Z}_p}$-structure. Thus the finite-level automorphisms
in $\Gamma_b/\Gamma_b(p^m)$ are those that lift to $\mathbb{X}_b$. These boundary components are related to $\Ig_{\tilde{Z},C_Z}^{b}$ in the following way: Denote by $\widehat{\Ig}^{b,\tor}_{Z}$ the completion of $\Ig^{b,\tor}$ along the locally closed perfect subscheme $\Ig^{b,\tor}_{Z}$. We have the following:
\begin{theorem}{\cite[Theorem~3.2.36]{San}}
\label{thm:modpIgboundary}
    \begin{enumerate}
        \item The Igusa cusp labels $\widetilde Z$ of level
$K^p(N)$ living over $Z$ give a
decomposition
        \begin{equation*}
        |\widehat{\Ig}^{b,\tor}_{Z}| =\bigsqcup_{\tilde{Z}}|\widehat{\Ig}^{b,\tor}_{\tilde{Z}}|,
        \end{equation*}
        of the geometric points of $\widehat{\Ig}^{b,\tor}_{Z}$. This partition is induced by a compatible choice of open and closed decompositions at levels $\Gamma_b(p^m)K^p(N)$ for all $m$.
        \item For a fixed Igusa cusp label $\widetilde Z$, each $\widehat{\Ig}^{b,\tor}_{\tilde{Z}}$ admits the following description: We have a commutative diagram with both squares Cartesian:
        \begin{equation}
        \label{eqn:Ig-cusp-construction}
            \begin{tikzcd}
    \Ig^b_{\tilde{Z},\Xi} \arrow[d, hook] \arrow[r] & \Xi_{Z}^{\mathrm{perf}} \arrow[d, hook] \\
    \Ig^b_{\tilde{Z},\Xi,\Sigma_Z} \arrow[d] \arrow[r]       & \Xi_{Z,\Sigma_Z}^{\mathrm{perf}} \arrow[d]       \\
    \Ig^b_{\tilde{Z},C_{Z}} \arrow[r]                  & C_Z^{\mathrm{perf}}, 
            \end{tikzcd}
        \end{equation}
        where on the right hand side the schemes denoted $(-)^{\mathrm{perf}}$ are the perfection of the mod $p$ fibers of the respective objects defined in Theorem \ref{thm:shimuraboundary}, the composition of the vertical maps are torsors for the torus $E_{\Phi}^{\mathrm{perf}}$, and the inclusion maps in the top square define a relative toroidal embedding. 
        
        Let $\widehat{\Ig}^b_{\tilde{Z},\Xi,\Sigma_Z,}$ be the completion of $\Ig^b_{\tilde{Z},\Xi,\Sigma_Z}$ along the distinguished boundary component $\partial^{\mathrm{dist}}\Ig^b_{\tilde{Z},\Xi,\Sigma_Z}:=\Ig^b_{\tilde{Z},\Xi,\Sigma_Z}\times_{\Xi_{Z,\Sigma_Z}^{\mathrm{perf}}} \partial^{\mathrm{dist}}\Xi_{Z,\Sigma_Z}^{\mathrm{perf}}$. Then, $\Gamma_{\tilde{Z}}$ acts on $\Ig^b_{\tilde{Z},\Xi}$ and $\Ig^b_{\tilde{Z},\Xi,\Sigma_Z}$, and we have an isomorphism
        \[\widehat{\Ig}^{b,\tor}_{\tilde{Z}}\simeq \widehat{\Ig}^b_{\tilde{Z},\Xi,\Sigma_Z}/\Gamma_{\tilde{Z}}.\]
        Here $\widehat{\operatorname{Ig}}^{b,\mathrm{tor}}_{\widetilde Z}$
denotes the formal chart associated with the fixed
Igusa cusp data $\widetilde Z$.
        Here, we define $\Gamma_{\tilde{Z}}$ to be the subgroup of $\Gamma_{Z}$ consisting of isomorphisms $(\gamma_X,\gamma_Y)$ such that $\tilde{\varphi}_{0}=\tilde{\varphi}_{0}\circ (\gamma_{Y}^{-1}\otimes 1)$.
    \end{enumerate}
\end{theorem}

For later use, we give also an alternate moduli description of $\Ig^b_{\tilde{Z},\Xi}$.

\begin{proposition} We have the following:
\label{prop:moduliIg}
    \begin{enumerate}
    \item \begin{enumerate}
    \item Over $\Ig^b_{\tilde{Z},C_Z}$, we have universal Raynaud extensions
    \[0\rightarrow T\rightarrow \mathcal{G}\rightarrow \mathcal{B}\rightarrow 0\]
    and its dual
    \[0\rightarrow T^\vee \rightarrow \mathcal{G}^\vee\rightarrow \mathcal{B}^\vee\rightarrow 0,\]
    equivalently maps $c:X\rightarrow \mathcal{B}^\vee$ and $c^\vee:Y\rightarrow\mathcal{B}$.
    \item Over $\Ig^b_{\tilde{Z},C_Z}$, there exists a universal $\mathcal{O}_{B_{\mathbb{Q}_p}}$-equivariant embedding $\rho:\mathcal{G}[p^\infty]\hookrightarrow \mathbb{X}_b\times_{\ol{\bb{F}}_p}\Ig^b_{\tilde{Z},C_Z}$ which is compatible with polarizations in the following sense: The induced filtration 
    \[Z_{\mathbb{X}_b}: T[p^\infty] \subset \mathcal{G}[p^\infty] \subset \mathbb{X}_b\times_{\ol{\bb{F}}_p}\Ig^b_{\tilde{Z},C_Z}\]
    is the base change of $Z_b$, and the induced isomorphism  
    \[\Gr^{Z_b}_{-1} \simeq \mathcal{G}[p^\infty]/T[p^\infty] = \mathcal{B}[p^\infty]\]
    is compatible with the polarizations
    \item There is a lift of $c^\vee:Y\rightarrow \mathcal{B}$ to $\tilde{c}^\vee:Y[\frac{1}{p}]\rightarrow\mathcal{B}$ corresponding to a symplectic splitting of the filtration $Z_{\mathbb{X}_b}$. 
    \end{enumerate}
    \item For $\tilde{c}^\vee$ as above, $\Ig^b_{\tilde{Z},\Xi}$ is the space of symmetric lifts $\iota:Y[\frac{1}{p}]\rightarrow\mathcal{G}$ lifting $\tilde{c}^\vee$.
    \end{enumerate} 
\end{proposition}
\begin{proof}
    This follows from \cite[Lemma 3.2.32,Proposition 3.2.33]{San}. More precisely, loc. cit. gives a moduli description for $\Ig^b_Z$, by considering a moduli functor of degeneration perfect Igusa-level structures \cite[Definition 3.2.30]{San}, and a torsor for $E_{\Phi}^\mathrm{perf}$ over this moduli space. Restricting this moduli functor to the data specified by
$\widetilde Z$ identifies the filtration with $Z_b$
(cf. the discussion after \cite[Definition~3.2.34]{San}).
\end{proof}

Let $\widehat{\mathscr{S}}(\mathbf{G},X)^\tor_{K^p}$ denote the inverse limit over $n$ of the $p$-adic completion of toroidal integral models with level $K(p^nN)$,
\[\widehat{\mathscr{S}}(\mathbf{G},X)^\tor_{K^p}=\lim\limits_{\substack{\longleftarrow\\ n}}\widehat{\mathscr{S}}(\mathbf{G},X)^\tor_{K(p^nN)}.\]
Fix a lift $\mathbb{X}_{O_C}$ of $\mathbb{X}=\mathbb{X}_b$ to a $p$-divisible group with $\mathbf{G}_{\mathbb{Z}_p}$-structure over $O_C$, and an isomorphism $\alpha : T_p(\mathbb{X}_{O_C}) \simeq L\otimes_{\mathbb{Z}} \mathbb{Z}_p$. Following \cite[Theorem~4.3.8]{San}, we have a map of formal schemes
\[g:\mIg^{b,\tor} \rightarrow \widehat{\mathscr{S}}(\mathbf{G},X)^\tor_{K^p}.\]
This map $g$ is related to the natural map $f:\Ig^{b,\tor}\rightarrow \mathscr{C}^{\tor,\mathrm{perf}}\rightarrow\mathscr{S}(\mathbf{G},X)^\tor_{K}$ in that on the geometric special fiber $g$ factors $f$, i.e., the composition of the geometric special fiber of $g$ with the natural projection $\mathscr{S}(\mathbf{G},X)_{K^p}^{\tor}\rightarrow \mathscr{S}(\mathbf{G},X)_{K}^{\tor}$ is $f$.

For later use, we will need to understand this map $g$ on toroidal boundary components. Moreover, we also want to describe the lifting of $\widehat{\Ig}^{b,\tor}_{\tilde{Z}}$ to a formal scheme over $O_C$. Let $\mIg_{Z,O_C}^{b,\tor}$ be the base change to $O_C$ of the canonical flat formal scheme lifting $\Ig^{b,\tor}_Z$. 

Recall that for every Shimura cusp label $Z$, we have the schemes $C_Z$, $\Xi_{Z}$ and $\Xi_{Z,\Sigma_Z}$ from Theorem \ref{thm:shimuraboundary}. Now suppose that we have a Shimura cusp label $Z'$ at level $K^p(N)$ living over $Z$ at level $K(N)$. Observe that this is the same as the data of a compatible system of Shimura cusp labels $Z_m$ of level $K(p^mN)$ for all $m\geq 0$, thus we define 
\[C_{Z'}:=\lim\limits_{\substack{\longleftarrow\\ m}} C_{Z_m}\qquad \Xi_{Z'}:=\lim\limits_{\substack{\longleftarrow\\ m}} \Xi_{Z_m}\qquad \Xi_{Z',\Sigma_Z}:=\lim\limits_{\substack{\longleftarrow\\ m}} \Xi_{Z_m,\Sigma_Z},\]
while $\widehat{C}_{Z'}$, $\widehat{\Xi}_{Z'}$, $\widehat{\Xi}_{Z',\Sigma_Z}$ denote the $p$-adic completions.

We have the following:
\begin{proposition}
\label{prop:mIg-cusp-construction}
    There is a flat formal lift of $\widehat{\Ig}^{b,\tor}_{\tilde{Z}}$ over $\Spf O_C$, denoted $\widehat{\mIg}^{b,\tor}_{\tilde{Z},O_C}$, which may be described as follows. We have a commutative diagram of formal schemes
    \begin{equation}
    \label{eqn:mIg-cusp-construction}
        \begin{tikzcd}
    \mIg^b_{\tilde{Z},\Xi,O_C} \arrow[d, hook] \arrow[r] &\widehat{\Xi}_{Z',O_C} \arrow[d, hook]\\
    \mIg^b_{\tilde{Z},\Xi,\Sigma_Z,O_C} \arrow[d] \arrow[r] &\widehat{\Xi}_{Z',\Sigma_{Z},O_C} \arrow[d]\\
    \mIg^b_{\tilde{Z},C_Z,O_C} \arrow [r]& \widehat{C}_{Z',O_C}, 
            \end{tikzcd}
    \end{equation}
    and $\widehat{\mIg}^{b,\tor}_{\tilde{Z},O_C}$ is isomorphic to $\mathfrak{X}_{\tilde{Z}}/\Gamma_{\tilde{Z}}$, where $\mathfrak{X}_{\tilde{Z}}$ is the completion of $\mIg^b_{\tilde{Z},\Xi,\Sigma_Z,O_C}$ along the distinguished boundary component.

    Moreover, at each finite $p^m$-level, there is an open and closed decomposition according to the Igusa cusp labels of level
$\Gamma_b(p^m)K^p(N)$. These decompositions are compatible as $m$ varies
and determine a partition of geometric points at
infinite level.
    
    Finally, the map $g$ is compatible with these finite-level
boundary decomposition. In particular, restricting to the data
specified by a label $\widetilde Z$ gives the map
    \[g_{\tilde{Z}}:\widehat{\mIg}^{b,\tor}_{\tilde{Z},O_C}\rightarrow \widehat{\mathscr{S}}(\mathbf{G},X)^\tor_{K^p,Z',O_C},\]
    induced by the maps in \eqref{eqn:mIg-cusp-construction} on toroidal boundary charts.
\end{proposition}

In the above proposition, we used $Z'$ to denote the Shimura cusp label at level $K^p(N)$ corresponding to $\tilde{Z}$ in the following way. To define a Shimura cusp label at level $K^p(N)$, it suffices to define the admissible filtration on $L\otimes_{\mathbb{Z}}\mathbb{Z}_p$ and the isomorphisms $\varphi_0,\varphi_{-2}$, but given the isomorphism $\alpha : T_p(\mathbb{X}_{O_C}) \simeq L\otimes_{\mathbb{Z}} \mathbb{Z}_p$ which we fixed before, we obtain this from $Z_b$ and $\tilde{\varphi_0}$.

\begin{proof}
    This result can be extracted from the construction of $g$ in the proof of \cite[Theorem~4.3.8]{San}, which we explicate here for the reader's convenience.

    The left column is the formal $p$-adic lift of the left column of \eqref{eqn:Ig-cusp-construction} in Theorem \ref{thm:modpIgboundary}, and in particular the composition is given by a $E^\mathrm{perf}_{\Phi,O_C}$-torsor, where $E^\mathrm{perf}_{\Phi,O_C}$ is the lift of the perfection of the torus $E^\mathrm{perf}_\Phi$.
    
    We will now give a more concrete description of the functors defining these lifts. Recall from Proposition \ref{prop:moduliIg} that the $E^\mathrm{perf}_{\Phi}$-torsor $\Ig^b_{\tilde{Z},\Xi}\rightarrow \Ig^b_{\tilde{Z},C_{Z}}$ can be described as the space parametrizing symmetric lifts $\iota:Y[\frac{1}{p}]\rightarrow\mathcal{G}$ of $\tilde{c}^\vee: Y[\frac{1}{p}] \rightarrow \mathcal{B}$. Now, we want to lift these to $O_C$. Fix a map $k\rightarrow O_C/p$. Recall that we have fixed a lifting $\mathbb{X}_{O_C}$ of $\mathbb{X}$, and we also fix as well as an $O_B$-linear splitting
    \[\delta_{\mathbb{X}_{O_C}}: \mathbb{X}_{O_C}\simeq \mathbb{X}^\mu_{O_C}\oplus \mathbb{X}^{(0,1)}_{O_C}\oplus \mathbb{X}^{\acute{e}t}_{O_C}.\]
    Let $\delta_{\mathbb{X}}=\delta_{\mathbb{X}_{O_C}}\times k$. Firstly, by \cite[Proposition 4.3.1]{CS2} there exists some rational $\epsilon\in(0,1]$ such that we have an isomorphism $\mathbb{X}\times_{k}O_C/p^\epsilon \simeq \mathbb{X}_{O_C} \times_{O_C} O_C /p^\epsilon$ of $p$-divisible groups with $G$-structures, lifting the identity, such that the splittings satisfy
    \[\delta_{\mathbb{X}_{O_C}}\times_{O_C}O_C/p^\epsilon \simeq \delta_{\mathbb{X}}\times_k {O_C/p^\epsilon}.\]
    Thus, we first base change to $O_C/p^\epsilon$. Observe that since the quotient $\mathbb{X}/\mathcal{G}[p^\infty]$ is \'{e}tale, it deforms uniquely to a $p$-divisible group over $O_C$. Thus, we have a unique lift of $\mathcal{G}[p^\infty]$ to $\mathcal{G}[p^\infty]_{O_C}$ such that the embedding $\rho$ lifts to an embedding $\rho_{O_C}:\mathcal{G}[p^\infty]_{O_C}\hookrightarrow \mathbb{X}_{O_C}$. By Serre-Tate theory, this lift $\mathcal{G}[p^\infty]$ induces a unique lift of the Raynaud extension
    \[0\rightarrow T_{O_C}\rightarrow \mathcal{G}_{O_C}\rightarrow\mathcal{B}_{O_C}\rightarrow 0,\]
    and by similarly considering duals, also a unique lift of the dual Raynaud extension
    \[0\rightarrow T_{O_C}^\vee\rightarrow \mathcal{G}^\vee_{O_C}\rightarrow\mathcal{B}^\vee_{O_C}\rightarrow 0.\]
    In particular, we have a map $c_{O_C}^\vee:Y\rightarrow\mathcal{B}_{O_C}$ lifting the map $c^\vee$ from Proposition \ref{prop:moduliIg}. 

    We now lift a splitting of 
    \[T[p^\infty]\subset \mathcal{G}[p^\infty]\subset \mathbb{X}\]
    over $O_C$, this will define a lifting of $c_{O_C}^\vee$ to a map $\tilde{c}_{O_C}^\vee:Y[\frac{1}{p}]\rightarrow\mathcal{B}_{O_C}$. Note that given the splitting $\delta_{\mathbb{X}_{O_C}}$, to get a splitting of
    \[T[p^\infty]_{O_C}\subset \mathcal{G}[p^\infty]_{O_C}\subset \mathbb{X}_{O_C}\]
    we only need to defined a splitting on the \'{e}tale parts. This splitting can be defined as a lift of the restriction to étale parts of the splitting mod $p$ as constructed in Proposition \ref{prop:moduliIg} (1)(c), and thus there is a unique such lift to $O_C$.
     
    Thus, we see that the $E^\mathrm{perf}_{\Phi,O_C}$-torsor $\mIg^b_{\tilde{Z},\Xi,O_C}\rightarrow \mIg^b_{\tilde{Z},C_Z,O_C}$ can be described as the torsor of liftings of $\tilde{c}_{O_C}^\vee$ to $\iota:Y[\frac{1}{p}]\rightarrow\mathcal{G}_{O_C}$. Now, we can consider a relative (perfect) toroidal embedding $\mIg^b_{\tilde{Z},\Xi,O_C}\hookrightarrow \mIg^b_{\tilde{Z},\Xi,\Sigma_Z,O_C}$ defined using the fan $\Sigma_Z$ in $S_{\Phi}^{\vee}\otimes\mathbb{R}$, where $S_{\Phi}[\frac{1}{p}]$ is the character lattice of $E_{\Phi}^{\mathrm{perf}}$, and thus also by construction the character lattice of $E^\mathrm{perf}_{\Phi,O_C}$. Observe that from the uniqueness of relative (perfect) toroidal embeddings, since $\mIg^b_{\tilde{Z},\Xi,\Sigma_Z,O_C}$ is flat over $O_C$ it is the unique lift of $\Ig^b_{\tilde{Z},\Xi,\Sigma_Z}$, since the (perfect) toroidal embedding $\Ig^b_{\tilde{Z},\Xi}\hookrightarrow \Ig^b_{\tilde{Z},\Xi,\Sigma_Z}$ is exactly defined also using the fan $\Sigma_Z$.

    Now fix any $m \ge 0$. Let $\tilde{Z}_m$ be the associated Igusa cusp label of level $\Gamma_b(p^m)K^p(N)$. Note that from $\tilde{Z}_m$ we obtain a Shimura cusp label $Z_m$ of level $K(p^mN)$ using the isomorphism $\mathbb{X}_{\mathcal{O}_C}[p^m](C) \cong L/p^m$. Moreover, the data of an Igusa cusp label of level $\Gamma_b(p^m)K^p(N)$ $\tilde{Z}_m$ gives us a filtration $Z_{b,m}$ on $\mathbb{X}[p^m]$, with an $\mathcal{O}_{B_{\mathbb{Q}_p}}$ -splitting $\delta_{b,m}$ of $Z_{b,m}$, and an isomorphism  $\psi_m : \mathrm{Gr}^{Z_{b,m}}_0\simeq Y /p^m$.

    Thus, we first want to understand a map to $\widehat{\mathscr{S}}^{\mathrm{tor}}_{K(p^mN),\,Z_m,\,O_C}.$ For this, firstly note that the prime-to-$p$ structures on both sides are identical, so it suffices to understand the structures at $p$. Consider the symplectic splitting
$L/p^m \cong \bigoplus_{i=-2}^{0} \mathrm{Gr}^{Z_{p^m}}_{i}$ of the filtration we obtain on $L/p^m$ using the trivialization $\alpha$ and the splitting $\delta_{b,m}$ of the filtration $Z_{b,m}$, and we assume it is compatible with the splitting $\delta$. Thus, we see that we have a splitting
$\mathrm{Gr}^{Z_{b,m}}_{-1} \to Z_{b,m,-1}$, which then induces a splitting of
\[
0 \longrightarrow T[p^m] \longrightarrow \mathcal{G}_{{O}_C}[p^m]
\longrightarrow \mathcal{B}_{O_C}[p^m] \longrightarrow 0,
\]
and also by duality, a splitting of the dual sequence  
\[0 \rightarrow T^\vee[p^m] \rightarrow \mathcal{G}^\vee_{O_C}[p^m] \rightarrow \mathcal{B}^\vee_{O_C}[p^m]\rightarrow 0.\]
We thus obtain lifts of the maps defining the Raynaud extension and its dual, $c$ and $c^\vee$ to maps  $c_m : \frac{1}{p^m}X \rightarrow \mathcal{B}^\vee_{O_C}$ and $c_m^\vee : \frac{1}{p^m}Y \rightarrow \mathcal{B}_{O_C}$. Thus, we have a map to the abelian scheme $C_{Z_m,O_C}$ on the toroidal boundary chart defining $\mathscr{S}^{\tor}_{K(p^mN),Z_m}$. Now, the torus torsor $\Xi_{Z_m,O_C}$ is defined as the moduli space parametrizing lifts $\iota_m : \frac{1}{p^m}Y \rightarrow \mathcal{G}_{O_C}$.

Since $\mIg^b_{\tilde{Z},\Xi,O_C}\rightarrow \mIg^b_{\tilde{Z},C_Z,O_C}$ can be described as the torsor of liftings of $\tilde{c}_{O_C}^\vee$ to $\iota:Y[\frac{1}{p}]\rightarrow\mathcal{G}_{O_C}$, we obtain a map from $\mIg^b_{\tilde{Z},\Xi,O_C}$ to $\Xi_{Z_m,O_C}$. Observe that this extends to a map of relative toroidal embeddings $\mIg^b_{\tilde{Z},\Xi,\Sigma_Z,O_C}$ to $\Xi_{Z_m,\Sigma_Z,O_C}$, since a map of torus torsors extends exactly when the induced map on fans maps every cone in one into a cone of the other. This last condition clearly holds true for the map $\mIg^b_{\tilde{Z},\Xi,O_C}\rightarrow \Xi_{Z_m,O_C}$, since we in fact get the identity map on the fan $\Sigma_Z$. These maps are compatible as $m$ varies. Passing to the
inverse limit gives the horizontal maps in (\ref{eqn:mIg-cusp-construction}), and hence
the desired map $g_{\widetilde Z}$. 
\end{proof}
\begin{remark}{\label{rem: monomials}}
For later use, write $g_0$ for the composition of $g$ with projection to level $K(N)$
\[g_0:\mIg^{b,\tor} \rightarrow \widehat{\mathscr{S}}(\mathbf{G},X)^\tor_{K^p}\rightarrow \widehat{\mathscr{S}}(\mathbf{G},X)^\tor_{K(N)};\]
this map lifts $f$ on the special fiber. Restrict to a cusp label $Z$, an Igusa cusp label $\tilde{Z}$ over it, let $M=S_\Phi$ and, for a cone
$\sigma\in\Sigma_Z$, put $P_\sigma=\sigma^\vee\cap M$. Note that from \eqref{eqn:mIg-cusp-construction} we have a commutative diagram 
\begin{equation*}
    \begin{tikzcd}
    \mIg^b_{\tilde{Z},\Xi,\Sigma_Z,\sigma,O_C} \arrow[d] \arrow[r] &\widehat{\Xi}_{Z,\Sigma_{Z},\sigma, O_C} \arrow[d]\\
    \mIg^b_{\tilde{Z},C_Z,O_C} \arrow [r]& \widehat{C}_{Z,O_C}, 
    \end{tikzcd}
\end{equation*}
where the left column is a toric scheme for $P_\sigma[1/p]$, and the right column a toric scheme for $P_\sigma$. In terms of moduli interpretation on the torus torsor $\mIg^b_{\tilde{Z},\Xi,O_C}$, we are restricting the symmetric lift $\iota:Y[1/p]\to \mathcal{G}$ to $Y$. In particular, this map of toric schemes is associated with a map of monoids which is the natural inclusion
\[
P_\sigma\hookrightarrow P_\sigma[1/p].
\]
Locally on $\mIg^b_{\tilde{Z},C_Z,O_C}$ and $\widehat{C}_{Z,O_C}$ we can trivialize the torus torsors, so for $m\in P_\sigma$, via the identification of $P_\sigma$ with functions on the toric variety for $P_\sigma$, we let $t^m$ denote the function on $\widehat{\Xi}_{Z,\Sigma_{Z},\sigma, O_C}$ which restricted to the geometric fibers over $\widehat{C}_{Z,O_C}$ is given by $m$. We similarly define $q^m$ the associated function on $\mIg^b_{\tilde{Z},\Xi,\Sigma_Z,\sigma,O_C}$, via the inclusion map $P_\sigma\hookrightarrow P_\sigma[1/p]$.

Observe that we have the relation
\[
g_0^*(t^m)=u_mq^m,
\]
where $u_m$ is a unit.
\end{remark}

Finally, we want to understand the inclusion of adic generic fibers $\partial\mIg^{b,\tor}_C\subset \mIg^{b,\tor}_C$. Let $\widehat{\partial\mathscr{S}}(\mathbf{G},X)^\tor_{K,O_C}$ denote the (base-change over to $O_C$ of the) closed formal subscheme of $\widehat{\mathscr{S}}(\mathbf{G},X)^\tor_{K,O_C}$ given by taking the formal completion of the boundary component $\partial\mathscr{S}(\mathbf{G},X)_K^{\tor}$ (which is flat over $O_{E_\mathfrak{p}}$) along its special fiber, then we may consider the formal subscheme
\[\tilde{\partial}\mIg^{b,\tor}_{O_C}:=\mIg^{b,\tor}_{O_C}\times_{\widehat{\mathscr{S}}(\mathbf{G},X)_{K^p,O_C}^\tor}\widehat{\partial\mathscr{S}}(\mathbf{G},X)^\tor_{K^p,O_C}. \]
Since the inclusion $\widehat{\partial\mathscr{S}}(\mathbf{G},X)^\tor_{K,O_C}\hookrightarrow \widehat{\mathscr{S}}(\mathbf{G},X)_{K,O_C}^\tor$ is an adic morphism of $p$-adic formal schemes, taking limits over $K_p$ and base-changing we see that the induced map $\tilde{\partial}\mIg^{b,\tor}_{O_C}\hookrightarrow \mIg^{b,\tor}_{O_C}$ is also an adic morphism, and in particular $\tilde{\partial}\mIg^{b,\tor}_{O_C}$ is also a $p$-adic formal scheme.

Finally, observe that by the previous proposition we know that image of an Igusa boundary component labelled by the Igusa cusp label $\tilde{Z}$ maps to a boundary component of the Shimura variety labelled by the Shimura cusp label $Z'$ as defined from $\tilde{Z}$ in the paragraph after Proposition \ref{prop:mIg-cusp-construction} , hence the perfection of the special fiber of $\tilde{\partial}\mIg^{b,\tor}_{O_C}$ is 
\begin{equation}{\label{eqn: reducedspecialfiber}}
 (\Ig^{b,\tor}\times_{\mathscr{S}(\mathbf{G},X)_{K^p,\ol{\mathbb{F}}_p}^\tor}\partial\mathscr{S}(\mathbf{G},X)^\tor_{K^p,\ol{\mathbb{F}}_p})^\mathrm{perf}\simeq \partial\Ig^{b,\tor}.
\end{equation}
\begin{proposition}
\label{prop:geom-pts-boundary}
    We have a closed inclusion of $p$-adic formal schemes
    \[\partial\mIg^{b,\tor}_{O_C}\subset\tilde{\partial}\mIg^{b,\tor}_{O_C},\]
    such that both have the same reduced special fiber. Moreover, this inclusion induces an equality of $\Spf O_{C'}$-points, for $C'$ any complete algebraically closed non-archimedean field containing $C$.
\end{proposition}
\begin{proof}

We first verify that this is a closed immersion on formal toroidal charts
over a fixed Shimura cusp label $Z$. At each finite
$p^m$-level, the corresponding Igusa cusp labels of level
$\Gamma_b(p^m)K^p(N)$ give an
open and closed decomposition of these charts
\cite[Theorem~3.2.22]{San}. It therefore suffices to check on each of these finite-level charts that the map $g$ preserves the boundary scheme-theoretically, and that these
identifications are compatible as the level varies. For the following moduli description, we choose a compatible system of Igusa
cusp labels $\widetilde Z$.

    To see this, we claim that the map
    \[g_{\tilde{Z}}:\mIg^{b,\tor}_{\tilde{Z},O_C}\rightarrow \widehat{\mathscr{S}}(\mathbf{G},X)^\tor_{K^p,Z',O_C},\]
    respects boundary components. For this, we will look at the moduli interpretations of both sides, and we want to show that the map in the middle row of \eqref{eqn:mIg-cusp-construction} respects the boundary. On the left hand side, from the previous proposition and the proof of \cite[Theorem 4.3.10]{San}, for any $p$-complete $O_C$-algebra $R$, we see that $\Spf R$-points of this are prescribed by the following data: 
    \begin{enumerate}
\item An abelian scheme $\mathcal{B}$ over $R$ with $\mathcal{O}_{B}$-action and a prime-to-$p$ polarization
\item An $\mathcal{O}_{B}$ -linear extension
\begin{equation*}
    0 \rightarrow T \rightarrow \mathcal{G} \rightarrow \mathcal{B} \rightarrow 0
\end{equation*}
where $\mathbb{X}^*(T)=X$; equivalently, an $\mathcal{O}_{B}$ -linear map $c : X \rightarrow \mathcal{B}^\vee$. 
\item An $\mathcal{O}_{B_{\mathbb{Q}_p}}$-linear isomorphism
\begin{equation*}
\rho:\mathcal{G}[p^\infty] \simeq Z_{b,-1}
\end{equation*}
that extends the isomorphism $T[p^\infty]\simeq Z_{b,-2}$ obtained by applying Cartier duality to $\tilde{\varphi}_0$, and which on $\mathcal{B}$ commutes with polarizations up to an element in $\mathbb{Z}_p^\times$. By taking a compatible choice of splitting $\delta_b$, this induces a splitting of
\begin{equation*}
0 \rightarrow T [p^\infty] \rightarrow \mathcal{G}[p^\infty] \rightarrow \mathcal{B}[p^\infty] \rightarrow 0,    
\end{equation*}
and in particular $c$ extends to a map $c:X[1/p]\rightarrow\mathcal{B}^\vee$.
\item An $\mathcal{O}_{B}$ -linear extension
\begin{equation*}
    0 \rightarrow T^\vee \rightarrow \mathcal{G}^\vee \rightarrow \mathcal{B}^\vee \rightarrow 0
\end{equation*}
where $\mathbb{X}^*(T^\vee)=Y$. Equivalently, an $\mathcal{O}_{B}$ -linear map $c^\vee : Y \rightarrow \mathcal{B}$. By the splitting $\delta_b$, as well as by duality (using that $\mathcal{B}[p^\infty]$ will be principally polarized), we have a splitting of 
\begin{equation*}
0 \rightarrow T^\vee[p^\infty] \rightarrow \mathcal{G}^\vee[p^\infty] \rightarrow \mathcal{B}^\vee[p^\infty] \rightarrow 0    
\end{equation*}
and in particular we extend $c^\vee$ to a map $c^\vee:Y[1/p]\rightarrow \mathcal{B}$.
\item  An $R$ -point of $\mathcal{P}'_{\Sigma_Z}$. Here, we note that away from the boundary we have a torsor $\mathcal{P}'$ over $R$ for the torus $E_\Phi^\mathrm{perf}$ with character group $S_\Phi[1/p]$, parametrizing lifts of $c^\vee$ to $\iota:Y[1/p] \rightarrow \mathcal{G}$. $\mathcal{P}'_{\Sigma_{Z}}\supset \mathcal{P}'$ is the toroidal embedding defined by the admissible rational polyhedral cone decomposition $\Sigma_Z$ for the cusp $(Z,\Phi)$. 
\end{enumerate}
Moreover, we see that the boundary component is given by $\Spf R$-points which in (5) lie entirely in the boundary $\mathcal{P}'_{\Sigma_Z}\backslash \mathcal{P}'$. By construction, the image of these boundary points under $g$ lies in $\Xi_{Z',\Sigma_{Z},O_C}\backslash \Xi_{Z',O_C}$, since being in the boundary of $\Xi_{Z',\Sigma_{Z},O_C}$ is also exactly determined by the data in (5) of the $R$-point of $\mathcal{P}'_{\Sigma_Z}$.

More precisely, this boundary preservation holds
scheme-theoretically. In the notation of Remark \ref{rem: monomials}, the inclusion
$P_\sigma\hookrightarrow P_\sigma[1/p]$ gives $g_0^*(t^m)=u_mq^m$ for $m\in P_\sigma$,
where $u_{m}$ is a unit. If $t^m$ vanishes along the
toroidal boundary, then $m$ is strictly positive on
every nonzero ray of $\sigma$, and $q^m$ belongs to
the ideal of the Igusa boundary. The analogous
containment of boundary ideals at level $K(p^nN)$
follows from the compatible maps of relative toroidal
embeddings constructed in Proposition~\ref{prop:mIg-cusp-construction}.
        
Here $\partial$ denotes the total toroidal boundary. When passing to formal cusp charts, we restrict it to the completion along the distinguished boundary component. Thus, for $\partial\mIg^b_{\tilde{Z},\Xi,\Sigma_Z,O_C}$ and $\partial\Xi_{Z',\Sigma_Z,O_C}:=\Xi_{Z',\Sigma_Z,O_C}\backslash \Xi_{Z',O_C}$ we have a commutative diagram
\[
\begin{tikzcd}
\partial\mIg^b_{\tilde{Z},\Xi,\Sigma_Z,O_C}\arrow[d] \arrow[r]       & \widehat{\partial\Xi}_{Z',\Sigma_Z,O_C} \arrow[d]       \\
\mIg^b_{\tilde{Z},C_{Z},O_C} \arrow[r]                  & \widehat{C}_{Z',O_C}, 
\end{tikzcd}
\]
where on the right column we take completions along the special fiber.

The quotient $\partial\mathfrak{Ig}^{b}_{\widetilde Z,\Xi,\Sigma_Z,\mathcal O_C}
/\Gamma_{\widetilde Z}$ is the flat formal boundary chart attached to the fixed
data $\widetilde Z$, and its map to
$\mathfrak{Ig}^{b,\mathrm{tor}}_{\mathcal O_C}$
factors through
$\widetilde{\partial}\mathfrak{Ig}^{b,\mathrm{tor}}_{\mathcal O_C}$
by the preceding calculation. At each
finite level the corresponding open and closed charts
cover the formal neighbourhood over $Z$. On each of these charts, the map $g$ sends the
Igusa boundary into the Shimura boundary, and
this is compatible with the transition maps and
the arithmetic group actions. Passing to the inverse limit and descending from the
toroidal charts therefore gives the required closed
inclusion
\[
\partial\mathfrak{Ig}^{b,\mathrm{tor}}_{\mathcal O_C}
\hookrightarrow
\widetilde{\partial}\mathfrak{Ig}^{b,\mathrm{tor}}_{\mathcal O_C}.
\]
The equality of reduced special fibres follows from
the identification (\ref{eqn: reducedspecialfiber}), noting that reduced closed subschemes
of a perfect scheme are perfect.

    It remains to prove the assertion on
$\operatorname{Spf}(\mathcal O_{C'})$-points. Let
\[
x:\operatorname{Spf}(\mathcal O_{C'})
\longrightarrow
\widetilde{\partial\mathfrak{Ig}}^{b,\mathrm{tor}}_{\mathcal O_C} 
\]
be such a point. By the equality of reduced special fibers already established,
the special point of $x$ lies in an Igusa boundary stratum.
We may therefore work on compatible finite-level
toroidal charts of Proposition~\ref{prop:mIg-cusp-construction}. We first check the
assertion at the generic point $x_\eta$, after lifting to
a geometric point of a toroidal chart and trivializing
the torus torsors. The resulting boundary condition
descends through the chart quotients, since the total
toroidal boundary is invariant under the quotient groups.

We use the notation of Remark \ref{rem: monomials}. Choose $a\in P_\sigma$ strictly positive
on every nonzero ray of $\sigma$. The torus in the
finite-level affine toric chart is $D(t^a)$, so its
boundary has underlying closed subset $V(t^a)$.
On the perfected special-fiber chart, the reduced boundary is
defined by the monomials
\[
q^{a/p^n}
\]
with $n \geq 0$. Its canonical formal lift is defined by the closed ideal
generated by the corresponding lifted monomials, by the
toroidal description in Proposition~2.12. 

The image of $x_\eta$ under $g_0$ lies in the target
boundary. Since
\[
g_0^*(t^a)=u_aq^a
\]
with $u_a$ a unit, we obtain $q^a(x_\eta)=0$. For every
$n\geq 0$, the identity
\[
\bigl(q^{a/p^n}\bigr)^{p^n}=q^a
\]
then implies $q^{a/p^n}(x_\eta)=0$, since the geometric
point takes values in a field. Thus $x_\eta$ factors
through the generic fiber of
$\partial\mathfrak{Ig}^{b,\mathrm{tor}}_{\mathcal O_C}$. The pullback along $x$ of the defining ideal of the formal
Igusa boundary vanishes in $C'$, hence in $\mathcal{O}_{C'}$
since $\mathcal{O}_{C'}\hookrightarrow C'$ is injective.
Thus $x$ itself factors through the boundary. Together with
the inclusion already proved, this gives the asserted equality
of $\operatorname{Spf}(\mathcal{O}_{C'})$-points.
\end{proof}

\section{Mantovan's Formula and the Hodge-Tate Period Morphism}
We now review some basic facts about the geometry of the Hodge-Tate period morphism, establishing some facts about the fiber of the Hodge-Tate period morphism in \S \ref{subsec: HodgeTatePeriodMap}. In \S \ref{sec: manprodformula}, we then apply these basic geometric facts to establish a version of Mantovan's product formula for non-compact Shimura varieties (Theorem \ref{thm: mantprodform}). This will in particular allow us to bound the degrees of the torsion cohomology of the Shimura variety using the semiperversity of certain sheaves coming from Igusa varieties and the perverse t-exactness of Hecke operators on $\Bun_{G}$. 

We recall that $(\mathbf{G},X)$ denotes a Shimura datum with reflex field $E/\bb{Q}$. We let $\Sh(\mathbf{G},X)_{K}/\Spec(E)$ denote the associated Shimura variety for a sufficiently small open compact subgroup $K \subset \mathbf{G}(\bb{A}_{f})$. We will assume throughout this section that the Shimura datum $(\mathbf{G},X)$ is of PEL type AC.

We recall that $\ell \neq p$ are distinct primes, $E_{\mf{p}}$ is the completion at the place $\mf{p}|p$ determined by a fixed choice of isomorphism $j: \ol{\bb{Q}}_{p} \simeq \bb{C}$, and $C$ is the completed algebraic closure of $E_{\mf{p}}$. The group $G := \mathbf{G}_{\bb{Q}_{p}}$ will denote the local group, and we let $K = K^{p}K_{p}$ be the decomposition into the level $K^{p} \subset \mathbf{G}(\bb{A}_{f}^{p})$ away from $p$ and the level $K_{p} \subset G(\bb{Q}_{p})$ at $p$. 
\subsection{The Hodge-Tate Period Morphism}{\label{subsec: HodgeTatePeriodMap}}
For $K \subset \mathbf{G}(\mathbb{A}_{f})$ a sufficiently small open compact, we define $\mathcal{S}(\mathbf{G},X)_{K} := (\Sh(\mathbf{G},X)_{K} \otimes_{E} E_{\mf{p}})^{\mathrm{ad}}$ to be the adic space over $\Spa(E_{\mf{p}})$ attached to the Shimura variety. For brevity, in this section we will often denote this simply by $\mathcal{S}_{K}$. When $K=K_{p}^\mathrm{hs}K^p$ with $K_p^{\mathrm{hs}}$ a hyperspecial subgroup, the space $\mathcal{S}_{K}$ has a canonical integral model $\mathscr{S}_{K}$ over $\mathcal{O}_{E,\mf{p}}$. Let $\mathcal{S}_{K}^{\circ} \subset \mathcal{S}_{K}$ be the good reduction locus, i.e., the open subspace of $\mathcal{S}_{K}$ obtained from the adic generic fiber of the $p$-adic completion $\mathscr{S}^\wedge_{K}$ of the scheme $\mathscr{S}_{K}$. We define $\mathcal{S}_{K'}^{\circ} \subset \mathcal{S}_{K'}$ for $K'\subset K$ by taking the preimage under the natural map from $\mathcal{S}_{K'}$ to $\mathcal{S}_{K}$. We also consider the adic spaces $\mathcal{S}^{*}_{K} := (\Sh(\mathbf{G},X)^{*}_{K} \otimes_{E} E_{\mf{p}})^{\mathrm{ad}}$ and $\Stor_{K} := (\Sh(\mathbf{G},X)^{\mathrm{tor}}_{K} \otimes_{E} E_{\mf{p}})^{\mathrm{ad}}$ attached to the minimal and toroidal compactifications of the Shimura variety $\Sh(\mathbf{G},X)_{K}$ over $E$, respectively, as described in \S \ref{sss:compactifications}, where, for the latter, we fix a compatible choice $\Sigma = \{\Sigma_{\Phi}\}$ of admissible smooth rational polyhedral cone decomposition data, as described in \emph{loc.cit}. We use the superscript $(-)^{\Diamond}$ for the diamond attached to these analytic adic spaces, as defined in \cite[Lecture~X]{SW}. 

Associated to the $\mathbf{G}(\mathbb{R})$-conjugacy class $X$, we have a conjugacy class of a minuscule cocharacter $\mu$ of $G_{C}$ with reflex field $E_\mf{p}$ attached to the inverse of the Hodge cocharacter of $X$. Let $\mathcal{F}\ell_{G,\mu^{-1}}$ be the adic flag variety over $\Spa(C)$ associated to the conjugacy class of $\mu^{-1}$, the dominant inverse of $\mu$ (i.e., we quotient out by the parabolic defined by the set of $g \in G_{C}$ such that $\lim_{t \ra 0} \mathrm{ad}(\mu^{-1}(t))g$ exists). Since $\mu$ is minuscule, via the Bialynicki-Birula isomorphism (\cite[Theorem~3.4.5]{CS1}), the diamond attached to the flag variety $\mathcal{F}\ell_{G,\mu^{-1}}^{\Diamond}$ represents the following functor on $\Perf_{C}$. Given any $S \in \Perf_{C}$, $\mathcal{F}\ell_{G,\mu^{-1}}^{\Diamond}(S)$ is the set of modifications of $G$-bundles $\mathcal{E} \dashrightarrow \mathcal{E}_{0,S}$ of meromorphy $\mu$ on $X_{S}$, the relative Fargues-Fontaine curve over $S$, such that the modification occurs over the degree one Cartier divisor in $X_{S}$ corresponding to the untilt of $S$ defined by the map $S \ra \Spd(C)$.

Let 
\[
    \mathcal{S}_{K^p,C}^\circ:=\lim\limits_{\substack{\longleftarrow\\ K_p}}\mathcal{S}_{K^pK_p,C}^{\circ,\Diamond} 
 \subset 
 \mathcal{S}_{K^p,C} :=\lim\limits_{\substack{\longleftarrow\\ K_p}}\mathcal{S}_{K^pK_p,C}^{\Diamond} \subset \mathcal{S}_{K^{p},C}^{\mathrm{tor}} := \lim\limits_{\substack{\longleftarrow\\ K_{p}}} \mathcal{S}^{\mathrm{tor},\Diamond}_{K^{p}K_{p},C} \ra \mathcal{S}_{K^{p},C}^{*} := \lim\limits_{\substack{\longleftarrow\\ K_{p}}} \mathcal{S}^{*,\Diamond}_{K^{p}K_{p},C}
\]
be the associated perfectoid Shimura varieties\footnote{If $(\mathbf{G},X)$ is more generally of Hodge type then this is representable by a perfectoid space by \cite[Theorem~IV.1.1]{Schtor}, and in general they have the structure of diamonds (see \cite{Han1} and \cite[Section~2]{PR23}).}, where this limit is computed in the category of sheaves on $\Perf_{C}$ with respect to the $v$-topology on $\Perf_{C}$, as defined in \cite[Definition~8.1]{Ecod} and we recall that any diamond defines such a $v$-sheaf in light of \cite[Proposition~11.9]{Ecod}. Throughout, we identify adic spaces attached to perfectoids and their associated $v$-sheaves under the Yoneda embedding, recalling that the $v$-topology is sub-canonical, by \cite[Theorem~1.2]{Ecod}.

We also consider $\ol{\mathcal{S}}^{\circ}_{K^{p},C}$, the canonical compactification of the good reduction locus over $\Spd(C)$ in the sense of \cite[Proposition~18.6]{Ecod}. This will be a subspace of $\mathcal{S}_{K^{p},C}$, since $\mathcal{S}_{K^{p},C}$ is partially proper. Caraiani-Scholze \cite[\S2.1]{CS1} consider the Hodge-Tate period morphism on $\mathcal{S}_{K^p}$
\[
\pi_{\mathrm{HT}}: \mathcal{S}_{K^{p},C} \ra \mathcal{F}\ell_{G,\mu^{-1}}^{\Diamond}, \]
which records the relative position of the Hodge-Tate filtration associated with the abelian varieties that $\mathcal{S}_{K^{p},C}$ parametrizes. As explained in \cite[Theorem 2.7.1]{CS2} and \cite[Theorem 4.4.40]{BoxerPilloni2021}, this extends to a Hodge-Tate period morphism on the minimal compactification
\[
\pi_{\mathrm{HT}}^{*}: \mathcal{S}_{K^{p},C}^{*} \ra \mathcal{F}\ell_{G,\mu^{-1}}^{\Diamond}
\]
and toroidal compactification
\[
\pi_{\mathrm{HT}}^{\mathrm{tor}}: \mathcal{S}_{K^{p},C}^{\mathrm{tor}} \ra \mathcal{F}\ell_{G,\mu^{-1}}^{\Diamond}.
\]
We write $\pi_{\mathrm{HT}}^{\circ}$ for the restriction to the good reduction locus, and $\ol{\pi}_{\mathrm{HT}}^{\circ}$ for the canonical compactification of $\pi_{\mathrm{HT}}^{\circ}$, where we note that this again maps to $\mathcal{F}\ell_{G,\mu^{-1}}^{\Diamond}$ as this is proper over $\Spa(C)$. 

These maps have the following properties:
\begin{enumerate}
\item $\overline{\pi}_{\mathrm{HT}}^{\circ}$, $\pi_{\mathrm{HT}}^{*}$, and $\pi_{\mathrm{HT}}^{\mathrm{tor}}$ are partially proper and qcqs; hence, proper.
\item $\pi_{\mathrm{HT}}$ is partially proper, but not always qcqs (Unless the Shimura variety is compact).
\item $\pi_{\mathrm{HT}}^{\circ}$  is qcqs, but not partially proper (Unless the Shimura variety is compact). 
\end{enumerate}
With these properties in mind, let us  study the fibers of these maps. For our purposes, we will focus on the compactly supported cohomology of $\mathcal{S}_{K^{p},C}$ and in turn the sheaf $R\pi_{\mathrm{HT}!}(\ol{\mathbb{F}}_{\ell})$ on $\mathcal{F}\ell_{G,\mu^{-1}}^{\Diamond}$.

Our goal is to describe the stalks of $R\pi_{\mathrm{HT}!}(\ol{\mathbb{F}}_{\ell})$ at a geometric rank $1$ point $x: \Spa(C,O_{C}) \ra \mathcal{F}\ell^{\Diamond}_{G,\mu^{-1}}$. We assume the geometric point $x$ factors through the adic Newton strata $\mathcal{F}\ell_{G,\mu^{-1}}^{b,\Diamond}$ for $b \in B(G,\mu)$, and choose a completely slope divisible $p$-divisible group $\mathbb{X}_{b}$ over $\ol{\mathbb{F}}_{p}$ corresponding to $b \in B(G,\mu)$. Let $\Ig^{b}$ be the associated perfect Igusa variety as defined in \S\ref{ss: defIgusavarieties}, with toroidal compactification $\Ig^{b,\tor}$ and minimal compactification $\Ig^{b,*}$. Recall that we have associated perfectoid Igusa varieties, $\mIg^b_{C}$,$\mIg^{b,\tor}_{C}$, and $\mIg^{b,*}_{C}$, which should model the fibers of $\ol{\pi}_{\mathrm{HT}}^{\circ}$, $\pi_{\mathrm{HT}}^{\mathrm{tor}}$, and $\pi^{*}_{\mathrm{HT}}$, respectively. We let $\partial\mIg_{C}^{b,*}$ and $\partial\mIg_{C}^{b,\mathrm{tor}}$ be the Zariski closed subspaces attached to the boundaries $\partial\Ig^{b,*}$ and $\partial\Ig^{b,\mathrm{tor}}$, respectively. 

Let $g_{b}: \Ig^{b} \hookrightarrow \Ig^{b,*}$ be the natural open immersion of $\ol{\mathbb{F}}_p$-schemes. We define the partially compactly supported cohomology
\[ R\Gamma_{c-\partial}(\Ig^{b},\ol{\mathbb{F}}_{\ell}) := R\Gamma(\Ig^{b,*},g_{b!}(\ol{\mathbb{F}}_{\ell})). \]
Our goal is to show that this computes the fibers of $R\pi_{\mathrm{HT}!}(\ol{\mathbb{F}}_{\ell})$ at geometric points of $\mathcal{F}\ell_{G,\mu^{-1}}^{b,\Diamond}$. 

To get a clearer picture of how these spaces interact with each other, we have the following theorem.
\begin{theorem}{\cite{CS2,San}}{\label{thm: fiberdescrip}}
For a geometric rank $1$ point $x: \Spa(C,O_{C}) \ra \mathcal{F}\ell_{G,\mu^{-1}}^{\Diamond}$, there exists a diagram of spaces of the form 
\[ \begin{tikzcd}
& (\pi_{\mathrm{HT}}^{\circ})^{-1}(x) \arrow[r,hook] & \pi_{\mathrm{HT}}^{-1}(x) \arrow[r,hook] 
 & (\pi_{\mathrm{HT}}^{\mathrm{tor}})^{-1}(x) \arrow[r] & (\pi_{\mathrm{HT}}^{*})^{-1}(x) \\
  & \mIg^{b}_{C} \arrow[u,hook,"i"] \arrow[rr,hook] & & \mIg^{b,\mathrm{tor}}_{C} \arrow[u,hook,"\phantom{}^{\mathrm{tor}}i"]\arrow[r] & \mIg^{b,*}_{C},  \arrow[u,hook,"\phantom{}^{*}i"]\\
\end{tikzcd}
\]
where the maps $i$, $\phantom{}^{*}i$, and $\phantom{}^{\mathrm{tor}}i$ are open immersions whose image contains all rank $1$ points. Moreover, the fibers $(\pi_{\mathrm{HT}}^{*})^{-1}(x)$ and $(\pi_{\mathrm{HT}}^{\mathrm{tor}})^{-1}(x)$ are partially proper, so in particular $\phantom{}^{*}i$ and $\phantom{}^{\mathrm{tor}}i$ are canonical compactifications in the sense of \cite[Proposition~18.6]{Ecod}. 
\end{theorem}
\begin{proof}
This theorem in the case where $(\mathbf{G},X)$ is of PEL type $A$ attached to a globally quasi-split unitary group of even dimension is \cite[Theorems~2.7.2,Theorem~4.5.1]{CS2}, and the general case of PEL type AC is proven in \cite[Theorems~4.3.10,4.3.12]{San}. 
\end{proof}
We will also combine this with the following result.
\begin{theorem}{\cite{CS2,San}}{\label{thm: affineness}}
The partially minimally compactified Igusa variety $\Ig^{b,*}$ is affine; in particular, the attached adic space $\mIg^{b,*}_{C}$ is affinoid perfectoid. Moreover, there exists a proper map  
\[ \mIg^{b,\mathrm{tor}}_{C} \ra \mIg^{b,*}_{C}, \]
which induces an isomorphism on global sections of the structure sheaves $\mathcal{O}_{\mIg^{b,\mathrm{tor}}_{C}}$ and  $\mathcal{O}_{\mIg^{b,*}_{C}}$ of the adic spaces. 
\end{theorem}
\begin{proof}
For the affineness, the case of PEL type $A$ attached to a global quasi-split unitary group of even dimension is covered by \cite[Theorem~1.7]{CS2} and \cite[Lemma~4.5.2]{CS2}. The general case of PEL type AC is covered in \cite[Lemma~3.3.7]{San}. To see the properness, we note that the map
\[ \mIg_{C}^{b,\mathrm{tor}} \ra \mIg_{C}^{b,*} \]
is the one appearing in the Stein factorization described in \cite[Proposition~3.3.4]{CS2} and \cite[Proposition~3.3.5]{San}. This also shows that one has an isomorphism on global sections of the structure sheaves. 
\end{proof}
We will also need the following corollary.
\begin{corollary}{\label{cor: quasicompact}}
The boundary $\partial\mIg_{C}^{b,\mathrm{tor}}$ is quasi-compact and $\partial\mIg_{C}^{b,*}$ is affinoid perfectoid (In particular, quasi-compact).
\end{corollary}
\begin{proof}
The fact that $\mIg_{C}^{b,*}$ is affinoid perfectoid follows from the previous Theorem. Moreover, $\partial\mIg_{C}^{b,*}$ is a Zariski closed subspace, since it came from considering the adic generic fiber of a formal model of the perfect closed subscheme $\partial\Ig^{b,*} \subset \Ig^{b,*}$, and a Zariski closed subspace of a perfectoid is still perfectoid by \cite[Remark~7.5]{BSPrism}. The claim for the toroidal compactification follows since the map $\mIg_{C}^{b,\mathrm{tor}} \ra \mIg_{C}^{b,*}$ is proper, and maps the boundary $\partial\mIg_{C}^{b,\mathrm{tor}}$ to $\partial\mIg_{C}^{b,*}$.  
\end{proof}
It is natural to wonder how one could describe the fiber of $\pi_{\mathrm{HT}}$ in terms of the spaces described above. In particular, we now deduce the following corollary. 
\begin{corollary}{\label{cor: openhodgetatefibers}}
 For $x: \Spa(C,O_{C}) \ra \mathcal{F}\ell_{G,\mu^{-1}}^{b,\Diamond}$ a geometric point, we have isomorphisms 
 \[ \pi_{\mathrm{HT}}^{-1}(x) \simeq \ol{\mIg}_{C}^{b,*} \setminus \ol{\partial\mIg}_{C}^{b,*} \simeq \ol{\mIg}_{C}^{b,\mathrm{tor}} \setminus \ol{\partial\mIg}_{C}^{b,\mathrm{tor}} \]
 induced by the natural open immersions $\pi_{\mathrm{HT}}^{-1}(x) \hookrightarrow (\pi_{\mathrm{HT}}^{*})^{-1}(x) \simeq \ol{\mIg}_{C}^{b,*}$ (resp. $\pi_{\mathrm{HT}}^{-1}(x) \hookrightarrow (\pi_{\mathrm{HT}}^{\mathrm{tor}})^{-1}(x) \simeq \ol{\mIg}_{C}^{b,\mathrm{tor}}$), as given by Theorem \ref{thm: fiberdescrip}. Here $\ol{\partial\mIg}_{C}^{b,*}$
(resp. $\ol{\partial\mIg}_{C}^{b,\mathrm{tor}}$)
is the Zariski closed subset of $\ol{\mIg}_{C}^{b,*}$
(resp. $\ol{\mIg}_{C}^{b,\mathrm{tor}}$)
defined by the canonical compactification of the boundary
$\partial\mIg^{b,*}_{C} \subset \mIg^{b,*}_{C}$
(resp. $\partial\mIg_{C}^{b,\mathrm{tor}}
\subset \mIg^{b,\mathrm{tor}}_{C}$).
\end{corollary}
\begin{proof}
We first establish the claim for the toroidal compactification. We consider the closed immersion 
\[ \ol{\partial\mIg}_{C}^{b,\mathrm{tor}} \hookrightarrow  \ol{\mIg}_{C}^{b,\mathrm{tor}} \times_{\Stor_{K^{p},C}} \partial\Stor_{K^{p},C} \]
induced by the inclusion of Proposition~\ref{prop:geom-pts-boundary} and
 applying Theorem \ref{thm: fiberdescrip}. To show the desired claim, it suffices to show this is an isomorphism. Both sides are qcqs, by Corollary~\ref{cor: quasicompact}, Theorem~\ref{thm: affineness},
and the properness of canonical compactifications of
qcqs spaces \cite[Corollary~18.8(vi)]{Ecod}.
Moreover, by Theorem~\ref{thm: fiberdescrip},
$\overline{\mathfrak{Ig}}^{b,\mathrm{tor}}_C$
identifies with the geometric fiber of
$\pi_{\mathrm{HT}}^{\mathrm{tor}}$, which is perfectoid.
Moreover, both sides of the displayed closed immersion are
Zariski closed subspaces of this fiber, so they
are perfectoid by \cite[Remark~7.5]{BSPrism}. Moreover, both the LHS and RHS are partially proper; therefore, to show this map is an isomorphism, it suffices to show it induces an isomorphism on rank one geometric points using \cite[Lemma~5.4]{Ecod} (where we note that condition (iv) of \cite[Lemma~5.4]{Ecod} will be implied by the claim on rank one points and the partial properness). In particular, given $C'/C$ a complete algebraically closed non-archimedean field, we claim that there exists a Cartesian diagram of the form
\[ \begin{tikzcd}
\partial\mIg^{b,\mathrm{tor}}_{C}(C',\mathcal{O}_{C'}) \arrow[r] \arrow[d] &   \mIg^{b,\mathrm{tor}}_{C}(C',\mathcal{O}_{C'}) \arrow[d] & \\
\partial\Stor_{K^{p},C}(C',\mathcal{O}_{C'}) \arrow[r] & \Stor_{K^{p},C}(C',\mathcal{O}_{C'}). &
\end{tikzcd}. \]
Observe that, since both $\widehat{\mathscr{S}}(\mathbf{G},X)_{K^p}^\tor$ and the boundary $\widehat{\partial\mathscr{S}}(\mathbf{G},X)_{K^p}^\tor$ are proper (where notation is as in \S \ref{sss:boundarycomponents}), we have that the $\Spa(C',O_{C'})$ points of the adic generic fiber $\widehat{\mathscr{S}}(\mathbf{G},X)_{K^p}^\tor$ and $\Stor_{K^p}$ are the same, and similarly for the adic generic fiber of $\widehat{\partial\mathscr{S}}(\mathbf{G},X)^\tor_{K^p}$ and $\partial\Stor_{K^{p},C}$. Moreover, the $\Spa(C',O_{C'})$-points of $\widehat{\mathscr{S}}(\mathbf{G},X)_{K^p}^\tor$ are also the same as the $\Spf O_{C'}$-points of the formal scheme $\widehat{\mathscr{S}}(\mathbf{G},X)_{K^p,O_C}^\tor$ . A similar result also holds for the boundary. Thus, we can apply Proposition \ref{prop:geom-pts-boundary} to conclude, since the space $\tilde{\partial}\mIg^{b,\tor}_{O_C}$ considered there is defined as the fiber product of formal schemes $\mIg^{b,\tor}_{O_C}\times_{\widehat{\mathscr{S}}(\mathbf{G},X)_{K^p,O_C}^{\tor}}\widehat{\partial\mathscr{S}}(\mathbf{G},X)^\tor_{K^p,O_C}$.

It remains to see the analogous claim for the minimal compactification. This follows easily using Theorem \ref{thm: fiberdescrip} and the fact that under the proper surjective map $\mIg_{C}^{b,\mathrm{tor}} \ra \mIg_{C}^{b,*}$ the preimage of  $\partial\mIg_{C}^{b,*}$ is given by $\partial\mIg_{C}^{b,\mathrm{tor}}$, by construction. 
\end{proof}    
We have the following corollary.
\begin{corollary}{\label{cor: stalkspartial}}
For a geometric point $x:\Spa(C,O_{C}) \ra \mathcal{F}\ell_{G,\mu^{-1}}^{b,\Diamond}$, we have an identification: 
\[ R\Gamma_{c-\partial}(\Ig^{b},\ol{\mathbb{F}}_{\ell}) \simeq R\pi_{\mathrm{HT}!}(\ol{\mathbb{F}}_{\ell})_{x}. \]
\end{corollary}
\begin{proof}
We have an identification $\pi_{\mathrm{HT}}^{-1}(x) \simeq \ol{\mIg}^{b,*}_{C} \setminus \ol{\partial\mIg}^{b,*}_{C}$ by the previous corollary, so, by proper base-change, we are tasked with computing the compactly supported cohomology of this space. We note, by Theorem \ref{thm: affineness}, the adic spaces $\mIg^{b,*}_{C}$ are affinoid perfectoid. It follows that the canonical compactification $\ol{\mIg}^{b,*}_{C} \simeq (\pi_{\mathrm{HT}}^{*})^{-1}(x)$ is also affinoid perfectoid, by \cite[Proposition~18.7 (iv)]{Ecod}. In particular, it is quasi-compact and partially proper, so in particular proper. It therefore follows by excision\footnote{One easily checks that the excision sequence is exact on geometric points, and this is sufficient by \cite[Proposition~14.3]{Ecod}.} that we have a distinguished triangle
\begin{equation}{\label{eqn: exctriangle}}
 R\Gamma_{c}(\pi_{\mathrm{HT}}^{-1}(x),\ol{\mathbb{F}}_{\ell}) \ra R\Gamma(\ol{\mIg}^{b,*}_{C},\ol{\mathbb{F}}_{\ell}) \ra R\Gamma(\ol{\partial\mIg}^{b,*}_{C},\ol{\mathbb{F}}_{\ell}) \xrightarrow{+1} .
\end{equation}
We observe that $k: \mIg^{b,*} \hookrightarrow \ol{\mIg}^{b,*}_{C}$ is a qcqs open immersion of perfectoid spaces inducing an isomorphism on rank $1$ points, and it follows that the same is true for the induced map on the Zariski closed subspaces $\phantom{}^{\partial}k: \partial\mIg^{b,*} \hookrightarrow \ol{\partial\mIg}^{b,*}_{C}$. Therefore, we can apply \cite[Lemma~4.4.2]{CS1}, this tells us that the natural maps
\[ \ol{\mathbb{F}}_{\ell} \ra Rk_{*}(\ol{\mathbb{F}}_{\ell}) \]
\[ \ol{\mathbb{F}}_{\ell} \ra R\phantom{}^{\partial}k_{*}(\ol{\mathbb{F}}_{\ell}) \]
are isomorphisms, giving identifications $R\Gamma(\ol{\mIg}_{C}^{b,*},\ol{\mathbb{F}}_{\ell}) \simeq R\Gamma(\mIg^{b,*}_{C},\ol{\mathbb{F}}_{\ell})$ and $R\Gamma(\ol{\partial\mIg}_{C}^{b,*},\ol{\mathbb{F}}_{\ell}) \simeq R\Gamma(\partial\mIg^{b,*}_{C},\ol{\mathbb{F}}_{\ell})$. Now, by \cite[Lemma~4.4.3]{CS1}, we have further identifications of $R\Gamma(\partial\mIg^{b,*}_{C},\ol{\mathbb{F}}_{\ell})$ and $R\Gamma(\mIg^{b,*}_{C},\ol{\mathbb{F}}_{\ell})$ with the cohomology of the perfect schemes $\partial\Ig^{b,*}$ and $\Ig^{b,*}$, respectively. Substituting this into the triangle (\ref{eqn: exctriangle}), we get a distinguished triangle 
\begin{equation*}
 R\Gamma_{c}(\pi_{\mathrm{HT}}^{-1}(x),\ol{\mathbb{F}}_{\ell}) \ra R\Gamma(\Ig^{b,*},\ol{\mathbb{F}}_{\ell}) \ra R\Gamma(\partial\Ig^{b,*},\ol{\mathbb{F}}_{\ell}) \xrightarrow{+1}.  
\end{equation*}
By applying quasi-compact base-change \cite[Proposition~17.6]{Ecod} and then using that the inclusion $\partial\mIg_{C}^{b,*} \subset \mIg_{C}^{b,*}$ is induced from taking the rigid generic fiber over $C$ of Witt vectors applied to $\partial\Ig^{b,*} \subset \Ig^{b,*}$, we identify the last map with the natural  restriction map on the cohomology. However, this identifies the first term with precisely the partially compactly supported cohomology, as desired. 
\end{proof}
We will combine this with the following Proposition, which already hints at our expectation that $R\pi_{\mathrm{HT}!}(\ol{\mathbb{F}}_{\ell})$ is connective in some suitable perverse $t$-structure. 
\begin{proposition}{\label{prop: artvanish}}
Let $d_{b} := \langle 2\rho_{G},\nu_{b} \rangle $, where $\nu_{b}$ is the slope homomorphism of $b \in B(G,\mu)$ and $\rho_{G}$ is the half sum of all positive roots. We have $d_b= \dim(\Ig^{b,*}) = \dim(\Ig^{b})$, and the cohomology of the complex 
\[ R\Gamma_{c-\partial}(\Ig^{b},\ol{\mathbb{F}}_{\ell}) \simeq R\pi_{\mathrm{HT}!}(\ol{\mathbb{F}}_{\ell})_{x} \]
is concentrated in degrees $\leq d_{b}$.
\end{proposition}
\begin{proof}
By \cite[Corollary~7.8]{Ham2015}, the central leaf
$\mathscr{C}_{\mathbb{X}_b}$ has dimension $d_b$.
The natural morphism factors as
\[
\Ig^b \longrightarrow
\mathscr{C}_{\mathbb{X}_b}^{\mathrm{perf}}
\longrightarrow \mathscr{C}_{\mathbb{X}_b}.
\]
The first arrow is a pro-finite \'etale torsor under
$\underline{\mathrm{Aut}_G(\mathbb{X}_b)(\ol{\mathbb{F}}_p)}$ endowed with its profinite topology
(see the discussion preceding Definition~3.2.25 in \cite{San}).
The second arrow is the perfection morphism, which is a
universal homeomorphism. Since the first arrow is integral
and surjective, we obtain
\[
\dim(\Ig^b)
=
\dim(\mathscr{C}_{\mathbb{X}_b}^{\mathrm{perf}})
=
\dim(\mathscr{C}_{\mathbb{X}_b})
=
d_b.
\]
The equality $\dim(\Ig^{b,*})=d_b$ follows from the
finite-level normalizations described below.

We also saw in Theorem \ref{thm: affineness} that $\Ig^{b,*}$ is an affine scheme. So we would like to apply Artin vanishing \cite[Expos\'e XIV, Corrolary 3.2]{SGA4}; however, $\Ig^{b,*}$ is also a perfect scheme so not in general of finite type. Recall from \cite[Proposition 4.3.8]{CS1} that $\Ig^b$ is obtained as the perfection of the limit of the finite \'etale covers 
\[ \Ig^{b}_{m} \ra \mathscr{C}_{\mathbb{X}_{b}}, \]
where $\Ig^{b}_{m}$ is as defined in \cite[Remark 2.3.6]{CS2}. These spaces are of finite type over $\ol{\mathbb{F}}_{p}$. Define $\Ig^{b,*}_{m}$ to be the normalization of $\mathscr{C}_{\mathbb{X}}^{*}$ in $\Ig^{b}_{m}$ of the finite \'etale cover $\Ig^{b}_{m} \ra \mathscr{C}_{\mathbb{X}_{b}}$. By \cite[Theorem~3.3.3]{San},  $\mathscr{C}_{\mathbb{X}}^{*}$ is affine of dimension $d_b$; therefore, it follows that $\Ig^{b,*}_{m}$ is a normal and affine scheme of dimension $d_b$ which will be of finite type since $\Ig^{b}_{m}$ is. Let $g_{b,m}: \Ig^{b}_{m} \ra \Ig^{b,*}_{m}$ be the natural map, then Artin vanishing implies that the cohomology of the sheaf $g_{b,m!}(\ol{\mathbb{F}}_{\ell})$ is concentrated in degrees $\leq d_b$. By definition of $\Ig^{b,*}$, it is the perfection of $\lim_{m \geq 1} \Ig_{m}^{b,*}$. Therefore, since \'etale cohomology is invariant under perfection, we can conclude by applying \cite[Tag~09YQ]{stacks-project} to the system of sheaves $g_{b,m!}(\ol{\mathbb{F}}_{\ell})$.
\end{proof}
Now we would like to link this analysis with the semi-perversity of certain sheaves on $\Bun_{G}$ via a version of the Mantovan product formula.
\subsection{Mantovan's Product Formula}{\label{sec: manprodformula}}
We consider the Hodge-Tate period morphism 
\[ \pi_{\mathrm{HT}}: [\mathcal{S}_{K^{p}}/\ul{G(\mathbb{Q}_{p})}] \ra [\mathcal{F}\ell_{G,\mu^{-1}}^{\Diamond}/\ul{G(\mathbb{Q}_{p})}] \]
from the previous section quotiented out by $G(\mathbb{Q}_{p})$, where we abusively use $\pi_{\mathrm{HT}}$ for this quotiented out map. We let $h^{\ra}: [\mathcal{F}\ell_{G,\mu^{-1}}^{\Diamond}/\ul{G(\mathbb{Q}_{p})}] \ra [\Spd(C)/\ul{G(\mathbb{Q}_{p})}]$ be the structure map quotiented out by $G(\mathbb{Q}_{p})$. Note this is a proper map, since $\mathcal{F}\ell_{G,\mu^{-1}}^{\Diamond}$ is proper over $\Spd(C)$.

Then we have an identification
\begin{equation}
\label{eq:isom1}
    R\Gamma_c(\mathcal{S}_{K^{p},C},\ol{\mathbb{F}}_{\ell})\simeq Rh^{\ra}_{*}R\pi_{\mathrm{HT}!}(\ol{\mathbb{F}}_{\ell})
\end{equation}
of complexes of smooth $G(\mathbb{Q}_{p})$-representations under the equivalence $\D([\Spd(C)/\ul{G(\bb{Q}_{p})}],\ol{\bb{F}}_{\ell}) \simeq \D(G(\bb{Q}_{p}),\ol{\bb{F}}_{\ell})$ of \cite[Theorem~V.1.1]{FS}.
\begin{remark}{\label{rem: cohomologyatinfty}}
We note that it is always true that the compactly supported cohomology at infinite level is the colimit of the compactly supported cohomology at finite levels, but, for usual cohomology one needs to assume the spaces are qcqs for this to be true (e.g the tower defined by the good reduction locus or the minimal/toroidal compactifications).
\end{remark}

Similarly, we have a map  
\[ h^{\la}: [\mathcal{F}\ell_{G,\mu^{-1}}^{\Diamond}/\ul{G(\mathbb{Q}_{p})}] \ra \Bun_{G,C} \]
remembering the isomorphism class of the bundle $\mathcal{E}$ in the moduli interpretation of $\mathcal{F}\ell_{G,\mu^{-1}}^{\Diamond}$ as a diamond described in \S \ref{subsec: HodgeTatePeriodMap}. Here $\Bun_{G,C}$ is the base-change of $\Bun_{G}$ along $\Spd(C) \ra \Spd(\ol{\bb{F}}_{p})$, where $\Bun_{G}$ is the moduli $v$-stack on $\Perf$ parameterizing $G$-bundles on the Fargues-Fontaine curve.

The morphism $h^{\la}$ is a cohomologically smooth map by \cite[Theorem~IV.1.19]{FS}, and  by \cite[Proposition~A.9]{RpadicEtaleApenndix} the image identifies with the open subset $B(G,\mu) \subset B(G)$ under the identification $|\Bun_{G}| \simeq B(G)$ of topological spaces \cite{Vi}, where $|\Bun_{G}|$ denotes the underlying topological space of the $v$-stack $\Bun_{G}$ and $B(G)$ has the topology given by its natural partial ordering (i.e it is generated by the opens $U_{b} := \{x \in B(G)| x \leq b\}$, where $x \leq b$ means that their Kottwitz invariants $\kappa(x) = \kappa(b)$ agree and the differences of slopes $\nu_{b} - \nu_{x}$ is given by a $\mathbb{Q}_{\geq 0}$-linear combination of positive coroots (see \cite[Section~I.5]{FS} for details)). For each $b$, we have a locally closed Harder-Narasimhan stratum $j_{b}: \Bun_{G,C}^{b} \hookrightarrow \Bun_{G,C}$, and we can define the locally closed subset  $[\mathcal{F}\ell_{G,\mu^{-1}}^{b,\Diamond}/\ul{G(\mathbb{Q}_{p})}]$, by pulling back this HN-strata along $h^{\la}$. This defines a locally closed stratification of $[\mathcal{F}\ell_{G,\mu^{-1}}^{\Diamond}/\ul{G(\mathbb{Q}_{p})}]$. Let $i_b: [\mathcal{F}\ell_{G,\mu^{-1}}^{b,\Diamond}/\ul{G(\mathbb{Q}_{p})}] \hookrightarrow [\mathcal{F}\ell_{G,\mu^{-1}}^{\Diamond}/\ul{G(\mathbb{Q}_{p})}]$ denote the associated locally closed immersion. We write $\pi_{\mathrm{HT}}^{b}: [\mathcal{S}^{b}_{K^{p},C}/\ul{G(\mathbb{Q}_{p})}] \ra [\mathcal{F}\ell_{G,\mu^{-1}}^{b,\Diamond}/\ul{G(\mathbb{Q}_{p})}]$ (resp. $\pi_{\mathrm{HT}}^{b,*}: [\mathcal{S}^{b,*}_{K^{p},C}/\ul{G(\mathbb{Q}_{p})}] \ra [\mathcal{F}\ell_{G,\mu^{-1}}^{b,\Diamond}/\ul{G(\mathbb{Q}_{p})}]$) for the pullbacks of $\pi_{\mathrm{HT}}$ (resp. $\pi^{*}_{\mathrm{HT}}$) along $i_{b}$. On the good reduction locus, we also have an additional stratification coming from pulling back the Newton stratification on the special fiber along the specialization map. We recall the rather subtle point that this does not agree with the pullback of the locally closed strata $\mathcal{F}\ell_{G,\mu^{-1}}^{b,\Diamond}$ (namely, the closure relationships are opposite with respect to the partial ordering on $B(G)$). We write $\mathcal{S}^{b,\circ,\mathrm{rd}}_{K^{p},C}$ for these Newton strata coming from the special fiber. There exists a natural map
\begin{equation}{\label{eqn: comparisonofadicandgenericNewtonstratification}}
 \mathcal{S}^{b,\circ,\mathrm{rd}}_{K^{p},C} \times_{\mathcal{F}\ell_{G,\mu^{-1}}^{\Diamond}} \mathcal{F}\ell_{G,\mu^{-1}}^{b,\Diamond} \hookrightarrow \mathcal{S}^{b,\circ}_{K^{p},C}
\end{equation}
which is a qcqs open immersion containing all rank $1$ points of the target (cf. \cite[Page~8]{Ko}). In particular, the fact that this is a qcqs open immersion can be seen by using that the specialization map applied to the integral model at hyperspecial level is a spectral map of spectral topological spaces (\cite[Proposition~4.2]{GleasonSpecializationMaps}), and therefore it is a qcqs map of topological spaces. In particular, by combining with \cite[Proposition~11.19]{Ecod}, this tells us that $\mathcal{S}^{\geq b,\circ,\mathrm{rd}}_{K^{p},C} \hookrightarrow \mathcal{S}^{\circ}_{K^{p},C}$, the open subdiamond defined by the preimage of the open $\geq b$ Newton strata in the special fiber is a qcqs open as a $v$-sheaf (see the beginning of \cite[Section~6]{Ko}). One then concludes by showing that (\ref{eqn: comparisonofadicandgenericNewtonstratification}) is obtained by the base-change of this qcqs open immersion along the inclusion $\mathcal{S}^{b,\circ}_{K^{p},C} \hookrightarrow \mathcal{S}^{\circ}_{K^{p},C}$ of the adic Newton strata. This follows from interpreting the strata $\mathcal{S}^{\geq b,\circ,\mathrm{rd}}_{K^{p},C}$ (resp.  $\mathcal{S}^{b,\circ}_{K^{p},C}$) as the preimage of the locus at hyperspecial level defined by geometric points $\Spa(C,C^{+}) \ra \mathcal{S}^{\circ}_{K^{p}K_{p}^{\mathrm{hs}}}$ where the special point (resp. generic point) of the induced map $\mathrm{Spf}(C^{+}) \ra \mathscr{S}_{K}$ to the integral model lies in the open (resp. locally closed) locus of the topological space of the special fiber $|\mathscr{S}_{K^{p}K_{p}^{\mathrm{hs}},\ol{\bb{F}}_{p}}|$ defined by $|\mathscr{S}_{K^{p}K_{p}^{\mathrm{hs}},\ol{\bb{F}}_{p}}^{\geq b}|$ (resp. $|\mathscr{S}_{K^{p}K_{p}^{\mathrm{hs}},\ol{\mathbb{F}}_{p}}^{b}|$). 

We write $\pi_{\mathrm{HT}}^{b,\circ}: [(\mathcal{S}^{b,\circ,\mathrm{rd}}_{K^{p},C} \times_{\mathcal{F}\ell^{\Diamond}_{G,\mu^{-1}}} \mathcal{F}\ell_{G,\mu^{-1}}^{b,\Diamond})/\ul{G(\mathbb{Q}_{p})}] \ra [\mathcal{F}\ell_{G,\mu^{-1}}^{b,\Diamond}/\ul{G(\mathbb{Q}_{p})}]$ for the Hodge-Tate period map on this locus, and similarly we write $\ol{\pi}_{\mathrm{HT}}^{b,\circ}: [\ol{\mathcal{S}}^{b,\circ}_{K^{p},C}/\ul{G(\mathbb{Q}_{p})}] \ra [\mathcal{F}\ell_{G,\mu^{-1}}^{b,\Diamond}/\ul{G(\mathbb{Q}_{p})}]$ for the induced map on the canonical compactification, where we note that this agrees with the canonical compactification of $[(\mathcal{S}^{b,\circ,\mathrm{rd}}_{K^{p},C} \times_{\mathcal{F}\ell^{\Diamond}_{G,\mu^{-1}}} \mathcal{F}\ell_{G,\mu^{-1}}^{b,\Diamond})/\ul{G(\mathbb{Q}_{p})}]$ by the above discussion and the fact that $\mathcal{F}\ell_{G,\mu^{-1}}^{b,\Diamond}$ is partially proper\footnote{The partial properness of these strata follows directly from the moduli interpretation, since the category of vector bundles on the Fargues-Fontaine curve is insensitive to the choice of $R^{+}$ for a Tate Huber pair $(R,R^{+})$.}.

Define the group diamond $\mathcal{J}_{b} := \mathrm{Aut}(\mathcal{E}_{b})$ on $\Perf_{C}$, as in \cite[Proposition~III.5.1]{FS} parametrizing automorphism of the $G$-bundle $\mathcal{E}_{b}$ on the Fargues-Fontaine curve attached to $b$. We have an isomorphism $j_{b}: \Bun_{G,C}^{b} \simeq [\Spd(C)/\mathcal{J}_{b}] \hookrightarrow \Bun_{G,C}$ with the locally closed HN-strata in $\Bun_{G}$ defined by $b$. There is a $\mathcal{J}_{b}$-torsor over the adic Newton strata $\mathcal{F}\ell_{G,\mu^{-1}}^{b,\Diamond}$ given by rigidifying the bundle $\mathcal{E}$ to be isomorphic to $\mathcal{E}_{b}$. This gives a map
\begin{equation}{\label{eqn: NewtonLeftHecke}}
    h^{\leftarrow}_{b}: [\mathcal{F}\ell_{G,\mu^{-1}}^{b,\Diamond}/\ul{G(\mathbb{Q}_{p})}] \rightarrow [\Spd(C)/\mathcal{J}_b] \simeq \Bun_{G,C}^{b}
\end{equation}
such that $h^{\la} \circ i_{b} = j_{b} \circ h^{\la}_{b}$.

The diamond attached to the perfectoid Igusa variety $\mIg^{b}_{C}$ comes equipped with an action of $\mathcal{J}_{b}$. In particular, we consider the formal group scheme $\mathrm{Aut}_{G}(\tilde{\mathbb{X}}_{b})$ over $\mathrm{Spf}(W(\ol{\bb{F}}_{p}))$ described in \cite[Definition~4.2.9]{CS1} parameterizing automorphisms of the universal cover $\tilde{\bb{X}}_{b}$ of $\bb{X}_{b}$. We write $\Aut_{G}(\tilde{\bb{X}}_{b})_{\ol{\bb{F}}_{p}}$ for the special fiber of this formal group scheme, which appeared in \ref{ss: defIgusavarieties}. This defines a formal group over $\ol{\bb{F}}_{p}$, which we regard as an fpqc sheaf on schemes over $\ol{\bb{F}}_{p}$. The formal group scheme $\mathrm{Aut}_{G}(\tilde{\mathbb{X}}_{b})$ over $\mathrm{Spf}(W(\ol{\bb{F}}_{p}))$ is then given by taking Witt vectors of this formal group.

As is done in \cite[Corollary~4.3.5]{CS1}, one can, using the moduli description of $\Ig^{b}$ provided in (\ref{igmoduliinterp}), show that the fpqc sheaf $\Aut_{G}(\tilde{\mathbb{X}}_{b})_{\overline{\bb{F}}_{p}}$ acts on $\Ig^{b}$, by noting that quasi-isogenies of $\bb{X}_{b}$ are the same as automorphisms of the universal cover.  

Functorality of Witt vectors then gives rise to an action of $\mathrm{Aut}_{G}(\tilde{\mathbb{X}}_{b})$ on the Witt vector lift of $\Ig^{b}$, and passing to rigid generic fibers over $C$ gives an action of the rigid generic fiber of $\mathrm{Aut}_{G}(\tilde{\mathbb{X}}_{b})$ on $\mIg^{b}_{C}$, which we denote by $\mathrm{Aut}_{G}(\tilde{\mathbb{X}}_{b})_{C}$. By \cite[Lemma~4.2.10]{CS1}, $\mathrm{Aut}_{G}(\tilde{\mathbb{X}}_{b})_{C}$ is a perfectoid space which we identify with its associated $v$-sheaf on $\Perf_{C}$ using the Yoneda embedding. The following seems to be well known, but we sketch a proof as we could not find it explicitly written down in the literature.
\begin{proposition}{\label{prop: automorphismProposition}}
There is an identification 
\begin{equation}{\label{eqn: AutomorphismsofPdivisibleAutomorphismsofJb}}
\mathrm{Aut}_{G}(\tilde{\mathbb{X}}_{b})_{C} \simeq \mathcal{J}_{b} 
\end{equation}
of group-valued functors on $\Perf_{C}$, where $\mathcal{J}_{b}$ is the group diamond parametrizing automorphisms of the $G$-bundle $\mathcal{E}_{b}$ on the Fargues-Fontaine curve attached to $b$.
\end{proposition}
\begin{proof}{(Sketch)}
We fix $(\mathbb{X}_{b})_{O_{C}}$ to be a lift of $\mathbb{X}_{b}$ up to quasi-isogeny (with its $G$-structure). For an object $S = \Spa(R,R^{+}) \in \Perf_{C}$, the $S$-points of the  LHS of (\ref{eqn: AutomorphismsofPdivisibleAutomorphismsofJb}) identifies (cf. the proof of Proposition \ref{prop: actextend} below) with the set of formal quasi-isogenies 
\[ (\mathbb{X}_{b})_{O_{C}} \times_{\Spf(O_{C})} \Spf(R^{+}) \dashrightarrow (\mathbb{X}_{b})_{O_{C}}  \times_{\Spf(O_{C})} \Spf(R^{+}) \]
in the sense of \cite[Section~2.5]{MingjiaPolI}, respecting the $G$-structure. In other words, this is a quasi-isogeny
\[ (\mathbb{X}_{b})_{O_{C}} \times_{O_{C}} R^{+\flat}/\varpi \dashrightarrow (\mathbb{X}_{b})_{O_{C}}  \times_{O_{C}} R^{+\flat}/\varpi \]
respecting the additional structures, where $\varpi \in R^{+\flat}$ is a suitable choice of pseudo-uniformizer in the tilt $R^{+\flat}$ of $R^{+}$, and the quotient map is given by $R^{+} \ra R^{+\flat}/\varpi$ induced by Fontaine's map $\theta_{(R,R^{+})}: W(R^{+\flat}) \ra R^{+}$, and the quasi-isogeny doesn't depend on the choice of pseudo-uniformizer by Serre-Tate theory.

As described in \cite[Section~2.5.3]{MingjiaPolI}, to a $p$-divisible $G_{0}$ group over $R^{+\flat}/\varpi$, one can associate a vector bundle $\mathcal{E}(G_{0})$ on $X_{S}$ using crystalline Dieudonn\'e theory. As explained in \cite[Section~2.5.6]{MingjiaPolI}, a quasi-isogeny $G_{0} \dashrightarrow G_{0}$ induces an isomorphism of vector bundles
\[ \mathcal{E}(G_{0}) \xrightarrow{\simeq} \mathcal{E}(G_{0}) \]
on $X_{S}$. By applying the Tannakian-formalism, it follows that for a $p$-divisible group $G_{0}$ with $G$-structure and a quasi-isogeny $G_{0} \dashrightarrow G_{0}$ respecting the additional structures, that one obtains a $G$-bundle $\mathcal{E}_{G}(G_{0})$ and an isomorphism of $G$-bundles $\mathcal{E}_{G}(G_{0}) \xrightarrow{\simeq} \mathcal{E}_{G}(G_{0})$. Applying this to $G_{0} = (\mathbb{X}_{b})_{O_{C}} \times_{O_{C}} R^{+\flat}/\varpi$, one can check by our assumption that $(\bb{X}_{b})_{O_{C}}$ was a lift of $\bb{X}_{b}$ up to quasi-isogeny and \cite[Theorem~II.0.4]{FS}, that one has pro-\'etale locally  an identification $\mathcal{E}_{G}(G_{0}) = \mathcal{E}_{b,S}$, where $\mathcal{E}_{b,S}$ is the pullback of $\mathcal{E}_{b}$ the vector bundle attached to $b \in B(G)$ on $X_{C}$ along the natural map $X_{S} \ra X_{C}$. This defines the desired natural transformation from the functor of points of the LHS of (\ref{eqn: AutomorphismsofPdivisibleAutomorphismsofJb}) to the functor of points of the RHS of (\ref{eqn: AutomorphismsofPdivisibleAutomorphismsofJb}). In order to check it is an equivalence, we use that the LHS and RHS of (\ref{eqn: AutomorphismsofPdivisibleAutomorphismsofJb}) are partially proper over $\Spa(C,O_{C})$ to reduce to checking it on $(R,R^{\circ})$-points. For the RHS, this follows from the definition, as the category of vector bundles on $X_{(R,R^{+})}$ and $X_{(R,R^{\circ})}$ are equivalent, and for the LHS it follows from \cite[Proposition~4.2.22]{CS1}. The desired equivalence is then reduced to the fully faithfulness statement in \cite[Proposition~2.2.7]{PR23} via \cite[Lemma~2.5.5]{MingjiaPolI}.
\end{proof}

Using the previous proposition, one sees that the action of $\mathrm{Aut}_{G}(\tilde{\mathbb{X}}_{b})_{C}$ on $\mIg^{b}_{C}$ described above gives rise to an action of the group diamond $\mathcal{J}_{b}$ parametrizing automorphisms of $\mathcal{E}_{b}$. 

This allows us to form the $v$-stack quotient $[\mIg_{C}^{b}/\mathcal{J}_{b}]$, which is equipped with a natural map:
\[ \pi_{\mIg}^{b}: [\mIg_{C}^{b}/\mathcal{J}_{b}] \ra [\Spd(C)/\mathcal{J}_{b}]. \]
We would like to say $\pi_{\mIg}^{b}$ pulls back to the map $\pi^{b}_{\mathrm{HT}}$ along the map (\ref{eqn: NewtonLeftHecke}). However, as seen in Corollary \ref{cor: stalkspartial}, we need to account for the additional points in the fiber of the Hodge-Tate period morphism that are not seen by the perfectoid Igusa varieties $\mIg^{b}_{C}$. To capture this, we need to show that $\pi_{\mIg}^{b}$ also extends to the partial minimal compactification in a suitable sense. We recall that we have a decomposition 
\begin{equation}{\label{eqn: AutomorphismsDecomposition}}
 \mathcal{J}_{b} \simeq \underline{J_{b}(\bb{Q}_{p})} \ltimes \mathcal{J}_{b}^{> 0}, 
\end{equation}
where $\underline{J_{b}(\bb{Q}_{p})}$ is the constant group scheme coming from the automorphisms of $\mathcal{E}_{b}$ which preserve the graded pieces of the Harder-Narasimhan filtration and $\mathcal{J}_{b}^{> 0}$ is an $\ell$-adically contractible unipotent group diamond of dimension $d_{b} = \langle 2\rho_{G},\nu_{b} \rangle$. Analogously, we have a decomposition of fpqc sheaves 
\begin{equation}{\label{eqn: QuasiIsogeniesDecomposition}}
 \Aut_{G}(\tilde{\bb{X}}_{b})_{\overline{\mathbb{F}}_p} = \underline{J_b(\mathbb{Q}_p)}\ltimes \widetilde{\mathcal{U}}_b,   
\end{equation}
where $J_b(\mathbb{Q}_p)$ can be identified with the group of self-quasi-isogenies of $\mathbb{X}_{b}$ over $\overline{\mathbb{F}}_p$ that respect the $G$-structure, $\underline{J_b(\mathbb{Q}_p)}$ is the associated constant group scheme, and $\Tilde{\mc{U}_b}$ is a unipotent formal group scheme, see \cite[Proposition 4.2.11]{CS1}. We now have the following.
\begin{proposition}{\label{prop: actextend}}
Assuming \ref{assump: codim}, the action of $\mathcal{J}_{b}$ on the perfectoid Igusa variety $\mIg^{b}_{C}$ extends uniquely to an action on $\mIg^{b,*}_{C}$. In particular, by functorality of the formation of the canonical compactification (\cite[Proposition~18.6]{Ecod}), we have a map
\[ \pi_{\mIg}^{b,*}: [\ol{\mIg}_{C}^{b,*}/\mathcal{J}_{b}] \ra [\Spd(C)/\mathcal{J}_{b}] \]
extending $\pi_{\mIg}^{b}$. The restriction of this action to the subgroup $\underline{J_{b}(\bb{Q}_{p})} \subset \mathcal{J}_{b}$ also preserves the boundary $\partial\mIg^{b,*}$, so, in particular, we also get a map
\[ \tilde{\pi}_{\mIg}^{b,\partial}: [(\ol{\mIg}^{b,*}_{C} \setminus \ol{\partial\mIg}^{b,*}_{C})/\underline{J_{b}(\bb{Q}_{p})}] \ra [\Spd(C)/\underline{J_{b}(\bb{Q}_{p})}], \]
by pulling back $\pi_{\mIg}^{b,*}$ along the map $[\Spd(C)/\underline{J_{b}(\bb{Q}_{p})}] \ra [\Spd(C)/\mathcal{J}_{b}]$ induced by the decomposition (\ref{eqn: AutomorphismsDecomposition}) and restricting to the complement of the boundary.
\end{proposition}
\begin{proof}
We consider the open immersion
\[ g_{b}: \Ig^{b} \ra \Ig^{b,*} \]
of perfect schemes, which we claim induces an isomorphism on global sections. To show this, we write $g_{b}$ as the perfection of the limit of the corresponding maps at finite level 
\[ g_{b,m}: \Ig^{b}_{m} \hookrightarrow \Ig^{b,*}_{m}, \]
as explained in the proof of Proposition \ref{prop: artvanish}. Under Assumption \ref{assump: codim}, we can apply the algebraic form of Hartogs' principle (see for example \cite[Proposition~III.2.9]{Schtor}) to the open inclusion $g_{b,m}$ to conclude an isomorphism of global sections via restriction. This gives the corresponding claim for the map $g_{b}$ of perfect schemes, by taking inverse limits and perfectifying. In particular, we have an isomorphism
\begin{equation}{\label{schemeident}}
\mathcal{O}(\Ig^{b}) \simeq \mathcal{O}(\Ig^{b,*}) 
\end{equation}
on global sections. Since the perfect scheme $\Ig^{b,*}$ is affine, the isomorphism of global sections implies that the action of the formal group scheme $\mathrm{Aut}_{G}(\tilde{\bb{X}}_{b})$ regarded as a fpqc sheaf on the perfect scheme $\Ig^{b}$ extends to an action on the perfect affine scheme $\Ig^{b,*}$. As before, by taking Witt vectors and passing to the rigid generic fiber over $C$, this induces an action of $\Aut_{G}(\tilde{\bb{X}}_{b})_{C}$ on the perfectoid space $\mIg^{b,*}_{C}$, extending the action on $\mIg^{b}_{C}$. Passing through the identification $\Aut_{G}(\tilde{\bb{X}}_{b})_{C} \simeq \mathcal{J}_{b}$ of Proposition \ref{prop: automorphismProposition}, this gives us the desired action of $\mathcal{J}_{b}$ on $\mIg^{b,*}_{C}$.

Now, to understand the preservation of the boundary, we understand the construction of the action of $\mathcal{J}_{b}$ in a more abstract way. In particular, we consider the small diamond functor $(-)^{\diamond}$ of \cite[Section~3.1]{GleasonSpecializationMaps}. Recall $\Perf$ is the category of perfectoid spaces in characteristic $p$ over $\ol{\bb{F}}_{p}$, the $v$-sheaf $\Aut_{G}(\tilde{\bb{X}}_{b})^{\diamond}$ on $\Perf$ attached to the formal group $\Aut_{G}(\tilde{\bb{X}}_{b})$ is defined by the functor sending  $(R,R^{+}) \in \Perf$ to the group
\[\Aut_{G}(\tilde{\bb{X}}_{b})(\Spf R^+)=\varprojlim_n \Aut_{G}(\tilde{\bb{X}}_{b})(\Spec R^+/\varpi^n)\simeq \Aut_{G}(\tilde{\bb{X}}_{b})(\Spec R^+/\varpi),\]
for some pseudo-uniformizer $\varpi$, where the last equivalence follows from Serre-Tate lifting of $p$-divisible groups along $p$-nilpotent rings, as in \cite[Theorem 2.4.1]{CS2}. The $v$-sheaves $\Ig^{b,\diamond}$ (resp. $\Ig^{b,*,\diamond}$) are defined by sending $(R,R^{+})$ to $\Ig^{b}(\Spec R^{+})$ and $\Ig^{b,*}(\Spec R^{+})$. We note that, since $\Ig^{b}$ and $\Ig^{b,*}$ are perfect and $R^{+}$ is perfectoid, we have identifications $\Ig^{b}(\Spec R^{+}) \simeq \Ig^{b}(\Spec R^{+}/\varpi)$ and $\Ig^{b,*}(\Spec R^{+}) \simeq \Ig^{b,*}(\Spec R^{+}/\varpi)$. In particular, we may understand the action of $\Aut_{G}(\tilde{\bb{X}}_{b})^{\diamond}$ on the $v$-sheaves $\Ig^{b,\diamond}$ and $\Ig^{b,*,\diamond}$ in terms of the action of $\Aut_{G}(\tilde{\bb{X}}_{b})(\Spec R^+/\varpi)$ on $\Ig^{b}(\Spec R^{+}/\varpi)$ and $\Ig^{b,*}(\Spec R^{+}/\varpi)$.

The action of $\mathcal{J}_{b}$ on $\mIg^{b}_{C}$ and $\mIg^{b,*}_{C}$ described above may now be given by considering $\Aut_{G}(\tilde{\bb{X}}_{b})^{\diamond}$, $\Ig^{b,\diamond}$, $\Ig^{b,*,\diamond}$ and base-changing along $\Spd(C) \ra \Spd(\ol{\bb{F}}_{p})$ in the category of $v$-sheaves and identifying the resulting functors on $\Perf_{C}$ with $\mathcal{J}_{b} \simeq \Aut_{G}(\tilde{\bb{X}}_{b})_{C}$, $\mIg^{b}_{C}$, and $\mIg^{b,*}_{C}$, respectively. In particular, the relationship between the construction described above in terms of taking Witt vectors and the passing to the rigid generic fiber over $C$ follows from the relationship between untilts of a characteristic $p$ perfectoid $(R,R^{+})$ and Cartier divisors in the Witt ring $W(R^{+})$, as described in \cite[Lemma~6.28]{SW}.

We now use this perspective to show that the action of sub group-sheaf $\underline{J_{b}(\bb{Q}_{p})} \subset \Aut_{G}(\tilde{\bb{X}}_{b})^{\diamond}$ given by the decomposition (\ref{eqn: QuasiIsogeniesDecomposition}) preserves the closed subsheaf $\partial\Ig^{b,*,\diamond} \subset \Ig^{b,*,\diamond}$. By base-changing along $\Spd(C) \ra \Spd(\ol{\bb{F}}_{p})$, this will give the desired claim. 

To see this,
we note that we have a natural identification $\underline{J_{b}(\bb{Q}_{p})}(\Spec R^{+}) \simeq \varprojlim_{n \geq 1} \underline{J_{b}(\bb{Q}_{p})}(\Spec R^{+}/\varpi^{n})$. Therefore, the restriction of this action to $\underline{J_{b}(\bb{Q}_{p})} \subset \mathrm{Aut}_{G}(\tilde{\bb{X}}_{b})^{\diamond}$ on $\Ig^{b,*,\diamond}$ is induced from the action of $\underline{J_{b}(\bb{Q}_{p})}(\Spec R^{+})$ on $\Ig^{b,*}(\Spec R^{+}) \simeq \varprojlim_{n \geq 1} \Ig^{b,*}(\Spec R^{+}/\varpi^{n})$. Now, we note that $\partial\Ig^{b,*} \subset \Ig^{b,*}$ is perfect and is therefore also reduced, by \cite[Remark~7.8]{BSPrism}. Similarly, the ring $R^{+}$ is perfect and is therefore also reduced. Moreover, by construction, the action of the fpqc sheaf $\Aut_{G}(\tilde{\bb{X}}_{b})$ on $\Ig^{b,*}$ extends the action on $\Ig^{b} \hookrightarrow \Ig^{b,*}$; in particular, it preserves $\Ig^{b}$ by construction. We may therefore conclude that $\underline{J_{b}(\bb{Q}_{p})}$ preserves the closed subsheaf $\partial\Ig^{b,*,\diamond} \subset \Ig^{b,*,\diamond}$ by appealing to the following abstract lemma applied to $G = \Aut_{G}(\tilde{\bb{X}}_{b})$, $\Ig^{b,*} = X$, and $Z = \partial\Ig^{b,*}$.
\begin{lemma}
Suppose we have a group-valued fpqc sheaf $G/\ol{\bb{F}}_{p}$ acting on a scheme $X/\ol{\bb{F}}_{p}$ via $m: G \times X \ra X$ and suppose that $Z \subset X$ is a closed subscheme which is reduced. We write $U := X \setminus Z$ for the open complement and assume the action of $G$ preserves $U \subset X$. Then, for all $T/\ol{\bb{F}}_{p}$ reduced, the action of the $T$-points $G(T)$ on $X(T)$ preserves $Z(T)$. 
\end{lemma}
\begin{proof}
For $T$ as in the statement, consider a $T$-point $x: T \ra Z \ra X$ and an element $g \in G(T)$. We need to show that $m(g,x): T \ra X$ factors through $Z \ra X$. Since $T$ and $Z$ are reduced, it suffices by \cite[Lemma 0356]{stacks-project}, to show that the induced map $|m(g,x)|: |T| \ra |X|$ on the underlying topological space factors through $|Z| \ra |X|$. Suppose for the sake of contradiction that it does not, then there exists a geometric point $\Spec(k) \ra T$ whose composite with $m(g,x):T\ra X$ factors through $U\ra X$. Now if we consider $m(g^{-1},m(g,x)) = x: T \ra X$, then we arrive at a contradiction since the precomposition of this map with $\Spec(k) \ra T$ must factor through $U \subset X$, by the assumption that $G$ preserves $U$. This however contradicts the decomposition $|X| = |U| \sqcup |Z|$ on topological spaces.
\end{proof}
\end{proof}
\begin{remark}
It is true that the full action of $\mathcal{J}_{b}$ preserves the Zariski closed boundary $\partial\mIg^{b,*}_{C} \subset \mIg^{b,*}_{C}$. This is shown in \cite[Corollary~4.24]{CHZIntCohofIgusaVarieties} (In fact, a stronger claim is shown that $\mathcal{J}_{b}$ preserves all the closed boundary strata of $\partial\mIg^{b,*}_{C}$). However, this requires a much more detailed argument involving lifting to the toroidal compactification. As a result, we will not address this here.
\end{remark}
Lastly, we will consider the map $\ol{\pi}^{b}_{\mIg}: [\ol{\mIg}^{b}_{C}/\mathcal{J}_{b}] \ra [\Spd(C)/\mathcal{J}_{b}]$, given by taking the canonical compactification of $\pi^{b}_{\mIg}$, where we note, by \cite[Proposition~4.2.22]{CS1}, the $v$-stack $[\Spd(C)/\mathcal{J}_{b}]$ is partially proper over $\Spd(C)$. We now have the following Proposition. 
\begin{proposition}
The maps constructed above fit into the following Cartesian squares\footnote{We emphasize that these are really diagrams of $v$-stacks and that all fiber products are formed in this category.}
\begin{equation}\label{eqn:manpdtfor}
 \begin{tikzcd}
& \left[(\mathcal{S}^{b,\circ,\mathrm{rd}}_{K^{p},C} \times_{\mathcal{F}\ell^{\Diamond}_{G,\mu^{-1}}} \mathcal{F}\ell_{G,\mu^{-1}}^{b,\Diamond})/\ul{G(\mathbb{Q}_{p})}\right] \arrow[r,"\pi_{\mathrm{HT}}^{b,\circ}"] \arrow[d,"\tilde{h}_{b}^{\la}"] & \left[\mathcal{F}\ell^{b,\Diamond}_{G,\mu^{-1}}/\ul{G(\mathbb{Q}_{p})}\right]\arrow[d,"h^{\la}_{b}"]  \\
& \left[\mIg_{C}^{b}/\mathcal{J}_{b}\right] \arrow[r,"\pi_{\mIg}^{b}"]  & \left[\Spd(C)/\mathcal{J}_{b}\right] 
\end{tikzcd} 
\end{equation}
and
\begin{equation}\label{eqn:manpdtfor0}
 \begin{tikzcd}
& \left[\ol{\mathcal{S}}^{b,\circ}_{K^{p},C}/\ul{G(\mathbb{Q}_{p})}\right] \arrow[r,"\ol{\pi}_{\mathrm{HT}}^{b,\circ}"] \arrow[d,"\tilde{h}_{b}^{\la}"] & \left[\mathcal{F}\ell^{b,\Diamond}_{G,\mu^{-1}}/\ul{G(\mathbb{Q}_{p})}\right]\arrow[d,"h^{\la}_{b}"]  \\
& \left[\ol{\mIg}_{C}^{b}/\mathcal{J}_{b}\right] \arrow[r,"\ol{\pi}_{\mIg}^{b}"]  & \left[\Spd(C)/\mathcal{J}_{b}\right]. 
\end{tikzcd} 
\end{equation}
\end{proposition}
\begin{proof} 
Consider the moduli space of local shtukas $\Sht(G,b,\mu)_{\infty,C}$, as defined in \cite[Definition~23.1.1]{SW}. This represents the functor sending $S \in \Perf_C$ to the set of all pairs $(S^\#,\alpha)$ where $S^\#$ is the untilt of $S$ coming from the map $S \ra \Spd(C)$, and $\alpha$ is a modification from $\mathcal{E}_b \dashrightarrow \mathcal{E}_{0}$ with meromorphy along $S^\#$ and bounded by $\mu$. We have a local Hodge-Tate period morphism 
\[ \Sht(G,b,\mu)_{\infty,C} \ra \mathcal{F}\ell_{G,\mu^{-1}}^{b,\Diamond}, \]
which fits into the following Cartesian diagram coming from the definition of $\Sht(G,b,\mu)_{\infty,C}$.
\begin{equation}\label{eqn:localHT}
    \begin{tikzcd}
        & \Sht(G,b,\mu)_{\infty,C} \arrow[r]\arrow[d]& \Spd(C)\arrow[d]\\
        & \mathcal{F}\ell_{G,\mu^{-1}}^{b,\Diamond}\arrow[r]&\left[\Spd(C)/\mathcal{J}_{b}\right]
    \end{tikzcd}
\end{equation}
Let $\Sht(G,b,\mu)_{\infty,C} \times^{\mathcal{J}_{b}}\mIg_{C}^{b}$ denote the quotient of $\Sht(G,b,\mu)_{\infty,C} \times_{C} \mIg_{C}^{b}$ by the diagonal $\mathcal{J}_{b}$-action $(x,y)\mapsto(xj,yj)$. To see that the diagram \eqref{eqn:manpdtfor} above is Cartesian, observe that \eqref{eqn:localHT} implies we have an isomorphism
\begin{equation*}
\mathcal{F}\ell^{b,\Diamond}_{G,\mu^{-1}}\times_{[\Spd(C)/\mathcal{J}_{b}]} [\mIg_{C}^{b}/\mathcal{J}_{b}] \simeq \Sht(G,b,\mu)_{\infty,C} \times^{\mathcal{J}_{b}}\mIg_{C}^{b}.
\end{equation*}
Moreover, we see, by \cite[Corollary~4.3.19,Lemma~4.3.20]{CS1}, that we have an isomorphism 
\begin{equation}
\label{eqn:CSalmostpdt}
    \Sht(G,b,\mu)_{\infty,C}\times_C \mIg_{C}^b \simeq \mathcal{S}^{b,\circ,\mathrm{rd}}_{K^p} \times_{\mathcal{F}\ell^{\Diamond}_{G,\mu^{-1}}}\Sht(G,b,\mu)_{\infty,C}.
\end{equation}
Moreover, this map is $\mathcal{J}_{b}$-equivariant. Here $\mathcal{J}_{b}$ acts on the LHS via $(x,y) \mapsto (xj,yj)$, and $\mathcal{J}_{b}$ acts on the RHS by the trivial action on $\mathcal{S}^{b,\circ,\mathrm{rd}}_{K^p}$ and the natural action on $\Sht(G,b,\mu)_{\infty,C}$ by automorphisms of bundles.

Now applying \eqref{eqn:localHT} implies that $\mathcal{S}^{b,\circ,\mathrm{rd}}_{K^p} \times_{\mathcal{F}\ell^{\Diamond}_{G,\mu^{-1}}} \mathcal{F}\ell_{G,\mu^{-1}}^{b,\Diamond}$ is isomorphic to the quotient of $\mathcal{S}^{b,\circ,\mathrm{rd}}_{K^p}\times_{\mathcal{F}\ell^{\Diamond}_{G,\mu^{-1}}}\Sht(G,b,\mu)_{\infty,C}$ by the action of $\mathcal{J}_b$ (here $\mathcal{J}_b$ acts via the action on the second factor). Thus, we have an isomorphism:
\begin{equation*}
    \Sht(G,b,\mu)_{\infty,C} \times^{\mathcal{J}_{b}}\mIg_{C}^{b} \simeq \mathcal{S}^{b,\circ,\mathrm{rd}}_{K^p} \times_{\mathcal{F}\ell^{\Diamond}_{G,\mu^{-1}}} \mathcal{F}\ell_{G,\mu^{-1}}^{b,\Diamond}. 
\end{equation*}
This gives us the Cartesian diagram (\ref{eqn:manpdtfor}). Now, the natural map, 
\[\mathcal{S}_{K^{p},C}^{b,\circ,\mathrm{rd}} \times_{\mathcal{F}\ell^{\Diamond}_{G,\mu^{-1}}} \mathcal{F}\ell_{G,\mu^{-1}}^{b,\Diamond} \hookrightarrow \mathcal{S}^{b,\circ}_{K^{p},C}  \]
(\ref{eqn: comparisonofadicandgenericNewtonstratification}) is a qcqs open immersion which is an isomorphism on rank $1$ points, as explained in the discussion proceeding (\ref{eqn: comparisonofadicandgenericNewtonstratification}). In turn, it induces an isomorphism of canonical compactifications over the partially proper strata $\mathcal{F}\ell_{G,\mu^{-1}}^{b,\Diamond}$. Therefore, by passing to canonical compactifications over $\mathcal{F}\ell_{G,\mu^{-1}}^{b,\Diamond}$, we deduce that the diagram (\ref{eqn:manpdtfor0}) is also Cartesian, using that the formation of canonical compactifications commutes with limits \cite[Proposition~18.7 (v)]{Ecod}. 
\end{proof}
\begin{remark}{\label{rem: SignConvention}}
We note that our conventions for moduli spaces of shtukas differ from those in \cite[Definition~23.1.1]{SW}; in particular, in \emph{loc.cit} the space $\Sht(G,b,\mu)_{\infty}$ denotes the moduli space parameterizing modifications $\mathcal{E}_{0} \dashrightarrow \mathcal{E}_{b}$ of meromorphy $\leq \mu$, which would identify more naturally with the space we denote by $\Sht(G,b,\mu^{-1})_{\infty}$. This is consistent with \cite{Ham1,Ham2,BMNH}, where this convention is nice in that the Hecke operator $j_{1}^{*}T_{\mu}$ is related to $\Sht(G,b,\mu)_{\infty}$ for $b \in B(G,\mu)$ instead of $\Sht(G,b,\mu^{-1})$, as in the usual convention. Here we also note that we have uniformized the Schubert stratification of the $\bb{B}_{\mathrm{dR}}^{+}$-Grassmannian so that $G(\bb{B}_{\mathrm{dR},C}) := \bigsqcup_{\mu \in \bb{X}_{*}(T_{\ol{\bb{Q}}_{p}})^{+}} G(\bb{B}^{+}_{\mathrm{dR},C})\mu(\xi)^{-1}G(\bb{B}_{\mathrm{dR},C}^{+})$, where the coset $G(\bb{B}^{+}_{\mathrm{dR},C})\mu(\xi)^{-1}G(\bb{B}_{\mathrm{dR},C}^{+})$ defines $\Gr_{G,\mu}$, as in the discussion proceeding \cite[Definition~3.4.1]{CS1}. We warn the reader that this differs by an inverse from the uniformization used in \cite{SW} and \cite{MingjiaPolI}, which is why the index set for the adic Newton stratification of $\mathcal{F}\ell_{G,\mu^{-1}}^{\Diamond}$ is for us $B(G,\mu)$ instead of $B(G,\mu^{-1})$.
\end{remark}
We now invoke a result of Zhang \cite{Zha}. 
\begin{theorem}{\cite[Theorem~1.3, Section~9.4]{Zha}}
Assuming \ref{assump: codim}, for all $b \in B(G,\mu)$ the Cartesian diagram (\ref{eqn:manpdtfor0}) extends to a Cartesian diagram
\begin{equation}\label{eqn:manpdtfor1}
 \begin{tikzcd}
& \left[\mathcal{S}^{b,*}_{K^{p},C}/\ul{G(\mathbb{Q}_{p})}\right] \arrow[r,"\pi_{\mathrm{HT}}^{b,*}"] \arrow[d,"\phantom{}^{*}\tilde{h}_{b}^{\la}"] & \left[\mathcal{F}\ell^{b,\Diamond}_{G,\mu^{-1}}/\ul{G(\mathbb{Q}_{p})}\right]\arrow[d,"h^{\la}_{b}"]  \\
& \left[\ol{\mIg}_{C}^{b,*}/\mathcal{J}_{b}\right] \arrow[r,"\pi_{\mIg}^{b,*}"]  & \left[\Spd(C)/\mathcal{J}_{b}\right] 
\end{tikzcd} 
\end{equation}
of $v$-stacks.
\end{theorem} 
\begin{remark}
In fact Zhang shows a much stronger claim, that there exists a series of larger Cartesian diagrams living over $\Bun_{G,C}$ such that the diagrams (\ref{eqn:manpdtfor0}) and (\ref{eqn:manpdtfor1}), are the base-change along the inclusions $j_{b}: \Bun_{G,C}^{b} \hookrightarrow \Bun_{G,C}$ of HN-strata for $b \in B(G,\mu)$ varying. 
\end{remark}
\begin{remark}
The rough idea behind proving this is to apply a relative $\Spa$ construction in the category of $v$-sheaves to the horizontal maps of the diagram (\ref{eqn:manpdtfor1}) by invoking Hartogs' principle, as in Proposition \ref{prop: actextend}.  
\end{remark}
We now state the key corollary that we will need. We write $\tilde{s}_{b}: [\widetilde{\mathcal{F}\ell}_{G,\mu^{-1}}^{b}/\underline{G(\bb{Q}_{p})}] \ra [\mathcal{F}\ell_{G,\mu^{-1}}^{b,\Diamond}/\underline{G(\bb{Q}_{p})}]$ for the torsor under a possibly twisted form of $\mathcal{J}_{b}^{> 0}$ defined by pulling back $h^{\la}_{b}$ along the natural map $s_{b}: [\Spd(C)/\underline{J_{b}(\bb{Q}_{p})}] \ra [\Spd(C)/\mathcal{J}_{b}]$ induced by (\ref{eqn: AutomorphismsDecomposition}). We write $\tilde{\pi}_{\mathrm{HT}}^{b}: [\widetilde{\mathcal{S}}^{b}_{K^{p},C}/\underline{G(\bb{Q}_{p})}] \ra [\widetilde{\mathcal{F}\ell}_{G,\mu^{-1}}^{b,\Diamond}/\underline{G(\bb{Q}_{p})}]$ and $\tilde{\pi}_{\mathrm{HT}}^{b,*}: [\widetilde{\mathcal{S}}^{b,*}_{K^{p},C}/\underline{G(\bb{Q}_{p})}] \ra [\widetilde{\mathcal{F}\ell}^{b,\Diamond}_{G,\mu^{-1}}/\underline{G(\bb{Q}_{p})}]$ for the pullbacks of $\pi_{\mathrm{HT}}^{b}$ and $\pi_{\mathrm{HT}}^{b,*}$ along the map $[\widetilde{\mathcal{F}\ell}_{G,\mu^{-1}}^{b,\Diamond}/\underline{G(\bb{Q}_{p})}] \ra [\mathcal{F}\ell_{G,\mu^{-1}}^{b,\Diamond}/\underline{G(\bb{Q}_{p})}]$.

We now have the following.
\begin{corollary}{\label{cor: openCartesiandiagram}}
Assuming \ref{assump: codim}, for all $b \in B(G,\mu)$ we have a Cartesian diagram 
\begin{equation}\label{eqn:manpdtfor2}
 \begin{tikzcd}
& \left[\widetilde{\mathcal{S}}^{b}_{K^{p},C}/\ul{G(\mathbb{Q}_{p})}\right] \arrow[r,"\tilde{\pi}_{\HT}^{b}"] \arrow[d] & \left[\widetilde{\mathcal{F}\ell}^{b,\Diamond}_{G,\mu^{-1}}/\underline{G(\mathbb{Q}_{p})}\right] \arrow[d]  \\
& \left[(\ol{\mIg}_{C}^{b,*} \setminus \ol{\partial\mIg}_{C}^{b,*})/\underline{J_{b}(\bb{Q}_{p})}\right] \arrow[r,"\tilde{\pi}_{\mIg}^{b,\partial}"]  & \left[\Spd(C)/\underline{J_{b}(\bb{Q}_{p})}\right]. 
\end{tikzcd} 
\end{equation}
with maps as defined above.
\end{corollary}
\begin{proof}
This follows from the Cartesian diagram (\ref{eqn:manpdtfor1}) and Corollary \ref{cor: openhodgetatefibers}. More precisely, consider the open substack
\[ [\widetilde{\mathcal{S}}_{K^{p},C}^{b,*}/\underline{G(\bb{Q}_{p})}] \times_{[\ol{\mIg}^{b,*}_{C}/\underline{J_{b}(\bb{Q}_{p})}]} \left[(\ol{\mIg}_{C}^{b,*} \setminus \ol{\partial\mIg}_{C}^{b,*})/\underline{J_{b}(\bb{Q}_{p})}\right] \hookrightarrow [\widetilde{\mathcal{S}}^{b,*}_{K^{p},C}/\underline{G(\mathbb{Q}_{p})}]  \]
obtained by base change of the open immersion $[(\overline{\mIg}_{C}^{b,*} \setminus \overline{\partial\mIg}^{b,*}_{C})/\underline{J_{b}(\bb{Q}_{p})}] \hookrightarrow [\overline{\mIg}^{b,*}_{C}/\underline{J_{b}(\bb{Q}_{p})}]$. The claim follows from showing that this open substack identifies with the open substack
\[ [\widetilde{\mathcal{S}}^{b}_{K^{p},C}/\underline{G(\bb{Q}_{p})}] \hookrightarrow [\widetilde{\mathcal{S}}^{b,*}_{K^{p},C}/\underline{G(\bb{Q}_{p})}]. \]
Using \cite[Proposition~12.9]{Ecod}, it suffices to check that the open subspaces of the underlying topological space of $[\widetilde{\mathcal{S}}^{b,*}_{K^{p},C}/\underline{G(\bb{Q}_{p})}]$ that these two substacks define agree with one another. Since both subspaces are partially proper (as they are torsors under forms of the partially proper group $\mathcal{J}_{b}^{> 0}$ over partially proper spaces), it suffices to show that they have the same rank $1$ points, and therefore it suffices to show that they identify with the same subspace after taking the fiber of $\tilde{\pi}_{\mathrm{HT}}^{b,*}$ over a rank $1$ point of $\widetilde{\mathcal{F}\ell}_{G,\mu^{-1}}^{b,\Diamond}$, and this reduces us to Corollary \ref{cor: openhodgetatefibers}.
\end{proof}
By the Cartesian diagram (\ref{eqn:manpdtfor2}), if we look at the sheaf
\[ \tilde{s}_{b}^{*}i_{b}^{*}R\pi_{\mathrm{HT}!}(\ol{\mathbb{F}}_{\ell}) \]
we can see that this is canonically identified with $R\tilde{\pi}_{\mIg!}^{b,\partial}(\ol{\mathbb{F}}_{\ell}) \in \D(J_{b}(\bb{Q}_{p}),\ol{\bb{F}}_{\ell})$ pulled back along the natural map $[\widetilde{\mathcal{F}\ell}_{G,\mu^{-1}}^{b,\Diamond}/\underline{G(\bb{Q}_{p})}] \ra [\Spd(C)/\underline{J_{b}(\bb{Q}_{p})}]$. Here we recall that  $i_{b}: [\mathcal{F}\ell_{G,\mu^{-1}}^{b,\Diamond}/\ul{G(\mathbb{Q}_{p})}] \hookrightarrow [\mathcal{F}\ell_{G,\mu^{-1}}^{\Diamond}/\ul{G(\mathbb{Q}_{p})}]$ is the inclusion of the adic Newton strata corresponding to $b \in B(G,\mu)$.

Moreover, we can identify $R\tilde{\pi}_{\mIg!}^{b,\partial}(\ol{\mathbb{F}}_{\ell})$ simply with the complex $V_{b} := R\Gamma_{c-\partial}(\Ig^{b},\ol{\mathbb{F}}_{\ell})$ of smooth $J_{b}(\mathbb{Q}_{p})$-modules, as in Corollary \ref{cor: stalkspartial}. We recall that we have an identification $\D(\Bun_{G,C}^{b},\ol{\mathbb{F}}_{\ell}) \simeq \D(J_{b}(\mathbb{Q}_{p}),\ol{\mathbb{F}}_{\ell})$, which is induced via pullback along $p_{b}: \Bun_{G,C}^{b} \simeq [\Spd(C)/\mathcal{J}_{b}] \ra [\Spd(C)/\underline{J_{b}(\bb{Q}_{p})}]$, as in \cite[Proposition~V.2.2]{FS}. 

The morphism $\tilde{s}_b$ is a torsor under a form of
$\mathcal{J}_b^{>0}$ and this implies that $\tilde{s}_{b}^{*}$ is fully faithful. Indeed, fully faithfullness is $v$-local on the target. After passing to a $v$-cover trivializing the form and the torsor, apply
\cite[Proposition~V.2.1]{FS} successively to its positive
Banach--Colmez quotients, as in the proof of
\cite[Proposition~V.2.2]{FS}. Then $v$-descent shows that
$\tilde{s}_b^*$ is fully faithful. Combining this with the above observations, the complexes
$i_b^*R\pi_{\mathrm{HT}!}(\ol{\mathbb{F}}_{\ell})$ and
$(h_b^\leftarrow)^*p_b^*(V_b)$ have canonically isomorphic
pullbacks along $\widetilde{s}_b$, and hence are isomorphic
by fully faithfulness. Abusively identifying $V_{b}$ and $p_{b}^{*}(V_{b})$ as in the identification $\D(\Bun_{G}^{b},\ol{\bb{F}}_{\ell}) \simeq \D(J_{b}(\bb{Q}_{p}),\ol{\bb{F}}_{\ell})$, this allows us to deduce an isomorphism 
\begin{equation}{\label{eqn: identifyinggraded}}
 h^{\la*}_{b}(V_{b}) \simeq i_{b}^{*}R\pi_{\HT!}(\ol{\bb{F}}_{\ell}). 
\end{equation}
We can further refine this using the following Lemma.
\begin{lemma}{\label{lem: excisioncomparison}}
We have isomorphisms
\[ i_{b!}h_{b}^{\la*}(V_{b}) \simeq h^{\la*}j_{b!}(V_{b}) \]
and 
\[ Ri_{b*}h_{b}^{\la*}(V_{b}) \simeq h^{\la*}Rj_{b*}(V_{b}) \]
of sheaves on
$[\mathcal{F}\ell_{G,\mu^{-1}}^{\Diamond}/
\ul{G(\mathbb{Q}_p)}]$.
\end{lemma}
\begin{proof}
The first isomorphism follows from proper base-change (\cite[Theorem~1.9 (ii)]{Ecod}). For the second isomorphism, we note that 
\[ h^{\la}: [\mathcal{F}\ell_{G,\mu^{-1}}^{\Diamond}/\ul{G(\mathbb{Q}_{p})}] \ra \Bun_{G,C} \]
is cohomologically smooth separated and representable in locally spatial diamonds; therefore, the result follows by smooth base-change \cite[Proposition~23.16 (2)]{Ecod}.
\end{proof}
\begin{remark}{\label{rem: excision}}
Combining Lemma \ref{lem: excisioncomparison} with
(\ref{eqn: identifyinggraded}), the graded pieces of the
filtration on
$R\Gamma_c(\mathcal{S}_{K^p,C},\ol{\mathbb{F}}_\ell)$
obtained by excision along the Newton stratification are
\[
Rh^{\ra}_{*}
\bigl(i_{b!}i_b^*R\pi_{\mathrm{HT}!}
(\ol{\mathbb{F}}_\ell)\bigr)
\simeq Rh^{\ra}_{*}h^{\la*}j_{b!}(V_b).
\]
Similarly,
\[
Rh^{\ra}_{*}
\bigl(Ri_{b*}i_b^*R\pi_{\mathrm{HT}!}
(\ol{\mathbb{F}}_\ell)\bigr)
\simeq Rh^{\ra}_{*}h^{\la*}Rj_{b*}(V_b).
\]
Here these identifications take place in
$\D([\Spd(C)/\ul{G(\mathbb{Q}_p)}],
\ol{\mathbb{F}}_\ell)
\simeq \D(G(\mathbb{Q}_p),\ol{\mathbb{F}}_\ell)$.
\end{remark}
All in all, by applying excision we get the following.
\begin{proposition}
Assuming \ref{assump: codim}, we have a filtration on $R\Gamma_{c}(\mathcal{S}(\mathbf{G},X)_{K^{p},C},\ol{\mathbb{F}}_{\ell})$ by complexes of smooth representations of $G(\mathbb{Q}_p)$, with graded pieces isomorphic to $Rh^{\ra}_{*}h^{\la*}j_{b!}(V_{b})$, where $V_{b} := R\Gamma_{c-\partial}(\Ig^{b},\ol{\mathbb{F}}_{\ell}) \in \D(J_{b}(\bb{Q}_{p}),\ol{\bb{F}}_{\ell})$ is regarded as a sheaf on $\Bun_{G,C}^{b}$ via the natural identification $\D(\Bun_{G,C}^{b},\ol{\bb{F}}_{\ell}) \simeq \D(J_{b}(\bb{Q}_{p}),\ol{\bb{F}}_{\ell})$.
\end{proposition}

The functor $Rh^{\ra}_*h^{\la*}(-)$ appearing on the graded pieces is manifestly related to the action on $\D(\Bun_{G},\ol{\mathbb{F}}_{\ell})$ by Hecke operators. In particular, for each geometric dominant cocharacter $\mu \in \domcochar$, we have a correspondence 
\begin{equation}{\label{eqn: HeckeCorrespondenceDiagram}}
\begin{tikzcd}
& & \arrow[dl,"h_{\mu}^{\leftarrow}"] \Hck_{G,\leq \mu} \arrow[dr,"h_{\mu}^{\rightarrow}"] & & \\
& \Bun_{G} & & \Bun_{G,C}  & 
\end{tikzcd} 
\end{equation}
where $\Hck_{G, \leq \mu}$ is the stack on $\Perf_{C}$ parametrizing modifications $\mathcal{E}_{1} \ra \mathcal{E}_{2}$ of a pair of $G$-bundles on $X_{S}$ with meromorphy bounded by $\mu$ along the Cartier divisor of $X_{S}$ determined by the fixed untilt defined by $C$. We define the Hecke operator \cite[Section~IX.2]{FS}
\[ T_{\mu}: \D(\Bun_{G},\ol{\mathbb{F}}_{\ell}) \ra   \D(\Bun_{G},\ol{\mathbb{F}}_{\ell})^{BW_{E_{\mu}}} \]
\[ A \mapsto Rh_{\mu*}^{\ra}(h_{\mu}^{\la*}(A) \otimes^{\mathbb{L}} \mathcal{S}_{\mu}),  \]
where $E_{\mu}$ is the reflex field of $\mu$ and $\mathcal{S}_{\mu}$ is a sheaf on $\Hck_{G,\leq \mu}$ attached to the highest weight tilting module $\mathcal{T}_{\mu}$ by geometric Satake. 
\begin{remark}{\label{rem: HeckeOperatorsWellDefined}}
To see why this operator gives a functor of the desired shape, we note that pullback along $\Bun_{G,C} \ra \Bun_{G}$ gives a natural isomorphism $\D(\Bun_{G},\ol{\bb{F}}_{\ell}) \simeq \D(\Bun_{G,C},\ol{\bb{F}}_{\ell})$ by \cite[Corollary~V.2.3.]{FS}, and, under this isomorphism, the operator $T_{\mu}$ a priori only defines a functor $\D(\Bun_{G},\ol{\bb{F}}_{\ell}) \ra \D(\Bun_{G},\ol{\bb{F}}_{\ell})$; however, it follows from \cite[Corollary~IX.2.3]{FS} and the surrounding discussion that this factors through the forgetful functor $\D(\Bun_{G},\ol{\bb{F}}_{\ell})^{BW_{E_{\mu}}} \ra \D(\Bun_{G},\ol{\bb{F}}_{\ell})$ forgetting the continuous $W_{E_{\mu}}$-action. 
\end{remark}
If we now let $\mu$ be the minuscule cocharacter appearing above then the Bialynicki-Birula map gives an isomorphism of $v$-stacks between the open locus of $\Hck_{G,\leq \mu}$ where $\mathcal{E}_{2}$ is isomorphic to the trivial bundle and $[\mathcal{F}\ell_{G,\mu^{-1}}^{\Diamond}/\ul{G(\mathbb{Q}_{p})}]$, which identifies $h_{\mu}^{\ra}$ (resp. $h_{\mu}^{\la}$) with $h^{\ra}$ (resp. $h^{\la}$) after replacing $\Bun_{G}$ by $\Bun_{G,C}$ on the LHS of (\ref{eqn: HeckeCorrespondenceDiagram}). Moreover, this is a cohomologically smooth $v$-stack of dimension $d := \langle 2\rho_{G},\mu \rangle$, and we have an isomorphism $\mathcal{S}_{\mu} \simeq \ol{\mathbb{F}}_{\ell}[d](\frac{d}{2})$\footnote{This is true for Satake sheaf attached to the highest weight module $V_{\mu}$ by definition and this agrees with the highest weight tilting module $\mathcal{T}_{\mu}$, since $\mu$ is minuscule.}. It follows, by proper base-change, that we have an isomorphism
\[ Rh^{\ra}_*h^{\la*}j_{b!}(V_{b}) \simeq j_{1}^{*}T_{\mu}j_{b!}(V_{b})[-d](-\frac{d}{2}) \]
of $G(\mathbb{Q}_{p})$-modules, where $1 \in B(G)$ denotes the trivial element. Here we have passed through the equivalence $\D(\Bun_{G,C},\ol{\bb{F}}_{\ell}) \simeq \D(\Bun_{G},\ol{\bb{F}}_{\ell})$, and abusively used $j_{b}$ for both the inclusion $\Bun_{G,C}^{b} \ra \Bun_{G,C}$ and $\Bun_{G}^{b} \ra \Bun_{G}$. This gives us Theorem \ref{thm: mantprodform}.
\begin{corollary}{\label{cor: filtheckeops}}
Assuming \ref{assump: codim}, the complex $R\Gamma_{c}(\mathcal{S}(\mathbf{G},X)_{K^{p},C},\ol{\mathbb{F}}_{\ell})$ has a $G(\mathbb{Q}_{p})$-equivariant filtration with graded pieces given by $j_{1}^{*}T_{\mu}j_{b!}(V_{b})[-d](-\frac{d}{2})$ for varying $b \in B(G,\mu)$, where $V_{b} \simeq R\Gamma_{c-\partial}(\Ig^{b},\ol{\mathbb{F}}_{\ell})$ and we implicitly identify $T_{\mu}$ with its postcomposition with the forgetful functor $\D(\Bun_{G},\ol{\bb{F}}_{\ell})^{BW_{E_{\mf{p}}}} \ra \D(\Bun_{G},\ol{\bb{F}}_{\ell})$. 
\\\\
Moreover, we obtain that each graded piece is isomorphic to
\[ (R\Gamma_{c}(G,b,\mu) \otimes^{\mathbb{L}}_{\mathcal{H}(J_{b})} V_{b})[2d_{b}] \]
as $G(\mathbb{Q}_{p})$-modules. Here 
\[ R\Gamma_{c}(G,b,\mu) := \colim_{K_{p} \ra \{1\}} R\Gamma_{c}(\Sht(G,b,\mu)_{\infty,C}/\ul{K_{p}},\ol{\mathbb{F}}_{\ell}(d_{b})) \]
is a complex of $G(\mathbb{Q}_{p}) \times J_{b}(\mathbb{Q}_{p}) \times W_{E_{\mf{p}}}$-modules, where $\Sht(G,b,\mu)_{\infty,C}$ is as defined above, and $\ol{\mathbb{F}}_{\ell}(d_{b})$ is the sheaf with $J_{b}(\mathbb{Q}_{p})$-action given as in \cite[Lemma~7.4]{Ko}.
\end{corollary}
\begin{proof}
It remains to explain the description of the graded pieces $j_{1}^{*}T_{\mu}j_{b!}(V_{b})[-d](-\frac{d}{2})$. By applying the first isomorphism of \cite[Lemma~10.1]{Ham2} and noting that $\mathcal{S}_{\mu} \simeq \ol{\mathbb{F}}_{\ell}[d](\frac{d}{2})$ since $\mu$ is minuscule (where we recall that, since $\mu$ is minuscule, the representation $\mathcal{T}_{\mu}$ agrees with the usual highest weight representation), we obtain that the graded pieces are isomorphic to
\[ \colim_{K_{p} \ra \{1\}} R\Gamma_{c}(\Sht(G,b,\mu)_{\infty,C}/\ul{K_{p}},\ol{\mathbb{F}}_{\ell}) \otimes^{\mathbb{L}}_{\mathcal{H}(J_{b})} (V_{b} \otimes \kappa^{-1})[2d_{b}]  \]
where the shtuka cohomology is regarded as a right
$\mathcal{H}(J_b)$-module by converting its left $J_b(\mathbb{Q}_p)$-action
via $g\mapsto g^{-1}$, while $V_b\otimes\kappa^{-1}$ is a left
$\mathcal{H}(J_b)$-module. Here $\kappa$ is the character of $J_{b}(\mathbb{Q}_{p})$ defined by the action of $J_{b}(\mathbb{Q}_{p})$ on the compactly supported cohomology of the $\ell$-adically contractible group diamond $\mathcal{J}_{b}^{> 0}$, where $\mathcal{J}_{b} \simeq \mathcal{J}_{b}^{> 0} \ltimes J_{b}(\mathbb{Q}_{p})$ is the semi-direct product decomposition of (\ref{eqn: AutomorphismsDecomposition}). The projection formula moves the character twist to the coefficients of the cohomology of shtukas. Moreover, \cite[Lemmas~7.4,7.6]{Ko}
and its proof identifies this coefficient system with
$\ol{\mathbb{F}}_{\ell}(d_b)$, endowed with the $J_b(\mathbb{Q}_p)$-action
specified in the corollary. Thus we can rewrite this as
\[ (\colim_{K_{p} \ra \{1\}} R\Gamma_{c}(\Sht(G,b,\mu)_{\infty,C}/\ul{K_{p}},\ol{\mathbb{F}}_{\ell}(d_{b})) \otimes^{\mathbb{L}}_{\mathcal{H}(J_{b})} V_{b})[2d_{b}], \]
as desired (see \cite[Corollary~1.9 (1)]{HamannImaiDualizing}).
\end{proof}
\begin{remark}{\label{rem: Weilgroupaction}}
The analogous statement involving the full Hecke operator $T_{\mu}: \D(\Bun_{G},\ol{\bb{F}}_{\ell}) \ra \D(\Bun_{G},\ol{\bb{F}}_{\ell})^{BW_{E_{\mu}}}$ and the natural Weil group action on $R\Gamma_{c}(\mathcal{S}(\mathbf{G},X)_{K^{p},C})$ also holds. The process of matching the actions of the inertia group $I_{E_{\mf{p}}} \subset W_{E_{\mf{p}}}$ is fairly straight forward; in particular, one just needs to check that all the above Cartesian diagrams descend to $\Breve{E}_{\mf{p}}$. The process of matching the Frobenius descent datum is a bit tricky however so we do not address this here; for a relevant discussion of this see \cite[Sections~8.2,8.4.9]{MingjiaPolI}. 
\end{remark}
\section{The Local Results}
In this section, we establish the key local results (Corollaries \ref{cor: appliedperversetexactness},\ref{cor: appliedsplitofsemiorthog}) necessary for the proof of Theorem \ref{theorem: mainthm}. We will let $G/\mathbb{Q}_{p}$ be a quasi-split connected reductive group with a choice of Borel $B$ and maximal torus $T$ as before. For the application of these local results to Theorem \ref{theorem: mainthm}, we will in practice take $G$ to be the base-change of the global group $\mathbf{G}$ to $\bb{Q}_{p}$.

More specifically, in \S 4.1, we review the spectral action of the category of perfect complexes on the stack of L-parameters on $\D(\Bun_{G})$, as well as its relationship to the action of Hecke operators on $\D(\Bun_{G})$. This will allow us, for a semisimple $L$-parameter $\phi$, to introduce a certain localization functor $\D(\Bun_{G}) \ra \D(\Bun_{G})_{\phi}$ that Hansen will study more extensively in Appendix \ref{append: spectralactionproperties}. We establish some basic properties of the localization functor that will be used in the later sections, and then turn our attention to studying the derived category $\D(\Bun_{G})_{\phi}$ for $\phi$ induced from a generic toral L-parameter. In particular, in \S 4.2, we discuss how, assuming the Fargues-Scholze local Langlands correspondence is compatible with a suitably nice form of local Langlands, one can show that every object in $\D(\Bun_{G})_{\phi}$ can be described in terms of constituents of certain explicit parabolic inductions. We then collect various instances for which this compatibility holds, using the results of \cite{BMNH,Ham2,HKW}. 

In \S 4.3, we turn to the perverse $t$-exactness of the Hecke operators on the subcategory $\D(\Bun_{G})_{\phi}$, explaining how the theory of geometric Eisenstein series allows one to verify the perverse t-exactness for the subcategory spanned by the inductions described in \S 4.2, under some possible constraints including the genericity and regularity of a toral parameter $\phi_{T}$ inducing $\phi$. This will allow us to prove our main local results under some possible additional constraints beyond genericity. In Section~4.4, we investigate when
genericity implies the additional
regularity conditions needed for our
local results, and formulate explicit
sufficient hypotheses for the remaining
groups.
\subsection{The Spectral Action}{\label{subsection: SpectralActionSection}}
We will work with $\D(\Bun_{G},\ol{\mathbb{F}}_{\ell})$, the derived category of \'etale $\ol{\mathbb{F}}_{\ell}$-sheaves on the moduli stack of $G$-bundles. Our goal in this section will be to describe a localization $\D(\Bun_{G},\ol{\mathbb{F}}_{\ell})_{\phi_{\mf{m}}} \subset \D(\Bun_{G},\ol{\mathbb{F}}_{\ell})$ for $\mf{m} \subset H_{K_{p}^{\mathrm{hs}}}$ a generic maximal ideal with associated semisimple L-parameter $\phi_{\mf{m}}$ in the case that $G$ is unramified. We will do this in slightly more generality using the spectral action \cite[Section~X.3]{FS}. Throughout this section, we will assume that $\ell \nmid \pi_{0}(Z(G))$ where $Z(G)$ denotes the center of $G$, so that we have access to the spectral action constructed in \cite[Theorem~I.10.1]{FS}. In particular, we consider the moduli stack $\mf{X}_{\hat{G}}/\Spec{\ol{\mathbb{F}}_{\ell}}$ of $\ol{\mathbb{F}}_{\ell}$-valued Langlands parameters, as defined in \cite{DH,Zhu1}, and let $\Perf(\mf{X}_{\hat{G}})$ denote the stable $\infty$-category of perfect complexes on this stack. We write $\Perf(\mf{X}_{\hat{G}})^{BW_{\mathbb{Q}_{p}}^{I}}$ for the stable $\infty$-category of objects with a continuous $W_{\mathbb{Q}_{p}}^{I}$ action for a finite index set $I$ (as defined in \cite[Section~IX.1]{FS}), and $\D(\Bun_{G},\ol{\mathbb{F}}_{\ell})^{\omega}$ for the stable idempotent complete subcategory of compact objects in $\D(\Bun_{G},\ol{\mathbb{F}}_{\ell})$. By the integral spectral-action theorem
\cite[Theorem~X.0.2, cf. Theorem~I.10.1]{FS} and its construction at the end
of \cite[Section~X.3]{FS}, followed by extension of scalars from integral
coefficients to $\ol{\mathbb{F}}_{\ell}$, there exists an
$\ol{\mathbb{F}}_{\ell}$-linear action
\[ \Perf(\mf{X}_{\hat{G}}) \ra \mathrm{End}(\D(\Bun_{G},\ol{\mathbb{F}}_{\ell})^{\omega}) \]
\[ C \mapsto \{A \mapsto C\star A \}\]
which, extending by filtered colimits, gives
\[ \mathrm{Ind}(\Perf(\mf{X}_{\hat{G}})) \ra \mathrm{End}(\D(\Bun_{G},\ol{\mathbb{F}}_{\ell})), \]
where the notation $\mathrm{Ind}(-)$ is the Ind-completion operation described in \cite[Section~5.3]{HTT}. We recall the following basic properties of this action\footnote{These properties follow from the integral construction at the end of
\cite[Section~X.3]{FS}, using \cite[Theorem~X.0.2]{FS}, and are preserved by
extension of scalars to $\ol{\bb{F}}_{\ell}$. Compatibility with the
functorial in $I$ Hecke action gives (1), and monoidality gives (2).}.
\begin{enumerate}
    \item For $V = \boxtimes_{i \in I} V_{i}\in \Rep_{\ol{\mathbb{F}}_{\ell}}(\phantom{}^{L}G^{I})$, there is an attached vector bundle $C_{V} \in \Perf(\mf{X}_{\hat{G}})^{BW_{\mathbb{Q}_{p}}^{I}}$ whose pullback to a $\ol{\mathbb{F}}_{\ell}$-point of $\mf{X}_{\hat{G}}$ corresponding to a (not necessarily semisimple) $L$-parameter $\tilde{\phi}: W_{\mathbb{Q}_{p}} \ra \phantom{}^{L}G(\ol{\mathbb{F}}_{\ell})$ is the vector space $V$ with $W_{\mathbb{Q}_{p}}^{I}$-action given by $\boxtimes_{i \in I} r_{V_{i}} \circ \tilde{\phi}$, where $r_{V_{i}}: \phantom{}^{L}G \ra \GL(V_{i})$ is morphism defined by $V_{i}$. The endomorphism
    \[ C_{V} \star (-): \D(\Bun_{G},\ol{\mathbb{F}}_{\ell}) \ra \D(\Bun_{G},\ol{\mathbb{F}}_{\ell})^{BW_{\mathbb{Q}_{p}}^{I}} \]
    is the Hecke operator $T_{V}$ attached to $V$.  
    \item The action is monoidal in the sense that, given $C_{1},C_{2} \in \Perf(\mf{X}_{\hat{G}})$, we have a natural equivalence of endofunctors 
    \[ (C_{1} \otimes^{\mathbb{L}} C_{2}) \star (-) \simeq C_{1} \star (C_{2} \star (-)). \]
\end{enumerate}
\begin{remark}
The Hecke operators appearing here are slightly different from the ones discussed in \S \ref{sec: manprodformula}. For a geometric dominant cocharacter $\mu$, take the
Weil-equivariant highest weight tilting module $\mathcal T_\mu$
of $\hat G\rtimes W_{E_\mu}$ defining $\mathcal S_\mu$ under
geometric Satake, where $E_\mu$ is the reflex field of $\mu$
(cf. \cite[Section~9.1]{Ham2}). This can then be induced along the inclusion $W_{E_{\mu}} \subset W_{\mathbb{Q}_{p}}$ to get a full representation of $\phantom{}^{L}G$ which we denote by $V$. Correspondingly, there is a Hecke operator
\[ T_{V}: \D(\Bun_{G},\ol{\mathbb{F}}_{\ell}) \ra \D(\Bun_{G},\ol{\mathbb{F}}_{\ell})^{BW_{\mathbb{Q}_{p}}} \]
attached to the representation of $\phantom{}^{L}G$. This Hecke operator factors as 
\[T_{\mu}: \D(\Bun_{G},\ol{\mathbb{F}}_{\ell}) \ra \D(\Bun_{G},\ol{\mathbb{F}}_{\ell})^{BW_{E_{\mu}}} \]
composed with the map $\D(\Bun_{G},\ol{\mathbb{F}}_{\ell})^{BW_{E_{\mu}}} \ra  \D(\Bun_{G},\ol{\mathbb{F}}_{\ell})^{BW_{\mathbb{Q}_{p}}}$ induced by induction along the inclusion $W_{E_{\mu}} \subset W_{\mathbb{Q}_{p}}$ (cf. the discussion on \cite[Page~322]{FS}). The Hecke operator $T_{\mu}: \D(\Bun_{G},\ol{\mathbb{F}}_{\ell}) \ra \D(\Bun_{G},\ol{\mathbb{F}}_{\ell})^{BW_{E_{\mu}}}$ was the one appearing in \S \ref{sec: manprodformula}. We will often cite results for Hecke operators $T_{V}$ for $V \in \Rep_{\ol{\mathbb{F}}_{\ell}}(\phantom{}^{L}G)$ to deduce results about the $T_{\mu}$, implicitly using this relationship in the process.
\end{remark}
If $\phi: W_{\mathbb{Q}_{p}} \ra \phantom{}^{L}G(\ol{\mathbb{F}}_{\ell})$ is a semisimple $L$-parameter then it defines a closed $\ol{\mathbb{F}}_{\ell}$-point $x$ in the moduli stack of Langlands parameters, which maps to a closed point in the coarse moduli space by \cite[Proposition~VIII.3.2]{FS}. We let $\mf{m}_{\phi} \subset \mathcal{O}_{\mf{X}_{\hat{G}}}(\mf{X}_{\hat{G}})$ denote the corresponding maximal ideal. We recall that, for all $f \in \mathcal{O}_{\mf{X}_{\hat{G}}}(\mf{X}_{\hat{G}})$ and $A \in \D(\Bun_{G},\ol{\mathbb{F}}_{\ell})$, one obtains an endomorphism 
\[ A \simeq \mathcal{O}_{\mf{X}_{\hat{G}}} \star A \ra  \mathcal{O}_{\mf{X}_{\hat{G}}} \star A \simeq A \]
induced by multiplication by $f$. Under the description of $\mathcal{O}_{\mf{X}_{\hat{G}}}(\mf{X}_{\hat{G}})$ in terms of the excursion algebra, this encodes the action of the excursion algebra on $\D(\Bun_{G},\ol{\mathbb{F}}_{\ell})$ by \cite[Theorem~5.2.1]{Zou}. More precisely, we recall that, to any Schur-irreducible $A \in \D(\Bun_{G},\ol{\mathbb{F}}_{\ell})$, we can, by \cite[Proposition~I.9.3]{FS}, attach a conjugacy class of semisimple $L$-parameters 
\[ \phi_{A}^{\mathrm{FS}}: W_{\mathbb{Q}_{p}} \ra \phantom{}^{L}G(\ol{\mathbb{F}}_{\ell}) \]
called the Fargues-Scholze parameter of $A$. By \cite[Theorem~VIII.3.6]{FS}, we have an identification between the ring of global functions $\mathcal{O}_{\mf{X}_{\hat{G}}}(\mf{X}_{\hat{G}})$ and the ring of excursion operators. Since $A$ is Schur irreducible, the endomorphism corresponding to
$f \in \mathcal{O}_{\mf{X}_{\hat{G}}}(\mf{X}_{\hat{G}})$ acts by the scalar
$f(\phi_A^{\mathrm{FS}}) \in \ol{\mathbb{F}}_{\ell}$, which is nonzero
precisely when $f\notin\mf{m}_{\phi_A^{\mathrm{FS}}}$.

With this in hand, we can make our key definition.
\begin{definition}
We define $\iota_{\phi}: \D(\Bun_{G},\ol{\mathbb{F}}_{\ell})_{\phi} \hookrightarrow \D(\Bun_{G},\ol{\mathbb{F}}_{\ell})$ to be the full subcategory of objects $A$ for which the endomorphisms $A \ra A$ induced by $f \in \mathcal{O}_{\mf{X}_{\hat{G}}}(\mf{X}_{\hat{G}}) \setminus \mf{m}_{\phi}$ are isomorphisms.
\end{definition}
It is easy to check that the subcategory $\D(\Bun_{G},\ol{\mathbb{F}}_{\ell})_{\phi} \subset \D(\Bun_{G},\ol{\mathbb{F}}_{\ell})$ is preserved under colimits and limits, and therefore, by the $\infty$-categorical adjoint functor theorem \cite[Corollary~5.5.2.9]{HTT}, there exists a left adjoint to the inclusion $\iota_{\phi}$ denoted by $\mathcal{L}_{\phi}$ (see Appendix \ref{append: spectralactionproperties} for details). We define $(-)_{\phi} := \iota_{\phi}\mathcal{L}_{\phi}(-)$. This, by the full faithfulness of $\iota_{\phi}$, will define an idempotent functor. 

We now have the following key Lemma.
\begin{lemma}{\label{lemma: localization map properties}}
The following is true. 
\begin{enumerate}
\item Any Schur irreducible object $A \in \D(\Bun_{G},\ol{\mathbb{F}}_{\ell})_{\phi}$ has Fargues-Scholze parameter equal to $\phi$ as conjugacy classes of parameters. 
\item Given $V \in \Rep_{\ol{\mathbb{F}}_{\ell}}(\phantom{}^{L}G^{I})$, the Hecke operator $T_{V}: \D(\Bun_{G},\ol{\mathbb{F}}_{\ell}) \ra \D(\Bun_{G},\ol{\mathbb{F}}_{\ell})^{BW_{\mathbb{Q}_{p}}^{I}}$ takes the subcategory $\D(\Bun_{G},\ol{\mathbb{F}}_{\ell})_{\phi}$ to $\D(\Bun_{G},\ol{\mathbb{F}}_{\ell})_{\phi}^{BW_{\mathbb{Q}_{p}}^{I}}$, and there is a natural equivalence $T_{V}((-)_{\phi}) \simeq (T_{V}(-))_{\phi}$.
\item  Let $K_p\subset G(\mathbb Q_p)$ be any compact open subgroup,
and let $H_{K_p}$ be the Hecke algebra of level $K_{p}$. The action of $\mathcal{O}_{\mathfrak{X}_{\hat{G}}}(\mathfrak{X}_{\hat{G}})$ on $j_{1!}(\mathrm{cInd}_{K_{p}}^{G(\bb{Q}_{p})}(\ol{\mathbb{F}}_{\ell}))$ induces a natural homomorphism
\[
  \theta_{K_p}:
  \mathcal{O}_{\mathfrak{X}_{\hat G}}
      (\mathfrak{X}_{\hat G})
  \rightarrow  Z(H_{K_p}).
\]
to the center of the Hecke algebra. Let $\phi$ be a semisimple $L$-parameter, with associated
maximal ideal
$\mathfrak m_\phi\subset
\mathcal O_{\mathfrak{X}_{\hat{G}}}
(\mathfrak{X}_{\hat G})$. For any bounded below complex
$A\in D(G(\mathbb Q_p),\overline{\mathbb F}_\ell)$,
there is a natural isomorphism
\[
  R\Gamma(K_p,A)_{\mathfrak m_\phi}
  \simeq R\Gamma(K_p,A_\phi),
\]
where the localization on the left is taken with respect
to the action of the spectral center.

In particular, if
$\mathfrak n\subset Z(H_{K_p})$
is a maximal ideal satisfying
\[
  \theta_{K_p}^{-1}(\mathfrak n)
  =\mathfrak m_\phi,
\]
then
\[
  R\Gamma(K_p,A)_{\mathfrak n}
  \simeq
  R\Gamma(K_p,A_\phi)_{\mathfrak n}.
\]
Here we abuse notation and identify $A \in \D(G(\mathbb{Q}_{p}),\ol{\mathbb{F}}_{\ell})$ with its image under the fully faithful embedding $\D(G(\bb{Q}_{p}),\ol{\bb{F}}_{\ell}) \xrightarrow{j_{1!}} \D(\Bun_{G},\ol{\bb{F}}_{\ell})$, where the localization on this subcategory is well-defined using Proposition \ref{prop:localbasics}. 

Moreover, suppose that $G$ is unramified and $K_p$ is a parahoric
subgroup. Let $\mathfrak n\subset Z(H_{K_p})$ be a maximal ideal,
and let $\phi_{\mathfrak n}$ be the conjugacy class of unramified
semisimple $L$-parameters corresponding to $\mathfrak n$ under the
normalized Bernstein isomorphism (see Lemma~\ref{lem:iwahori-generic-fiber}~(1) and the argument below):
\[
  Z(H_{K_p})\simeq
  \mathcal O(\hat T_\sigma)^{W_{G}}
  =\ol{\bb{F}}_{\ell}[X_*(T)^\sigma]^{W_{G}}.
\]
Here $\sigma$ is the Frobenius acting on $\hat G$, $\hat{T}_{\sigma}=\hat{T}/(1-\sigma)\hat{T}$ is the coinvariant torus, where $(1-\sigma)\hat{T}$ is the image of $t\mapsto t\sigma(t)^{-1}$, and we recall that the relative Weyl group $W_{G}$
identifies with the $\sigma$-fixed points of the Weyl group of $\hat{T} \subset \hat{G}$; see \cite[Lemma~6.1.2 and Remark~6.1.3]{HainesRostami}. Then $\theta_{K_p}$ is surjective,
$\theta_{K_p}^{-1}(\mathfrak n)=\mathfrak m_{\phi_{\mathfrak n}}$,
and there is a natural isomorphism
\[
  R\Gamma(K_p,A)_{\mathfrak n}
  \simeq R\Gamma(K_p,A_{\phi_{\mathfrak n}}).
\]
In particular, when $K_p=K_p^{\mathrm{hs}}$ is hyperspecial,
$Z(H_{K_p})=H_{K_p}$ and the Bernstein isomorphism agrees with the
normalized Satake isomorphism, recovering the corresponding
statement for maximal ideals of the spherical Hecke algebra.
\item If $A \in \D(\Bun_{G},\ol{\mathbb{F}}_{\ell})$ is ULA then one has a direct sum decomposition
\[ A \simeq \bigoplus_{\phi} A_{\phi} \]
ranging over all semisimple $L$-parameters.
\end{enumerate}
\end{lemma}
First we need some preparations for the proof of the final assertion of part (3). suppose that $G$ is unramified and that $K=K_p$ is parahoric. After conjugating $K$, choose an Iwahori subgroup $I\subset K$ compatible with $T$ and $B$, and write $\Lambda:=X_*(T)^\sigma$, $R:=\mathcal O(\hat T_\sigma)= \ol{\bb{F}}_{\ell}[\Lambda]$, and $Z_{I}:=R^{W_{G}}$, so that $\Lambda$ identifies with the translation subgroup
$T(\mathbb Q_p)/T(\mathbb Z_p)$ of the Iwahori--Weyl group, where $T(\mathbb Z_p)$ denotes the integral points of the unramified
torus model of $T$. We write $q_\lambda:= |It^{\lambda}I/I|$ and let $T_{t^\lambda}:=\mathbf{1}_{It^\lambda I}$ be the characteristic function. For $\lambda\in\Lambda$, choose $\lambda_1,\lambda_2\in\Lambda$ antidominant with respect to $B$ such that $\lambda=\lambda_1-\lambda_2$. We consider the Bernstein elements
\[
  E(\lambda):=q_\lambda^{1/2}q_{\lambda_1}^{-1/2}q_{\lambda_2}^{1/2}
  T_{t^{\lambda_1}}T_{t^{\lambda_2}}^{-1}.
\]
All square roots use the fixed choice of $\sqrt p$. This expression is independent of the choice of $\lambda_1,\lambda_2$; see \cite[Section~4.3 and the proof of Proposition~4.4.3]{BoumasmoudParahoric}. In particular, $E(\lambda)=T_{t^\lambda}$ when $\lambda$ is antidominant. We set
\[
  \Theta_\lambda:=q_\lambda^{-1/2}E(\lambda).
\]
We denote the basis elements of the group algebra $\ol{\bb{F}}_{\ell}[\Lambda]$ by $e^{\lambda}$ for $\lambda \in \Lambda$. We consider the injective algebra homomorphism
\[
  \iota:R = \ol{\bb{F}}_{\ell}[\Lambda]\longrightarrow H_I \]
  \[ e^{\lambda} \mapsto\Theta_\lambda, \]
which is an algebra homomorphism because $\Theta_\lambda\Theta_\mu=\Theta_{\lambda+\mu}$ by \cite[equation~(3)]{VigHeckeII}.
We refer to its image as the Bernstein subalgebra of $H_I$, and identify it with $R$. We now record the following basic properties of the situation.
\begin{lemma}\label{lem:iwahori-generic-fiber}
The following is true.
\begin{enumerate}
\item We have an isomorphism
\[ Z_I:=Z(H_{I}) \simeq R^{W_{G}} \]
and an identification 
\begin{equation}{\label{eqn: BernsteinPresentation}}
 H_{I} = \bigoplus_{w \in W_{G}} RT_{w},
\end{equation}
where $T_w=\mathbf{1}_{I\dot wI}$, with $I$ of volume $1$ and representatives $\dot w$ of $w$ chosen in the normalizer of the maximal split torus inside a hyperspecial subgroup containing $I$.
\item We have that $Z_{I}=Z(H_I)$ is a normal domain, and $H_I$ is finite and torsion-free as a $Z_{I}$-module.
\item There is an isomorphism of $\mathrm{Frac}(Z_{I})$-algebras
\[
  H_I\otimes_{Z_{I}} \mathrm{Frac}(Z_{I})
  \simeq \mathrm{Frac}(R) \rtimes W_{G}
  \simeq \operatorname{End}_{\mathrm{Frac}(Z_{I})}(\mathrm{Frac}(R))
  \simeq M_{|W_{G}|}(\mathrm{Frac}(Z_{I})).
\]
In particular, the generic fiber of $H_I$ over its center is
split central simple.
\end{enumerate}
\end{lemma}

\begin{proof}
The identification $Z(H_I)=R^{W_{G}}$ follows from
\cite[Theorem~1.2 and Remark~1.4]{VigHeckeII}. The identification \eqref{eqn: BernsteinPresentation}
is the standard Bernstein presentation; see \cite[Corollary~4.3.6 and Proposition~4.4.3]{BoumasmoudParahoric}. This completes the proof of (1).

Since $R = \ol{\bb{F}}_{\ell}[\Lambda]$ is clearly a normal domain, its invariant
subring $Z_{I}=R^{W_{G}}$ is normal by
\cite[Proposition~2.3.11]{DerksenKemper}.

Since $W_{G}$ is finite, $R$ is integral over $Z_{I}=R^{W_{G}}$.
As $R$ is finitely generated over $\ol{\bb{F}}_{\ell} \subset Z_{I}$, it is therefore finite as a $Z_{I}$-module; see
\cite[Tag~0BRI, proof, and Tag~02JJ]{stacks-project}.
The Bernstein presentation \eqref{eqn: BernsteinPresentation}
then shows that $H_I$ is finite as a $Z_{I}$-module. It also shows torsion-freeness: multiplication by any nonzero element of $Z_{I}\subset R$ is injective on each summand $RT_w$. This completes the proof of (2).

For $0\neq r\in R$, the product
\[
  N(r):=\prod_{w\in W_{G}}w(r)
\]
belongs to $Z_{I}\setminus\{0\}$, and
\[
  r^{-1}
  =\frac{\prod_{w\ne 1}w(r)}{N(r)}.
\]
Consequently, $R\otimes_{Z_{I}} \mathrm{Frac}(Z_{I}) = \mathrm{Frac}(R)$. The same argument gives $\mathrm{Frac}(R)^{W_{G}} = \mathrm{Frac}(Z_{I})$.

The Bernstein presentation therefore gives
\begin{equation}{\label{eqn: FractionFieldBernsteinPresentation}}
  H_I\otimes_{Z_{I}} \mathrm{Frac}(Z_{I}) =\bigoplus_{w\in W_{G}}\mathrm{Frac}(R)T_w.
\end{equation}
In particular, $H_I\otimes_{Z_{I}} \mathrm{Frac}(Z_{I})$ contains $\mathrm{Frac}(R)$ and has dimension $|W_{G}|$
as a left $\mathrm{Frac}(R)$-vector space.

We first construct an action of $H_I\otimes_{Z_{I}}\mathrm{Frac}(Z_{I})$ on $\mathrm{Frac}(R)$.
Let $H_0\subset H_I$ be the finite Hecke subalgebra with
basis $\{T_w:w\in W_{G}\}$, and let $k_{\mathrm{triv}}$ be the one-dimensional module on which $T_s$ acts by $q_s:=|I\dot sI/I|$
for every simple reflection $s$.
The Bernstein decomposition identifies
\[
M := H_I\otimes_{H_0}k_{\mathrm{triv}}\simeq R
\]
\[ r \otimes 1 \mapsto r \]
as left $R$-modules. Tensoring over $Z_{I}$ with $\mathrm{Frac}(Z_{I})$ therefore gives an $H_{I} \otimes_{Z_{I}} \mathrm{Frac}(Z_{I})$-module
\[
  M\otimes_{Z_{I}} \mathrm{Frac}(Z_{I}) \simeq \mathrm{Frac}(R),
\]
and hence a $\mathrm{Frac}(Z_{I})$-algebra homomorphism
\[
  \rho: H_{I} \otimes_{Z_{I}} \mathrm{Frac}(Z_{I}) \rightarrow\operatorname{End}_{\mathrm{Frac}(Z_{I})}(\mathrm{Frac}(R)).
\]

For each simple reflection $s$, the Bernstein relation
\cite[Proposition~4.4.3 and Lemma~6.1.1]{BoumasmoudParahoric},
written in normalized Bernstein coordinates, gives
\[
  T_s f-s(f)T_s=c_s(f-s(f))
\]
for $f\in\mathrm{Frac}(R)$, where $c_s=G(\alpha_s)-1\in\mathrm{Frac}(R)$ and $\alpha_s\in\Lambda$ is the simple root attached to $s$ in the root datum of the affine Hecke algebra; see
\cite[Definition~2.2.1]{BarbaschCiubotaru}. The explicit formulas for $G(\alpha)$ in
\cite[Definition~2.2.1, equation~(2.2.4)]{BarbaschCiubotaru} are products and quotients of $1-ae^\nu$ with
$a\in \ol{\bb{F}}_{\ell}^\times$ and $0\ne\nu\in\Lambda$.
Those same formulas give
\[
  d_s :=q_s-c_s=s\bigl(G(\alpha_s)\bigr)\ne0.
\]
Thus, $d_s$ is invertible in $\mathrm{Frac}(R)$. Since $T_s$ acts on $1\in M$ by $q_s$, the previous Bernstein
relation gives, for every $f\in \mathrm{Frac}(R)$,
\[
  \rho(T_s)(f)
  =q_s\,s(f)+c_s(f-s(f))
  =c_s f+d_s\,s(f).
\]
Consequently, the element
\[
  \tau_s:=d_s^{-1}(T_s-c_s)\in H_I\otimes_{Z_{I}}\mathrm{Frac}(Z_I)
\]
acts on $\mathrm{Frac}(R)$ precisely by the automorphism $s$. In other words,
\[
  \rho(\tau_s)(f)=s(f).
\]
Thus the image of $\rho$ contains multiplication by
every element of $\mathrm{Frac}(R)$ and every automorphism in $W_{G}$.

We now show that these operators generate
$\operatorname{End}_{\mathrm{Frac}(Z_{I})}(\mathrm{Frac}(R))$.
The action of $W_{G}$ on $\mathrm{Frac}(R)$ is faithful, and we have already
shown that $\mathrm{Frac}(Z_{I}) = \mathrm{Frac}(R)^{W_{G}}$.
Artin's theorem therefore gives
\[
  [\mathrm{Frac}(R):\mathrm{Frac}(Z_{I})]=|W_{G}|;
\]
see \cite[Tag~09I3]{stacks-project}. This in particular gives an identification
\[
\operatorname{End}_{\mathrm{Frac}(Z_I)}(\mathrm{Frac}(R))
  \simeq M_{|W_{G}|}(\mathrm{Frac}(Z_I)).
\]
Consider the natural map
\[ \psi: \mathrm{Frac}(R) \rtimes W_{G} \rightarrow\operatorname{End}_{\mathrm{Frac}(Z_{I})}(\mathrm{Frac}(R)), \]
\[
  (a,w)\longmapsto\bigl(x\longmapsto a\,w(x)\bigr).
\]
It is an algebra homomorphism because the multiplication
in the skew group algebra is
\[
  (a,w)(b,w')=(aw(b),ww').
\]
It is injective by the linear independence of distinct
field automorphisms over $\mathrm{Frac}(R)$.
Moreover,
\[
\begin{aligned}
  \dim_{\mathrm{Frac}(Z_{I})}(\mathrm{Frac}(R)\rtimes W_{G})
  &=|W_{G}|[\mathrm{Frac}(R):\mathrm{Frac}(Z_{I})]
  =|W_{G}|^2\\
  &=\dim_{\mathrm{Frac}(Z_{I})}(\operatorname{End}_{\mathrm{Frac}(Z_{I})}(\mathrm{Frac}(R))),
\end{aligned}
\]
so $\psi$ is an isomorphism.

Since the image of $\rho$ contains all the operators
$x\mapsto a\,w(x)$ for $a\in\mathrm{Frac}(R)$ and $w\in W_{G}$, it follows that $\rho$ is surjective.
The Bernstein decomposition also gives
\[
  \dim_{\mathrm{Frac}(Z_{I})}(H_{I} \otimes_{Z_{I}} \mathrm{Frac}(Z_{I}))=|W_{G}|[\mathrm{Frac}(R):\mathrm{Frac}(Z_{I})]=|W_{G}|^2.
\]
Thus $\rho$ is an isomorphism as well.

In particular, faithfulness of $\rho$ shows that the
$\tau_s$ satisfy the Coxeter relations. They therefore define elements $\tau_w$, for $w\in W_{G}$,
satisfying
\[
  \tau_w f=w(f)\tau_w
\]
and 
\[  \tau_w\tau_v=\tau_{wv}. \]
This gives an isomorphism
\[
  \mathrm{Frac}(R) \rtimes W_{G} \xrightarrow{\sim} H_{I} \otimes_{Z_{I}} \mathrm{Frac}(Z_{I})
\]
which sends $(a,w)$ to $a\tau_w$. This completes the proof of (3).
\end{proof}

\begin{proof}[Proof of Lemma~\ref{lemma: localization map properties}]
Claims (2) and (4) follow from Proposition \ref{prop:localbasics} and Proposition \ref{prop:localizeula}, respectively, where for claim (2) we use the relationship between Hecke operators and the spectral action described above.

For (1), this follows since the action of $\mathcal{O}_{\mf{X}_{\hat{G}}}(\mf{X}_{\hat{G}})$ on $A$ will factor through the quotient by the maximal ideal $\mf{m}_{A}$ defined by the semisimple $L$-parameter $\phi_{A}^{\mathrm{FS}}$ attached to $A$ by the above discussion, and therefore having $A \in \D(\Bun_{G},\ol{\mathbb{F}}_{\ell})_{\phi}$ forces an equality of maximal ideals: $\mf{m}_{A} = \mf{m}_{\phi}$.

For~(3), write
\[
  \mathfrak Z_G =
  \mathcal O_{\mathfrak{X}_{\hat G}}
  (\mathfrak{X}_{\hat G}).
\]
Using the spectral action, we may act on $\mathrm{c\text{-}Ind}_{K_p}^{G(\mathbb Q_p)}
\overline{\mathbb F}_\ell$, this gives a homomorphism
\[
  \mathfrak Z_G
  \longrightarrow
  \operatorname{End}_{G(\mathbb Q_p)}
  \bigl(\mathrm{cInd}_{K_p}^{G(\mathbb Q_p)}
  \overline{\mathbb F}_\ell\bigr)
  \simeq H_{K_p}^{\mathrm{op}}.
\]
By naturality of the spectral action, its image is central.
This defines $\theta_{K_p}$.

Choose an open normal pro-$p$ subgroup
$K_p^0\triangleleft K_p$, and set $Q=K_p/K_p^0$.
Since $\ell\neq p$, taking $K_p^0$-invariants is
exact and commutes with filtered colimits. Moreover,
\[
  R\Gamma(K_p,A)
  \simeq R\Gamma(Q,A^{K_p^0}).
\]
For uniformly bounded below complexes, the finite-group
cochain complex computing $R\Gamma(Q,-)$ commutes with
filtered colimits, since its terms are finite direct
sums and only finitely many cochain degrees contribute
to each total degree. Consequently, $R\Gamma(K_p,-)$
commutes with filtered colimits of uniformly bounded
below complexes.

By Proposition~\ref{prop:localbasics}, localization in the spectral Bernstein center at
$\phi$ is obtained by taking the filtered colimit
inverting the elements of
$\mathfrak Z_G\setminus\mathfrak m_\phi$.
It follows that
\[
  R\Gamma(K_p,A_\phi)
  \simeq
  R\Gamma(K_p,A)_{\mathfrak m_\phi}.
\]
Finally, if
$\theta_{K_p}^{-1}(\mathfrak n)=\mathfrak m_\phi$,
the elements inverted in this localization
are already invertible after localization at
$\mathfrak n$. Localizing the preceding isomorphism
at $\mathfrak n$ gives
\[
  R\Gamma(K_p,A)_{\mathfrak n}
  \simeq
  R\Gamma(K_p,A_\phi)_{\mathfrak n},
\]
as desired.

It remains to prove the final assertion for a parahoric subgroup $K=K_p$. Choose $I\subset K$ as above, and use the notation of Lemma~\ref{lem:iwahori-generic-fiber}.
We first identify $\theta_I$. For any unramified character
$\chi:T(\mathbb Q_p) \ra \ol{\bb{F}}_{\ell}^\times$, put
$\pi_\chi=i_B^G(\chi)$, and let $\phi_\chi$ be the unramified
parameter obtained from the corresponding toral parameter.
The space $\pi_\chi^I$ is nonzero, and $Z_{I}$ acts on it by
evaluation at the $W_{G}$-orbit of $\chi$.
Compatibility of the Fargues-Scholze spectral center with local
class field theory for tori and with parabolic induction
(\cite[Proposition~IX.6.5 and Corollary~IX.7.3]{FS}) shows that
$f\in\mathfrak Z_G$ acts on $\pi_\chi$ by the scalar
$f(\phi_\chi)$. Since $Z_{I}$ is reduced by Lemma~\ref{lem:iwahori-generic-fiber}~(2) and is a finitely generated $\ol{\bb{F}}_{\ell}$-algebra by \cite[Theorem~1.2]{VigHeckeII}, these evaluations identify
$\theta_I$ with the composite
\[
  \mathfrak Z_G\longrightarrow
  \mathcal O(\hat G//_\sigma\hat G)
  \xrightarrow{\sim} Z_{I} = Z(H_I).
\]
The first map is surjective by \cite[Lemma~9.4.2]{MingjiaPolI}, and the middle isomorphism is twisted Chevalley restriction
\cite[Proposition~4.2.3]{XZ19}, together with
\cite[Lemma~9.4.6]{MingjiaPolI}. Thus $\theta_I$ is surjective.
The normalization of the Bernstein isomorphism and of parabolic induction
is understood to use the same choice of $\sqrt p$ as the geometric
Satake correspondence used in the definition of the spectral action; when switching between left and right Hecke
actions we use the usual involution $IhI\mapsto Ih^{-1}I$.

We next pass from $I$ to $K$. For $L=I,K$, put
$P_L:=\mathrm{cInd}_L^{G(\mathbb Q_p)}(\ol{\bb{F}}_{\ell})$, and write
$\mathcal H:=\operatorname{End}_{G(\mathbb Q_p)}(P_I)
\simeq H_I^{\mathrm{op}}$. Let $\mathfrak a$ be the alcove corresponding to $I$, and let $\mathfrak f\subset\overline{\mathfrak a}$ be the facet corresponding to $K$. We let $J$ be the set of reflections in the walls of $\mathfrak a$ containing $\mathfrak f$, and let $W_J:=\langle J\rangle$ be the corresponding finite
parabolic subgroup of the affine Weyl group; see \cite[Definition~2.5.2]{BoumasmoudParahoric}. Then
\[
  K=\bigsqcup_{w\in W_J}I\dot wI
\]
by \cite[Proposition~3.5.1]{BoumasmoudParahoric}. Using the preceding involution to transport basis $T_{\tilde{w}}$ for $\tilde{w}$ in the extended affine Weyl group to $\mathcal H$, we set
\[
  x_J:=\sum_{w\in W_J}T_w,\qquad M_J:=x_J\mathcal H.
\]
Taking Iwahori invariants identifies $P_K^I$ with the right
$\mathcal H$-module $M_J$ and gives
\[
  H_K^{\mathrm{op}}\simeq
  \operatorname{End}_{\mathcal H}(M_J)
\]
by \cite[Theorem~2.3 and Sections~4.2.13--4.2.14]{VigSchur}. Since $Z_{I}$ is noetherian, Lemma~\ref{lem:iwahori-generic-fiber}~(2) shows that, as a nonzero submodule of $\mathcal H$, the module $M_J$ is finite and
torsion-free over $Z_{I}$. Consequently
$E_J:=\operatorname{End}_{\mathcal H}(M_J)$ is also finite and
torsion-free over $Z_{I}$. By Lemma~\ref{lem:iwahori-generic-fiber}~(3), the generic fiber of $\mathcal H$ is a matrix algebra over $\mathrm{Frac}(Z_I)$. The generic fiber of $M_J$ is a nonzero finite module over this matrix algebra, so its endomorphism algebra, which is the generic fiber of $E_J$, is also a matrix algebra over $\mathrm{Frac}(Z_I)$.
It follows that $Z(E_J)$ embeds into $\mathrm{Frac}(Z_I)$
and is integral over $Z_{I}$. Normality of $Z_{I}$, again by Lemma~\ref{lem:iwahori-generic-fiber}~(2), therefore gives
\[
  Z_{I}\xrightarrow{\sim}Z(E_J)\simeq Z(H_K).
\]

The natural inclusion
\[
  \mathrm{cInd}_K^{G(\mathbb Q_p)}(\ol{\bb{F}}_{\ell})
  \hookrightarrow\mathrm{cInd}_I^{G(\mathbb Q_p)}(\ol{\bb{F}}_{\ell})
\]
identifies on Iwahori invariants with $M_J\subset\mathcal H$.
Naturality of the spectral action therefore identifies $\theta_K$
with $\theta_I$ followed by $Z_{I}\simeq Z(H_K)$. In particular,
$\theta_K$ is surjective and
$\theta_K^{-1}(\mathfrak n)=\mathfrak m_{\phi_{\mathfrak n}}$.
Surjectivity also gives
\[
  Z(H_K)\otimes_{\mathfrak Z_G}
  (\mathfrak Z_G)_{\mathfrak m_{\phi_{\mathfrak n}}}
  \simeq Z(H_K)_{\mathfrak n}.
\]
Combining this with the first assertion of~(3), proved above, we obtain
\[
  R\Gamma(K,A)_{\mathfrak n}
  \simeq R\Gamma(K,A)_{\mathfrak m_{\phi_{\mathfrak n}}}
  \simeq R\Gamma(K,A_{\phi_{\mathfrak n}}),
\]
as desired. At hyperspecial level the Bernstein isomorphism agrees
with normalized Satake; cf.\ \cite[Section~6.5]{BoumasmoudParahoric}.
\end{proof}

We note that we get the following corollary of this. 
\begin{corollary}{\label{cor: appliedspecdecomp}}
Let $A$ be a complex of smooth $G(\mathbb{Q}_{p})$-representations which is perfect admissible (i.e $A^{K}$ is a perfect complex for all open pro-$p$ $K \subset G(\mathbb{Q}_{p})$). We then have a decomposition 
\[ A \simeq \bigoplus_{\phi} A_{\phi} \]
running over semisimple $L$-parameters, where any irreducible constituent $\pi$ of $A_{\phi}$ has Fargues-Scholze parameter equal to $\phi$, as conjugacy classes of parameters. 
\end{corollary}
\begin{proof}
This follows by applying Lemma \ref{lemma: localization map properties} (1) and (4) to the full subcategory $\D(G(\mathbb{Q}_{p}),\ol{\mathbb{F}}_{\ell}) \subset \D(\Bun_{G},\ol{\mathbb{F}}_{\ell})$, where we recall that a sheaf $A \in \D(\Bun_{G},\ol{\bb{F}}_{\ell})$ lies in the ULA subcategory if and only if its pullbacks $j_{b}^{*}(A) \in \D(\Bun_{G}^{b},\ol{\bb{F}}_{\ell}) \simeq \D(J_{b}(\bb{Q}_{p}),\ol{\bb{F}}_{\ell})$ for all $b \in B(G)$ are perfect admissible as in the statement of the corollary, \cite[Theorem~V.7.1]{FS}.
\end{proof}
We now turn our attention to studying the subcategory $\D(\Bun_{G},\ol{\bb{F}}_{\ell})_{\phi}$ in certain cases coming from generic toral parameters, by using compatibility of the Fargues-Scholze local Langlands correspondence with more classical and well understood instances of local Langlands.
\subsection{Local-Global Compatibility of the Fargues-Scholze Local Langlands}
Now our goal is to describe the subcategory $\D(\Bun_{G},\ol{\mathbb{F}}_{\ell})_{\phi}$ more explicitly, using the results of \cite{Ham2} in the case that $\phi$ is induced from a generic toral parameter $\phi_{T}$, as in Definition \ref{def: generic}. To do this, we will need to have some information about the Fargues-Scholze local Langlands correspondence in the form of a local-global compatibility statement. We first introduce the precise compatibility we will need. 
\subsubsection{Implications of Compatibility for Generic Parameters}
We let $B(G)_{\mathrm{un}} := \Im{(B(T) \ra B(G))}$. We recall that these are precisely the elements $b \in B(G)$ such that the $\sigma$-centralizer $J_{b}$ is quasi-split (\cite[Lemma~2.13]{Ham2}). In particular, the fixed choice of Borel $B \subset G$ transfers to a Borel subgroup $B_{b}$ for all $b \in B(G)_{\mathrm{un}}$, and $J_{b} \simeq M_{b}$ under the inner twisting, where $M_{b} \subset G$ is the Levi subgroup of $G$ determined by the centralizer of the slope homomorphism of $b$ in $G$. We let $\delta_{P_{b}}$ denote the modulus character of the standard parabolic $P_{b}$ of $G$ with Levi factor $M_{b}$ transferred to $J_{b}$ under the inner twisting (see \cite[Definition~3.14]{HamannImaiDualizing} for a more precise discussion, where we use the same conventions for the modulus character as \textit{loc.cit}). We set $W_{b} := W_{M_{b}} \backslash W_{G}$\footnote{We note that here we use the left cosets, as opposed to \cite{Ham2} which uses the right cosets.} to be the quotient of the relative Weyl group of $G$ by the relative Weyl group of $M_{b}$. We identify $W_{b}$ with a set of minimal-length representatives $w \in W_{G}$. We set $\rho_{b,w} := i_{B_{b}}^{J_{b}}(\chi^{w}) \otimes \delta_{P_{b}}^{-1/2}$ to be the normalized parabolic induction of $\chi^{w}$, where $\chi$ is the character of $T(\mathbb{Q}_{p})$ attached to a toral parameter $\phi_{T}$ under local class field theory and $\delta_{P_{b}}$ is the modulus character of $P_b$, restricted to $M_b$ and transferred to $J_b$ (Here we have slightly abused notation by omitting $\chi$ from the notation, where it will usually be clear from context. If further clarification is necessary then we will write $\rho_{b,w}^{\chi}$.).

We will need to assume the following properties of the Fargues-Scholze local Langlands correspondence, as in \cite[Assumption~6.5]{Ham2}. 
\begin{assumption}{\label{compatibility}}
For a connected reductive group $H/\mathbb{Q}_{p}$, we have
\begin{itemize} 
\item the set $\Pi(H)$ of smooth irreducible $\ol{\mathbb{Q}}_{\ell}$-representations of $H(\mathbb{Q}_{p})$,
\item the set $\Phi(H)$ of conjugacy classes of continuous  maps 
\[ W_{\mathbb{Q}_{p}} \times \SL(2,\ol{\mathbb{Q}}_{\ell}) \ra \phantom{}^{L}H(\ol{\mathbb{Q}}_{\ell}) \]
where $\ol{\mathbb{Q}}_{\ell}$ has the discrete topology, $\mathrm{SL}(2,\ol{\mathbb{Q}}_{\ell})$ acts via an algebraic representation, $W_{\mathbb{Q}_{p}}$ acts through a Frobenius semisimple homomorphism.
\item the set $\Phi^{\mathrm{ss}}(H)$ of continuous (where again $\ol{\mathbb{Q}}_{\ell}$ has the discrete topology) semisimple homomorphisms
\[ W_{\mathbb{Q}_{p}} \ra \phantom{}^{L}H(\ol{\mathbb{Q}}_{\ell}), \]
(in the sense that if it factors through $\phantom{}^{L}P$ for a proper parabolic $P \subset H$ then it also factors through the associated Levi $\phantom{}^{L}M$), and the map respects the action of $W_{\mathbb{Q}_{p}}$ on $\phantom{}^{L}H(\ol{\mathbb{Q}}_{\ell})$, the $L$-group of $H$,
\item and the semi-simplification map $(-)^{\mathrm{ss}}: \Phi(H) \ra \Phi^{\mathrm{ss}}(H)$ defined by precomposition with
\[ W_{\mathbb{Q}_{p}} \ra W_{\mathbb{Q}_{p}} \times \mathrm{SL}(2,\ol{\mathbb{Q}}_{\ell}) \]
\[ g \mapsto (g,\begin{pmatrix} |g|^{1/2} & 0 \\ 0 & |g|^{-1/2} \end{pmatrix}). \]
\end{itemize} 
Then, we assume that there exists a collection of maps
\[ \mathrm{LLC}_{b}: \Pi(J_{b}) \ra \Phi(J_{b}) \]
\[ \rho \mapsto \phi_{\rho}, \]
for all $b \in B(G)$, satisfying the following properties:
\begin{enumerate} 
\item The diagram
\[ \begin{tikzcd}[ampersand replacement=\&]
            \Pi(J_{b})  \ar[rr, "\mathrm{LLC}_{b}"] \arrow[drr,"\mathrm{LLC}^{\mathrm{FS}}_{b}"] \& \&   \Phi(J_{b}) \ar[d,"(-)^{\mathrm{ss}}"] \\
            \& \& \Phi^{\mathrm{ss}}(J_{b})
\end{tikzcd}\] 
commutes, where $\mathrm{LLC}_{b}^{\mathrm{FS}}$ is the Fargues-Scholze local Langlands correspondence for $J_{b}$. 
\item For $\rho$ a non-zero smooth irreducible representation of $J_{b}(\bb{Q}_{p})$ for some $b \in B(G)$, consider $\phi_{\rho}$ as an element of $\Phi(G)$ given by composing with the twisted embedding $\phantom{}^{L}J_{b}(\ol{\mathbb{Q}}_{\ell}) \simeq \phantom{}^{L}M_{b}(\ol{\mathbb{Q}}_{\ell}) \ra \phantom{}^{L}G(\ol{\mathbb{Q}}_{\ell})$ (as defined in \cite[Section~IX.7.1]{FS}). If $\phi_\rho$ factors through the natural
(up to $\hat G$-conjugacy) embedding
${}^LT\longrightarrow{}^LG$, then
$b\in B(G)_{\mathrm{un}}$.
\item If $\rho$ is a non-zero smooth irreducible representation of $J_{b}(\bb{Q}_{p})$ such that $ W_{\mathbb{Q}_{p}} \times \mathrm{SL}(2,\ol{\mathbb{Q}}_{\ell}) \xrightarrow{\phi_{\rho}} \phantom{}^{L}J_{b}(\ol{\mathbb{Q}}_{\ell}) \ra \phantom{}^{L}G(\ol{\mathbb{Q}}_{\ell})$ factors through the embedding $\phantom{}^{L}T \ra \phantom{}^{L}G$ with induced parameter $\phi_{T}$, then, by (2), the element $b$ is unramified, and we require that $\rho$ is isomorphic to an irreducible subquotient of $i_{B_{b}}^{J_{b}}(\chi^{w}) \otimes \delta_{P_{b}}^{-1/2}$ for $w \in W_{b}$, where $\chi: T(\bb{Q}_{p}) \ra \ol{\bb{Q}}_{\ell}^{*}$ is the character attached to $\phi_{T}$ by class field theory and the notation is as above. Here the map $\phantom{}^{L}T \ra \phantom{}^{L}G$ is the non-twisted embedding and the map $\phantom{}^{L}J_{b} \ra \phantom{}^{L}G$ is the twisted embedding.
\end{enumerate} 
\end{assumption}
\begin{remark}{\label{rem: FarguesScholzewithCcoefficients}}
We note that we have endowed $\ol{\bb{Q}}_{\ell}$ with the discrete topology. In particular, the axioms as formulated above are designed to be satisfied for a suitably nice local Langlands correspondence with coefficients in $\bb{C}$ after applying a fixed choice of isomorphism $i: \ol{\bb{Q}}_{\ell} \xrightarrow{\simeq} \bb{C}$. We recall that the geometric Satake correspondence and in turn the Fargues-Scholze parameter also depend on a fixed choice of square root of $p$, which we fixed in the notation section above. In order to have the desired compatibility of a classical local Langlands correspondence over $\bb{C}$ with that of Fargues-Scholze, one needs to fix this choice of square root to be given as $i^{-1}(\sqrt{p})$ and we will implicitly do this in what follows. Note that the square roots of the modulus character appearing in the above induction are also normalized with respect to this choice (cf. \cite[Conventions and Notations, Theorem~1.1]{BMNH}), and that the twisted embedding also has square roots of $p$ normalized with respect to this choice (since they come from geometric Satake). 
\end{remark}

\begin{proposition}\label{prop: constituent proposition}
Assume Assumption~\ref{compatibility} and that $\ell \nmid \pi_{0}(Z(G))$. Let $\phi$ be induced from a toral parameter $\phi_T:W_{\mathbb Q_p}\to{}^LT(\ol{\bb{F}}_{\ell})$, with corresponding character
$\chi:T(\mathbb Q_p)\to \ol{\bb{F}}_{\ell}^*$.
In general, assume that $\phi_T$ is generic. If $G$ is split, it suffices to assume that $\phi_T$ is weakly generic, namely that
\[
  H^2(W_{\mathbb Q_p},\alpha\circ\phi_T)=0\qquad\text{for every coroot }\alpha,
\]
Then every object of $\D^{\ULA}(\Bun_G,\ol{\bb{F}}_{\ell})_\phi$ is only  supported on
$B(G)_{\mathrm{un}}$. Moreover $\D^{\mathrm{adm}}(\Bun_G^b,\ol{\bb{F}}_{\ell})_\phi$ contains an irreducible non-zero smooth representation if and only if $b\in B(G)_{\mathrm{un}}$. For such $b$, these representations are precisely the irreducible
subquotients of
\[
  \rho_{b,w}^{\chi}
  =i_{B_b}^{J_b}(\chi^w)\otimes\delta_{P_b}^{-1/2},
  \qquad w\in W_b.
\]
The assertion on the support of the sheaves only requires part (1) and (2) of Assumption~\ref{compatibility}. Here $\D^{\mathrm{adm}}(\Bun_G^b,\ol{\bb{F}}_{\ell})_\phi$ is defined using
$j_{b!}$ and Proposition~\ref{prop:localbasics} and $\Dadm(\Bun_{G}^{b},\ol{\mathbb{F}}_{\ell}) \subset \D(\Bun_{G}^{b},\ol{\mathbb{F}}_{\ell}) \simeq \D(J_{b}(\mathbb{Q}_{p}),\ol{\mathbb{F}}_{\ell})$ denotes the subcategory of perfect admissible complexes regarded as a full subcategory of $\D(\Bun_{G},\ol{\bb{F}}_{\ell})$ by $!$-pushforward. Lastly, $\Dadm(\Bun_{G}^{b},\ol{\mathbb{F}}_{\ell})_{\phi} := \Dadm(\Bun_{G}^{b},\ol{\mathbb{F}}_{\ell}) \cap \D(\Bun_{G}^{b},\ol{\bb{F}}_{\ell})_{\phi}$  is the intersection with the $\phi$-local subcategory.
\end{proposition}

\begin{proof}
The case of generic parameters follows from the proof of \cite[Corollary~6.7]{Ham2} using Lemma \ref{lemma: localization map properties} (1). To deal with the split case under the weakly generic hypothesis, we combine the argument of
\cite[proof of Corollary~6.7]{Ham2} with the argument of
Caraiani--Scholze given in the proof of \cite[Corollary~5.1.3]{CS2}. 
We assume $G$ is split for the rest of the proof.

Let $I\subset W_{\bb{Q}_{p}}$ be the inertia subgroup, and
choose an arithmetic Frobenius $\sigma$. We note that the condition of being weakly generic implies that $\alpha \circ \phi_{T} \not\simeq |\cdot|^{\pm 1}$, by local Tate-duality

Let $\rho$ be an irreducible smooth $\ol{\bb{F}}_{\ell}$-representation of
$J_b(\mathbb Q_p)$ whose ambient Fargues--Scholze parameter is
$\phi$ after composing with the embedding $\phantom{}^{L}J_{b} \ra \phantom{}^{L}G$. By \cite[Lemma~6.8(i)]{Dat}, there is an integral irreducible
smooth $\overline{\mathbb Q}_\ell$-representation
$\widetilde\rho$ whose reduction contains $\rho$ as a subquotient.
Let $\phi_{\tilde\rho}$ be its $L$-parameter under the full local Langlands correspondence given by assumption \ref{compatibility} after composing with $\phantom{}^{L}J_{b} \ra \phantom{}^{L}G$.
Assumption~\ref{compatibility}\textup{(1)} identifies $\phi_{\tilde{\rho}}^{\mathrm{ss}}$ with the Fargues-Scholze parameter $\phi_{\widetilde\rho}^{\mathrm{FS}}$.
Moreover, \cite[Theorem~IX.5.2]{FS} shows that the semisimplification of the reduction of $\phi_{\tilde{\rho}}^{\mathrm{FS}}$ is conjugate to $\phi$.

Let $E/\mathbb Q_\ell$ be a finite extension over which $\phi_{\tilde{\rho}}$ is
defined and $\mathcal{O}$ the ring of integers with uniformizer $\varpi$ and residue field $k$. We use $\overline{(-)}$ to denote reduction mod $\varpi$ and put $\Gamma = \phi_{\tilde{\rho}}^{\mathrm{ss}}(W_{\bb{Q}_{p}}) \subset \hat G(E)$. We note that $\Gamma$ is bounded. Indeed, this follows from the condition that $\phi_{\tilde{\rho}}^{\mathrm{ss}} := \phi_{\tilde{\rho}}^{\mathrm{FS}}$ and the fact that $\phi_{\tilde{\rho}}^{\mathrm{FS}}$ is necessarily an integral semi-simple parameter by the integrality of $\tilde{\rho}$ and therefore bounded (the image of inertia is finite, and its semisimple Frobenius has eigenvalues in $\mathcal{O}^{*}$). 

Since $\Gamma$ is bounded, the Bruhat--Tits fixed-point theorem (\cite[Proposition~3.2.4]{BT72}) gives a fixed point in the reduced building of $\hat G$. The unique open facet containing this point is $\Gamma$-stable, so the barycenter of its closure is also fixed. This barycenter has rational apartment coordinates. Moreover, $\Gamma$ acts trivially on the additional vector factor expressing the extended building in terms of the reduced building. In particular, if it didn't act trivially the resulting nonzero translation would have unbounded powers contradicting boundedness of $\Gamma$. Choosing zero in this vector factor therefore gives a fixed point $x$ of the enlarged building with rational apartment coordinates.

Choose an apartment containing $x$ and conjugate it to the
standard apartment of $\hat T$, identified with
$X_*(\hat T)\otimes_{\mathbb Z}\mathbb R$ by taking the
standard hyperspecial point as the origin. We then have that
\[
  x\in X_*(\hat T)\otimes_{\mathbb Z}\mathbb Q,
\]
as explained above. Choose an integer $e\geq 1$ such that $ex\in X_*(\hat T)$,
and let $E'/E$ be a totally ramified extension of degree $e$.
The map of apartments under extension of
scalars sends $x$ to $ex$.
Since $\hat T(E')$ acts by translations through
$X_*(\hat T)$, an element of $\hat T(E')$ carries this
point to the origin. After conjugation, $\Gamma$ therefore
fixes the standard point of the enlarged building, and therefore
its full stabilizer is $\hat G(\mathcal O_{E'})$;
see \cite[proof of Proposition~5.1.2]{HKW}.
In particular, after extending coefficients and conjugating,
we may assume that $\phi_{\tilde{\rho}}^{\mathrm{ss}}$ factors through $\hat{G}(\mathcal{O}_{E'})$. We replace $E$ by $E'$ and $\mathcal{O}$ by $\mathcal{O}_{E'}$.

We may further arrange that the parameter $\phi_{\tilde\rho}^{\mathrm{ss}}$ actual reduction is $\phi$ instead of just being given by $\phi$ after further semi-simplifying the reduction (note that the reduction of a semi-simple parameter is not necessarily semi-simple). Put $\psi=\overline{\phi_{\tilde\rho}^{\mathrm{ss}}}$.
By \cite[Proposition~VIII.3.2 and its proof]{FS}, after extending the residue field there is a cocharacter $\lambda$ with
$\psi(W_{\bb Q_p})\subset P_\lambda$ and $c_\lambda\circ\psi\sim\phi$, where $c_\lambda$ is the projection to the Levi.
After taking a further unramified extension of $E$ and conjugation by an element of $\hat G(\mathcal{O}_{E})$, lift $\lambda$ to $\hat T_{\mathcal O}$.
For each $w$, the open-cell decomposition of $\phi_{\tilde\rho}^{\mathrm{ss}}(w)$ has the form $u_-(w)m(w)u_+(w)$, with $u_-(w)\equiv1\pmod\varpi$.
Choose $e$ larger than every positive $\lambda$-weight on $U_\lambda$, and adjoin $\varpi'$ with $(\varpi')^e=\varpi$.
Conjugation by $\lambda(\varpi')$ makes a root coordinate of positive weight $a$ divisible by $(\varpi')^a$ and an root coordinate given by the negative divisible by $(\varpi')^{e-a}$.
Thus the conjugated parameter remains integral and reduces to $c_\lambda\circ\psi$. We keep the notation $E,\mathcal O,\varpi,k$ for the new choices.

We now use the relationship $\overline{(\phi_{\tilde{\rho}}^{\mathrm{ss}})} = \phi$ just established, to show that $\phi_{\tilde{\rho}}$ cannot have monodromy.
Let $N$ be the image of $\begin{pmatrix} 0 & 1 \\ 0 & 0 \end{pmatrix}$
under the differential of
$\phi_{\widetilde\rho}|_{\mathrm{SL}_2}$.
Then
\[
  \operatorname{Ad}(\phi_{\tilde{\rho}}^{\mathrm{ss}}(w))N=|w|N.
\]
Suppose $N\ne0$ for the sake of contradiction, then we scale it to be integral with nonzero reduction
$\overline N \neq 0$. This reduction is nilpotent in a faithful integral representation of $\hat{G}/\mathcal{O}$ and its reduction satisfies
\[
  \operatorname{Ad}(\phi(w))\overline N
  = |w|\overline N.
\]
Now if $\overline{N} \neq 0$ choose a root $\alpha$ such that $\overline{N}_{\alpha} \neq 0$. This would then imply that 
\[ \alpha \circ \phi_{T}(w)\overline{N}_{\alpha} = |w|\overline{N}_{\alpha},\]
which would contradict genericity. Therefore, $N=0$ and the algebraic homomorphism
$\phi_{\tilde\rho}|_{\mathrm{SL}_2}$ is trivial. We may therefore conflate $\phi_{\tilde{\rho}}$ and $\phi_{\tilde{\rho}}^{\mathrm{ss}}$ in the sequel.

We now show that $\phi_{\tilde{\rho}}$ factors through a maximal torus. We let $\phi_0:W\to\hat T(\mathcal O)$ be the Teichm\"uller lift of $\phi_T$. By definition of the Teichm\"uller lift, its image is of finite order prime to $\ell$,
and set $M= Z_{\hat G_{\mathcal O}}(\phi_0(W_{\bb{Q}_{p}}))^\circ$. Let $\mathcal A\subset\hat T_{\mathcal O}$ be the finite constant
subgroup scheme defined by the Teichm\"uller lifts $\phi_0(W_{\bb{Q}_{p}})$. Its order is prime to $\ell$, and its special fiber is the subgroup defined by $\phi(W_{\bb{Q}_{p}})$. Set
\[
  C=Z_{\hat G_{\mathcal O}}(\mathcal A),
  \qquad M=C^\circ,
\]
where $C^\circ$ denotes the relative identity component.
Since $\mathcal A$ is of multiplicative type, $C$ is smooth
and $M$ is reductive; see
\cite[Lemma~2.2.4, Proposition~3.1.3, and Remark~3.1.5]{ConradRGS}.
Formation of the centralizer and of the
relative identity component commutes with base change. Consequently,
\[
  C_{k}=Z_{\hat G_{k}}(\phi(W_{\bb{Q}_{p}})),
  \qquad
  M_{k}=Z_{\hat G_{k}}(\phi(W_{\bb{Q}_{p}}))^\circ.
\]
Moreover, $\hat T_{\mathcal O}$ is a split maximal torus
of $M$. Since maximal tori in reductive group schemes
are their own scheme-theoretic centralizers
\cite[Exp.~XIX, 2.8]{SGA3}, we have
\begin{equation}
  Z_{M}(\hat T_{\mathcal O})
  = \hat T_{\mathcal O}.
\end{equation}
In particular, we have that 
\begin{equation}{\label{eqn: centerII}} 
Z(M)\subseteq\hat T_{\mathcal O}.
\end{equation}

Note, by definition of $\phi_{0}$, we have for every root $\alpha$ and every $w\in W_{\bb{Q}_{p}}$,
\[
  \alpha(\phi_0(w))=[\alpha(\phi(w))],
\]
where $[-]$ denote Teichm\"uller lifting. Thus
$\alpha\circ\phi_0=1$ if and only if $\alpha\circ\phi=1$.
Indeed, if $\alpha\circ\phi_0 \neq 1$ there is some $w \in W_{\bb{Q}_{p}}$ for which $\alpha(\phi_0(w))-1\in\mathcal O^\times$ and hence the reduction of $\alpha(\phi_{0}(w))$ is non-trivial as well. Moreover, one easily sees that
\begin{equation}{\label{eqn: residualcoroots}}
  \operatorname{Lie}(M)_{k}
  =
  \operatorname{Lie}(\hat{T})_{k}
  \oplus
  \bigoplus_{\substack{
    \alpha\\
    \alpha\circ\phi=1}}
  \hat{\mathfrak g}_{\alpha,k},
\end{equation}
where $\alpha$ runs over coroots of $G$. In particular, every summand of $(\operatorname{Lie}(\hat G)/\operatorname{Lie}(M))_{\phi}$ given by coroots of $G$ therefore has vanishing $H^0$ and also vanishing $H^2$ by the weak genericity assumption, hence vanishing $H^1$. Consequently, $H^1(W_{\bb{Q}_{p}},(\operatorname{Lie}(\hat G)/
       \operatorname{Lie}(M))_{\phi})=0$.

Consider $X=\hat{G}/M$. The homomorphism $\phi_{\tilde{\rho}}$ defines an action of $W_{\bb{Q}_{p}}$ on $X$ and we know that $\overline{(\phi_{\tilde{\rho}})} = \phi$ takes values in $M(k)$. The identity coset in $X(k)$ is therefore fixed by $W_{\bb{Q}_{p}}$, and we may inductively lift this to get a $\mathcal{O}$-point of $X$ which is fixed under $W_{\bb{Q}_{p}}$ and reduces to the identity coset. Indeed, the obstruction to finding such a lifting lies in $H^1(
    W_{\bb{Q}_{p}},(\operatorname{Lie}(\hat G)/
       \operatorname{Lie}(M))_{\phi})=0$. By smoothness of $\hat G\to X$, choose a representative $g\in\hat G(\mathcal O)$ with $g\equiv1\pmod\varpi$. This means that for every $w \in W_{\bb{Q}_{p}}$ we have that $\phi_{\tilde{\rho}}(w)gM=gM$ or equivalently $g^{-1}\phi_{\tilde\rho}(W_{\bb{Q}_{p}})g\subset M(\mathcal O)$. Since $g\equiv1\pmod\varpi$, this conjugation preserves the
mod $\varpi$-reduction $\phi$ and therefore after conjugating $\phi_{\tilde{\rho}}$ we may assume that $\phi_{\tilde{\rho}}$ factors through $M$.

The projection $\phi^{\mathrm{ad}}_{\tilde{\rho}}: W_{\bb{Q}_{p}} \xrightarrow{\phi_{\tilde{\rho}}} M(\mathcal{O}) \xrightarrow{q}  M^{\mathrm{ad}}(\mathcal O)$
has trivial reduction mod $\varpi$. This follows from the containment that $\phi(W_{\bb{Q}_{p}})\subset Z(M_{k})(k)$. To see this, we note that we have  $\phi(W_{\bb{Q}_{p}}) \subset \hat T(k) \subset M(k)$, where the last containment holds since $\hat T_{\mathcal O}$ is a maximal torus of $M$. On the other hand, by the definition of $M_{k}$, every
element of $M(k)$ commutes with $\phi(w)$. Hence, $\phi(w)$ is central in $M_{k}$, as claimed. In particular, $\phi_{\tilde{\rho}}^{\mathrm{ad}}(I)$ is contained in the kernel of the map $M^{\mathrm{ad}}(\mathcal{O}) \ra M^{\mathrm{ad}}(k)$, which is a pro-$\ell$ group, by \cite[Lemma~3.8]{PlatonovRapinchuk}. Therefore, $\phi^{\mathrm{ad}}_{\tilde{\rho}}(I)$ is a finite $\ell$-group, wild inertia must therefore act trivially and therefore this is a cyclic group. If a generator of tame inertia had nontrivial image under $\phi^{\mathrm{ad}}_{\tilde{\rho}}$, its first non-zero term $X$ when expanding by powers of $\varpi^{n}$ after choosing a faithful representation of $\hat{G}/\mathcal{O}$ would satisfy
\[
  X=pX,
\]
by the tame Weil relation and the fact that
$\phi_{\tilde{\rho}}^{\mathrm{ad}}(\sigma)$ reduces to $1$ as explained above. On the other hand, $p - 1$ must act invertibly on $\Lie(M^{\mathrm{ad}})_{k}$ by (\ref{eqn: residualcoroots}). Therefore, $\phi_{\tilde{\rho}}^{\mathrm{ad}}(I) = 1$ and this in turn implies that $\phi_{\tilde{\rho}}(I)\subset Z(M)(\mathcal{O}) \subset \hat{T}(\mathcal{O})$, where the last containment is (\ref{eqn: centerII}). Now, we just conjugate the semisimple element $\phi_{\tilde\rho}(\sigma)$ into $\hat T$ inside $M$ which makes the whole parameter factor through a torus.

The full parameter $\phi_{\widetilde\rho}$ therefore factors through a maximal torus and the assertion about irreducible representations follows from Assumption \ref{compatibility} and Lemma \ref{lemma: localization map properties} (1), as in the proof of \cite[Corollary~6.7]{Ham2}.

For either hypothesis in the proposition, we now deduce the support assertion. The ULA hypothesis and \cite[Theorem~V.7.1]{FS} imply that the cohomology $V$ of each $j_b^*A$ is admissible, and Proposition~\ref{prop:localbasics} implies that it is $\phi$-local. For $K$ compact open pro-$p$, every $f\notin\mathfrak m_\phi$ has inverse polynomial in $f$ on the finite-dimensional space $V^K$. Exactness of $K$-invariants therefore shows that every subquotient of $V$ is $\phi$-local. If $V$ is nonzero, an irreducible subquotient has parameter $\phi$ by Lemma~\ref{lemma: localization map properties} (1); the preceding argument, using only parts (1) and (2) of Assumption~\ref{compatibility}, forces $b\in B(G)_{\mathrm{un}}$.
\end{proof}
\subsubsection{Assumption \ref{compatibility} under Central Isogenies}
We consider an injective map $\psi: G' \hookrightarrow G$ of connected reductive groups which induces an isomorphism of derived groups, and the induced map $\psi_{\Bun}: B(G') \ra B(G)$ on the associated Kottwitz sets. We now have the following Lemma.
\begin{lemma}{\label{lemma: adjderiso}}
If $\psi: G' \ra G$ is an injective map which induces an isomorphism on derived groups then it follows that $\psi_{\Bun}: B(G') \ra B(G)$ induces an injection $\psi_{b'}: J_{b'} \ra J_{b}$ for all $b = \psi_{\Bun}(b')$ and $b' \in B(G')$, which is an isomorphism on the derived groups.
\end{lemma}
\begin{proof}
Since $\psi$ is an inclusion it easily follows that it induces an inclusion $J_{b'} \ra J_{b}$ of $\sigma$-centralizers. To see that it induces an isomorphism on derived groups, recall that $J_{b}$ is an inner form of a Levi subgroup $M_{b}$ of $G$ given by the centralizer of the slope homomorphism of $b$. The preimage of $M_{b}$ under $\psi$ defines a Levi subgroup $M_{b'}$ of $G'$ which will be the centralizer of the slope homomorphism of $b'$, since $\psi$ induces an isomorphism on derived groups and in turn also on adjoint groups (since $(G^{\mathrm{der}})^{\mathrm{ad}} \simeq G^{\mathrm{ad}}$)). Moreover, the inner twisting from $J_{b}$ to $M_{b}$ and $J_{b'}$ to $M_{b'}$ are compatible with the inclusion in the sense that the inclusion $J_{b'} \ra J_{b}$ is given by applying the inner twist of $M_{b}$ to the inclusion $M_{b'} \ra M_{b}$. Since the formation of derived groups respects inner twists, this reduces us to showing that the map $M_{b'} \ra M_{b}$ on Levi subgroups induces an isomorphism on the derived groups, and this is clear.  
\end{proof}
We now consider a map $\LLC_{\Bun_{G}}: \bigsqcup_{b \in B(G)} \Pi(J_{b}) \ra \bigsqcup_{b \in B(G)} \Phi(J_{b})$ determined by components $\LLC_{b}: \Pi(J_{b}) \ra \Phi(J_{b})$ for $b \in B(G)$ and satisfying Assumption \ref{compatibility}. We now wish to define $\LLC_{\Bun_{G'}}: \bigsqcup_{b' \in B(G')} \Pi(J_{b'}) \ra \bigsqcup_{b' \in B(G')} \Phi(J_{b'})$ with components $\LLC_{b'}$ for $b' \in B(G')$ in terms of $\LLC_{\Bun_{G}}$, and show that it also satisfies Assumption \ref{compatibility}.  For $b'\in B(G')$, put $b=\psi_{\mathrm{Bun}}(b')$.
By Lemma~\ref{lemma: adjderiso} and \cite[Proposition~2.2]{Tad1}, every
$\pi_{b'}\in\Pi(J_{b'})$ occurs as an irreducible constituent
of $\pi_b|_{J_{b'}(\mathbb{Q}_p)}$ for some
$\pi_b\in\Pi(J_b)$. The set of irreducible constituents of this restriction is
finite by \cite[Lemma~2.1]{Tad1}.

We now construct a collection of maps
\[
  \mathrm{LLC}_{b'}:\Pi(J_{b'})\longrightarrow\Phi(J_{b'}),
  \qquad b'\in B(G').
\]
For $b'\in B(G')$, put $b=\psi_{\mathrm{Bun}}(b')$,
and write
\[
  {}^L\psi_{b'}:{}^LJ_b\longrightarrow{}^LJ_{b'}
\]
for the homomorphism induced by
$\psi_{b'}:J_{b'}\hookrightarrow J_b$.

For each $\pi_{b'}\in\Pi(J_{b'})$, choose once and for all
a representation $\widetilde\pi_{b'}\in\Pi(J_b)$ such that
$\pi_{b'}$ is an irreducible constituent of
\[
  \left.\widetilde\pi_{b'}\right|_{J_{b'}(\mathbb Q_p)}.
\]
Such a lift exists by Lemma~\ref{lemma: adjderiso} as discussed above. Define
\[
  \mathrm{LLC}_{b'}(\pi_{b'})
  :=
  {}^L\psi_{b'}\circ
  \mathrm{LLC}_b(\widetilde\pi_{b'}).
\]
These maps are defined after fixing the choices of lifts.
We will show that every such collection satisfies
Assumption~\ref{compatibility}.
\begin{proposition}{\label{prop: compatibcentralisog}}
Suppose that $\psi:G'\hookrightarrow G$ is an injective
homomorphism of quasi-split connected reductive groups
inducing an isomorphism on derived groups.
If there exists a collection
\[
  \mathrm{LLC}_b:\Pi(J_b)\longrightarrow\Phi(J_b),
  \qquad b\in B(G),
\]
satisfying Assumption~\ref{compatibility}, then there exists a collection
\[
  \mathrm{LLC}_{b'}:\Pi(J_{b'})\longrightarrow\Phi(J_{b'}),
  \qquad b'\in B(G'),
\]
satisfying Assumption~\ref{compatibility}.
More precisely, any collection constructed from fixed
choices of lifts as above has this property.
\end{proposition}

\begin{proof}
Fix $b'\in B(G')$, put $b=\psi_{\mathrm{Bun}}(b')$,
and let $\widetilde\pi_{b'}$ be the chosen lift of
$\pi_{b'}\in\Pi(J_{b'})$.

By Lemma~\ref{lemma: adjderiso}, the inclusion $J_{b'}\hookrightarrow J_b$
induces an isomorphism on adjoint groups.
Compatibility of the Fargues--Scholze correspondence
with central isogenies
\cite[Theorem~I.9.6(v)]{FS} gives
\[
  \phi_{\pi_{b'}}^{\mathrm{FS}}
  =
  {}^L\psi_{b'}\circ
  \phi_{\widetilde\pi_{b'}}^{\mathrm{FS}}.
\]
Since semisimplification commutes with composition,
Assumption~\ref{compatibility}(1) follows from the same assumption for
$\mathrm{LLC}_b$.

Choose compatible maximal tori $T'\subset G'$ and
$T\subset G$.
The dual homomorphism
$\hat\psi:\hat G\to\hat G'$
is a surjective central homomorphism with torus kernel, and hence
\[
  ({}^L\psi)^{-1}({}^LT')={}^LT.
\]
The normalized embeddings
${}^LJ_b\to{}^LG$ and ${}^LJ_{b'}\to{}^LG'$
are compatible with these maps: their normalizing
twists agree because $\psi$ identifies the roots and
the corresponding unipotent radicals.

Suppose that the parameter of $\pi_{b'}$, viewed as a
parameter for $G'$, factors through ${}^LT'$ up to
$\hat G'$-conjugacy.
A conjugating element lifts to $\hat G$.
The preceding inverse-image identity therefore implies
that the parameter of $\widetilde\pi_{b'}$, viewed as
a parameter for $G$, factors through ${}^LT$.
Assumption~\ref{compatibility}(2) for $G$ gives
$b\in B(G)_{\mathrm{un}}$.
Thus $J_b$ is quasi-split.
Since $J_{b'}$ and $J_b$ have isomorphic adjoint groups,
$J_{b'}$ is quasi-split as well, and consequently
$b'\in B(G')_{\mathrm{un}}$.
This proves Assumption~\ref{compatibility}(2).

In this situation, let $\phi_T$ be the toral parameter
of the chosen lift, and let $\phi_{T'}$ be its image.
If $\chi$ and $\chi'$ are the corresponding characters,
functoriality of local class field theory for tori gives
\[
  \chi'=\left.\chi\right|_{T'(\mathbb Q_p)}.
\]
Identify the corresponding relative Weyl groups and the
quotients $W_b$ and $W_{b'}$.
For compatible Borel subgroups, restriction of normalized
principal series gives
\[
  i_{B_b}^{J_b}(\chi^w)|_{J_{b'}(\mathbb Q_p)}
  \simeq
  i_{B_{b'}}^{J_{b'}}((\chi')^w).
\]
Indeed, the inclusion identifies the flag varieties,
and the modulus characters restrict compatibly.
Likewise,
\[
  \delta_{P_b}|_{J_{b'}(\mathbb Q_p)}
  =\delta_{P_{b'}}.
\]

By Assumption~\ref{compatibility}(3) for $G$, the representation
$\widetilde\pi_{b'}$ is an irreducible subquotient of
\[
  i_{B_b}^{J_b}(\chi^w)\otimes\delta_{P_b}^{-1/2}
\]
for some $w\in W_b$.
Restriction is exact, and $\pi_{b'}$ is a constituent
of the restriction of $\widetilde\pi_{b'}$.
It follows that $\pi_{b'}$ is an irreducible subquotient of
\[
  i_{B_{b'}}^{J_{b'}}((\chi')^w)
  \otimes\delta_{P_{b'}}^{-1/2}.
\]
This proves Assumption~\ref{compatibility}(3).
\end{proof}
Now that we have shown this compatibility assumption is somewhat flexible, we can state the groups we know to satisfy Assumption \ref{compatibility}. This result is largely contained in \cite{Ham1,FS,HKW,BMNH}, but we also want to consider an additional group $\GU_{2,L}$.

\subsubsection{An $\LLC$ for $\GU_2$} \label{subsubsec:LLC_GU_2}
Let $\GU_{2,L}$ be the unique unramified unitary similitude group over $L$. We want a construction of the $\LLC$ map for $\Res_{L/\mathbb{Q}_p}(\GU_{2,L})$. In the following, we will construct the $\LLC$ map directly over $L$ rather than for the restriction of scalars to $\mathbb{Q}_p$. Let $\sigma_L$ be the corresponding Frobenius element, i.e. if $L/\mathbb{Q}_p$ is unramified of degree $n$, then $\sigma_L=\sigma^n$. For a group $H$ over $L$, let $B_L(H)$ denote the $\sigma_L$-conjugacy classes of $H(\breve{\mathbb{Q}}_p)$, and $\Phi_L(H)$ denotes the conjugacy classes of L-parameters
\[
\phi_H:W_L\times\mathrm{SL}(2,\ol{\mathbb{Q}}_\ell)
\longrightarrow {}^LH.
\]
Similarly, for an element $b\in B_L(H)$ we can define the $\sigma_L$ centralizer $J_{b}^L$. This is a group over $L$, and let $\Pi_L(J_{b}^L)$ be the set of smooth irreducible $\ol{\mathbb{Q}}_\ell$-representations of $J_b^L(L)$. Note that by Shapiro's Lemma we have $\Phi_L(H)\simeq \Phi(\Res_{L/\mathbb{Q}_p}H)$, and an isomorphism $B_L(H)\simeq B(\Res_{L/\mathbb{Q}_p}H)$. If we let $c\in B(\Res_{L/\mathbb{Q}_p}H)$ be the image of $b\in B_L(H)$ under this isomorphism, then
\[J_b^L(L)\simeq J_c(\mathbb{Q}_p).\]
Thus, it suffices to construct maps
\[\LLC_b:\Pi_L(J_b^L)\rightarrow \Phi_L(J_b^L).\]

For the rest of this discussion we will fix the unramified field extension $L/\mathbb{Q}_p$, and we will simplify the notation $J_b^L$ to $J_b$. Firstly, recall that we have an isomorphism
\begin{equation*}
    \GU_{2,L}:=(\GL_{2,L}\times\Res_{L'/L}\mathbb{G}_{m,L'})/\mathbb{G}_{m,L},
\end{equation*}
where $\mathbb{G}_{m,L}$ is embedded in $H:=\GL_{2,L}\times\Res_{L'/L}(\mathbb{G}_{m,L'})$ via $a\mapsto (\mathrm{diag}(a,a),a^{-1})$, and $L'/L$ is an unramified quadratic extension. Let $\psi:B_L(H)\rightarrow B_L(\GU_{2,L})$ and let $\tilde{\psi}:B_L(H)\rightarrow B_L(\GL_{2,L})$ be the natural maps of Kottwitz sets. Given $b\in B_L(H)$, let $b'=\psi(b)$, $\tilde{b}=\tilde{\psi}(b)$.
\begin{lemma}{\label{lemma: GU2bijection}}
    There is a bijection between $\Pi_L(J_{b'})$ and the set of pairs $(\tilde{\pi}, \chi)$ such that $\tilde{\pi}\in\Pi(J_{\tilde{b}})$ and $\chi$ is a character of $(L')^\times$ such that $\chi|_{L^\times}=\omega_{\tilde{\pi}}|_{L^\times}$, where $\omega_{\tilde{\pi}}$ is the central character of $\tilde{\pi}$.
\end{lemma}
\begin{proof}
We will show that we have an isomorphism 
\begin{equation*}
    J_{b'}\simeq (J_{\tilde{b}}\times\Res_{L'/L}\mathbb{G}_m)/\mathbb{G}_m
\end{equation*}
of groups over $L$. Assuming this, we see that since by Hilbert's Theorem 90 we have $H^1(L,\mathbb{G}_m)=0$, the projection map from $(J_{\tilde{b}}\times\Res_{L'/L}\mathbb{G}_m)$ to $J_{b'}$ induces a short exact sequence on $L$-points
\[0\rightarrow L^\times\rightarrow J_{\tilde{b}}(L)\times(L')^\times\rightarrow J_{b'}(L)\rightarrow 0,\]
so we see that to give a smooth irreducible $\ol{\mathbb{Q}}_\ell$-representation of $J_{b'}(L)$ is exactly equivalent to giving a smooth irreducible $\ol{\mathbb{Q}}_\ell$-representation of $J_{\tilde{b}}(L)$ and $(L')^\times$ which agree on the restriction to the central copy of $L^\times$. 

To see this isomorphism, recall that $J_b$ (resp. $J_{b'},J_{\tilde{b}}$) is an inner form of $M_b$ (resp. $M_{b'},M_{\tilde{b}}$), the Levi subgroup of $H$ (resp. $\GU_{2,L},\GL_{2,L}$) given by the centralizer of the slope homomorphism of $b$ (resp. $b',\tilde{b}$). Note that we have $J_b=J_{\tilde{b}}\times\Res_{L'/L}\mathbb{G}_m$.
The map $H\rightarrow\GU_{2,L}$ induces a map $J_b\rightarrow J_{b'}$
by functoriality of $\sigma$-centralizers. Moreover, we see that we have an isomorphism
\begin{equation*}
    M_{b'}\simeq (M_{\tilde{b}}\times \Res_{L'/L}\mathbb{G}_m)/\mathbb{G}_m,
\end{equation*}
and thus we have a surjective map $M_{b}\rightarrow M_{b'}$, since $M_b=M_{\tilde{b}}\times \Res_{L'/L}\mathbb{G}_m$. Moreover, we see that under these maps, we have isomorphisms $M_{b}^{\mathrm{ad}}\simeq M_{b'}^{\mathrm{ad}}\simeq M_{\tilde{b}}^{\mathrm{ad}}$, and the inner twist defining a class in $H^1(L,M_b^{\mathrm{ad}})$ corresponding to $J_b$ is, under this identification, the inner twist inducing $J_{b'}$ and $J_{\tilde{b}}$. The identification of $J_{b'}$ then follows.  
\end{proof}
We observe that we have an exact sequence of dual groups
\begin{equation*}
    1\rightarrow \hat{J}_{b'}\rightarrow\hat{J}_{b}\rightarrow\mathbb{G}_m\rightarrow 1,
\end{equation*}
where the map $p:\hat{J}_b\rightarrow\mathbb{G}_m$ is defined as follows. We can write $\hat{J}_b=\hat{J}_{\tilde{b}}\times \mathbb{G}_m^2$, and we have maps $\hat{i}_1:\hat{J}_{\tilde{b}}\rightarrow \mathbb{G}_m$, $\hat{i}_2:\hat{\Res_{L'/L}\mathbb{G}_m}=\mathbb{G}_m^2\rightarrow \mathbb{G}_m$ induced from the inclusion maps $i_1: \mathbb{G}_m\hookrightarrow J_{\tilde{b}}$ and $i_2:\mathbb{G}_m\hookrightarrow \Res_{L'/L}\mathbb{G}_m$, and $p(g,h)=\hat{i}_1(g)\hat{i}_2(h)^{-1}$.

Now, we want to define $\LLC_{\Bun_{\GU_{2,L}}}$ in terms of $\LLC_{\Bun_{H}}$. Firstly, note that for every $b'\in B_L(\GU_{2,L})$, we can construct a lift of $b'$ to $b\in H(\breve{\mathbb Q}_p)$, which exists by Hilbert Theorem 90 applied to $\mathbb{G}_m$, the kernel of $H\rightarrow \GU_{2,L}$. We want to define $\LLC_{b'}: \Pi_L(J_{b'}) \rightarrow \Phi_L(J_{b'})$ in terms of $\LLC_{b}: \Pi_L(J_{b}) \rightarrow \Phi_L(J_{b})$ for $\LLC_{\Bun_{H}}$. For $\pi_{b'}=(\tilde{\pi},\chi) \in \Pi_L(J_{b'})$, we consider the image $\phi=\LLC_{\Bun_{H}}((\tilde{\pi},\chi))$, and we want to show that
$\phi:W_{L} \times \SL_{2}(\ol{\mathbb{Q}}_{\ell})\rightarrow\phantom{}^{L}J_b(\overline{\mathbb{Q}}_\ell)$ factors through $\phantom{}^{L}J_{b'}(\overline{\mathbb{Q}}_\ell)$. To see this, from the exact sequence above, it suffices to show that the composition of $\phi$ with the map $\phantom{}^{L}J_b(\overline{\mathbb{Q}}_\ell)\rightarrow \phantom{}^{L}\mathbb{G}_m(\overline{\mathbb{Q}}_\ell)$ is trivial. However, we observe that the condition that $\chi|_{L^\times}=\omega_{\tilde{\pi}}|_{L^\times}$ exactly implies that this image is trivial, since the composition is the $L$-parameter associated with the character $\omega_{\tilde{\pi}}\chi^{-1}|_{L^\times}$. Thus, we have an $L$-parameter $\phi':W_{L} \times \SL_{2}(\ol{\mathbb{Q}}_{\ell})\rightarrow\phantom{}^{L}J_{b'}(\overline{\mathbb{Q}}_\ell)$. We thus define the map $\LLC_{b'}$ to take $\pi_{b'}$ to $\phi'$.

For the various groups discussed above, we will also want to check some assumptions on the prime $\ell$ relative to the group $G$ that we will need to assume for the proof of our main results, since they are required for the arguments in \cite{Ham2}. In particular, we define the following. 
\begin{definition}{\label{defn: verydecent}}
We say that $\ell$ is \emph{very decent} with respect to $G$ if the following holds. 
\begin{enumerate}
\item We have that $\ell \nmid |\pi_{0}(Z(G))|$, where $Z(G)$ denotes the center of $G$. 
\item We have that $\ell \nmid |Q|$, where we recall that $Q$ is the smallest quotient through which the action of $W_{\bb{Q}_{p}}$ on $\hat{G}$ factors.
\end{enumerate}
\end{definition}
\begin{remark}
The motivation for the first condition was already explained in
\S \ref{subsection: SpectralActionSection}; in particular, it is
required for the construction of the spectral action and in turn
the Fargues-Scholze local Langlands correspondence and the localized
derived category $\D(\Bun_G,\ol{\bb{F}}_\ell)_\phi$.
The second condition is used in the formalism of tilting modules in
\cite[Section~9.1]{Ham2}.
Since $\ell\nmid |Q|$, representations of $Q$ and its subgroups
over $\ol{\mathbb{F}}_\ell$ are semisimple.
This allows the tilting formalism for $\hat{G}$ to extend to
the full $L$-group $\phantom{}^LG:=\hat{G}\ltimes Q$,
with decompositions into indecomposable tilting modules as described
there. The resulting tilting category is additive and generally
not abelian; its indecomposable objects need not be simple.
\end{remark}
We now have the following theorem about the groups we know to satisfy Assumption \ref{compatibility}, anticipating our discussion in the next section. 
\begin{theorem}{\cite{Ham1,FS,HKW,BMNH}}
\label{thm:AssumptionLLC}
Assumption \ref{compatibility} holds in the following cases. 
\begin{enumerate}
    \item The group $\mathrm{Res}_{L/\mathbb{Q}_{p}}(\GL_{n})$ for any extension $L/\bb{Q}_{p}$.
    \item The groups $\GU_{n}$ or $\U_{n}$ for $n$ odd, defined with respect to a split Hermitian form and an unramified quadratic extension $E/\mathbb{Q}_{p}$ and $\ell \neq 2$.
    \item The group $\Res_{L/\mathbb{Q}_p}(\GU_{2})$ defined with respect to a split Hermitian form and an unramified quadratic extension $L'/L$. 
    \item The group $\mathrm{Res}_{L/\mathbb{Q}_{p}}(\mathrm{GSp}_{4})$ for $L/\mathbb{Q}_{p}$ an unramified extension with $p > 2$ or $L = \mathbb{Q}_{p}$ for all $p$ and $\ell$.    
\end{enumerate}
\end{theorem}
\begin{proof}
We recall (Remark \ref{rem: FarguesScholzewithCcoefficients}) that we implicitly fix a choice of isomorphism $i: \ol{\bb{Q}}_{\ell} \xrightarrow{\simeq} \bb{C}$ and normalize the Fargues-Scholze correspondence with respect to the choice of square root of $p$ given by $i^{-1}(\sqrt{p})$ to make sense of the compatibility with the local Langlands correspondences constructed for $\bb{C}$-coefficients. 

We first discuss how Assumption \ref{compatibility} (2) can be checked. Firstly, recall that if $J_{b}$ is a non-quasi-split group then the fibers of the $\LLC_{b}$ over an $L$-parameter $\phi: W_{\bb{Q}_{p}} \times \SL_{2}(\ol{\mathbb{Q}}_{\ell}) \ra \phantom{}^{L}J_{b}(\ol{\mathbb{Q}}_{\ell})$ should be empty if $\phi$ factors through $\phantom{}^{L}M(\ol{\mathbb{Q}}_{\ell}) \ra \phantom{}^{L}G(\ol{\mathbb{Q}}_{\ell})$ via the natural (equivalently twisted) embedding for a Levi subgroup $M \subset G$ which does not transfer to a Levi subgroup of $J_{b}$ (recalling that $G$ is quasi-split in this case, so that $J_{b}$ is an inner form of the quasi-split Levi of $G$). In particular, such parameters are called irrelevant, and we expect the fiber to be empty if and only if $\phi$ is irrelevant \cite[Conjecture~A.2]{Kal}. Assumption \ref{compatibility} (2) reduces to checking that the fibers over the irrelevant $L$-parameters are empty for all $J_{b}$ and the $\LLC_{b}$ in question. This is because $b \in B(G)_{\mathrm{un}}$ if and only if $J_{b}$ is quasi-split or a torus (see \cite[Lemma~2.13]{Ham2}) implying that if $b \notin B(G)_{\mathrm{un}}$ then all $L$-parameters $\phi \in \Phi(J_{b})$ which factor through $\phantom{}^{L}T$ are irrelevant.

We also note that verification of Assumption \ref{compatibility} (3) can be reduced to checking that $\LLC_b$ satisfies a certain compatibility with parabolic induction. Firstly, note that any parameter $\phi: W_{\bb{Q}_{p}} \times \SL_{2}(\ol{\mathbb{Q}}_{\ell}) \ra \phantom{}^{L}G(\ol{\mathbb{Q}}_{\ell})$ induced from a toral parameter $\phi_{T}$ has necessarily trivial monodromy, since there is no nontrivial algebraic homomorphism $\SL_{2} \ra \hat{T}$. Moreover, since $J_{b}$ is an inner form of $M_{b}$, it follows that the conjugacy classes of parameters $\phi': W_{\bb{Q}_{p}} \times \SL_{2}(\ol{\mathbb{Q}}_{\ell}) \ra \phantom{}^{L}J_{b}(\ol{\mathbb{Q}}_{\ell})$, which can give rise to $\phi$ under the twisted embedding $\phantom{}^{L}J_{b}(\ol{\mathbb{Q}}_{\ell}) \ra \phantom{}^{L}G(\ol{\mathbb{Q}}_{\ell})$ are represented, possibly with repetitions, by the Weyl conjugates of
$\phi_T$ indexed by a set of minimal length representatives of
$W_{b} = W_{M_{b}} \backslash W_{G}$, with the inverse modulus twist
required by the twisted embedding. We expect (see \cite[Conjecture~A.5]{Kal}) that the fiber of $\LLC_{b}$ over such a $\phi'$ inducing $\phi$ to be the irreducible constituents of the normalized induction of the $L$-packet of $\phi_{T}^{w}$, which is just $\chi^{w}$ by local class field theory for $w \in W_{b}$. If this holds, then so does Assumption \ref{compatibility} (3).

With these reductions, we now proceed to check each of the cases.

{\bf Case (1): }$\GL_n$. In this case, the $\LLC_b$ map will be the Harris-Taylor correspondence precomposed with Badulescu's Jacquet-Langlands map \cite{BadulescuJacquetLanglands}. Assumption \ref{compatibility} (1) follows from \cite[Theorem~I.9.6]{FS} and \cite[Theorem~1.0.3]{HKW}.

Now, for Assumption \ref{compatibility} (2), we first observe that the fibers over irrelevant parameters are empty by the standard properties of the Jacquet-Langlands correspondence. In particular, it follows from the discussion before \cite[Theorem~4.2.1]{KudlaLocalLanglands} that the Harris-Taylor parameter for a discrete series representation of $\GL_{n}$ always defines a discrete L-parameter (in the sense that its L-parameter does not factor through a proper Levi subgroup) and is compatible with parabolic induction. The claim on $\LLC$ for $\GL_{m}(D_{\frac{h}{n}})$ now follows from the fact that Jacquet--Langlands carries essentially square-integrable representations of every $\GL_r(D_{\frac{h}{n}})$ to essentially square-integrable representations, together with the Langlands classification and the compatibility with parabolic induction discussed in \cite[Section~3.1]{BadulescuJacquetLanglands}. This implies the claim on irrelevant parameters for $J_{b}$ in the case of $G = \GL_{n}$, since $J_{b}$ in this case will always be a product of groups of the form $\GL_{m}(D_{\frac{h}{n}})$.

Finally, for Assumption \ref{compatibility} (3), the required compatibility with parabolic induction is known, also from \cite[Section~3.1]{BadulescuJacquetLanglands}.

{\bf Case (2): }$\GU_n$ or $\U_n$, with $n$ odd. In this case, the $\LLC_b$ map will be the correspondence of Mok \cite{Mok} and Kaletha-Minguez-Shin-White \cite{KMSW}. Assumption \ref{compatibility} (1) follows from \cite[Theorem~1.1]{BMNH}.

Now, for Assumption \ref{compatibility} (2), observe that since odd unitary groups are always quasi-split there are no irrelevant $L$-parameters for these groups and there is nothing to check. Therefore, in this case, we can reduce to the claim on irrelevant $L$-parameters to $\GL_{n}$ and its inner forms, since these occur as direct factors of $J_{b}$ for such $G$, which was already explained above.

Finally, for Assumption \ref{compatibility} (3), compatibility with parabolic induction for the Kaletha-Minguez-Shin-White/Mok correspondence follows from \cite[Section~2.2.3]{BMNH}.

{\bf Case (3): }$\GU_2$. In this case, the $\LLC_b$ map will be the correspondence constructed in Section \ref{subsubsec:LLC_GU_2}. Note that Assumption \ref{compatibility} (1) holds for $\GL_2$ and thus for $H:=\GL_{2,L}\times\Res_{L'/L}(\mathbb{G}_{m,L'})$. Moreover, we see that the map $\Phi_L^{\mathrm{ss}}(J_{b'})\to \Phi^{\mathrm{ss}}_L(J_{b})$ is injective, since the map $p : Z(\hat J_b)^{W_L} \to \mathbb{G}_m$ is surjective, so any conjugator can be multiplied by a central element fixed by $W_L$ to lie in $\ker p$. We now apply \cite[Theorem I.9.6(v)]{FS} to the map $J_{b}\rightarrow J_{b'}$, to get that $\phi^{\mathrm{FS}}_{\pi_{b'}}= (\phi')^{\mathrm{ss}}$, so Assumption \ref{compatibility} (1) also holds for this.

Now, for Assumption \ref{compatibility}, since $\GU_{2,L},H$ and $\GL_{2,L}$ all have the same adjoint group, $b'\in B_L(\GU_2)_{\mathrm{un}}$ is unramified exactly when $b$ is unramified with notation as in Lemma \ref{lemma: GU2bijection}. Now, let $\tilde{T}$ be a maximal split torus of $\GL_{2}$, and observe that $T'=(\tilde{T}\times\Res_{L'/L}\mathbb{G}_m)/\mathbb{G}_m$ is a maximal torus of $\GU_2$ over $L$. Given $\pi_{b'}=(\tilde{\pi},\chi)$, and $\phi_{\pi_{b'}}:W_{L} \times \SL_{2}(\ol{\mathbb{Q}}_{\ell})\rightarrow \phantom{}^{L}J_{b'}(\ol{\mathbb{Q}}_{\ell})$, we see from the construction that this factors through $\phantom{}^{L}T'(\ol{\mathbb{Q}}_{\ell})$ exactly when the associated $L$-parameter for $H$, $\phi_{(\tilde{\pi},\chi)}:W_{L} \times \SL_{2}(\ol{\mathbb{Q}}_{\ell})\rightarrow \phantom{}^{L}J_{b}(\ol{\mathbb{Q}}_{\ell})$, factors through $\phantom{}^{L}T(\ol{\mathbb{Q}}_{\ell})$, where $T=\tilde{T}\times\Res_{L'/L}\mathbb{G}_m$. Since Assumption \ref{compatibility} (2) is compatible with taking products,
the group $H$ satisfies this assumption. The preceding equivalence then shows
that $\GU_2$ satisfies it as well.

Finally, for Assumption \ref{compatibility} (3), we see that when $b'$ is unramified, we have isomorphisms of flag varieties
\begin{equation*}
    J_{b'}/B_{b'}\simeq J_{\tilde{b}}/B_{\tilde{b}}\simeq J_{b}/B_{b}.
\end{equation*}
In the above situation where $\phi_{\pi_{b'}}$ factors through $\phantom{}^{L}T'$, we see that since $H,\GL_2$ satisfy Assumption \ref{compatibility} (3), the corresponding representation of $J_b(L)$ is of the form $(\tilde{\pi},\chi)$, where $\tilde{\pi}$ is an irreducible constituent of $i_{B_{\tilde{b}}}^{J_{\tilde{b}}}(\chi_1^{w}) \otimes \delta_{P_{\tilde{b}}}^{-1/2}$, for the associated character $\chi_1$ of $\tilde{T}$. In particular, we see that $\pi_{b'}$ is a constituent of $i_{B_{b'}}^{J_{b'}}(\chi_1^{w}\otimes \chi) \otimes \delta_{P_{b'}}^{-1/2}$, as desired.

{\bf Case (4): }$\GSp_4$. In this case, the $\LLC_b$ map will be the correspondence of Gan-Takeda \cite{GT1} and Gan-Tantono \cite{GT2} for the basic elements, and the Harris-Taylor correspondence for the non-basic $b$. Assumption \ref{compatibility} (1) follows from \cite[Theorem~1.1]{Ham1}\footnote{In \cite{Ham1}, the assumption that $p > 2$ is only used to invoke basic uniformization of abelian type Shimura varieties, but this condition is actually unnecessary if $L = \bb{Q}_{p}$, following the construction of Pappas-Rapoport \cite{PR23} of the uniformization in Hodge type case when $p=2$ and the group is unramified.}.

Now, for Assumption \ref{compatibility} (2), we need to show the claim on irrelevant $L$-parameters for $\LLC_{\mathrm{GU}_{2}(D)}$, where $\mathrm{GU}_{2}(D)$ is the unique non-split inner form of $\mathrm{GSp}_{4}$. Here the claim on irrelevant $L$-parameters follows directly from the construction of Gan-Tantono (see the discussion before the main theorem in \cite{GT2}).

Finally, for Assumption \ref{compatibility} (3), compatibility with parabolic induction for the Gan-Takeda/Gan-Tantono correspondence follows directly from the construction and the analogous claim for the Harris-Taylor correspondence.
\end{proof}
We now turn our attention to deriving the key local results from Assumption \ref{compatibility}.
\subsection{Perverse $t$-exactness}{\label{subsec: perversetexactness}}
We assume that $\ell \nmid |\pi_{0}(Z(G))|$ throughout, so that the localized category $\D(\Bun_{G},\ol{\mathbb{F}}_{\ell})_{\phi}$ described in \S 4.1 is well-defined. We recall that $\Bun_{G,C}^{b} \simeq [\Spd(C)/\mathcal{J}_{b}]$, where $\mathcal{J}_{b} := \mathrm{Aut}(\mathcal{E}_{b})$ is the group diamond parameterizing automorphisms of the bundle $\mathcal{E}_{b}$ attached to $b \in B(G)$ on $X_{C}$, as in (\ref{eqn: AutomorphismsofPdivisibleAutomorphismsofJb}). The unipotent part $\mathcal{J}_{b}^{> 0}$ has pure cohomological $\ell$-dimension over $\Spd(C)$ (in the sense of \cite[Definition~IV.1.17]{FS}) equal to $\langle 2\rho_{G}, \nu_{b} \rangle$, where $\nu_{b}$ is the slope homomorphism of $b$.
Moreover, we have that $\Bun_{G}$ is cohomologically smooth of pure $\ell$-dimension equal to $0$ over the base. This motivates the following definition.
\begin{definition}
We define a perverse $t$-structure $(\pD^{\leq 0}(\Bun_{G},\ol{\mathbb{F}}_{\ell}),\pD^{\geq 0}(\Bun_{G},\ol{\mathbb{F}}_{\ell}))$ on $\D(\Bun_{G},\ol{\mathbb{F}}_{\ell})$ such that $A \in \D(\Bun_{G},\ol{\mathbb{F}}_{\ell})$ lies in $\pD^{\leq 0}(\Bun_{G},\ol{\mathbb{F}}_{\ell})$ (resp. $\pD^{\geq 0}(\Bun_{G},\ol{\mathbb{F}}_{\ell})$) if and only if $j_{b}^{*}(A)$ (resp. $j_{b}^{!}(A)$) sits in cohomological degrees $\leq \langle 2\rho_{G},\nu_{b} \rangle$ (resp. $\geq \langle 2\rho_{G},\nu_{b} \rangle$) (see \cite[Proposition~8.1.5]{MingjiaPolI} for a detailed proof that this defines a $t$-structure).
\end{definition}
For the rest of this section, we assume that $\ell$ is very decent in the sense of Definition \ref{defn: verydecent} so that we may invoke the results of \cite{Ham2} and \S \ref{subsection: SpectralActionSection}.

We let $\phi$ be a semisimple L-parameter, and we write $(\pD^{\geq 0}(\Bun_{G},\ol{\mathbb{F}}_{\ell})_{\phi},\pD^{\leq 0}(\Bun_{G},\ol{\mathbb{F}}_{\ell})_{\phi})$ for the restriction of this $t$-structure to the full subcategory $\D(\Bun_{G},\ol{\mathbb{F}}_{\ell})_{\phi}$. One can check that this gives rise to a well-defined $t$-structure on $\D(\Bun_{G},\ol{\mathbb{F}}_{\ell})_{\phi}$ such that the localization map $(-)_{\phi}$ is $t$-exact. Indeed, this follows by applying Proposition \ref{prop:localizenaive} to the two sets of compact generators $\{j_{b!}(\mathrm{cInd}_{K}^{J_{b}(\bb{Q}_{p})}(\Lambda))\}$ and $\{j_{b\sharp}(\mathrm{cInd}_{K}^{J_{b}(\bb{Q}_{p})}(\Lambda))\}$  for $K \subset J_{b}(\bb{Q}_{p})$ a varying open pro-$p$ subgroup of $\D(\Bun_{G},\ol{\bb{F}}_{\ell})$ (see \cite[Theorem~V.4.1]{FS}) and using that localization is an exact functor. Here $j_{b\sharp}$ denotes the exceptional left adjoint to $j_{b}^{*}$ considered in \cite[Proposition~VII.7.2]{FS}. Moreover, one easily checks that this also gives rise to a $t$-structure after further intersecting with $\D^{\mathrm{ULA}}(\Bun_{G},\ol{\bb{F}}_{\ell}) \subset \D(\Bun_{G},\ol{\bb{F}}_{\ell})$ the full subcategory of sheaves $A$ such that $j_{b}^{*}A$ satisfies the property that its invariants under any pro-$p$ group $K \subset J_{b}(\bb{Q}_{p})$ is a perfect complex, by an analogous argument where one uses that the category of perfect complexes of $\ol{\bb{F}}_{\ell}$-vector spaces is preserved under standard truncation (more generally this holds for coefficients in a regular Noetherian ring). We let $\Perv(\Bun_{G},\ol{\mathbb{F}}_{\ell})_{\phi}$ and $\Perv^{\mathrm{ULA}}(\Bun_{G},\ol{\bb{F}}_{\ell})_{\phi}$ denote the hearts of these $t$-structures. 
\subsubsection{The generic regular case}
One of the technical conditions introduced in \cite{Ham2} is the following.
\begin{definition}{\label{defn: weaknormreg}}
We say that a toral parameter $\phi_{T}: W_{\mathbb{Q}_{p}} \ra \phantom{}^{L}T(\ol{\mathbb{F}}_{\ell})$ is generic regular if it is generic in the sense of Definition \ref{def: generic} and $\phi_{T}$ is also regular in the sense that, if $\chi$ denotes the character attached to $\phi_{T}$ under local class field theory, we have that $\chi$ is regular, i.e for all $w \in W_{G}$ non-trivial, we have that
\begin{equation}{\label{weylgroupreln}}
 \chi \not\simeq \chi^{w}.
\end{equation}
\end{definition}
We let $\phi: W_{\mathbb{Q}_{p}} \ra \phantom{}^{L}G(\ol{\mathbb{F}}_{\ell})$ be the L-parameter induced by a toral parameter $\phi_{T}$. To motivate the condition, we recall that, if $\phi_{T}$ is generic regular, we have by \cite[Theorem~9.12]{Ham2} an object $\nmEis(\mathcal{S}_{\phi_{T}}) \in \Perv(\Bun_{G},\ol{\mathbb{F}}_{\ell})$, which is a perverse filtered Hecke eigensheaf on $\Bun_{G}$, assuming \ref{compatibility} holds. Moreover, it is supported on the set of unramified elements and, for $b \in B(G)_{\mathrm{un}}$, its stalks are given by
\[ \Red_{b,\phi}^{\mathrm{tw}} := \bigoplus_{w \in W_{b}} \rho_{b,w}^{\chi}[- \langle 2\rho_{G}, \nu_{b} \rangle] \in \D(J_{b}(\bb{Q}_{p}),\ol{\mathbb{F}}_{\ell}), \]
where we recall that $\rho_{b,w}^{\chi} := i_{B_{b}}^{J_{b}}(\chi^{w}) \otimes \delta_{P_{b}}^{-1/2}$ and $\chi$ is the character attached to $\phi_{T}$ via local class field theory. In particular, by compatibility of the Fargues-Scholze correspondence with local Langlands for tori and parabolic induction \cite[Theorem~I.9.6 (i), (viii)]{FS} as well as \cite[Theorem IX.7.2]{FS}, it defines an object in the localized category $\Perv(\Bun_{G},\ol{\mathbb{F}}_{\ell})_{\phi}$. Moreover, it splits as a direct product 
\begin{equation}
 \nmEis(\mathcal{S}_{\phi_{T}}) \simeq \prod_{b \in B(G)_{\mathrm{un}}} j_{b!}(\Red_{b,\phi}^{\mathrm{tw}}) \in \D(\Bun_{G},\ol{\mathbb{F}}_{\ell})_{\phi}.
\end{equation}
To show the desired perverse $t$-exactness property, we would like to use the filtered Hecke eigensheaf property of $\nmEis(\mathcal{S}_{\phi_{T}})$. Given a geometric dominant cocharacter $\mu$ of $G$, we consider the highest weight tilting module $\mathcal{T}_{\mu}$ attached to $\mu$ and let 
\[ T_{\mu}: \D(\Bun_{G},\ol{\mathbb{F}}_{\ell}) \ra \D(\Bun_{G},\ol{\mathbb{F}}_{\ell})^{BW_{E_{\mu}}} \]
be the Hecke operator attached to the representation $\mathcal{T}_{\mu}$,
where $E_{\mu}$ denotes the reflex field of $\mu$. Write
\[
r_\mu:\hat{G}\rtimes W_{E_\mu}\longrightarrow\GL(\mathcal{T}_\mu)
\]
for the representation on the tilting module, with the
$W_{E_\mu}$-equivariant structure used to define this Hecke operator.
By the filtered eigensheaf property, the object $T_{\mu}(\nmEis(\mathcal{S}_{\phi_{T}}))$ carries a finite filtration whose graded are isomorphic to $\nmEis(\mathcal{S}_{\phi_{T}})^{\oplus n}$ after forgetting the $W_{E_{\mu}}$-action. Hence, this is a "filtered eigensheaf". We will now use this to deduce the following result.
\begin{proposition}{\label{prop: genlocal}}
Assume $G/\bb{Q}_{p}$ satisfies assumption \ref{compatibility}. For any $\phi$ induced from a generic regular $\phi_{T}$, the following is true. 
\begin{enumerate}
\item For all $b \in B(G)_{\mathrm{un}}$ and $w \in W_{b}$, the representations $\rho_{b,w}$ are semisimple (in fact irreducible).
\item The functors $j_{b}^{*}$ for $b \in B(G)$ induce a direct product decomposition 
\[ \prod_{b \in B(G)_{\mathrm{un}}} \Dadm(\Bun_{G}^{b},\ol{\mathbb{F}}_{\ell})_{\phi} \simeq \DULA(\Bun_{G},\ol{\mathbb{F}}_{\ell})_{\phi}. \]
\end{enumerate}
Moreover, for any $A \in \DULA(\Bun_{G}^{b},\ol{\mathbb{F}}_{\ell})_{\phi} \simeq \Dadm(J_{b}(\mathbb{Q}_{p}),\ol{\mathbb{F}}_{\ell})_{\phi}$, we have that the $!$ and $*$ pushforwards agree with respect to the inclusion $j_{b}: \Bun_{G}^{b} \ra \Bun_{G}$.     
\end{proposition}
\begin{proof}
The first part of the Proposition follows from \cite[Corollary~10.20(2)]{Ham2}. For the second part, we use the semi-orthogonal decomposition of $\D(\Bun_{G},\ol{\mathbb{F}}_{\ell})$ into $\D(\Bun_{G}^{b},\ol{\mathbb{F}}_{\ell}) \simeq \D(J_{b}(\mathbb{Q}_{p}),\ol{\mathbb{F}}_{\ell})$ via the excision triangles. This reduces us to showing that the $!$ and $*$-pushforwards agree for all objects $A \in \Dadm(\Bun_{G}^{b},\ol{\mathbb{F}}_{\ell})_{\phi}$. In particular, this guarantees that the excision spectral sequence degenerates and the first part of the claim follows.

We now use Proposition \ref{prop: constituent proposition} to see that an object $A \in \DULA(\Bun_{G},\ol{\mathbb{F}}_{\ell})_{\phi}$ can only be supported on the HN-strata $\Bun_{G}^{b}$ for $b \in B(G)_{\mathrm{un}}$ and that the restriction of $A$ to $\Bun_{G}^{b}$ has cohomology sheaves with irreducible constituents valued in the (irreducible by Part (1)) representations $\rho_{b,w}$ for $w \in W_{b}$ varying.

For the representations $\rho_{b,w}$, we use the following claim of \cite{Ham2}.
\begin{proposition}{\cite[Proposition~10.11]{Ham2}}{\label{prop: cleaness}}
For all $b \in B(G)_{\mathrm{un}}$ and $w \in W_{b}$, the natural map
\[ j_{b!}(\rho_{b,w}) \ra Rj_{b*}(\rho_{b,w}) \]
is an isomorphism assuming $\phi_{T}$ is generic regular and Assumption \ref{compatibility} is true.
\end{proposition}
This is enough to conclude the claim for any $A \in \Dadm(J_{b}(\mathbb{Q}_{p}),\ol{\mathbb{F}}_{\ell})_{\phi}$ using the following claim. 
\begin{lemma}{\label{lemma: howefinitelength}}
Assuming \ref{compatibility}, for $\phi$ a generic parameter and any $A \in \D^{\mathrm{adm}}(J_{b}(\mathbb{Q}_{p}),\ol{\mathbb{F}}_{\ell})_{\phi}$ the cohomology of $A$ in the standard $t$-structure has finite length and the complex is bounded. 
\end{lemma}
\begin{proof}
By Proposition \ref{prop: constituent proposition}, any irreducible subquotient of the cohomology of $A$ is an irreducible constituent of one of the finite-length representations $\rho_{b,w}$ for $w \in W_{b}$. Therefore, by choosing $K \subset J_{b}(\mathbb{Q}_{p})$ a sufficiently small open compact such that all these irreducible constituents have an invariant vector and taking invariants is exact (e.g $K$ is pro-$p$), we deduce, since $A^{K}$ is a perfect complex by assumption, that $A$ must have bounded length with finite length cohomology sheaves.
\end{proof}
\end{proof}

We now have the following key claim.
\begin{theorem}{\label{thm: generalperverseexact}}
Let $\mu$ be a geometric dominant cocharacter. We write 
\[ T_{\mu}: \D(\Bun_{G},\ol{\mathbb{F}}_{\ell}) \ra \D(\Bun_{G},\ol{\mathbb{F}}_{\ell})^{BW_{E_{\mu}}} \] 
for the Hecke operator attached to the highest weight tilting module $\mathcal{T}_{\mu}$ of highest weight $\mu$, where $E_{\mu}$ denotes the reflex field of $\mu$. Then the operator restricted to $\DULA(\Bun_{G},\ol{\mathbb{F}}_{\ell})_{\phi}$ is perverse $t$-exact if $\phi_{T}$ is generic regular and Assumption \ref{compatibility} is true. Here $\phi: W_{\mathbb{Q}_{p}} \ra \phantom{}^{L}G(\ol{\mathbb{F}}_{\ell})$ is the parameter induced by $\phi_{T}$. 
\end{theorem}
\begin{proof}
Using Lemma \ref{lemma: howefinitelength}, the commutation of Hecke operators with colimits (which follows since it is a left adjoint (see the discussion preceding \cite[Theorem~I.7.2]{FS})), Proposition \ref{prop: constituent proposition}, Proposition \ref{prop: genlocal}, we can reduce to showing, for all $b \in B(G)_{\mathrm{un}}$, that we have a containment 
\[ T_{\mu}(j_{b!}(\Red^{\mathrm{tw}}_{b,\phi})) \in \Perv^{\ULA}(\Bun_{G},\ol{\mathbb{F}}_{\ell})_{\phi} \]
for the fixed $\mu$. However, $\Red^{\mathrm{tw}}_{b,\phi}$ are direct summands of the perverse filtered Hecke eigensheaf $\nmEis(\mathcal{S}_{\phi_{T}})$ given by \cite[Theorem~9.12]{Ham2}. In particular, it suffices to show that $T_{\mu}(\nmEis(\mathcal{S}_{\phi_{T}})) \in \Perv^{\ULA}(\Bun_{G},\ol{\mathbb{F}}_{\ell})_{\phi}$. However, the sheaf $T_{\mu}(\nmEis(\mathcal{S}_{\phi_{T}}))$ admits a finite filtration whose graded are isomorphic to $\nmEis(\mathcal{S}_{\phi_{T}})$ occurring with some multiplicity, which are perverse. Since $\Perv^{\ULA}(\Bun_{G},\ol{\mathbb{F}}_{\ell})_{\phi}$ is closed under finite extensions, this shows the desired claim.
\end{proof}
We now seek to remove regularity from these results.
\subsubsection{Removing Regularity}
Throughout this subsection, put $\Lambda=\ol{\bb{F}}_\ell$.
We want to understand the extent to which Proposition \ref{prop: genlocal} and \ref{thm: generalperverseexact} are true when one drops the assumption of regularity on the toral parameter $\phi_{T}$ and just assumes genericity. In order to do this, we will need the following assumptions on our group $G/\bb{Q}_{p}$ and the prime $\ell$.
\begin{assumption}{\label{assumption: unramifiedness}}
We assume $G/\bb{Q}_{p}$ and $\ell$ satisfy the following: 
\begin{enumerate}
\item $G/\bb{Q}_{p}$ is unramified (in particular, quasi-split). 
\item The auxiliary group $\tilde G$ constructed below, with $S$ split over the unramified splitting field of $G$, can be chosen to satisfy Assumption~\ref{compatibility}.
\item $\ell$ is very decent with respect to $G$.
\end{enumerate}
\end{assumption}
Now, under this assumption, we may, as in the proof of \cite[Theorem~3.1.2]{HHS}, choose a closed embedding $Z(G)\hookrightarrow S$ into an unramified induced torus. We set $\tilde G=(G\times S)/Z(G)$ and $D=S/Z(G)$ via the anti-diagonal embedding of $Z(G)$. We note the following properties of the situation
\begin{equation}{\label{eqn: embedding}}
 1\longrightarrow G\xrightarrow{i}\tilde G
   \longrightarrow D\longrightarrow1,
 \qquad Z(\tilde G)=S,
 \qquad G^{\mathrm{der}}\simeq\widetilde G^{\mathrm{der}}.
\end{equation}
\begin{remark}{\label{rem: propogatingassumptions}}
The construction of the induced torus $S$ can be made over the unramified splitting field $F/\bb{Q}_{p}$ of $G$. Indeed, choose a surjection $\mathbb Z[\operatorname{Gal}(F/\bb{Q}_{p})]^r\twoheadrightarrow X^*(Z(G))$ and take $S=(\operatorname{Res}_{F/\bb{Q}_{p}}\mathbb G_m)^r$. In particular, if $\ell$ is very decent for  $G$ it is very decent for $\tilde G$ as well. Now we show the following in this situation. By Proposition~\ref{prop: compatibcentralisog}, Assumption~\ref{compatibility} then holds for $G$ as well.
\end{remark}
\begin{lemma}\label{lem:unramified-regular-lift}
Every unramified generic toral parameter $\phi_T$ for $G$ lifts
to an unramified generic regular toral parameter $\phi_{\tilde T}$
for $\tilde{G}$, where $\tilde T=(T\times S)/Z(G)$ is the maximal torus of $\tilde{G}$.
\end{lemma}
\begin{proof}
Let $\sigma$ be a generator of the Galois group of the unramified extension of $\bb{Q}_{p}$ splitting $T$. Write $\phi_T(\mathrm{Frob})=t\rtimes\sigma$ and choose
$\tilde{t} \in\hat{\tilde T}(\ol{\bb{F}}_{\ell})$ above
$t\in\hat T(\ol{\bb{F}}_{\ell})$, which exists since the map induced by $T \ra \tilde{T}$ on dual torii is surjective.
Sending Frobenius to $\tilde{t} \rtimes\sigma$  defines the  unramified lift $\phi_{\tilde{T}}$ of $\phi_{T}$.
The dual map $\hat{\tilde G}\to\hat G$ is a central
surjection, which in particular implies that we have an identification of the coroots and that genericity of $\phi_{T}$ implies genericity of $\phi_{\tilde{T}}$.

Put $H=\hat{\tilde G}$ and $A=H/H^{\mathrm{der}}$.
Since $Z(\tilde G)=S$ is connected, $\hat{\tilde G}^{\mathrm{der}}$ is
simply connected. Moreover, $H/H^{\mathrm{der}} \simeq\hat S$ is isomorphic to several copes of $\mathbb{G}_{m}$ and Frobenius acts by permutation. In particular, the fixed points $\hat{S}^\sigma$ is connected. The automorphism $\theta=\operatorname{Ad}(\widetilde t)\circ\sigma$
has finite order prime to $\ell$. Indeed, if $n$ is the order
of $\sigma$, then $(\widetilde t\rtimes\sigma)^n$ lies in the
torus and has finite order prime to $\ell$, while $\ell\nmid n$
by very decency. Consequently its fixed-point groups are smooth (see \cite[Proposition~VIII.5.11]{FS}\footnote{We are citing v2 of the arXiv}), and taking
invariants in the exact Lie-algebra sequence shows that
\[
 H^\theta\longrightarrow A^\sigma
\]
is smooth. Its image is an open subgroup of the connected torus
$A^\sigma$, hence is all of $A^\sigma$.
Its kernel $(H^{\mathrm{der}})^\theta$ is connected by
Steinberg's connectedness theorem
\cite[Theorem~8.1]{SteinbergEndomorphisms}.
It follows that $H^\theta$ is connected.

Write $T^\vee=\hat{\tilde T} \subset \hat{\tilde{G}}$. We now show that $H^{\theta} =((T^\vee)^\sigma)^0$. Let $\mathfrak h=\operatorname{Lie}(H)$ and $\mathfrak t=\operatorname{Lie}(T^\vee)$. Grouping the root
spaces into $\sigma$-orbits gives a $d\theta$-stable decomposition
\[
  \mathfrak h
  =
  \mathfrak t
  \oplus
  \bigoplus_{\mathcal O} V_{\mathcal O},
  \qquad
  V_{\mathcal O}
  :=
  \bigoplus_{\alpha\in\mathcal O}\mathfrak h_\alpha,
\]
where $\mathcal O$ runs over the $\sigma$-orbits in
$\Phi(H,T^\vee)$. On each $V_{\mathcal O}$, the action of
$d\theta$ agrees with the Frobenius action on the 
Weil group representation representation associated with
$\phi_{\widetilde T}$. Since the parameter is unramified,
inertia acts trivially, and genericity therefore gives $V_{\mathcal O}^{d\theta} = H^0(W_{\bb{Q}_{p}},V_{\mathcal O}) = 0$.

Moreover, $\theta$ restricts to $\sigma$ on $T^\vee$ since it is commutative.
Consequently,
\[
  \operatorname{Lie}(H^\theta)
  =
  \mathfrak h^{d\theta}
  =
  \mathfrak t^{d\sigma}.
\]
The inclusion $((T^\vee)^\sigma)^0\subset H^{\theta}$ thus identifies their Lie algebras. Both groups are smooth using \cite[Proposition~VIII.5.11]{FS}, so they have the same dimension. As $H^{\theta}$ is connected, this closed inclusion must therefore be an equality. In particular,
\[
  H^{\theta} =((T^\vee)^\sigma)^0\subset T^\vee.
\]

Suppose now that an element
$w\in W(H,T^\vee)^\sigma \simeq W_{\tilde{G}}$ of the relative Weyl group stabilizes the toral parameter $\phi_{\widetilde T}$. We may choose representatives of the Weyl group satisfying
$\sigma(n_w)=n_{\sigma(w)}$, so $\sigma(n_w)=n_w$.
Since $w$ fixes the $T^\vee$-conjugacy class of the parameter,
there exists $s\in T^\vee(\ol{\bb{F}_{\ell}})$ such that
\[
  w(\widetilde t)
  =
  s\widetilde t\,\sigma(s)^{-1}.
\]
Set $g=s^{-1}n_w$. Then
\[
  g\widetilde t\,\sigma(g)^{-1}
  =
  s^{-1}w(\widetilde t)\sigma(s)
  =
  \widetilde t.
\]
Thus $g\in H^{\theta} \subset T^\vee$, and hence
$n_w=sg\in T^\vee$ and therefore $w=1$. Therefore, $\phi_{\widetilde T}$ has trivial relative
Weyl stabilizer, as desired.
\end{proof}
We now want to use our results that are known for $\tilde{G}$ and generic regular $\phi_{T}$. We start first with the following basic Lemma.
\begin{lemma}\label{lem:finite-central-clean-pullback}
Let $a:L \ra H$ be a finite central isogeny of reductive groups over $\bb{Q}_{p}$, and write $a_{\Bun}:\Bun_{L} \ra \Bun_H$ for the induced map.
Then $a_{\Bun}$ is cohomologically smooth of relative dimension zero.
Let $c\in B(L)$ and set $d=a(c)$. We let $\Bun_{L}^{(c)} \hookrightarrow \Bun_{L}$ be the connected component $c$ lies in, and write $a_{C}: \Bun_{L}^{(c)} \ra \Bun_{H}$ for the restriction of $a$. Similarly, we write $a_{c}: \Bun_{L}^{c} \ra \Bun_{H}^{d}$ for the induced map on HN-strata. The natural commutative diagram
\[
\begin{tikzcd}
\Bun_L^c \ar[r,"a_c"] \ar[d,"j_c"']
  & \Bun_H^d \ar[d,"j_d"] \\
\Bun_{L}^{(c)} \ar[r,"a_C"'] & \Bun_H
\end{tikzcd}
\]
is Cartesian. Moreover, for $A\in\D(\Bun_H^d,\Lambda)$, the base-change maps give
\[
 a_C^*j_{d!}A\simeq j_{c!}a_c^*A,
 \qquad
 a_C^*Rj_{d*}A\simeq Rj_{c*}a_c^*A.
\]
In particular, the natural map $j_{c!}a_c^*A\to Rj_{c*}a_c^*A$
is an isomorphism whenever $j_{d!}A\to Rj_{d*}A$ is an isomorphism.
\end{lemma}

\begin{proof}
Write $K=\ker(a)$, choose an embedding of $K$ into a torus $D$,
and define $L_1=(L\times D)/K$  $Q=D/K$, using the antidiagonal embedding of $K$ in $L\times D$.
There are exact sequences
\[
 1\longrightarrow D\longrightarrow L_1\longrightarrow H\longrightarrow1,
 \qquad
 1\longrightarrow L\longrightarrow L_1\longrightarrow Q\longrightarrow1.
\]
By \cite[Lemma~III.2.10]{FS} and the discussion following it,
$\Bun_{L_1}\to\Bun_H$ is a surjective torsor under the Picard stack
$\Bun_D$. It is easy to check that one has a commutative diagram
\[
\begin{tikzcd}
  \Bun_L \ar[d,"a_{\Bun}"']
  & \Bun_L\times_{\Bun_H}\Bun_{L_1}
      \ar[l]
      \ar[r]
      \ar[d]
  & \Bun_D \ar[d]
  \\
  \Bun_H
  & \Bun_{L_1}
      \ar[l]
      \ar[r]
  & \Bun_Q .
\end{tikzcd}
\]
which is Cartesian. Indeed, a lift of an $H$-bundle along $L_1 \ra H$ can be twisted
by a $D$-bundle, and the obstruction to reducing that lift along $L \ra L_{1}$ is the induced $Q$-bundle of this $D$-bundle. In particular, since $\Bun_{L_{1}} \ra \Bun_{H}$ is $v$-surjective (it is $\Bun_{D}$ torsor), it suffices by \cite[Proposition~23.15]{Ecod} to show that the middle arrow is cohomologically smooth of dimension $0$. It therefore suffices to show that $\Bun_{D} \ra \Bun_{Q}$ is cohomologically smooth of dimension $0$. Since $D$ and $Q$ are both torii this breaks as a disjoint union of
\[ [\ast/\ul{D(\bb{Q}_{p})}] \ra [\ast/\mathrm{Im}(D(\bb{Q}_{p}) \ra Q(\bb{Q}_{p}))] \ra [\ast/\ul{Q(\bb{Q}_{p})}], \]
where the first map is a fibration in the classifying stack attached to the finite kernel of the map $D(\bb{Q}_{p}) \ra Q(\bb{Q}_{p})$, and the second map is finite \'etale since the iamge of $D(\bb{Q}_{p})$ in $Q(\bb{Q}_{p})$ is open of finite index. In particular, we conclude that $[\ast/\ul{D(\bb{Q}_{p})}] \ra [\ast/\ul{Q(\bb{Q}_{p})}]$ is cohomologically smooth of dimension $0$, as desired.

A finite central isogeny identifies the rational cocharacter lattices defining the Newton points. An element of $B(L)$ mapping to $d$ and having Kottwitz invariant
$\kappa_L(c)$ therefore has the same Newton and Kottwitz invariants
as $c$, and hence equals $c$. This proves the second assertion. 

The remaining claims follow from smooth and proper base-change
\cite[Proposition~23.16(ii)]{Ecod}.
\end{proof}
To proceed further, we choose a subtorus $C\subset S=Z(\widetilde G)$ whose
map to $D$ is an isogeny. Multiplication gives a
finite central isogeny
\begin{equation}\label{eq:unramified-central-bridge}
 m:G\times C\longrightarrow\widetilde G.
\end{equation}
For any $\delta\in B(C)$, we note that the projection
\begin{equation}{\label{eqn: projection}}
 p_\delta:\Bun_G\times\Bun_C^\delta\longrightarrow\Bun_G
\end{equation}
is cohomologically smooth, surjective. In particular, the pullback functor $p_{\delta}^{*}$ is conservative.

\begin{proposition}\label{prop:unramified-clean-semisimple}
Let $\phi_T$ be an unramified generic toral parameter and let
$\phi$ be its induced parameter for $G$ and $\ell$ satisfying assumption \ref{assumption: unramifiedness}. Then the following is true. 
\begin{enumerate}
\item For every $b\in B(G)_{\mathrm{un}}$ and $w\in W_b$, the
representation $\rho_{b,w}^{\chi}$ is semisimple of finite length.
\item The natural map
\[
 j_{b!}\rho_{b,w}^{\chi}\longrightarrow Rj_{b*}\rho_{b,w}^{\chi}
\]
is an isomorphism. Moreover, the same is true for every irreducible
constituent of $\rho_{b,w}^{\chi}$.
\end{enumerate}
\end{proposition}

\begin{proof}
Let $\chi$ denote the character attached to $\phi_{T}$, and choose $\phi_{\widetilde T}$ as in
Lemma~\ref{lem:unramified-regular-lift}, and let $\widetilde\chi$
be its associated smooth character. Write $\eta=\widetilde\chi|_{C(\bb{Q}_{p})}$. For $\widetilde b=i(b)$, the map
$J_b\times C\to J_{\widetilde b}$ induced by
\eqref{eq:unramified-central-bridge} is a finite central isogeny, by arguing as in the proof of Lemma \ref{lemma: adjderiso}. Similarly, as in the proof of Proposition \ref{prop: compatibcentralisog}, we have that
\[ \rho_{\tilde b,w}^{\tilde\chi}
       |_{J_b(\bb{Q}_{p})\times C(\bb{Q}_{p})}
 \simeq\rho_{b,w}^{\chi}\boxtimes\eta. \]
By Proposition \ref{prop: genlocal} (1), $\rho_{\widetilde b,w}^{\widetilde\chi}$ is irreducible, since $\phi_{\tilde{T}}$ is generic regular and Remark \ref{rem: propogatingassumptions}. We now deduce the semi-simplicity of $\rho_{b,w}^{\chi}$ from the following.
\begin{lemma}\label{lem:finite-index-clifford}
Let $k$ be a field,  $H$ a group, and let
$N\triangleleft H$ be a normal subgroup of finite index.
If $V$ is a (not necessarily finite-dimensional) irreducible $k[H]$-module, then its restriction
to $N$ is semisimple of finite length, with
\[
  \operatorname{length}_{k[N]}(V)\leq [H:N].
\]
\end{lemma}
\begin{proof}
Put $r=[H:N]$ and choose coset representatives
$h_{1},\ldots,h_r$. For any nonzero $v\in V$,
irreducibility gives $V=k[H]v=\sum_{i=1}^r k[N](h_iv)$. Thus, $V$ is finitely generated as a $k[N]$-module. Consequently, it admits a maximal proper $k[N]$-submodule
$M$. Indeed, by finite generation, the union of a chain of proper
submodules remains proper, so Zorn's lemma applies.

Since $N$ is normal in $H$, translation by any $h \in H$
permutes the $k[N]$-submodules of $V$. In particular,
$hM$ is again a maximal proper $k[N]$-submodule, and therefore
$V/hM$ is simple. Moreover, $nM=M$ for all $n\in N$,
so every translate of $M$ is one of the $h_iM$.
The intersection
\[
  K:=\bigcap_{i=1}^r h_iM
    =\bigcap_{h\in H}hM
\]
is therefore an $H$-submodule of $V$. Since $K\subset M\neq V$, irreducibility of $V$ implies that $K=0$.

Hence, the diagonal $k[N]$-linear map
\[
  V\longrightarrow\bigoplus_{i=1}^r V/h_iM, \]
\[   x\mapsto(x+h_iM)_{i=1}^r, \]
is injective. Its target is a direct sum of $r$ simple
$k[N]$-modules. Since every submodule of a semisimple
module is semisimple, the claim follows.
\end{proof}
We apply this with
$N=\operatorname{im}(J_b(\bb{Q}_{p})\times C(\bb{Q}_{p})\to J_{\tilde b}(\bb{Q}_{p})) \subset H  = J_{\tilde{b}}(\bb{Q}_{p})$ the index is easily checked to be finite using Galois cohomology. Since $C(\bb{Q}_{p})$ acts through the smooth character $\eta$, this
also proves semisimplicity and finite length of $\rho_{b,w}^{\chi}$.

We note that $!$ and $*$-pushforward of the $\rho_{\tilde b,w}^{\tilde{\chi}}$ agree by Proposition \ref{prop: cleaness}. Therefore, Lemma~\ref{lem:finite-central-clean-pullback}, applied at
$(b,0)\in B(G\times C)$, gives that the $!$ and $*$-pushforward of
$\rho_{b,w}^{\chi}\boxtimes\eta$ agree on
$\Bun_G\times\Bun_C^0$. The obvious conservativity/smoothness of the pullback $p_0^*$ of (\ref{eqn: projection}) therefore gives cleanness of the $\rho_{b,w}$. 
\end{proof}

\begin{theorem}
\label{thm:unramified-perverse-without-regularity}
For $G$ and $\ell$ satisfying assumption \ref{assumption: unramifiedness}, we let $\phi$ be induced
from an unramified generic toral parameter $\phi_T$. Then the following hold.
\begin{enumerate}
\item The functors $j_{b}^{*}$ for varying $b \in B(G)$ induce an equivalence
\[
  \DULA(\Bun_G,\Lambda)_\phi \simeq 
  \prod_{b\in B(G)_{\mathrm{un}}}
  \Dadm(\Bun_G^b,\Lambda)_\phi.
\]
\item
For every geometric dominant cocharacter $\mu$, the Hecke
operator attached to the highest weight tilting module
$\mathcal T_\mu$,
\[
  T_\mu:
  \DULA(\Bun_G,\Lambda)_\phi
  \longrightarrow
  \DULA(\Bun_G,\Lambda)_\phi^{BW_{E_\mu}},
\]
is perverse $t$-exact.
\end{enumerate}.
\end{theorem}

\begin{proof}
Put $d_b=\langle 2\rho_G,\nu_b\rangle$, and let $\chi$ be the
character associated with $\phi_T$. The proof of (1) follows exactly as in Proposition \ref{prop: genlocal}, where it reduces to Proposition \ref{prop:unramified-clean-semisimple} instead of Proposition \ref{prop: cleaness}.

We now prove (2). Choose the generic regular lift
$\phi_{\widetilde T}$ furnished by Lemma~\ref{lem:unramified-regular-lift} and consider the Eisenstein functors
\begin{equation}{\label{eqn: EisensteinSeries}}
  E_G:=\nmEis_G(\mathcal S_{\phi_T}),
  \qquad
  E_{\widetilde G}:=
  \nmEis_{\widetilde G}(\mathcal S_{\phi_{\widetilde T}}).
\end{equation}
applied to the Hecke eigensheaves $\mathcal{S}_{\phi_{T}}$ (resp. $\mathcal{S}_{\phi_{\widetilde{T}}}$), as in \cite{Ham2}. Let $\widetilde\chi$ be the character associated with
$\phi_{\widetilde T}$, put
$\eta=\widetilde\chi|_{C(\bb{Q}_{p})}$, and let $\mathcal S_\eta$
be the corresponding torus eigensheaf on $\Bun_C$.

Recall the finite central isogeny
$m:G\times C\to\tilde G$ from
\eqref{eq:unramified-central-bridge}.
We have a commutative diagram
\begin{equation}{\label{eqn: EisensteinDiagram}}
\begin{tikzcd}
\Bun_{T} \times \Bun_{C} \ar[r,"m_{T}"] & \Bun_{\tilde{T}} \\
  \Bun_{B} \times \Bun_{C} \ar[r] \ar[d] \arrow[u]
    & \Bun_{\tilde B} \ar[d] \arrow[u] \\
  \Bun_{G} \times \Bun_{C} \ar[r,"m_{\Bun}"']
    & \Bun_{\tilde G}.
\end{tikzcd}
\end{equation}
where the bottom row is Cartesian in light of the isomorphism
\[
  (G\times C)/(B\times C)
  \simeq \tilde G/\tilde B.
\]
We recall the geometric Eisenstein functors (\ref{eqn: EisensteinSeries}) are given by pulling and pushing along this diagram the eigensheaves $\mathcal{S}_{\phi_{T}}$ (resp. $\mathcal{S}_{\phi_{\tilde{T}}}$) on $\Bun_{T}$ (resp. $\Bun_{\tilde{T}}$), which are just given by the characters $\chi$ (resp. $\tilde{\chi}$) on the various connected components of $\Bun_{T}$ (resp. $\Bun_{\tilde{T}}$).

In particular, we note that $m_T^{*}(\mathcal{S}_{\phi_{\tilde{T}}}) \simeq \mathcal{S}_{\phi_{T}} \boxtimes \mathcal{S}_{\eta}$, where $\mathcal{S}_{\eta}$ has value given by $\eta$ on the connected components of $\Bun_{C}$. Proper base-change applied to (\ref{eqn: EisensteinDiagram}) satisfies
\begin{equation}\label{eq:unramified-eisenstein-pullback}
  m_{\Bun}^*E_{\tilde G}
  \simeq E_G\boxtimes\mathcal S_\eta.
\end{equation}
We write $\tilde{\phi}$ for the parameter induced by $\phi_{\tilde{T}}$. Now, since $\phi_{\tilde{T}}$ is generic and regular, we have an isomorphism
\[ E_{\tilde{G}} \simeq \bigoplus_{\tilde{b} \in B(\tilde{G})_{\mathrm{un}}} j_{\tilde{b}!}(\Red^{\mathrm{tw}}_{\tilde b,\tilde{\phi}}) \]
In particular, for $b\in B(G)_{\mathrm{un}}$, restrict
\eqref{eq:unramified-eisenstein-pullback} to the stratum
indexed by $(b,0)$ and then pull back along the point
classifying the trivial $C$-bundle. This gives
\[
  j_b^*E_G
  \simeq \Red_{b,\phi}^{\mathrm{tw}},
\]
where we note $i$ identifies the relevant relative Weyl groups and
Levi Weyl groups, and $d_{i(b)}=d_b$. Similarly, for $b\notin B(G)_{\mathrm{un}}$, one similarly one sees that $j_{b}^{*}E_{G}$ is $0$.

Now, by applying excision to $E_{G}$ and using Part (1) to see that the excision spectral sequence degenerates, we have that
\[
  E_G
  \simeq
  \bigoplus_{b\in B(G)_{\mathrm{un}}}
  \bigoplus_{w\in W_b}
  j_{b!}\rho_{b,w}^{\chi}[-d_b].
\]
In particular, $E_G$ is perverse by Part (1). Now we note that $E_{G}$ is again a filtered Hecke eigensheaf. Indeed, this follows from
\cite[Corollary~6.4]{Ham2}, and does not use any genericity. The proof of Part (2) now follows exactly as in Theorem \ref{thm: generalperverseexact}.
\end{proof}
\begin{remark}{\label{rem: regularityishownecessary}}
The only part where we crucially used that $G/\bb{Q}_{p}$ is unramified and that the parameter $\phi_{T}$ is unramified is Lemma~\ref{lem:unramified-regular-lift}. In particular, we could have also formulated a slightly more axiomatic version of Theorem \ref{thm:unramified-perverse-without-regularity} in terms of the existence of a generic regular lift $\phi_{\tilde{T}}$ of a generic $\phi_{T}$. We have not done this for simplicity. The existence of such a lift is not guaranteed if either $G$ is assumed to be ramified or $\phi_{T}$ is assumed to be ramified. For example, if $G$ is ramified and $\phi_{T}$ is unramified a sufficient condition to guarantee the existence of this lift seems to be that $\phi_{T}$ is split generic (see Definition \ref{defn: splittinggeneric}).
\end{remark}
In order to also have some results for not necessarily unramified parameters $\phi_{T}$, we also discuss when the additional assumptions of regularity is superfluous for the groups $G$ that we consider. We will also discuss precisely what the $\ell$ is very decent condition means in these cases.
\subsection{Verification of additional assumptions}
\label{section:verification}
We will now want to apply \ref{thm: generalperverseexact}. However, we want to first investigate the extent to which regularity is a necessary condition. In other words, when is it implied by generic? Before proceeding, we will explicate the genericity condition further. Let $F$ be a minimal splitting field of $G$. Given a $\Gamma$-orbit of absolute coroots
$\alpha\in\mathbb{X}_{*}(T_{\ol{\mathbb{Q}}_p})/\Gamma$, choose a
representative $\tilde\alpha\in\cochar$ and let $F_\alpha$ be its reflex field.
The associated representation of $\phantom{}^LT$ is the inflation of
\[
\alpha_{\Ad}:=
\Ind_{\hat T\rtimes(W_{F_\alpha}/W_F)}^{\hat T\rtimes(W_{\mathbb Q_p}/W_F)}
(\tilde\alpha\otimes\tau_\alpha).
\]
Here $\tilde\alpha$ is viewed as a character of $\hat T$, extended trivially
to $W_{F_\alpha}/W_F$, and $\tau_\alpha$ is the character of
$W_{F_\alpha}/W_F$ given by its action on the root space
$\hat{\mf g}_{\tilde\alpha}$ in the Lie algebra of $\hat G$, inflated to
$\hat T\rtimes(W_{F_\alpha}/W_F)$.

We have the following.
\begin{lemma}
\label{lemma:genericE}
For $\alpha \in \mathbb{X}_{*}(T_{\ol{\mathbb{Q}}_{p}})/ \Gamma$ a Galois orbit of coroots, we let $F_{\alpha}$ be the reflex field. Then $\phi_{T}$ is generic if and only if $R\Gamma(W_{F_{\alpha}},\tilde{\alpha} \circ \phi_{T} \otimes \tau_{\alpha}|_{W_{F_{\alpha}}})$ is trivial for all absolute coroots $\tilde{\alpha} \in \mathbb{X}_{*}(T_{\ol{\mathbb{Q}}_{p}})$.
\end{lemma}
\begin{proof}
This is Shapiro's Lemma.
\end{proof}
We note that it is not always the case that $F_{\alpha} = F$ is the minimal splitting field of $G$ for all Galois orbits of coroots $\alpha$ for all the groups we will consider (e.g. for the quasi-split unitary group $G = U_{3}$ attached to $E/\bb{Q}_{p}$, the positive absolute coroots form two Galois orbits: $\{\alpha_{1},\alpha_{2}\}$ with reflex field $E$ and $\{\alpha_{1}+\alpha_{2}\}$ with reflex field $\bb{Q}_{p}$, where $\alpha_{1}$ and $\alpha_{2}$ are the simple absolute coroots). For this reason, we introduce the following slightly stronger condition that will streamline our verification of the additional assumptions by allowing us to change to splitting fields.
\begin{definition}{\label{defn: splittinggeneric}}
We say a toral parameter $\phi_{T}: W_{\bb{Q}_{p}} \ra \phantom{}^{L}T(\ol{\bb{F}}_{\ell})$ is split generic if, for all geometric coroots $\tilde{\alpha} \in \mathbb{X}_{*}(T_{\ol{\mathbb{Q}}_{p}})$, the Galois cohomology of the representation $R\Gamma(W_{F},\tilde{\alpha} \circ \phi_{T}|_{W_{F}})$ is trivial, where $F$ denotes the minimal splitting field of $G$. Similarly, we say a semisimple L-parameter $\phi: W_{\bb{Q}_{p}} \ra \phantom{}^{L}G(\ol{\bb{F}}_{\ell})$ is split generic if it is induced from a split generic toral parameter.
\end{definition}
\begin{remark}{\label{rem: splittinggenericgeneric}}
We note that if the reflex field $F_{\alpha}$  of all the Galois orbits of coroots $\alpha$ is equal to the minimal splitting field $F$ then Lemma \ref{lemma:genericE} tells us that splitting generic is the same as generic. 
\end{remark}
We have the following basic lemma, which states that in the relevant situations this condition is strictly stronger than usual genericity. 
\begin{lemma}{\label{lemma: splitgenercimpliesgeneric}}
Suppose that the order of $[F:\bb{Q}_{p}]$ is coprime to $\ell$ (e.g $\ell$ is very decent with respect to $G$), where $F$ is the minimal splitting field of $G$. Then split generic implies generic.
\end{lemma}
\begin{proof}
Using Lemma \ref{lemma:genericE},  genericity is equivalent
to the vanishing $R\Gamma(
    W_{F_\alpha},
    \widetilde{\alpha}\circ\phi_T
    \otimes \tau_\alpha|_{W_{F_\alpha}}
  )\simeq 0$
for all Galois orbits of coroots $\alpha$ Now, for all $i \in \bb{N}_{\geq 0}$, we recall that we have the maps defined by restriction and corestriction
\[ H^{i}(W_{F_{\alpha}},\tilde{\alpha} \circ \phi_{T} \otimes \tau_{\alpha}|_{W_{F_{\alpha}}}) \xrightarrow{\mathrm{res}} H^{i}(W_{F},\tilde{\alpha} \circ \phi_{T}|_{W_{F}}) \xrightarrow{\mathrm{cores}} H^{i}(W_{F_{\alpha}},\tilde{\alpha} \circ \phi_{T} \otimes \tau_{\alpha}|_{W_{F_{\alpha}}}), \]
and that the composite is given by multiplication by $[F:F_{\alpha}]$ (see \cite[Chapter~VIII,Proposition~4]{SerreLocalFields}). Therefore, since $\ell \nmid [F:F_{\alpha}]$ by our assumption, it follows that the vanishing of $R\Gamma(W_{F},\tilde{\alpha} \circ \phi_{T}|_{W_{F}})$ implies the vanishing of $R\Gamma(W_{F_{\alpha}},\tilde{\alpha} \circ \phi_{T}|_{W_{F_{\alpha}}} \otimes \tau_{\alpha})$, since multiplication by $[F:F_{\alpha}]$ is invertible on $\ol{\bb{F}}_{\ell}$-vector spaces. This shows the desired claim.
\end{proof}
We want to consider certain similitude variants of the groups listed in Theorem \ref{thm:AssumptionLLC}. Let $L/\mathbb{Q}_p$ a finite unramified extension. We consider the following four types of groups over $\mathbb{Q}_p$, also introduced in Table \ref{constrainttable}.
\begin{enumerate}
  \item $\Res_{L/\mathbb{Q}_p}(\mathbb{G}_{m,L}\times \GL_{n,L})$. Projecting onto the first component gives a map 
  \[\mathbb{G}_{m,L}\times \GL_{n,L}\rightarrow \mathbb{G}_{m,L},\]
  which, upon taking restriction of scalars from $L$ to $\mathbb{Q}_p$, gives us a similitude map
  \[\nu:\Res_{L/\mathbb{Q}_p}(\mathbb{G}_{m,L}\times \GL_{n,L})\rightarrow \Res_{L/\mathbb{Q}_p}(\mathbb{G}_{m,L}).\]
  \item $\GU_n$, for any $n$. By definition this has a unitary similitude character $\nu:\GU_n\rightarrow \mathbb{G}_m$.
  \item $\Res_{L/\mathbb{Q}_{p}}(\GU_{2,L})$. Over $L$, there is a map
  \[\GU_{2,L}\simeq (\GL_{2,L}\times\Res_{L'/L}\mathbb{G}_{m,L'})/\mathbb{G}_{m,L}\rightarrow \mathbb{G}_{m,L},\]
  taking $(g,h)$ to $\det(g)\mathrm{N}_{L'/L}(h)$, and taking restriction of scalars gives us a map
  \[\nu:\Res_{L/\mathbb{Q}_p}(\GU_{2,L})\rightarrow \Res_{L/\mathbb{Q}_p}(\mathbb{G}_{m,L}).\]
  \item $\Res_{L/\mathbb{Q}_p}(\GSp_{2n,L})$. Over $L$, there is a symplectic similitude ($n\geq 2$)  or determinant $(n=1)$ character, and taking restriction of scalars gives us a map
  \[\nu:\Res_{L/\mathbb{Q}_{p}}(\GSp_{2n,L})\rightarrow \Res_{L/\mathbb{Q}_p}(\mathbb{G}_{m,L}).\]
\end{enumerate}
For later reference, we define the following similitude groups. Let $\mathbf L=(L_1,\ldots,L_r)$ be a tuple of
finite unramified extensions of $\mathbb Q_p$.
We write
\[
G(\mathrm{SL}_{2,\mathbf L})
:=
\left(
\prod_{i=1}^r
\operatorname{Res}_{L_i/\mathbb Q_p}\mathrm{GL}_2
\right)
\times_{\prod_i
\operatorname{Res}_{L_i/\mathbb Q_p}\mathbb G_m}
\mathbb G_m,
\]
where the maps from the general linear
groups are the determinants. Similarly, put
\[
G(\mathrm{Sp}_{4,\mathbf L})
:=
\left(
\prod_{i=1}^r
\operatorname{Res}_{L_i/\mathbb Q_p}\mathrm{GSp}_4
\right)
\times_{\prod_i
\operatorname{Res}_{L_i/\mathbb Q_p}\mathbb G_m}
\mathbb G_m,
\]
using the symplectic similitude characters.

For each $i$, let $E_i/L_i$ be the
unramified quadratic extension, and put
\[
G(U_{2,\mathbf L})
:=
\left(
\prod_{i=1}^r
\operatorname{Res}_{L_i/\mathbb Q_p}
\mathrm{GU}_2(E_i/L_i)
\right)
\times_{\prod_i
\operatorname{Res}_{L_i/\mathbb Q_p}\mathbb G_m}
\mathbb G_m.
\]
In each case the map from $\mathbb G_m$
is diagonal. 

As explained in the Introduction, we will consider similitude fiber products of the groups from (1)-(4), namely, given a finite family $\{G_j\}_{j=1}^r$ of such groups, with similitude characters $\nu_j:G_j\rightarrow S_j$, we look at 
\[G=\left(\prod_{j=1}^r G_j\right)\times_{\prod_j S_j}\mathbb{G}_m,\]
where the map from $\mathbb{G}_m$ into $\prod_{j} S_j$ is the diagonal. This is the setup of Assumption \ref{assumption: keygroupassumption}. We now treat the relationship between genericity and regularity, as well as the very decent assumption.

\begin{lemma}
\label{lem:unified-decency}
Let $G$, $p$, and $\ell$ satisfy
Assumption~\ref{assumption: keygroupassumption}. Then $\ell$ is very
decent for $G$ in the sense of
Definition~\ref{defn: verydecent}.
\end{lemma}
\begin{proof}
Firstly, observe that $Z(G_{\overline{\mathbb Q}_p})$ is the subgroup of $\prod_j Z(G_j)_{\overline{\mathbb Q}_p}\times \mathbb{G}_{m,\overline{\mathbb Q}_p}$ cut out by the vanishing of the equations $\nu_j=t$ for all $j$, where $t$ is the variable for the $\mathbb{G}_{m,\overline{\mathbb Q}_p}$ factor. In particular, if $t$ is fixed, then the number of connected components depends on the number of connected components of the kernel of $\nu_j$.

Note that for $\mathrm{GL}_2$ and $\mathrm{GSp}_4$, the center consists of scalar matrices and for any scalar matrix $zI$, we have $\nu(zI)=z^2$. Thus, each for each factor of the form $\operatorname{Res}_{L_{j}/\bb{Q}_{p}}(\mathrm{GL}_{2,L_j})$ or $\operatorname{Res}_{L_{j}/\bb{Q}_{p}}(\mathrm{GSp}_{4,L_j})$, the kernel of $\nu_j$ has $2^{[L_j:\mathbb{Q}_p]}$ geometric connected components.

For the remaining types of groups,  
the restriction of the similitude
to the geometric center has
connected kernel: over $\ol{\mathbb{Q}}_p$ all groups are of the form $\mathbb G_m\times\mathrm{GL}_n$
and the similitude map is given by
$(u,v)\mapsto uv$. In particular, the kernel is isomorphic to $\mathbb{G}_m$, which is connected.

As in (\ref{eqn: adhocinvariant}), we define
\[s=s(G):=
\sum_{\substack{j:\ G_j\text{ is of type }\\
\operatorname{Res}_{L_{j}/\bb{Q}_{p}}(\mathrm{GL}_{2,L_j}) \text{ or }
\operatorname{Res}_{L_{j}/\bb{Q}_{p}}(\mathrm{GSp}_{4,L_j})}}
[L_j:\mathbb Q_p],\]

If $s\geq1$, choose one scalar
coordinate $z_1$ with multiplier
$t=z_1^2$. Each of the other
$s-1$ scalar coordinates $z_i$
then satisfies
\[
(z_i/z_1)^2=1.
\]
The remaining central coordinates
contribute connected torus factors.

When $s=0$, the center is
connected. Consequently,
\[
|
\pi_0(
Z(G_{\overline{\mathbb Q}_p})
)|
=
2^{\max\{s-1,0\}}.
\]

Thus, Assumption~\ref{assumption: keygroupassumption} (2b), implies that its order is
prime to $\ell$. 

Moreover, we see that the action of
$W_{\mathbb Q_p}$ on $\hat{G}$
factors through the finite
unramified splitting quotient
associated with the factors $G_j$. Its order divides the least common
multiple of the degrees of the
splitting fields of these factors,
which is prime to $\ell$ by the
restrictions in Table~\ref{constrainttable}. Hence, $\ell$ is very decent for $G$.
\end{proof}
We now want to study the regularity condition. 
Let $\phi_T$ be a generic toral
parameter with associated smooth
character $\chi:T(\mathbb Q_p)\longrightarrow
\overline{\mathbb F}_\ell^\times$.
Assume moreover that $\phi_T$ is
split generic on the absolute coroots
belonging to the unitary factors,
tested over the splitting field of
each such factor. Let $\nu:G\to\mathbb G_m$ denote
the similitude character.

Define the multiset $\mathcal I$ consist of one copy of $j$
for each $\operatorname{Res}_{L_j/\mathbb Q_p}
\mathrm{GL}_2$ factor and two copies of $j$
for each
$\operatorname{Res}_{L_j/\mathbb Q_p}
\mathrm{GSp}_4$ factor appearing in $G$. For $c\in\mathcal{I}$, write $L_c$ for the corresponding field extension and for each $c$ we have a character
\[
\eta_c:L_c^\times\longrightarrow
\overline{\mathbb F}_\ell^\times
\]
defined using the coroots as follows. For a $\GL_2$ factor, consider the image of the embedding of $\mathbb{G}_{m,L_c}$ into $\GL_{2,L}$ given by
\[x\mapsto \mathrm{diag}(x,x^{-1}),\]
while for a $\GSp_4$ factor, consider the image of the two embeddings of $\mathbb{G}_{m,L_c}$ into $\GSp_{4,L}$ given by 
\[x\mapsto \operatorname{diag}(x,1,1,x^{-1})
\text{ or }
\operatorname{diag}(1,x,x^{-1},1),\]
respectively. Note that in both cases the image lands in the subgroup of $G_j$ with similitude 1, and thus we can view this as a subgroup of $G$ if we take the values on the other factors to be $1$. In either case, we define $\eta_c$ to be the restriction of $\chi$ to this subgroup. 

Observe that if we consider the composition of $\phi_T$ with the corresponding absolute coroot $\alpha$, the corresponding character under local class field theory is precisely $\eta_c$.
Thus, if $\eta_c$ were trivial, the associated
coroot representation would have nonzero
$H^0(W_{L_c},-)$, contradicting genericity.

Set
\[
\mathcal Q_\chi
=\{c\in\mathcal I:\eta_c^2=1\},
\qquad
q_\chi=|\mathcal Q_\chi|,
\qquad
H_G=\nu(T(\mathbb Q_p))
\subset\mathbb Q_p^\times.
\]
By genericity, $\eta_c\neq1$ for
every $c\in \mathcal{I}$.

Therefore, for $c\in Q_\chi$,
the character $\eta_c$ is
nontrivial quadratic. Restrict to $H_G\subseteq \mathbb{Q}_p^\times\subseteq L_c^\times$. This allows us to define
\begin{equation}\label{eq:mixed-stabilizer}
\begin{aligned}
\Theta_\chi:
(\mathbb Z/2\mathbb Z)^{\mathcal Q_\chi}
&\longrightarrow
\operatorname{Hom}
(H_G,\overline{\mathbb F}_\ell^\times),
\\
(e_c)_c
&\longmapsto
\prod_{c\in\mathcal Q_\chi}
(\eta_c|_{H_G})^{e_c}.
\end{aligned}
\end{equation}
\begin{lemma}\label{lem:unified-regularity}
Let $G$ satisfy
Assumption~\ref{assumption: keygroupassumption}. Let $\phi_T$ be a generic toral parameter for $G$ which is moreover split generic on the absolute coroots
belonging to the factors $G_j$ which are unitary,
tested over the splitting field of each such factor. We consider the map $\Theta_\chi$ as in (\ref{eq:mixed-stabilizer}) above.
Then
\[
\operatorname{Stab}_{W_G}(\chi)
\simeq\ker\Theta_\chi.
\]
In particular, $\chi$ is regular if and only if
$\Theta_\chi$ is injective. 
\end{lemma}
\begin{proof}

The relative Weyl group of $G$ is the product
of the relative Weyl groups of its factors,
because their derived groups are unchanged
by forming the common-similitude fiber product.

Suppose $w\in W_G$ stabilizes $\chi$.
Restrict the equality $\chi^w=\chi$ to the
subtorus of similitude 1 elements of a factor of the form 
\[
\operatorname{Res}_{L/\mathbb Q_p}
(\mathbb G_m\times\mathrm{GL}_n)\qquad\text{or}\qquad\operatorname{Res}_{L/\mathbb Q_p}(\GU_n)
\]

In the first case, note that this subtorus can be identified with the diagonal torus $\mathbb{G}_m^n$ in $\GL_n$, and writing $\chi=(\chi_i)_{i=1}^n$ on this diagonal torus, genericity implies that the $\chi_i$ are pairwise distinct. In the second case, this subtorus can also be identified after passing to a splitting field with the diagonal torus $\mathbb{G}_m^n$ in $\GL_n$, so again writing $\chi=(\chi_i)_{i=1}^n$ on this diagonal torus, split genericity means that the $\chi_i$ are pairwise distinct. Since the corresponding Weyl groups act
by permutations of these characters,
$w$ is trivial on these factors. Thus, the Weyl elements coming from such factors do not lie in $\operatorname{Stab}_{W_G}(\chi)$.

For a factor of the form $\Res_{L/\mathbb{Q}_p}\GL_2$, we can write the diagonal torus as 
\[
(t_1,t)\longmapsto
\operatorname{diag}
(t_1,t/t_1),
\]
where $t\in H_G$ is the similitude factor. Observe that the only non-trivial Weyl element $w$ acts by sending $t_1$ to $t/t_1$, and $\chi^w=\chi$ implies $\eta^2=1$, $\eta|_{H_G}=1$. 

For a factor of the form
$\operatorname{Res}_{L/\mathbb Q_p}\mathrm{GSp}_4$,
write the diagonal torus as
\[
(t_1,t_2,t)\longmapsto
\operatorname{diag}
(t_1,t_2,t/t_2,t/t_1),
\]
where $t\in H_G$ is the similitude factor.
The seven non-identity Weyl group elements of $\GSp_4$
act as follows. The right-hand column
records necessary consequences of
$\chi^w=\chi$.
\[
\begin{array}{c|c}
(t_1,t_2)\text{ to}
&
\text{Consequence of }\chi^w=\chi
\\
\hline
(t_2,t_1)
&
\eta_1\eta_2^{-1}=1
\\
(t/t_1,t_2)
&
\eta_1^2=1,\quad \eta_1|_{H_G}=1
\\
(t_1,t/t_2)
&
\eta_2^2=1,\quad \eta_2|_{H_G}=1
\\
(t/t_1,t/t_2)
&
\eta_1^2=\eta_2^2=1,\quad
(\eta_1\eta_2)|_{H_G}=1
\\
(t/t_2,t_1)
&
\eta_1\eta_2^{-1}=1
\\
(t_2,t/t_1)
&
\eta_1\eta_2^{-1}=1
\\
(t/t_2,t/t_1)
&
\eta_1\eta_2=1
\end{array}
\]

Indeed, setting $t=1$ and all coordinates
of the other factors equal to the identity,
the variables $t_1,t_2\in L^\times$
vary independently.

For the four transformations involving
a permutation, comparison of these
independent characters forces
$\eta_1\eta_2^{-1}=1$ or
$\eta_1\eta_2=1$. Both contradict genericity,
since these are coroot characters.

The remaining three transformations
are precisely the independent sign
changes
\[
t_i\longmapsto t/t_i.
\]
Their conditions are recorded in the
second, third, and fourth rows.

Consequently, the only possible
elements in $\operatorname{Stab}_{W_G}(\chi)$ are associated with the sign change transformations coming from $\mathrm{GL}_2$ and $\mathrm{GSp}_4$ factors.

For a subset $S\subset\mathcal I$,
we can consider the corresponding sign change transformation which sends
$a_c\mapsto t/a_c$ for $c\in S$, where $a_c$ is the corresponding torus coordinate, 
and leaves all remaining coordinates
unchanged. Writing $w_S$ for the Weyl element
given by these sign changes, we obtain
\[
\chi^{w_S}\chi^{-1}
=
\prod_{c\in S}
\eta_c(t)\eta_c(a_c)^{-2}.
\]

The $a_c\in L_c^\times$ vary
independently. Thus this ratio is
identically $1$ precisely when
$S\subseteq\mathcal Q_\chi$ and
\[
\prod_{c\in S}\eta_c|_{H_G}=1.
\]
This proves the stabilizer formula.
\end{proof}

\subsubsection{The Local Results}{\label{subsec: thelocalresults}}
We now formulate the precise assumptions we will need for our local results.
\begin{assumption}{\label{assumption: generalparameters}}
Let $G$ satisfy Assumption~\ref{assumption: keygroupassumption},
and let $\phi$ be a semisimple
$L$-parameter induced from a toral
parameter $\phi_T$.

If $\phi_{T}$ is unramified then we assume $\phi$ is generic. If $\phi_{T}$ is ramified then we assume that $\phi_T$ is
generic regular. Alternatively, it suffices to assume
that 
\begin{enumerate}
\item $\phi_T$ is generic and split generic on the absolute coroots belonging to the unitary factors, tested over their splitting fields.
\item  the map $\Theta_\chi$ of
(\ref{eq:mixed-stabilizer}) is injective, where $\chi$
is the character attached to
$\phi_T$.
\end{enumerate}
Indeed, these conditions imply
generic regularity by Lemma~\ref{lem:unified-regularity}
\end{assumption}
We now apply Theorem~\ref{thm: generalperverseexact} to
obtain our main local results. 
\begin{corollary}{\label{cor: appliedperversetexactness}}
Assume $G$, $p$, and $\ell$ satisfy Assumption \ref{assumption: keygroupassumption}. Let $\phi$ be a generic $L$-parameter as in Assumption \ref{assumption: generalparameters}. We then have that the Hecke operator
\[ T_{\mu}: \DULA(\Bun_{G},\ol{\mathbb{F}}_{\ell})_{\phi} \ra \D^{\mathrm{ULA}}(\Bun_{G},\ol{\mathbb{F}}_{\ell})_{\phi} \]
is perverse t-exact. In particular, it is perverse $t$-exact for all minuscule $\mu$, and if we look at the natural functor 
\[ j_{1}^{*}T_{\mu}: \DULA(\Bun_{G},\ol{\mathbb{F}}_{\ell})_{\phi} \ra \D^{\mathrm{adm}}(G(\mathbb{Q}_{p}),\ol{\mathbb{F}}_{\ell})_{\phi} \]
then this is exact for the perverse $t$-structure on the source and the perverse (= standard $t$-structure) on the target.
\end{corollary}
\begin{proof}
The very decency of $\ell$ follows
from Lemma~\ref{lem:unified-decency}.

Write
\[
G=
\left(\prod_{j\in J}G_j\right)
\times_{\prod_{j\in J}\nu_j}
\mathbb G_m
\]
as in Assumption~\ref{assumption: keygroupassumption}, and put
\[
\widetilde G=\prod_{j\in J}G_j.
\]

The group $G$ is connected,
quasi-split and unramified.
Moreover, the natural inclusion
$G\hookrightarrow\widetilde G$
induces an isomorphism
\[
G^{\mathrm{der}}
\xrightarrow{\sim}
\widetilde G^{\mathrm{der}}.
\]

Assumption~\ref{compatibility} holds for
$\widetilde G$ by Theorem~\ref{thm:AssumptionLLC},
local class field theory for
the auxiliary torus factors,
and the obvious compatibility of Assumption \ref{compatibility} with finite
direct products.
Moreover, $Z(\widetilde G)=\prod_j Z(G_j)$ is an unramified
induced torus, and multiplication identifies
$(G\times Z(\widetilde G))/Z(G)$ with $\widetilde G$.
The conditions in Table~\ref{constrainttable} also ensure that
$\ell$ is very decent for $\widetilde G$. Thus this choice of
auxiliary group verifies Assumption~\ref{assumption: unramifiedness}.

Proposition~\ref{prop: compatibcentralisog} therefore
implies that Assumption~\ref{compatibility}
holds for $G$.

If $\phi_{T}$ is unramified the result now follows from Theorem~\ref{thm:unramified-perverse-without-regularity} and assumption~\ref{assumption: generalparameters}. If $\phi_{T}$ is ramified then  $\phi_T$ is generic
regular by Assumption~\ref{assumption: generalparameters}, and the result follows from Theorem~\ref{thm: generalperverseexact}.
\end{proof}
\begin{remark}{\label{rem: perversetexactnessintheSplitCase}}
We note that if the local group $G$ is split then the set $B(G,\mu)_{\mathrm{un}}$ is a singleton for any minuscule cocharacter $\mu$, by invoking \cite[Corollary~2.10]{Ham2}. We denote the unique element in this set as $b_{\mu}$. In this case, for a semisimple $L$-parameter $\phi$ induced from a generic $\phi_{T}$, by using Proposition \ref{prop: constituent proposition} and that Assumption \ref{compatibility} holds, we see that the adjunction map $j_{b_{\mu}!}j_{b_{\mu}}^{!} \rightarrow \mathrm{id}$ induces a natural isomorphism 
\[ j_{1}^{*}T_{\mu}j_{b_{\mu}!}j_{b_{\mu}}^{*} \simeq j_{1}^{*}T_{\mu} \]
of operators on the full subcategory $\DULA(\Bun_{G})_{\phi} \subset \D(\Bun_{G})$ (Here we have implicitly used the isomorphism $j_{1}^{*}T_{\mu} \simeq j_{1}^{*}T_{\mu}j_{\leq \mu!}j_{\leq \mu}^{*}$, by \cite[Proposition~A.9]{RpadicEtaleApenndix} , to rewrite $j_{b_{\mu}}^{!}$ as $j_{b_{\mu}}^{*}$ using Proposition \ref{prop: constituent proposition}, where $j_{\leq \mu}: \Bun_{G,\leq \mu} \hookrightarrow \Bun_{G}$ is the open immersion of the substack corresponding to $B(G,\mu) \subset B(G)$). We have that the $\sigma$-centralizer $J_{b_{\mu}}$ is isomorphic to the Levi subgroup $M_{b_{\mu}} \subset G$ and the LHS of this isomorphism can be explicitly computed (up to a shift) in terms of the parabolic induction from $J_{b_{\mu}} \simeq M_{b_{\mu}}$ to $G$, using the results of \cite{IG}. In this case, the second part of Corollary \ref{cor: appliedperversetexactness} is a simple consequence of the exactness of the parabolic induction functor. In particular, for any of cases appearing in Table (\ref{constrainttable}) where the group $G$ is split, the assertion about $j_1^*T_\mu$ in Corollary \ref{cor: appliedperversetexactness} also holds immediately for minuscule $\mu$ upon knowing that an object in $\DULA(\Bun_{G})_{\phi}$ is supported on just the unramified elements for $\phi$ induced from a  weakly generic toral parameter $\phi_{T}$, which follows from Assumption \ref{compatibility} and Proposition \ref{prop: constituent proposition}.

If $G$ is a product of $\GL_n$s and $\phi_T$ is unramified,
\cite[Lemma~3.1]{Ko} implies that
$\D(\Bun_G,\ol{\bb{F}}_\ell)_\phi$ is supported on the
unramified elements under the weaker hypothesis
\[
H^2(W_{\bb{Q}_p},\alpha_{\mathrm{Ad}}\circ\phi_T)=0
\]
for every coroot $\alpha$
(cf. Remark~\ref{rem: weaklyLanglandsShahidicase}).

For a general split group $G$ satisfying Assumption~\ref{compatibility}, Proposition~\ref{prop: constituent proposition} supplies this support statement for ULA objects. Thus the same argument gives the corresponding torsion-vanishing result under weak genericity whenever the other hypotheses of the global application hold (see \cite[Propositions~10.2.3,10.2.4]{MingjiaPolI} for details). However, under this weaker hypothesis, the splitting of the semi-orthogonal decomposition (see Corollary \ref{cor: appliedsplitofsemiorthog} below) need not hold in general (cf. Remark \ref{rem: weaklyLanglandsShahidicase}) and this becomes important when $B(G,\mu)_{\mathrm{un}}$ has more than one element.
\end{remark}
We also have the following.
\begin{corollary}{\label{cor: appliedsplitofsemiorthog}}
Assume $G$, $p$, and $\ell$ satisfy Assumption \ref{assumption: keygroupassumption}. Then, for a generic $\phi$ as in Assumption \ref{assumption: generalparameters}, we have that
\[ \DULA(\Bun_{G},\ol{\mathbb{F}}_{\ell})_{\phi} \simeq \prod_{b \in B(G)_{\mathrm{un}}} \Dadm(J_{b}(\mathbb{Q}_{p}),\ol{\mathbb{F}}_{\ell})_{\phi}. \] 
Moreover, the $!$ and $*$ pushforwards agree for any $A \in \Dadm(J_{b}(\mathbb{Q}_{p}),\ol{\mathbb{F}}_{\ell})_{\phi} \simeq \DULA(\Bun_{G}^{b},\ol{\mathbb{F}}_{\ell})_{\phi}$
\end{corollary}
\begin{proof}
This follows from Theorem~\ref{thm:unramified-perverse-without-regularity} and its proof when $\phi_T$ is unramified, and from Proposition~\ref{prop: genlocal} (2) when $\phi_T$ is ramified. The necessary assumptions were verified in the proof of Corollary~\ref{cor: appliedperversetexactness}.
\end{proof}

\section{The Proof of the Main Theorems}
We recall that the pair $(\mathbf{G},X)$ will denote a Shimura datum with reflex field $E/\bb{Q}$. We let $\ell \neq p$ be distinct prime numbers, and $E_{\mf{p}}$ the completion at the place $\mf{p}$ in $E$ dividing $p$ determined by the fixed isomorphism $j: \ol{\bb{Q}}_{p} \xrightarrow{\simeq} \bb{C}$ with completed algebraic closure denoted by $C$.  We will let $G := \mathbf{G}_{\bb{Q}_{p}}$ be the local group with maximal torus $T$, and we let $\mu$ be the conjugacy class of geometric dominant cocharacters of $G$ attached to the inverse of the Hodge cocharacter of $X$ and the isomorphism $j: \ol{\bb{Q}}_{p} \xrightarrow{\simeq} \bb{C}$. We assume that $G$ is unramified, and let $K_{p}^{\mathrm{hs}} \subset G(\bb{Q}_{p})$ denote a hyperspecial level. We let $H_{K_{p}^{\mathrm{hs}}}$ denote the spherical Hecke algebra with coefficients in $\ol{\bb{F}}_{\ell}$, and recall that $B(G)_{\mathrm{un}} := \mathrm{Im}(B(T) \ra B(G))$ denotes the set of unramified elements.

For $K \subset \mathbf{G}(\bb{A}_{f})$ a sufficiently small open compact, we write $\Sh(\mathbf{G},X)_{K}/\Spec(E)$ for the Shimura variety attached to the Shimura datum, with its associated adic space $\mathcal{S}(\mathbf{G},X)_{K}$ over $\Spa(E_{\mf{p}})$. and if $K = K^{p}K_{p}$, where $K_{p}$ (resp. $K^{p}$) denotes the level at $p$ (resp. away from $p$), we write $\mathcal{S}(\mathbf{G},X)_{K^{p}}$ for the associated Shimura variety at infinite level introduced in \S \ref{subsec: HodgeTatePeriodMap}.

In this section, we will reap the fruit of our work in the previous section and finally prove our main results. More precisely, in \S 5.1, we will establish our main results on the splitting of Mantovan's filtration (Theorem \ref{thm: appliedmantprodformbody}) for the torsion cohomology of a PEL type AC Shimura variety, as well as show the bounds for the torsion cohomology after generic localization (Theorem \ref{theorem: mainthmbody}). In \S 5.2, we push our results a bit further and deduce additional applications to abelian type Shimura varieties (Corollaries \ref{cor:abeliantypebody}, \ref{cor: CarTamcomp}). 

\subsection{Proof of Theorems 1.10 and 1.19}
We can now reap the fruit of our work in the previous sections and finally prove our main results. We start with Theorem \ref{thm: appliedmantprodform}.
\begin{theorem}{\label{thm: appliedmantprodformbody}}
Suppose $(\mathbf{G},X)$ is a PEL datum of type AC satisfying Assumption \ref{assump: codim} such that $\mathbf{G}_{\mathbb{Q}_{p}}$, $p$, and $\ell$ satisfy Assumption \ref{assumption: keygroupassumption}, and let $\phi$ be a semisimple parameter of $G$ satisfying Assumption \ref{assumption: generalparameters}. Then the complex $R\Gamma_{c}(\mathrm{Sh}(\mathbf{G},X)_{K^{p},C},\ol{\mathbb{F}}_{\ell})_{\phi}$ breaks up as a direct sum 
\[ \bigoplus_{b \in B(G,\mu)_{\mathrm{un}}} (R\Gamma_{c}(\Sht(G,b,\mu)_{\infty,C},\ol{\mathbb{F}}_{\ell}(d_{b})) \otimes^{\mathbb{L}}_{\mathcal{H}(J_{b})} R\Gamma_{c-\partial}(\Ig^{b},\ol{\mathbb{F}}_{\ell}))_{\phi}[2d_{b}] \]
of $G(\mathbb{Q}_{p})$-modules.
\end{theorem}
\begin{proof}
By Corollary \ref{cor: filtheckeops}, the complex $R\Gamma_{c}(\mathcal{S}(\mathbf{G},X)_{K^{p},C},\ol{\mathbb{F}}_{\ell})$ has a $G(\mathbb{Q}_{p})$-equivariant filtration with graded pieces isomorphic to $j_{1}^{*}T_{\mu}j_{b!}(V_{b})[-d](-\frac{d}{2})$. The cohomology of Igusa varieties $V_{b}$ and the global Shimura variety is perfect admissible as a complex (For the $V_{b}$, we can write $V_{b}$ as the fiber of the natural excision map $R\Gamma(\Ig^{b,*},\ol{\bb{F}}_{\ell}) \ra R\Gamma(\partial\Ig^{b,*},\ol{\bb{F}}_{\ell})$ as in  the proof of Corollary \ref{cor: stalkspartial} and then analogously to \cite[Proposition~8.21]{Zha} show that $R\Gamma(\Ig^{b,*},\ol{\bb{F}}_{\ell})$ and  $R\Gamma(\partial\Ig^{b,*},\ol{\bb{F}}_{\ell})$ are perfect admissible, since the condition of perfect admissibility is preserved under finite limits, the claim follows), and therefore it follows that $j_{b!}(V_{b}) \in \DULA(\Bun_{G},\ol{\bb{F}}_{\ell})$, and so we can apply the results of the previous section to them. We consider the localization
\[ (j_{1}^{*}T_{\mu}j_{b!}(V_{b}))_{\phi}[-d](-\frac{d}{2}). \]
This defines a filtration on $R\Gamma_{c}(\mathcal{S}(\mathbf{G},X)_{K^{p},C},\ol{\mathbb{F}}_{\ell})_{\phi}$. The filtration on $R\Gamma_{c}(\mathcal{S}(\mathbf{G},X)_{K^{p},C},\ol{\mathbb{F}}_{\ell})_{\phi}$ considered above comes from applying $(-)_{\phi}$ to $Rh^{\ra}_{*}(i_{b!}i_{b}^{*}(R\pi_{\mathrm{HT!}}(\ol{\mathbb{F}}_{\ell})))$ viewed as a $G(\mathbb{Q}_{p})$-representation. Using Corollary \ref{cor: appliedsplitofsemiorthog}, we see that these graded pieces are also isomorphic to
\[ (j_{1}^{*}T_{\mu}Rj_{b*}(V_{b}))_{\phi}[-d](-\frac{d}{2}) \]
via the natural transformation $j_{b!} \ra Rj_{b*}$ and are trivial for $b \notin B(G,\mu)_{\mathrm{un}}$. However, using Lemma \ref{lem: excisioncomparison}, this implies that the natural map
\[ Rh^{\ra}_{*}(i_{b!}i_{b}^{*}(R\pi_{\mathrm{HT!}}(\ol{\mathbb{F}}_{\ell})))_{\phi} \ra  Rh^{\ra}_{*}(Ri_{b*}i_{b}^{*}(R\pi_{\mathrm{HT!}}(\ol{\mathbb{F}}_{\ell})))_{\phi} \]
is an isomorphism (see Remark \ref{rem: excision}). For a stratum open in the remaining closed union, the two adjunction maps through the middle term of the excision triangle compose to this isomorphism after pushforward and localization. Hence the first map has a retraction, and induction gives us a direct sum decomposition
\[ R\Gamma_{c}(\mathcal{S}(\mathbf{G},X)_{K^{p},C},\ol{\mathbb{F}}_{\ell})_{\phi} \simeq \bigoplus_{b \in B(G,\mu)_{\mathrm{un}}} j_{1}^{*}T_{\mu}j_{b!}(V_{b})[-d](-\frac{d}{2})_{\phi},  \]
as desired. We may rewrite these graded pieces into the desired shape by arguing as in the proof of Corollary \ref{cor: filtheckeops}.
\end{proof} 
We now turn our attention to Theorem \ref{theorem: mainthm}.
\begin{theorem}{\label{theorem: mainthmbody}}
Suppose $(\mathbf{G},X)$ is a PEL datum of type AC satisfying Assumption \ref{assump: codim} such that $\mathbf{G}_{\mathbb{Q}_{p}}$, $p$ and $\ell$ satisfy Assumption \ref{assumption: keygroupassumption}. Let $\mf{m} \subset H_{K_{p}^{\mathrm{hs}}}$ be a generic maximal ideal. Then, for a level $K = K^{p}K_{p}^{\mathrm{hs}} \subset \mathbf{G}(\bb{A}_{f})$, the cohomology of  $R\Gamma(\mathrm{Sh}(\mathbf{G},X)_{K,\ol{E}},\ol{\mathbb{F}}_{\ell})_{\mf{m}}$ (resp. $R\Gamma_{c}(\mathrm{Sh}(\mathbf{G},X)_{K,\ol{E}},\ol{\mathbb{F}}_{\ell})_{\mf{m}}$) is concentrated in degrees $d \leq i \leq 2d$ (resp. $0 \leq i \leq d$).
\end{theorem}
\begin{proof}
We recall, by Proposition \ref{prop: artvanish}, that $V_{b}$ is a complex of smooth $J_{b}(\mathbb{Q}_{p})$-representations concentrated in degree $\leq d_{b}$. It follows that we have $\bigoplus_{b \in B(G,\mu)} j_{b!}(V_{b})_{\phi_{\mf{m}}} \in \pD^{\leq 0,\mathrm{ULA}}(\Bun_{G},\ol{\mathbb{F}}_{\ell})_{\phi_{\mf{m}}}$, using Proposition \ref{prop:localizeula}. Corollary \ref{cor: appliedperversetexactness} implies that 
\[ j_{1}^{*}T_{\mu}j_{b!}(V_{b})[-d]_{\phi_{\mf{m}}} \in \D^{\leq d}(G(\mathbb{Q}_{p}),\ol{\mathbb{F}}_{\ell})_{\phi_{\mf{m}}} \]
after forgetting the Weil group action. Therefore, we conclude that 
\[ R\Gamma_{c}(\mathcal{S}(\mathbf{G},X)_{K^{p},C},\ol{\mathbb{F}}_{\ell})_{\phi_{\mf{m}}} \]
is concentrated in degrees $0 \leq i \leq d$. By applying Poincar\'e duality at the pro-$p$ finite levels and Corollary \ref{cor: contragradients}, this allows us to conclude that the non-compactly supported cohomology
\[ R\Gamma(\mathcal{S}(\mathbf{G},X)_{K^{p},C},\ol{\mathbb{F}}_{\ell})_{\phi_{\mf{m}}^{\vee}} \]
localized at $\phi_{\mf{m}}^{\vee}$ (where the dual $\phi_{\mf{m}}^{\vee}$ is given by twisting $\phi_{\mf{m}}$ by the Chevalley involution of $\phantom{}^{L}G$ (see \cite[Section~2]{ContragradientVoganAdams})) is concentrated in degrees $d \leq i \leq 2d$, where we define this to be the colimit over the non-compactly supported cohomology of finite levels (cf. Remark \ref{rem: cohomologyatinfty}). Moreover, we note that Assumption \ref{assumption: generalparameters} is preserved under taking duals. The Chevalley involution preserves unramifiedness, and duality
interchanges positive and negative
coroots, preserving both genericity and
split genericity. The Chevalley involution
also preserves the property of having
trivial Weyl stabilizer, and hence
preserves regularity. It therefore follows that 
\[ R\Gamma(K_{p}^{\mathrm{hs}},R\Gamma(\mathcal{S}(\mathbf{G},X)_{K^{p},C},\ol{\mathbb{F}}_{\ell})_{\phi_{\mf{m}}^{\vee}}) \]
is also concentrated in degrees $\geq d$, but this is isomorphic to 
\[ R\Gamma(\mathcal{S}(\mathbf{G},X)_{K^{p}K_{p}^{\mathrm{hs}},C},\ol{\mathbb{F}}_{\ell})_{\mf{m}^{\vee}} \]
by Lemma \ref{lemma: localization map properties} (3). This establishes Theorem \ref{theorem: mainthm}, by applying Poincar\'e duality to the Shimura variety at finite level $K_{p}^{\mathrm{hs}}$. 
\end{proof}
We note that we didn't actually need to restrict to unramified parameters in the above discussion or hyperspecial level. In particular, we have the slightly more general claim. 
\begin{theorem}\label{thm:general-p-level}
Suppose that $(\mathbf{G},X)$ is a Shimura datum of PEL type AC
satisfying Assumption \ref{assump: codim}, and that $\mathbf{G}_{\mathbb Q_p}$, $p$, and $\ell$ satisfy Assumption \ref{assumption: keygroupassumption}.
Let $K=K^pK_p$ be a sufficiently small level, where
$K_p\subset G(\mathbb Q_p)$ is any compact open subgroup. For any $\phi$ satisfying Assumption \ref{assumption: generalparameters}, then $R\Gamma(\operatorname{Sh}(\mathbf{G},X)_{K^{p},\overline E},\ol{\bb{F}}_{\ell})_{\phi}$ (resp. $R\Gamma_c(
\operatorname{Sh}(\mathbf{G},X)_{K^{p},\overline E},\overline{\mathbb F}_\ell)_{\phi}$) is concentrated in degrees $d \leq i \leq 2d$ (resp. $0 \leq i \leq d$), where $R\Gamma(
\operatorname{Sh}(\mathbf{G},X)_{K^{p},\overline E},\overline{\mathbb F}_\ell)_{\phi}$ is defined to be the colimit of the cohomology of finite levels.  Moreover, let
\[
\mathfrak Z_G
:=\mathcal O_{\mathfrak{X}_{\hat G}}
  (\mathfrak{X}_{\hat G})
\]
denote the spectral Bernstein center where $\mathfrak{X}_{\hat{G}}$ is the stack of $L$-parameters, and write
\[
\theta_{K_p}:\mathfrak Z_G
       \longrightarrow Z(H_{K_p}),\qquad
H_{K_p}:=
\overline{\mathbb F}_\ell
[K_p\backslash G(\mathbb Q_p)/K_p],
\]
for the homomorphism induced by action of $\mathfrak{Z}_{G}$ on $\mathrm{cInd}_{K_p}^{G(\bb{Q}_{p})}(\ol{\mathbb{F}}_{\ell})$. Let $\mathfrak n\subset Z(H_{K_p})$ be a maximal ideal such that $\theta_{K_p}^{-1}(\mathfrak n)=\mathfrak m_\phi$,
where $\mathfrak m_\phi$ is the maximal ideal attached
to a semisimple $L$-parameter $\phi$ induced from a
toral parameter satisfying Assumption~\ref{assumption: generalparameters}. Then $R\Gamma(\operatorname{Sh}(\mathbf{G},X)_{K,\overline E},\ol{\bb{F}}_{\ell})_{\mf{n}}$ (resp. $R\Gamma_c(
  \operatorname{Sh}(\mathbf{G},X)_{K,\overline E},
  \overline{\mathbb F}_\ell)_{\mathfrak n}$) is concentrated in degrees $d \leq i \leq 2d$ (resp. $0 \leq i \leq d$). Moreover, if $K_{p} \subset G(\bb{Q}_{p})$ is parahoric and $\mf{n} \subset Z(H_{K_{p}})$ corresponds to a generic unramified $L$-parameter under the Bernstein isomorphism this claim holds.
\end{theorem}
\begin{proof}
Arguing exactly as in the proof of Theorem~\ref{theorem: mainthm}
the compactly supported cohomology of the infinite $p$-level tower,
after localization at $\phi$, is
concentrated in degrees $\leq d$. Apply Poincar\'e duality at finite level, this gives the same bound $\geq d$ for the regular cohomology.  Passing to derived $K_p$-invariants, which is left
$t$-exact, and using the more general form of Lemma \ref{lemma: localization map properties} (3) this gives the same lower bound at level $K_p$, 
further localization at $\mathfrak n$ is exact.
The compactly supported bound follows from
Poincar\'e duality and the fact that
Assumption \ref{assumption: generalparameters} is preserved under passage to contragredients. 
\end{proof}
\begin{remark}{\label{rem: MeaninglessLocalizations}}
We note the importance of verifying the compatibility between the
Fargues-Scholze spectral action and the usual Hecke action
in Lemma~\ref{lemma: localization map properties} (3) at parahoric level. Indeed, our vanishing results in Theorem \ref{thm:general-p-level} are first
established after localization at a maximal ideal
$\mathfrak m_\phi$ of the spectral Bernstein center.
In order to translate these results into statements about
localization at a maximal ideal
$\mathfrak n\subset Z(H_{K_p})$, we must identify
the corresponding spectral parameter and establish
\[
  R\Gamma(K_p,A)_{\mathfrak n}
  \simeq
  R\Gamma(K_p,A_\phi)_{\mathfrak n}.
\]
At parahoric level, Lemma~\ref{lemma: localization map properties} (3)
identifies $\phi$ with the unramified parameter attached to
$\mathfrak n$ under the Bernstein isomorphism, and the localization
at $\mathfrak n$ on the right can be omitted.
For a general compact open subgroup, without information about the Fargues-Scholze
correspondence (which we do at least have with rational coefficients for the groups in question since we have verified compatibility for these groups), the condition
\[
  \theta_{K_p}^{-1}(\mathfrak n)=\mathfrak m_\phi
\]
could be difficult to verify, and there would be no
a priori reason for the maximal ideals arising from
the usual Hecke action to correspond to parameters
satisfying our genericity assumptions. Consequently,
although Theorem~\ref{thm:general-p-level} would remain a formally valid
statement,  its formulation
in terms of maximal ideals of the center of the Hecke algebra could conceivably be vacuous without the additional justification.
\end{remark}
We also formulate the axiomatic version of the above results for future use.
\begin{theorem}{\label{thm: axiomatictheorem}}
Suppose that $(\mathbf{G},X)$ is a PEL type Shimura datum of type AC satisfying Assumption \ref{assump: codim} and unramified at $p$ and $K^{p}$ a sufficiently small away from $p$ level.  We assume that the local group $\mathbf{G}_{\bb{Q}_{p}}$ satisfies Assumption \ref{compatibility} and that $\ell$ is very decent with respect to $\mathbf{G}_{\bb{Q}_{p}}$ in the sense of Definition \ref{defn: verydecent}. Then, if $\phi$ is induced from a generic toral parameter $\phi_T$ which is either regular, or unramified with $G,\ell$ satisfying Assumption~\ref{assumption: unramifiedness}, the following is true. 
\begin{enumerate}
\item The complex $R\Gamma_{c}(\Sh(\mathbf{G},X),\ol{\bb{F}}_{\ell})_{K^{p},\phi}$ breaks up as a direct sum 
\[ \bigoplus_{b \in B(G,\mu)_{\mathrm{un}}} (R\Gamma_{c}(\Sht(G,b,\mu)_{\infty,C},\ol{\mathbb{F}}_{\ell}(d_{b})) \otimes^{\mathbb{L}}_{\mathcal{H}(J_{b})} R\Gamma_{c-\partial}(\Ig^{b},\ol{\mathbb{F}}_{\ell}))_{\phi}[2d_{b}] \]
of $G(\mathbb{Q}_{p})$-modules.
\item The complex $R\Gamma_{c}(\Sh(\mathbf{G},X)_{K^{p}},\ol{\bb{F}}_{\ell})_{\phi}$ (resp. $R\Gamma(\Sh(\mathbf{G},X)_{K^{p}},\ol{\bb{F}}_{\ell})_{\phi}$) is concentrated in degrees $\leq d$ (resp. $\geq d$).
\item Moreover, if $K_p \subset G(\mathbb Q_p)$ is any compact
open subgroup and $\mathfrak n \subset Z(H_{K_p})$ is a maximal
ideal such that
$\theta_{K_p}^{-1}(\mathfrak n)=\mathfrak m_\phi$, then,
writing $K=K^pK_p$, the localization $R\Gamma(\Sh(\mathbf G,X)_{K,\overline E},
\overline{\mathbb F}_\ell)_{\mathfrak n}$ (resp. $R\Gamma_c(\operatorname{Sh}(\mathbf G,X)_{K,\overline E},
\overline{\mathbb F}_\ell)_{\mathfrak n})$) is concentrated in degrees $d\leq i\leq 2d$ (resp. $0\leq i\leq d$).
In particular, this is true for any parahoric subgroup $K_p$
and any maximal ideal $\mathfrak n\subset Z(H_{K_p})$
whose associated unramified semisimple parameter
$\phi_{\mathfrak n}$, under the normalized Bernstein
isomorphism, is induced from a generic toral parameter which
is regular or for which $G,\ell$ satisfy
Assumption~\ref{assumption: unramifiedness}.
\item If $\mathbf G_{\mathbb Q_p}$ is split, the conclusions
of (1)--(3) also hold when the relevant toral parameter is
merely generic, without assuming regularity.
\end{enumerate}
\end{theorem}
\begin{proof}
The generic regular case follows as above by invoking Theorem~\ref{thm: generalperverseexact} and Proposition~\ref{prop: genlocal}. In the unramified generic case, use Theorem~\ref{thm:unramified-perverse-without-regularity} instead. Item (4) follows similarly using Remark~\ref{rem: perversetexactnessintheSplitCase}.
\end{proof}

We now turn our attention to the abelian type case.

\subsection{Proof of Corollary \ref{cor:abeliantype}}{\label{s: abeliantype}}
We would like to obtain the main Theorem for some Shimura varieties of non-PEL type, especially Hilbert-modular and Hilbert-Siegel modular varieties. We will show this in a more general setup, as follows. For the rest of this section, $K_p=K_p^{\mathrm{hs}}$ is a hyperspecial subgroup of $G(\mathbb{Q}_p)$, and similarly $K_{1p}$ is a hyperspecial subgroup of $G_1(\mathbb{Q}_p)$.

Let $(\mathbf{G},X)$ and $(\mathbf{G}_1,X_1)$ be Shimura
data equipped with a central isogeny
of derived groups
\[
\mathbf{G}^{\mathrm{der}}\longrightarrow
\mathbf{G}_1^{\mathrm{der}}
\]
inducing an isomorphism of adjoint
Shimura data
\[
(\mathbf{G}^{\mathrm{ad}},X^{\mathrm{ad}})
\simeq
(\mathbf{G}_1^{\mathrm{ad}},X_1^{\mathrm{ad}}).
\]
We further assume that both groups $G,G_1$ are unramified. Observe that from \cite[Lemma 3.4.13(3)]{Kis}, for any abelian type Shimura datum $(\mathbf{G}_1,X_1)$, we can always arrange a Hodge type Shimura datum $(\mathbf{G},X)$ which satisfies this condition. Consider the associated Shimura varieties $\Sh(\mathbf{G},X)_{K}$ and $\Sh(\mathbf{G}_1,X_1)_{K_1}$, where we will choose levels $K,K_1$ such that the levels at $p$ are compatible hyperspecial subgroups. More precisely, let $\mathbf{G}^\mathrm{sc}$ be the simply connected cover of $\mathbf{G}^\der$, and $G^\mathrm{sc}=\mathbf{G}^\mathrm{sc}_{\bb{Q}_p}$. Observe that we have a natural map of Bruhat-Tits buildings 
\[\mathcal{B}(G^\mathrm{sc},\breve{\mathbb{Q}}_p)\ra \mathcal{B}(G_1,\breve{\mathbb{Q}}_p)\]
which sends a hyperspecial vertex to a hyperspecial vertex. In particular, we first choose a hyperspecial vertex for $G^\mathrm{sc}$ corresponding to a hyperspecial subgroup $K^{\mathrm{sc}}_p$ of $G^{\mathrm{sc}}(\mathbb{Q}_p)$, and we let $K_{1p}$ be the corresponding hyperspecial subgroup in $G_1(\mathbb{Q}_p)$. Similarly, since we observe that $\mathbf{G}^{\mathrm{sc}}$ is also the simply connected cover of $\mathbf{G}_1^\der$ as well, we obtain a hyperspecial subgroup $K_{p}$ in $G(\mathbb{Q}_p)$.

Now, let $H_{K_p^{\mathrm{sc}}}$ denote the Hecke algebra $C_c^\infty(K_p^{\mathrm{sc}}\backslash G^\mathrm{sc}(\mathbb{Q}_p)/K_p^{\mathrm{sc}},\ol{\bb{F}}_\ell)$.  Observe that we have a natural algebra homomorphism $H_{K_p^{\mathrm{sc}}}\rightarrow H_{K_p}$, obtained by sending the indicator function $1_{K_p^{\mathrm{sc}}gK_p^{\mathrm{sc}}}$ to $1_{K_pg'K_p}$, where $g'$ is the image of $g$ in $G(\mathbb{Q}_p)$. Alternatively, if we denote by $T^{\mathrm{sc}}$ the maximal torus of $G^{\mathrm{sc}}$ given by taking preimage of $T$, then via the Satake isomorphism, this map of Hecke algebras is induced by the natural map of cocharacters $X_*(T^{\mathrm{sc}})\rightarrow X_*(T)$. Moreover, given a maximal ideal $\mathfrak{m}\subset H_{K_p}$, there exists a corresponding maximal ideal $\mathfrak{m}'$ of $H_{K^{\mathrm{sc}}_p}$ such that $\mathfrak{m}$ contains the image of $\mathfrak{m}'$.

At hyperspecial level, genericity transfers through the common
simply connected derived group as follows.
\begin{lemma}\label{lem:genericity-derived-isogeny}
Let $\mathfrak m_1\subset H_{K_{1,p}}$ be a generic maximal
ideal, and let
$\mathfrak m'=\mathfrak m_1\cap H_{K_p^{\mathrm{sc}}}$.
Then every maximal ideal $\mathfrak m\subset H_{K_p}$
restricting to $\mathfrak m'$ is generic.
\end{lemma}
\begin{proof}
The central isogenies of derived groups identify the absolute
coroots and their Galois actions. Under the induced maps of
dual tori, the representations attached to the corresponding
Galois orbits of coroots agree. Thus the complexes occurring
in Definition~\ref{def: generic} agree for the associated
toral parameters. Genericity therefore passes from
$\mathfrak m_1$ to $\mathfrak m'$ and from $\mathfrak m'$
to every lift $\mathfrak m$.
\end{proof}
We will also use the following commutative algebra lemma.
\begin{lemma}{\label{lemma: basiccommutativealgebra}}
Let $R\to S$ be a homomorphism of commutative
$\overline{\mathbb F}_\ell$-algebras, and let $M$
be an $S$-module finite-dimensional over
$\overline{\mathbb F}_\ell$. For every maximal ideal $\mathfrak{m}' \subset R$
there is a natural isomorphism
\[
  M_{\mathfrak{m}'}\simeq
  \bigoplus_{\substack{
    \mathfrak m\in\operatorname{Supp}_S(M)\\
    \mathfrak m\cap R=\mathfrak \mf{m}'
  }}
  M_{\mathfrak m}.
\]
\end{lemma}
\begin{proof}
The image of $S$ in
$\operatorname{End}_{\overline{\mathbb F}_\ell}(M)$
is a finite-dimensional commutative
$\overline{\mathbb F}_\ell$-algebra and hence
Artinian.

Its decomposition into local factors yields
\[
  M\simeq\bigoplus_{\mathfrak m}M_{\mathfrak m}.
\]

On $M_{\mathfrak m}$, the elements outside
$\mathfrak m\cap R$ act invertibly, whereas
$\mathfrak m\cap R$ acts nilpotently.
Localization at $\mathfrak{m}'$ therefore gives exactly the summands
$\mathfrak{m} \cap R=\mathfrak{m}'$.

\end{proof}

This passage to the simply connected cover allows us to pass to connected components of Shimura varieties. To see this, fix a connected component $X^+\subset X$. This also fixes a $X^+_1\subset X_1$, and an isomorphism $X^+\simeq X^+_1$, since $\mathbf{G},\mathbf{G}_1$ have isomorphic adjoint quotients. For any compact open subgroup $K\subset \mathbf{G}(\mathbb{A}_f)$, we let $\Sh^+(\mathbf{G},X)_K$ be the geometrically connected component which is the image of $X^+\times 1$. Moreover, we will let 
\begin{equation*}
    \Sh^+(\mathbf{G},X)_{K_p}=\varprojlim_{K^p}\Sh^+(\mathbf{G},X)_{K_pK^p},\qquad \Sh^+(\mathbf{G},X)=\varprojlim_{K}\Sh^+(\mathbf{G},X)_{K}.
\end{equation*}
Note that since all the transition morphisms are finite \'{e}tale and hence affine, $\Sh^+(\mathbf{G},X)_{K_p}$ is also a qcqs scheme by \cite[Lemma 01YX]{stacks-project}.

Now, for any element of $H_{K^{\mathrm{sc}}_p}$, we consider the
associated algebraic correspondence on $\Sh(\mathbf{G},X)_{K_p}$
via the algebra homomorphism
$H_{K^{\mathrm{sc}}_p}\rightarrow H_{K_p}$.
We first observe that the natural action on the geometric connected components of
$\Sh(\mathbf{G},X)_{K_p}$ is trivial. Consequently, the correspondence restricts by base change to each geometric connected component, and in particular to
$\Sh^+(\mathbf{G},X)_{K_p}$. At any finite prime-to-$p$ level $K^p$, the resulting correspondence on $\Sh(\mathbf{G},X)_{K_pK^p}$ is a finite
disjoint union of correspondences on its geometric connected
components. To see this, it suffices to analyze the action of
$G^{\mathrm{sc}}(\mathbb{Q}_p)$ on the set of geometric
connected components of
\begin{equation*}
    \Sh(\mathbf{G},X):=\varprojlim_{K}\Sh(\mathbf{G},X)_{K}.
\end{equation*}
Let $p_{K_p}$ denote the projection $\Sh(\mathbf{G},X)\rightarrow \Sh(\mathbf{G},X)_{K_p}$, and similarly $p_{K_{1p}}:\Sh(\mathbf{G}_1,X_1)\rightarrow \Sh(\mathbf{G}_1,X_1)_{K_{1p}}$. We first recall from \cite[\S2.1.3]{Del} that the connected components are given by the profinite group
\[\pi_0(\Sh(\mathbf{G},X))=\mathbf{G}(\bb{Q})^-_{+}\backslash \mathbf{G}(\bb{A}_f),\]
where as in the notation of \cite{Del} we denote by $\mathbf{G}(\mathbb{R})_+$ the inverse image of the identity component $\mathbf{G}^{\mathrm{ad}}(\mathbb{R})^+$, $\mathbf{G}(\mathbb{Q})_+=\mathbf{G}(\mathbb{R})_+\cap \mathbf{G}(\mathbb{Q})$ and $\mathbf{G}(\bb{Q})^-_{+}$ is the closure of $\mathbf{G}(\bb{Q})_+$ in $\mathbf{G}(\bb{A}_f)$. The strong approximation theorem (see \cite[\S2.5.1]{Del}) implies that $\mathbf{G}(\bb{Q})^-_{+}$ contains the image of $\mathbf{G}^{\mathrm{sc}}(\mathbb{A}_f)$. In particular, this implies that $\mathbf{G}^{\mathrm{sc}}(\mathbb{Q}_p)$ acts trivially on the set of connected components, so the Hecke correspondence associated with any $1_{K_p^{\mathrm{sc}}gK_p^{\mathrm{sc}}}$ in $H_{K^{\mathrm{sc}}_p}$ will send $\Sh(\mathbf{G},X)^+_{K_p}$ back to itself.

Now, using the notation of \cite[\S3.3]{Kis}, we see that there exist groups $\mathscr{A}(\mathbf{G}_{\mathbb{Z}_{(p)}})^\circ,\mathscr{A}(\mathbf{G}_{1,{\mathbb{Z}_{(p)}}}),\mathscr{A}(\mathbf{G}_{\mathbb{Z}_{(p)}})$, whose definition we do not recall here, such that from \cite[Proposition 3.3.10]{Kis}, we know that we have an isomorphism
\begin{equation}
\label{eqn:abeliantype}
    \Sh(\mathbf{G}_1,X_1)_{K_{1p}}\simeq (\Sh^+(\mathbf{G},X)_{K_{p}}\times\mathscr{A}(\mathbf{G}_{1,\mathbb{Z}_{(p)}}))/\mathscr{A}(\mathbf{G}_{\mathbb{Z}_{(p)}})^\circ,
\end{equation}
where we see from \cite[\S3.3.9]{Kis} that there is a map $\mathscr{A}(\mathbf{G}_{\mathbb{Z}_{(p)}})^\circ\simeq \mathscr{A}(\mathbf{G}^\der_{\mathbb{Z}_{(p)}})^\circ\rightarrow \mathscr{A}(\mathbf{G}^\der_{1,\mathbb{Z}_{(p)}})^\circ\hookrightarrow \mathscr{A}(\mathbf{G}_{1,\mathbb{Z}_{(p)}})$.
Define $\Delta=\ker(\mathscr{A}(\mathbf{G}_{\mathbb{Z}_{(p)}})^\circ\rightarrow \mathscr{A}(\mathbf{G}_{1,\mathbb{Z}_{(p)}}))$, and consider the quotient $\Sh^+(\mathbf{G},X)_{K_{p}}/\Delta$. We know that this is isomorphic to the connected component $\Sh^+(\mathbf{G}_1,X_1)_{K_{1p}}$, and the profinite set $\mathscr{A}(\mathbf{G}_{1,\mathbb{Z}_{(p)}})/\mathscr{A}(\mathbf{G}_{\mathbb{Z}_{(p)}})^\circ$ is the space of connected components of $\Sh(\mathbf{G}_1,X_1)_{K_{1p}}$.

We now observe that the algebraic correspondence on $\Sh^+(\mathbf{G}_1,X_1)_{K_{1p}}$ associated with any element of $H_{K^{\mathrm{sc}}_p}$ can be recovered from that of $\Sh^+(\mathbf{G},X)_{K_p}$. To see this, let us first describe the correspondence on $\Sh^+(\mathbf{G},X)_{K_p}$ as a one-to-many map. Consider any point $x\in\Sh^+(\mathbf{G},X)_{K_p}$, and let $\tilde{x}\in \Sh^+(\mathbf{G},X)$ be a preimage of $x$ under $p_{K_p}$. Decompose
\[K_p^{\mathrm{sc}}g'K_p^{\mathrm{sc}}=\bigsqcup_{i}h'_iK_p^{\mathrm{sc}}.\]
The Hecke correspondence associated with $K_p^{\mathrm{sc}}g'K_p^{\mathrm{sc}}$ sends $x$ to the set (with multiplicity) $\{p_{K_p}(\tilde{x}\cdot h_i)\}_i$, where $h_i$ is the image of $h_i'$ in $G(\mathbb{Q}_p)$. Let $q:\Sh^+(\mathbf{G},X)_{K_p}\rightarrow \Sh^+(\mathbf{G}_1,X_1)_{K_{1p}}$ denote the quotient projection, set $y=q(x)$, and let $\tilde{y}$ be the image of $\tilde{x}$ under the similar projection map $\Sh^+(\mathbf{G},X)\rightarrow \Sh^+(\mathbf{G}_1,X_1)$. Then the same consideration as above implies the Hecke correspondence associated with $K_p^{\mathrm{sc}}g'K_p^{\mathrm{sc}}$ on $\Sh^+(\mathbf{G}_1,X_1)_{K_{1p}}$ sends $y$ to the set $\{p_{K_{1p}}(\tilde{y}\cdot h_{1i})\}_i$, where $h_{1i}$ is the image of $h_i'$ in $G_1(\mathbb{Q}_p)$. Now, we see that the projection map of connected components $\Sh^+(\mathbf{G},X)\rightarrow \Sh^+(\mathbf{G}_1,X_1)$ is equivariant for the action of $G^{\mathrm{sc}}(\mathbb{Q}_p)$, so we have $p_{K_{1p}}(\tilde{y}\cdot h_{1i})=q(p_{K_p}(\tilde{x}\cdot h_i))$.

It remains to observe that from the proof of \cite[Lemma 3.3.3]{Kis}, we can choose representatives of $\mathscr{A}(\mathbf{G}_{1,\mathbb{Z}_{(p)}})/\mathscr{A}(\mathbf{G}_{\mathbb{Z}_{(p)}})^\circ$ to be of the form $(g^p,1)$, where $g^p\in \mathbf{G}_{1}(\mathbb{A}_f^p)$, so this commutes with the Hecke action at $p$. Thus the Hecke correspondence on other connected components is simply given by translation by $g^p$ of that on $\Sh^+(\mathbf{G}_1,X_1)_{K_{1p}}$.
 
\begin{lemma}{\label{lemma: HochschildSerre}}
There is an $ H_{K_p^{\mathrm{sc}}}$-equivariant
isomorphism
\[
R\Gamma(
  \operatorname{Sh}^{+}(\mathbf{G}_1,X_1)_{K_{1p}},
  \ol{\bb{F}}_{\ell})
\simeq
R\Gamma_{\mathrm{cont}}(
  \Delta,
  R\Gamma(
    \operatorname{Sh}^{+}(\mathbf{G},X)_{K_p},
    \ol{\bb{F}}_{\ell})),
\]
where $\Delta$ is endowed with its profinite topology.
\end{lemma}

\begin{proof}
The group $\Delta$ is profinite; see
\cite[Appendix~E.5--E.7]{Kis17}.
The quotient description above identifies
\[
\operatorname{Sh}^{+}(\mathbf{G},X)_{K_p}
\longrightarrow
\operatorname{Sh}^{+}(\mathbf{G}_1,X_1)_{K_{1p}}
\]
with the pro-\'etale $\Delta$-cover obtained from taking the inverse limit of the
connected Shimura varieties at finite prime-to-$p$ level.
The isomorphism therefore follows from
Hochschild--Serre. Its
$H_{K_p^{\mathrm{sc}}}$-equivariance follows from
the compatibility of the Hecke correspondences established
above.
\end{proof}
By localizing Lemma~\ref{lemma: HochschildSerre} (using that $\Delta$ is profinite to see this commutes with cohomology), we get that if $H^{i}(\Sh^{+}(\mathbf{G},X)_{K_{p}},\ol{\bb{F}}_{\ell})_{\mathfrak m'}=0$ for $i<d$, then $H^i(\Sh^{+}(\mathbf{G}_{1},X_{1})_{K_{1,p}},\ol{\bb{F}}_{\ell})_{\mathfrak m'}=0$ for $i < d$.

Now fix a sufficiently small compact open subgroup
$K_1^p\subset \mathbf{G}_1(\mathbb A_f^p)$. Let
\[
f:\mathscr{A}(\mathbf{G}_{\mathbb Z_{(p)}})^\circ
\longrightarrow
\mathscr{A}(\mathbf{G}_{1,\mathbb Z_{(p)}})
\]
be the map introduced above, with kernel $\Delta$, and put
\[
\mathcal K_1^p
=
\operatorname{im}\!\left(
K_1^p\longrightarrow
\mathscr{A}(\mathbf{G}_{1,\mathbb Z_{(p)}})
\right),
\qquad
P_{K_1^p}
=
\mathcal K_1^p\cap\operatorname{im}(f).
\]
Equivalently,
\[
P_{K_1^p}\simeq f^{-1}(\mathcal K_1^p)/\Delta.
\]
The description of connected components following
equation~(\ref{eqn:abeliantype}) shows that $P_{K_1^p}$ is precisely the
subgroup of $\mathcal K_1^p$ preserving $\Sh^{+}(\mathbf{G}_{1},X_{1})_{K_{1,p}}$.
It is a closed subgroup of $\mathcal K_1^p$, hence
profinite. For sufficiently small $K_1^p$, the projection
\[
\Sh^{+}(\mathbf{G}_{1},X_{1})_{K_{1,p}} \rightarrow
\operatorname{Sh}^{+}(\mathbf{G}_1,X_1)_{K_{1p}K_1^p}
\]
is a pro-\'etale $P_{K_1^p}$-torsor, as follows by
passing to the inverse limit of the connected finite
\'etale Galois covers at finite prime-to-$p$ sublevels.

A second application of Hochschild-Serre therefore gives
\[
H^i(
\operatorname{Sh}^{+}(\mathbf{G}_1,X_1)_{K_{1p}K_1^p},\ol{\bb{F}}_{\ell})_{\mathfrak m'}=0.
\]
for $i < d$.

The other connected components at finite level are
obtained from distinguished connected components,
possibly at conjugate prime-to-$p$ levels, by
prime-to-$p$ translations. These translations commute
with the $H_{K_p^{\mathrm{sc}}}$-action.
Taking the finite direct sum over connected components
and then the colimit over prime-to-$p$ levels yields
\[
H^i\!\left(
\operatorname{Sh}(\mathbf{G}_1,X_1)_{K_{1p}},\ol{\bb{F}}_{\ell}
\right)_{\mathfrak m'}=0
\]
for $i < d$.
This proves the following.
\begin{proposition}
\label{prop:connectedcomponentShimura}
    Suppose that we have two Shimura data $(\mathbf{G},X)$, $(\mathbf{G}_1,X_1)$ such that $(\mathbf{G}^\mathrm{ad},X^\mathrm{ad})\simeq (\mathbf{G}^\mathrm{ad}_1,X^\mathrm{ad}_1)$ and that we have a central isogeny $\mathbf{G}^\mathrm{der}\rightarrow \mathbf{G}_1^\mathrm{der}$, and $G,G_1$ are both unramified. Then the action of $H_{K^{\mathrm{sc}}_p}$ stabilizes each connected component, and, for any maximal ideal $\mf{m}'\subset H_{K^{\mathrm{sc}}_p}$, if $R\Gamma(\Sh^+(\mathbf{G},X)_{K_{p}},\ol{\mathbb{F}}_{\ell})_\mathfrak{m'}$ is concentrated in degrees $[d,2d]$, then so is $R\Gamma(\Sh(\mathbf{G}_1,X_1)_{K_{1p}},\ol{\mathbb{F}}_{\ell})_\mathfrak{m'}$.
\end{proposition}
We now establish the following claim.
\begin{proposition}{\label{prop: abeliantypecase}}
Suppose that $(\mathbf{G},X)$ is a PEL-type
Shimura datum of type AC satisfying
Assumption~\ref{assump: codim} and unramified at $p$,
such that $\mathbf{G}_{\mathbb Q_p}$ and $\ell$ satisfy
Assumption~\ref{assumption: unramifiedness}. We consider another Shimura datum $(\mathbf{G}_1,X_1)$ such that we have a central isogeny $\mathbf{G}^\mathrm{der}\rightarrow \mathbf{G}_1^\mathrm{der}$, inducing an isomorphism $(\mathbf{G}^\mathrm{ad},X^\mathrm{ad})\simeq (\mathbf{G}^\mathrm{ad}_1,X^\mathrm{ad}_1)$. If
$\mathfrak{m}_{1} \subset H_{K_{1,p}}$
is generic, then $R\Gamma(\Sh(\mathbf{G}_1,X_1)_{K_{1p}},\ol{\mathbb{F}}_{\ell})_{\mathfrak{m}_1}$ is concentrated in degrees $[d,2d]$. 
\end{proposition}
\begin{proof}
Put $\mathfrak{m}'
  =\mathfrak m_1\cap
    H_{K_p^{\mathrm{sc}}}$;
by Lemma~\ref{lem:genericity-derived-isogeny}, every spherical
maximal ideal $\mathfrak m$ restricting to $\mathfrak m'$ is generic.

Fix a sufficiently small prime-to-$p$ level
$K^p$ and $i<d$. The vector space
\[
  M^i=
  H^i\bigl(
    \operatorname{Sh}(\mathbf{G},X)_{
      K^pK_p,\overline E},
    \overline{\mathbb F}_\ell
  \bigr)
\]
is finite-dimensional.

For every maximal ideal $\mathfrak m$ lying above
$\mathfrak m'$, the associated toral parameter is unramified
and generic. Theorem~\ref{thm: axiomatictheorem}(3) therefore gives
\[
(M^i)_{\mathfrak m}=0.
\]

Lemma~\ref{lemma: basiccommutativealgebra} therefore implies
\[
  (M^i)_{\mathfrak{m}'}=0.
\]
Now, observe that $\Sh(\mathbf{G},X)_{K_pK^p}$ is a finite
disjoint union of geometric connected components.
As explained above, the correspondences associated with
$H_{K_p^{\mathrm{sc}}}$ preserve each of these components.
Thus,
$H^i(\Sh(\mathbf{G},X)_{K_pK^p},
\ol{\mathbb{F}}_\ell)_{\mathfrak{m}'}$
decomposes as a finite direct sum.
\[H^i(\Sh(\mathbf{G},X)_{K_pK^p},\ol{\mathbb{F}}_{\ell})_{\mathfrak{m}'}=\bigoplus_{\omega\in \pi_0(\Sh(\mathbf{G},X)_{K_pK^p})}H^i(\Sh(\mathbf{G},X)^\omega_{K_pK^p},\ol{\mathbb{F}}_{\ell})_{\mathfrak{m}'}, \]
so we have that $H^i(\Sh^+(\mathbf{G},X)_{K_pK^p},\ol{\mathbb{F}}_{\ell})_{\mathfrak{m}'}$ vanishes for $i<d$ as well. 

Taking colimits over $K^p$, we see that $H^i(\Sh^+(\mathbf{G},X)_{K_p},\ol{\mathbb{F}}_{\ell})_{\mathfrak{m}'}$ vanishes for $i<d$ as well. Applying Proposition \ref{prop:connectedcomponentShimura},  $H^i(\Sh(\mathbf{G}_1,X_1)_{K_{1p}},\ol{\mathbb{F}}_{\ell})_\mathfrak{m'}$ vanishes for $i < d$.

Now, let $\mathbf{Z}\subset\mathbf{G}_1$ be the center, and we denote by $\mathbf{Z}(\mathbb{Q})^-$ the closure of $\mathbf{Z}(\mathbb{Q})$ in $\mathbf{Z}(\mathbb{A}_f)$. Denote by $\ol{K_1^p}$ the quotient
\[
\ol{K_1^p}=(K_{1p}K^{p}_1/K_{1p}K^{p}_1\cap\mathbf{Z}(\mathbb{Q})^-)/(K_{1p}/K_{1p}\cap\mathbf{Z}(\mathbb{Q})^-).
\]

By the Hochschild-Serre spectral sequence (applied to the quotients $\ol{K_1^p}$), we see that for all sufficiently small $K_1^p$, $H^i(\Sh(\mathbf{G}_1,X_1)_{K_{1p}K_1^p},\ol{\mathbb{F}}_{\ell})_{\mathfrak{m}'}$ vanishes for $i<d$.

As above, we apply Lemma~\ref{lemma: basiccommutativealgebra} to the modules
\[
H^i\bigl(
  \operatorname{Sh}(\mathbf{G}_{1},X_1)_{
    K_1^pK_{1,p},\overline E},
  \overline{\mathbb F}_\ell
\bigr)
\]
and the map
\[
  H_{K_p^{\mathrm{sc}}}
  \longrightarrow
  H_{K_{1,p}},
\]
This shows that the localization of $H^i(\operatorname{Sh}(\mathbf{G}_{1},X_1)_{K_1^pK_{1,p},\overline E},\overline{\mathbb F}_\ell)$ at $\mathfrak m_1$ vanishes for every $i<d$.
\end{proof}
We obtain the following by combining this Proposition with Theorem \ref{theorem: mainthm} and Theorem \ref{thm: axiomatictheorem}. 
\begin{corollary}{\label{cor:abeliantypebody}}
Suppose that $(\mathbf{G}_1,X_1)$ is an
abelian-type Shimura datum admitting
a PEL-type covering datum $(\mathbf{G},X)$
of type AC satisfying Assumption \ref{assump: codim}. Assume that the two groups are unramified at $p$, that
$\mathbf{G}_{\mathbb Q_p}$ and $\ell$ satisfy
Assumption~\ref{assumption: unramifiedness}. Let $\mathfrak m_1$
be a generic maximal ideal of the spherical Hecke algebra
$H_{K_{1,p}^{\mathrm{hs}}}$. Then, for a level $K_1 = K_1^{p}K_{1p}^{\mathrm{hs}} \subset \mathbf{G}_{1}(\bb{A}_{f})$, the cohomology of  $R\Gamma(\mathrm{Sh}(\mathbf{G}_1,X_1)_{K_1,\ol{E}},\ol{\mathbb{F}}_{\ell})_{\mf{m}_{1}}$ (resp. $R\Gamma_{c}(\mathrm{Sh}(\mathbf{G}_1,X_1)_{K_1,\ol{E}},\ol{\mathbb{F}}_{\ell})_{\mf{m}_{1}}$) is concentrated in degrees $d \leq i \leq 2d$ (resp. $0 \leq i \leq d$).     
\end{corollary}
\begin{proof}
By Proposition \ref{prop: abeliantypecase}, we obtain the required lower
bound for ordinary cohomology. Genericity is preserved by the
Chevalley involution, which permutes the coroot representations.
Applying the same bound to the dual maximal ideal and then
Poincar\'e duality gives the compactly supported bound.
\end{proof}

In particular, we can strengthen previous results of Caraiani-Tamiozzo \cite[Theorem B]{CT}, who previously showed torsion vanishing for Hilbert modular varieties under the additional assumption that $p$ was split in the totally real field $F$ (though we also note that they showed torsion vanishing under a hypothesis on $\mathfrak{m}$ which is weaker than the genericity considered here, see Remark \ref{rmk:weakgeneric}). Here we recall certain similitude groups $G(\SL_{2,\mathbf{L}}),G(\mathrm{Sp}_{4,\mathbf{L}}),G(\U_{2,\mathbf{L}})$ defined in the discussion before Lemma \ref{lem:unified-regularity}. 
\begin{corollary}{\label{cor: CarTamcomp}}
Let $p,\ell$ be distinct primes. Suppose that $\mathrm{Sh}(\mathbf{G}_1,X_1)_{K_{1p}^\mathrm{hs}K_1^p}$ is such that $K_{1p}^{\mathrm{hs}}$ is hyperspecial. Let $\mathfrak{m}_1$ be a maximal ideal of the spherical Hecke algebra whose associated toral parameter is generic, and suppose that we are in one of the following cases:
\begin{enumerate}
\item $(\mathbf{G}_1,X_1)$ defines a Hilbert modular variety or Hilbert--Siegel modular variety of genus $2$ associated with a totally real field $F$ with $[F:\mathbb{Q}]>1$, where we moreover assume that $(\ell,2\prod_i d_i)=1$, where we set
\[
F\otimes_{\mathbb Q}\mathbb Q_p
=\prod_{i=1}^r L_i,
\qquad d_i=[L_i:\mathbb Q_p].
\]
In the Hilbert-Siegel case, assume further
that $p\ne 2$ if some $d_i>1$. 
\item $(\mathbf{G}_1,X_1)$ defines a quaternionic Shimura variety associated with a quaternion algebra over a totally real field $F$ which is non-split at at least one real place, and $[F:\mathbb{Q}]>1$. Assume moreover that $(\ell,2\prod_i d_i)=1$, where we set
\[
F\otimes_{\mathbb Q}\mathbb Q_p
=\prod_{i=1}^r L_i,
\qquad d_i=[L_i:\mathbb Q_p].
\]
\end{enumerate}

Then $R\Gamma(
\mathrm{Sh}(\mathbf{G}_1,X_1)_{K_1,\overline E},
\overline{\mathbb F}_\ell)_{\mathfrak m_1}$ (resp. $R\Gamma_{c}(
\mathrm{Sh}(\mathbf{G}_1,X_1)_{K_1,\overline E},
\overline{\mathbb F}_\ell)_{\mathfrak m_1}$) is concentrated in degrees $\geq d$ (resp. $\leq d$), where $d=\dim\mathrm{Sh}(\mathbf{G}_1,X_1)_{K_1}$.
\end{corollary}

\begin{proof}
In the Hilbert modular and
Hilbert--Siegel cases, we can construct a PEL covering Shimura variety of type C by taking $B$ to be $F$, and $V$ to be $F^2$ and $F^4$ respectively. The local
PEL covering groups are thus respectively
$G(\mathrm{SL}_{2,\mathbf L})$
and $G(\mathrm{Sp}_{4,\mathbf L})$. We also observe that the covering PEL Shimura varieties satisfy Assumption~\ref{assump: codim}.

The hypothesis on $\ell$ and $p$ ensure
that these groups satisfy
Assumption~\ref{assumption: keygroupassumption}

In the quaternionic case, as constructed in \cite{TX16}, we can find a PEL cover of type A such that the
local PEL covering group is
$G(\mathrm U_{2,\mathbf L})$,
where $\mathbf{L} := (L_{1},\ldots,L_{r})$ and all the $L_i/\mathbb Q_p$ are unramified.

In each case, the local PEL covering group satisfies
Assumption~\ref{assumption: keygroupassumption}, hence
Assumption~\ref{assumption: unramifiedness} by the verification
in the proof of Corollary~\ref{cor: appliedperversetexactness}.
Genericity transfers to every spherical lift on the PEL cover
by Lemma~\ref{lem:genericity-derived-isogeny}.
The conclusion follows from Corollary~\ref{cor:abeliantypebody}.
\end{proof}
\begin{remark}
    As remarked by a referee, there are Shimura data other than the ones listed in Corollary \ref{cor: CarTamcomp} to which we can apply Corollary \ref{cor:abeliantypebody}. For instance, one can consider the Shimura data where the group $\mathbf{G}_1$ is an inner form of $\mathrm{Res}_{F/\mathbb{Q}}\GSp_4$ for $F$ totally real, and which is split at all real places, and $X_1$ is the product of $[F:\mathbb{Q}]$ copies of the Siegel upper half space $\mathcal{H}_2$. In this case, one can show that there exists a PEL-type covering datum $(\mathbf{G},X)$ satisfying $(\mathbf{G}^{\mathrm{ad}},X^{\mathrm{ad}})\simeq (\mathbf{G}^{\mathrm{ad}}_1,X_1^{\mathrm{ad}})$.
\end{remark}

\section{Conjectures and Concluding Remarks}
In this section, we give some further complements to the main results established in the previous section. In particular, in \S 6.1, we explain how, by combining Theorem \ref{thm: appliedmantprodform}, with some further results of \cite{Ham2} one can obtain a more precise description of the contribution of a basic unramified element to the cohomology of the torsion cohomology of the global Shimura variety, that is compatible with an analogous description provided by Xiao and Zhu in \cite{XZ}. In \S 6.2, we discuss possible generalizations of our results, introducing the notion of a Langlands-Shahidi type parameter, which should be the most general condition for which the analogue of all of our main results hold (Conjectures \ref{conj: semiorthogonaldecomposition}, \ref{conj: generaltorsionvanish}). 
\subsection{Relationship to Xiao-Zhu}{\label{sec: comaprisonwithXiaoZhu}}
We keep the running notation and assumptions of the previous section. In particular, the pair $(\mathbf{G},X)$ will be a Shimura datum of PEL type AC satisfying Assumption \ref{assump: codim}, and $K = K^{p}K_{p} \subset \mathbf{G}(\bb{A}_{f})$ will denote a sufficiently small level. We assume that $G := \mathbf{G}_{\bb{Q}_{p}}$, $p$ and $\ell$ satisfy assumption \ref{assumption: keygroupassumption} and that $\mf{m} \subset H_{K_{p}^{\mathrm{hs}}}$ is a generic maximal ideal in the spherical Hecke algebra whose associated toral parameter $\phi_{\mathfrak m}^{T}$ is regular in the sense of Definition~\ref{defn: weaknormreg}. We will assume that the basic element $b \in B(G,\mu)_{\mathrm{un}}$ is unramified (see \cite[Remark~4.2.11]{XZ} for a classification). Let us look at the middle degree cohomology $H^{d}(R\Gamma_{c}(\mathcal{S}(\mathbf{G},X)_{K^{p},C},\ol{\mathbb{F}}_{\ell})_{\phi_{\mf{m}}})$. By Theorem \ref{thm: appliedmantprodform}, it has a summand isomorphic to 
\[ H^{d}((R\Gamma_{c}(G,b,\mu) \otimes^{\mathbb{L}}_{\mathcal{H}(J_{b})} R\Gamma_{c-\partial}(\Ig^{b},\ol{\mathbb{F}}_{\ell}))_{\phi_{\mf{m}}}).  \]
To describe this, let $\mathbf{G}'$ be the unique $\mathbb{Q}$-inner form of $\mathbf{G}$ such that 
$\mathbf{G}(\mathbb{A}^{p\infty}) \simeq \mathbf{G}'(\mathbb{A}^{p\infty})$, $\mathbf{G}'(\mathbb{R})$ is compact modulo center, and $\mathbf{G}'_{\mathbb{Q}_{p}} \simeq J_{b}$ (see \cite[Proposition~3.1]{Han} for the existence). We write $C(K^{p} \backslash \mathbf{G}'(\mathbb{A}_{f})/\mathbf{G}'(\mathbb{Q}),\ol{\mathbb{F}}_{\ell})$ for the set of all continuous functions on the profinite (since $\mathbf{G}'_{\bb{R}}$ is compact modulo center) set $K^{p}  \backslash \mathbf{G}'(\mathbb{A}_{f})/\mathbf{G}'(\mathbb{Q})$. It is easy to show that one has an isomorphism
\[ C(K^{p}\backslash\mathbf{G}'(\mathbb{A}_{f})/\mathbf{G}'(\mathbb{Q}),\ol{\mathbb{F}}_{\ell}) \simeq  R\Gamma_{c-\partial}(\Ig^{b},\ol{\mathbb{F}}_{\ell}) \]
for example by combining \cite[Theorem~3.4]{Han} and Corollary \ref{cor: stalkspartial}. We let $V_{\mu} \in \Rep_{\ol{\mathbb{F}}_{\ell}}(\hat{G})$ be the usual highest weight module of highest weight $\mu$, which in particular agrees with the highest weight tilting module, since $\mu$ is minuscule. We let $b_{T}$ denote the unique (since $b$ is basic) reduction of $b \in B(G)$ to $B(T)$, and regard it as an element in $B(T) \simeq \mathbb{X}^{*}(\hat{T}^{\Gamma})$ in what follows. It should be the case that, under possible additional constraints on $\mf{m}$ depending on $\mu$ (see for example \cite[Conjecture~1.18]{Ham2} and \cite[Definition~1.4.2]{XZ}), we have an isomorphism
\begin{align}{\label{XZ isom}} R\Gamma_{c}(\Sht(G,b,\mu)_{\infty,C}/\ul{K_{p}^{\mathrm{hs}}},\ol{\mathbb{F}}_{\ell})_{\mf{m}} \otimes_{\mathcal{H}(J_{b})}^{\mathbb{L}} C(K^{p}\backslash\mathbf{G}'(\mathbb{A}_{f})/\mathbf{G}'(\mathbb{Q}),\ol{\mathbb{F}}_{\ell})  & \simeq \\ C(K^{p}K_{p}^{\mathrm{hs}}\backslash\mathbf{G}'(\mathbb{A}_{f})/\mathbf{G}'(\mathbb{Q}),\ol{\mathbb{F}}_{\ell})_{\mf{m}} \otimes V_{\mu}|_{\hat{G}^{\Gamma}}(b_{T})[-d](-\frac{d}{2}) \nonumber
\end{align}
of $H_{K_{p}^{\mathrm{hs}}}$-representations\footnote{One should also be able describe the Weil group action, as in \cite[Conjecture~1.18]{Ham2}.} up to possibly taking the semi-simplifications of both sides as $H_{K_{p}^{\mathrm{hs}}}$-representations, where we note that $J_{b} \simeq G$ if $b \in B(G,\mu)_{\mathrm{un}}$ since $b$ is basic, and $J_{b}$ must be quasi-split since $b$ is unramified. Note that, if we were working with $\ol{\bb{Q}}_{\ell}$-coefficients\footnote{For this comparison, it would have been more natural to consider an analogue of Theorem \ref{thm: appliedmantprodform} with $\ol{\mathbb{Q}}_{\ell}$-coefficients. This is indeed doable assuming that $\phi_{\mf{m}}$ admits a $\ol{\mathbb{Z}}_{\ell}$-lattice as in \cite[Theorem~1.12]{Ham2}. This integrality condition is however an artifact of the theory of solid $\ol{\mathbb{Q}}_{\ell}$-sheaves not being properly understood (e.g excision fails) and should be removable with more technology.}, the space of automorphic forms $C(K^{p}\backslash\mathbf{G}'(\mathbb{A}_{f})/\mathbf{G}'(\mathbb{Q}),\ol{\mathbb{Q}}_{\ell})$ with rational coefficients will be semisimple as a $G(\bb{Q}_{p})$-representation, by virtue of the fact that $\mathbf{G}'(\bb{R})$ will be compact modulo center, and in this case the semi-simplification would be unnecessary.

In particular, by arguing as in \cite[Page~6]{Ko}, we know that $R\Gamma_{c}(\Sht(G,b,\mu)_{\infty,C}/\ul{K_{p}^{\mathrm{hs}}},\ol{\mathbb{F}}_{\ell})_{\mf{m}}$ will have cohomology sheaves with irreducible constituents given by the representations of $J_{b}(\mathbb{Q}_{p})$ with Fargues-Scholze parameter equal to $\phi_{\mf{m}}$ as conjugacy classes of parameters. Moreover, Assumption~\ref{compatibility} holds
for the present local group by
the proof of Corollary~\ref{cor: appliedperversetexactness}. By Proposition~\ref{prop: constituent proposition} and the
generic regularity of
$\phi_{\mathfrak m}^T$,
the relevant representations
are constituents of the
corresponding principal series,
which is irreducible under
these hypotheses.  In particular, the only irreducible $J_{b}(\bb{Q}_{p}) \simeq G(\bb{Q}_{p})$-representations that will contribute to the $H_{K_{p}^{\mathrm{hs}}}$-semisimplification of the LHS of (\ref{XZ isom}) will be given by $i_{B}^{G}(\chi)$, and \cite[Conjecture~1.18]{Ham2} would imply that $R\Gamma_{c}(G,b,\mu)[i_{B}^{G}(\chi)] \simeq i_{B}^{G}(\chi) \otimes V_{\mu}|_{\hat{G}^{\Gamma}}(b_{T})[-d](\frac{-d}{2})$
as $G(\mathbb{Q}_{p})$-modules. If we assume $\ell$ is banal (i.e., coprime to the pro-order of $K_{p}^{\mathrm{hs}}$) then, passing to $K_{p}^{\mathrm{hs}}$-invariants (recalling that this is exact under the banal hypothesis) gives us the isomorphism (\ref{XZ isom}) up to semi-simplification as $H_{K_{p}^{\mathrm{hs}}}$-modules.
\begin{remark}
If $B(G,\mu)_{\mathrm{un}}$ consists of only the basic element and the $\mu$-ordinary element and $\phi_{T}$ satisfies a condition known as strong $\mu$-regularity then \cite[Conjecture~1.18]{Ham2} is true. In particular, it follows from \cite[Corollary~1.17]{Ham2} that the isomorphism (\ref{XZ isom}) can be made unconditional. More precisely, for $(\mathbf{G},X)$ any PEL type datum of type AC satisfying Assumption \ref{assump: codim} such that the local group $G$, $p$, and $\ell$ satisfy Assumption \ref{assumption: keygroupassumption}, $\ell$ is banal with respect to $G$, and $\mf{m} \subset H_{K_{p}^{\mathrm{hs}}}$ a generic maximal ideal of the form described in Theorem \ref{thm: appliedperversetexactnessintro} with associated semisimple toral parameter $\phi_{\mf{m}}^{T}$ such that $\phi_{\mf{m}}^{T}$ is generic regular and strongly $\mu$-regular in the sense of (\cite[Definition~9.10]{Ham2}), then the isomorphism (\ref{XZ isom}) holds after semi-simplifying as a $H_{K_{p}^{\mathrm{hs}}}$-module in this case.
\end{remark}
We note that this description of the $K_{p}^{\mathrm{hs}}$-invariants of the basic summand in middle degree cohomology on the generic fiber of the Shimura variety at hyperspecial level parallels \cite[Theorem~1.1.4(1)]{XZ}, which describes the top Borel--Moore homology of the basic Newton stratum in the special fiber of the natural integral model.
\subsection{A General Torsion Vanishing Conjecture}
Consider now a general Shimura datum $(\mathbf{G},X)$. Let $\Lambda \in \{\ol{\mathbb{Q}}_{\ell},\ol{\mathbb{F}}_{\ell}\}$. If $\Lambda = \ol{\mathbb{F}}_{\ell}$ assume that $\ell$ is very decent as in Definition \ref{defn: verydecent}. We can then look at the complex of $G(\mathbb{Q}_{p})$-representations 
\[ R\Gamma_{c}(\mathcal{S}(\mathbf{G},X)_{K^{p},C},\Lambda) \]
defined by the cohomology at infinite level. By applying Proposition \ref{prop:localizeula} to the extension by $j_{1!}$ of this perfect-admissible complex, we obtain a $G(\mathbb{Q}_{p})$-equivariant decomposition of this complex
\[ R\Gamma_{c}(\mathcal{S}(\mathbf{G},X)_{K^{p},C},\Lambda) = \bigoplus_{\phi} R\Gamma_{c}(\mathcal{S}(\mathbf{G},X)_{K^{p},C},\Lambda)_{\phi}, \] 
running over semisimple $L$-parameters $\phi: W_{\mathbb{Q}_{p}} \ra \phantom{}^{L}G(\Lambda)$. For such a $\phi$, we let $(\phi_{M},M)$ denote a cuspidal support. That is $M$ is a Levi of $G$ and $\phi_{M}: W_{\mathbb{Q}_{p}} \ra \phantom{}^{L}M(\Lambda)$ is a supercuspidal $L$-parameter such that $\phi$ is induced (up to $\hat{G}$-conjugacy) by composing with the natural (up to $\hat{G}$-conjugacy) embedding $\phantom{}^{L}M(\Lambda) \ra \phantom{}^{L}G(\Lambda)$. We want to describe the degrees of cohomology that $R\Gamma_{c}(\mathcal{S}(\mathbf{G},X)_{K^{p},C},\Lambda)_{\phi}$ sits in for suitably nice $\phi$. The case where $\phi$ factors through $M = T$ is covered by Conjecture \ref{conj: torsionvanishing}. To go beyond this, we give the following definition. 
\begin{definition}{\label{def: LangShahi}}
For $P$ a parabolic with Levi factor $M$, let $\hat P=\hat M\hat N$ be the corresponding dual parabolic. We consider the representation $r_{P}$ given by the adjoint action of $\phantom{}^{L}M$ on $\operatorname{Lie}(\hat N)$. For a semisimple $L$-parameter $\phi$ with a cuspidal support $(M,\phi_{M})$, we say $\phi$ is of Langlands-Shahidi type if, for all parabolics $P$ with Levi factor $M$, the Galois cohomology groups
\[ R\Gamma(W_{\mathbb{Q}_{p}},r_{P} \circ \phi_{M}) \]
are trivial. Similarly, we say $\phi$ is of weakly Langlands-Shahidi type if
\[ H^{2}(R\Gamma(W_{\mathbb{Q}_{p}},r_{P} \circ \phi_{M}))  \]
is trivial.
\end{definition}
\begin{remark}
We note that since we enforced this condition for all parabolics $P$, it is independent of the choice of cuspidal support. Moreover, it is easy to check that, if $M = T$, the Langlands-Shahidi type condition precisely recovers Definition \ref{def: generic}.
\end{remark}
The terminology of "Langlands-Shahidi type" comes from the fact that the representation $r \circ \phi_{M}$ is precisely the representation which appears in the description of the constant terms of the usual Eisenstein series via the Langlands-Shahidi method (\cite{LanglandsEulerProducts}, \cite{ShahidiLFunction}, \cite{ShahidiRamanujanConjecture}). The motivation for this definition comes from considering the behavior of geometric Eisenstein series over the Fargues-Fontaine curve for general parabolics, by making analogies with the classical theory over function fields, as developed in \cite{BG,Lau}. In particular, this should be the correct definition that guarantees that the eigensheaves $\mathcal{S}_{\phi}$ on $\Bun_{G}$ with eigenvalue $\phi$ are as simple as possible, and the analysis carried out in \cite{Ham2} generalizes to the non-principal case. This is discussed in more detail in \cite[Chapter~3]{Ham4}. In addition, we expect that the consequences derived from the analysis in \cite{Ham2} in the principal case should also generalize. More precisely, we conjecture the following generalization of Proposition \ref{prop: genlocal} and Corollary \ref{cor: appliedperversetexactness}
\begin{conjecture}{\label{conj: semiorthogonaldecomposition}}
Let $B(G)_{M} := \mathrm{Im}(B(M)_{\mathrm{basic}} \ra B(G))$ be the set of $M$-reducible elements, and let $\phi$ be a semisimple $L$-parameter of Langlands-Shahidi type with cuspidal support $(M,\phi_{M})$. The category $\Dlis^{\mathrm{ULA}}(\Bun_{G},\Lambda)_{\phi}$ of $\phi$-local (as defined in Appendix \ref{append: spectralactionproperties}) lisse-\'etale ULA $\Lambda$-sheaves  breaks up as direct product
\[ \Dlis^{\mathrm{ULA}}(\Bun_{G},\Lambda)_{\phi} \simeq \prod_{b \in B(G)_{M}} \Dlis^{\mathrm{ULA}}(\Bun_{G}^{b},\Lambda)_{\phi} \]
via excision\footnote{We note that a priori there are no excision triangles for the HN-stratification of $\Bun_{G}$, since there is in general no well-behaved $!$-pushforward in the lisse formalism with non-torsion coefficients. However, one can show that such a pushforward exists for the inclusions of Harder-Narasimhan strata, and satisfies the usual excision triangles (see for example the discussion on \cite[Page~11]{ImaConv}), and so this is well-defined in the present context.}, and the $!$ and $*$ pushforwards agree for any smooth irreducible representation $\rho$ of $J_{b}(\mathbb{Q}_{p})$ lying in $\Dlis(\Bun_{G}^{b},\Lambda)_{\phi}$ for $b \in B(G)_{M}$.
\\\\
Given a tilting module $V \in \mathrm{Tilt}_{\Lambda}(\phantom{}^{L}G^{I})$ (where $\Tilt_{\Lambda}(\phantom{}^{L}G^{I})$ is defined as in \cite[Section~9]{Ham2}), if $\phi$ is of weakly Langlands-Shahidi type then the map induced by the associated Hecke operator 
\[ T_{V}: \Dlis^{\mathrm{ULA}}(\Bun_{G},\Lambda)_{\phi} \ra \Dlis^{\mathrm{ULA}}(\Bun_{G},\Lambda)^{BW_{\mathbb{Q}_{p}}^{I}}_{\phi} \]
on the $\phi$-localized ULA subcategory is perverse $t$-exact, where the fact the Hecke operator preserves this subcategory is proven as in Lemma \ref{lemma: localization map properties} (2).
\end{conjecture}
\begin{remark}{\label{rem: Generous?}}
While we were formulating our conjecture, Hansen formulated similar conjectures with rational coefficients \cite{HanBeij},
using the notions of generous parameters
\cite[Section~2.1]{HanBeij} and semisimple generic parameters
\cite[Section~2.3]{HanBeij}.
We explain their relationship with
Definition~\ref{def: LangShahi} and regularity. Assume that $\Lambda=\overline{\mathbb Q}_{\ell}$ as in \cite{HanBeij}, and let
$\phi:W_{\bb{Q}_{p}}\to{}^LG(\Lambda)$ be a semisimple parameter with
cuspidal support $(M,\phi_M)$. We say that $\phi_M$ is regular
if its $\hat M$-conjugacy class has trivial stabilizer
under the Weyl group $W_G(M)=N_{\hat G}({}^LM)/\hat M$ (which identifies with the relative Weyl group $W_{G}$ if $M = T$). 
Then
\begin{equation}{\label{generosity?1}}
  \phi\text{ is generous}
  \quad\Longleftrightarrow\quad
  \phi\text{ is of Langlands--Shahidi type and }
  \phi_M\text{ is regular}.
\end{equation}
We write $\mathfrak{X}_{\hat{G}}$ for the stack of $L$-prameters over $\ol{\bb{Q}}_{\ell}$ and $q: \mathfrak{X}_{\hat{G}} \ra X_{\hat{G}}$ for its map to the coarse quotient. We let $x_\phi$ be the point corresponding to a semi-simple parameter $\phi$.

We recall, from \cite[Definition~2.1.5]{HanBeij}, that a semisimple
parameter $\phi:W_{\bb{Q}_{p}} \to{}^LG(\overline{\mathbb Q}_{\ell})$
is called \emph{generous} if the natural morphism
\[
  BS_\phi \longrightarrow
  \mathfrak X_{\hat G}
  \times_{X_{\hat G}}\{x_\phi\}
\]
is an isomorphism. The RHS is the scheme-theoretic fiber of $q$ over $x_\phi$. Here
$S_\phi=\operatorname{Cent}_{\hat G}(\phi)$,
$q:\mathfrak X_{\hat G}\to X_{\hat G}$ is the map
from the stack of Langlands parameters to its coarse quotient,
and $x_\phi$ is the point corresponding to $\phi$.

To explain this equivalence, we write  $S_\phi=\operatorname{Cent}_{\hat G}(\phi)$, $S_{\phi_M}=\operatorname{Cent}_{\hat M}(\phi_M)$, and $A_M=(Z(\hat M)^{W_{\bb{Q}_{p}}})^\circ$. We recall that the tangent complex of $\mathfrak{X}_{\hat{G}}$ at $\phi$ is given by twisting the adjoint representation by $\phi$
\begin{equation}{\label{eqn: tangentcomplexdescription}}
  \mathbb T_\phi\mathfrak X_{\hat G}
  \simeq R\Gamma(W_{\bb{Q}_{p}},\operatorname{Ad}\circ\phi)[1];
\end{equation}
see \cite[Proposition~2.15, Remark~2.21, and
Proposition~2.41]{Zhu1}.
Together with the etale slice theorem
\cite[Theorem~1.2 and Proposition~4.13]{AHR20}, this gives
\begin{equation}{\label{eqn: TangentComplexCharacterizationofGernosity}}
  \phi\text{ is generous}
  \quad\Longleftrightarrow\quad
  \begin{cases}
    H^2(W_{\bb{Q}_{p}},\operatorname{Ad}\circ\phi)=0,\\
    S_\phi\text{ acts trivially on }
    H^1(W_{\bb{Q}_{p}},\operatorname{Ad}\circ\phi).
  \end{cases}
\end{equation}
Here we note that generosity implies smoothness by
\cite[Proposition~2.1.7~(3)]{HanBeij}.
Conversely, at a smooth parameter the slice theorem gives the
\'etale-local model $[V/S_\phi]\to V/\!/S_\phi$ of the fiber, where
$V=H^1(W_{\bb{Q}_{p}},\operatorname{Ad}\circ\phi)$.
In particular, its scheme-theoretic fiber over the origin is $BS_\phi$ exactly
when $S_\phi$ acts trivially on $V$.

Since $\phi_M$ is supercuspidal, we have
$S_{\phi_M}^\circ=A_M$ and
$H^2(W_{\bb{Q}_{p}},\operatorname{Ad}\circ\phi_M)=0$.
Decomposing $\operatorname{Lie}(\hat G)$ using opposite
parabolics with Levi $\hat M$, therefore gives
\begin{equation}{\label{eqn: characterizationofLStype}}
  \phi\text{ is of Langlands--Shahidi type}
  \quad\Longleftrightarrow\quad
  H^2(W_{\bb{Q}_{p}},\operatorname{Ad}\circ\phi)=0
  \ \text{ and }\ S_\phi^\circ=A_M.
\end{equation}
In particular, we note that we have the vanishing of $H^2(W_{\bb{Q}_{p}},\mathrm{Ad} \circ \phi) = 0$ appearing on the RHS of the characterization of generosity (\ref{eqn: TangentComplexCharacterizationofGernosity}) and of Langlands-Shahidi type (\ref{eqn: characterizationofLStype}). Under this assumption, th semisimplicity of $\phi$
and the Hochschild--Serre sequence applied to $1 \ra I_{\bb{Q}_{p}} \ra W_{\bb{Q}_{p}} \ra \bb{Z} \ra 1$  give an
$S_\phi$-equivariant identification
\begin{equation}{\label{eqn: UsefulCharacterizationunderH2vanishing}}
 H^2(W_{\bb{Q}_{p}},\operatorname{Ad}\circ\phi)=0 \quad \implies \quad  H^1(W_{\bb{Q}_{p}},\operatorname{Ad}\circ\phi)
  \simeq ((\operatorname{Ad}\circ\phi)^{I})_{\sigma} \simeq \operatorname{Lie}(S_\phi).
\end{equation}
for $\sigma$ a lift of Arithmetic Frobenius, where the last identification follows from the fact that $\sigma$ coinvariants identify with invariants by the semi-simplicity of $\phi$ (which in turn implies that $\sigma$ acts semi-simply on $(\operatorname{Ad}\circ\phi)^{I}$).

Now, suppose first that $\phi$ is of Langlands--Shahidi type. Since $S_\phi^\circ=A_M$ and
$\hat M=\operatorname{Cent}_{\hat G}(A_M)$,
the group $S_\phi$ normalizes ${}^LM$, and
\begin{equation}{\label{eqn: LStypecomputesregulars}}
  \phi\text{ is of Langlands--Shahidi type}
  \quad\implies\quad
  S_\phi/S_{\phi_M}
  \simeq \operatorname{Stab}_{W_G(M)}([\phi_M]).
\end{equation}
This quotient acts faithfully on $\operatorname{Lie}(A_M)$. Indeed, note that
$A_M=S_\phi^\circ$ (\ref{eqn: characterizationofLStype}), so $S_\phi$ by  normalizes $A_M$.
Since $\operatorname{Cent}_{\widehat G}(A_M)=\widehat M$,
the kernel of the conjugation action on $A_M$ is $S_\phi\cap\operatorname{Cent}_{\widehat G}(A_M)
  =S_\phi\cap\widehat M
  =S_{\phi_M}$. Thus $S_\phi/S_{\phi_M}$ acts faithfully on $A_M$.
Moreover, since we are in characteristic $0$, the induced map $\mathrm{Aut}(A_{M}) \ra \GL(\Lie(A_{M}))$ is injective (since $A_{M}$ is a torus), so this also acts faithfully on $\Lie(A_{M})$. Thus, regularity is equivalent to the triviality of the
$S_\phi$-action on 
$H^1(W_{\bb{Q}_{p}},\operatorname{Ad}\circ\phi)$ in light of (\ref{eqn: UsefulCharacterizationunderH2vanishing}) and the fact that $\Lie(S_{\phi}) = \Lie(A_{M})$, and hence to generosity.

Conversely, if $\phi$ is generous, the 
criterion (\ref{eqn: TangentComplexCharacterizationofGernosity}) in light of (\ref{eqn: UsefulCharacterizationunderH2vanishing}) shows that $S_\phi$ acts trivially on
$\operatorname{Lie}(S_\phi)$. Consequently,
$S_\phi^\circ$ is a torus which is central in $S_\phi$.
The fact $\phi_M$ is supercuspidal implies that $A_M$ is a maximal
torus of $S_\phi^\circ$, so $S_\phi^\circ=A_M$ and therefore $\phi$ is of Langlands-Shahidi type and we can now argue as above to show that $\phi$ must be regular since $S_\phi$ acts trivially on
$\operatorname{Lie}(S_\phi)$, as desired.

On the other hand, we have $\phi$ is \emph{semisimple generic} in the sense
of \cite[Section~2.3]{HanBeij} if its adjoint $L$-function
$L(s,\operatorname{Ad}\circ\phi)$ is regular at $s=1$ or equivalently, by \cite[Proposition~2.3.1]{HanBeij},
$\mathfrak{X}_{\hat{G}}$ is smooth at $\phi$. We easily see 
\[
  \phi\text{ is of weakly Langlands--Shahidi type}
  \quad\Longleftrightarrow\quad
  \phi\text{ is semisimple generic}.
\]
Indeed, both are easily verified to be equivalent to $H^{2}(W_{\bb{Q}_{p}},\mathrm{Ad} \circ \phi) \simeq 0$ using the description of the tangent complex in (\ref{eqn: tangentcomplexdescription}).
\end{remark}
In particular, by combining this with a generalization of Theorem \ref{thm: mantprodform} to arbitrary Shimura varieties and the analysis carried out in \S 5, we could deduce the following as a consequence. 
\begin{conjecture}{\label{conj: generaltorsionvanish}}
Let $\phi$ be a semisimple $L$-parameter of weakly Langlands-Shahidi type with cuspidal support $(M,\phi_{M})$. Then the complex $R\Gamma_{c}(\mathcal{S}(\mathbf{G},X)_{K^{p},C},\Lambda)_{\phi}$ (resp. $R\Gamma(\mathcal{S}(\mathbf{G},X)_{K^{p},C},\Lambda)_{\phi}$) is concentrated in degrees $0 \leq i \leq d$ (resp. $d \leq i \leq 2d$).
\end{conjecture}
\begin{remark}
 For $(\mathbf{G},X)$ of PEL type AC, assuming \ref{assump: codim} and that $\phi$ of Langlands-Shahidi type, we should also obtain a $W_{E_{\mf{p}}} \times G(\mathbb{Q}_{p})$-equivariant direct sum decomposition 
 \[ R\Gamma_{c}(\mathcal{S}(\mathbf{G},X)_{K^{p},C},\Lambda)_{\phi} \simeq \bigoplus_{b \in B(G,\mu)_{M}} (R\Gamma_{c}(G,b,\mu)_{\phi} \otimes_{\mathcal{H}(J_{b})}^{\mathbb{L}} V_{b})[2d_{b}], \]
 where $R\Gamma_{c}(G,b,\mu) := \colim_{K_{p} \ra \{1\}} R\Gamma_{c}(\Sht(G,b,\mu)_{\infty,C}/\ul{K_{p}},\Lambda(d_{b}))$ and $R\Gamma_{c}(G,b,\mu)_{\phi}$ is the projection applied to the complex viewed as a $G(\mathbb{Q}_{p})$-representation. This should also generalize once one has appropriate general definitions of $\Ig^{b}$ and $\Ig^{b,*}$ so that one can actually define $V_{b} := R\Gamma_{c-\partial}(\Ig^{b},\Lambda)$. Under possible additional constraints on $\phi$, one should also be able to describe the contribution of $R\Gamma_{c}(G,b,\mu)_{\phi}$ in terms of the decomposition $V_{\mu}|_{Z(\hat{M}^{\Gamma})} = \mathcal{T}_{\mu}|_{Z(\hat{M}^{\Gamma})}$ for $b \in B(G)_{M}$ (along the lines of \cite[Conjecture~1.18]{Ham2}), as is explained in the toral case in \S \ref{sec: comaprisonwithXiaoZhu}. It would be interesting to formulate an optimal conjecture.   
\end{remark}
\begin{remark}{\label{rem: weaklyLanglandsShahidicase}}
We believe that this conjecture should be true under just the weakly Langlands-Shahidi condition. However, the weakly Langlands-Shahidi condition alone need not imply the splitting of the semi-orthogonal decomposition or of Mantovan's filtration discussed in the previous Remark. In particular, in \cite[Section~2.2]{HanBeij} Hansen conjectures the existence of perverse sheaves lying $\Dlis(\Bun_{G},\Lambda)_{\phi}$, for which the semi-orthogonal decomposition does not split. Nonetheless, one still expects perverse $t$-exactness of Hecke operators to hold in these cases \cite[Conjecture~2.4.1]{HanBeij}. 
\end{remark}
\appendix
\section{Spectral Decomposition of Sheaves on $\mathrm{Bun}_G$, by David Hansen}{\label{append: spectralactionproperties}}
Let $G/\mathbb{Q}_{p}$ be a connected reductive group, $\Lambda/\mathbf{Z}_{\ell}$
an algebraically closed field. If $\mathrm{char}(\Lambda) \neq 0$ we assume
$\ell \nmid |\pi_{0}(Z(G))|$, where $Z(G)$ denotes the center of $G$, as in \cite[Theorem~I.10.1]{FS}.

Set $\D(\mathrm{Bun}_{G})= \Dlis(\mathrm{Bun}_{G},\Lambda)$ to be the derived category of lisse-\'etale $\Lambda$-sheaves, regarded
as a presentable stable $\infty$-category whenever convenient. Let $\mathfrak{X}_{\hat{G}}=Z^{1}(W_{\mathbb{Q}_{p}},\hat{G})_{\Lambda}/\hat{G}$
be the stack of $L$-parameters over $\Lambda$, and let $X_{\hat{G}}$ be its
coarse moduli space, $q:\mathfrak{X}_{\hat{G}}\to X_{\hat{G}}$ the natural map.
We will regard $\mathfrak{X}_{\hat{G}}$ as a disjoint union of finite type
algebraic stacks over $\Lambda$, and $X_{\hat{G}}$ as a disjoint union of finite
type affine $\Lambda$-schemes. As in \cite[Theorem IX.5.2 and Theorem X.0.2]{FS}, we have the spectral action of
$\mathrm{Perf}(\mathfrak{X}_{\hat{G}})$ on $\D(\mathrm{Bun}_{G})$, and
there is a natural map $\Psi_{G}: \mathcal{O}(\mf{X}_{\hat{G}}) = \mathcal{O}(X_{\hat{G}})\to\mathfrak{Z}(\D(\mathrm{Bun}_{G})) := \pi_{0}(\mathrm{End}(\mathrm{id}_{\D(\Bun_{G})}))$, where we recall that $Z^{1}(W_{\mathbb{Q}_{p}},\hat{G})_{\Lambda}$ is a disjoint union of affine schemes by \cite[Theorem~VIII.1.3]{FS}. These two structures are compatible (as proven by Zou \cite[Theorem~5.2.1]{Zou}).

By \cite[Prop. VIII.3.8]{FS}, the set of closed points $X_{\hat{G}}(\Lambda)$
is naturally in bijection with the set of isomorphism classes of semisimple
$L$-parameters $\phi:W_{\mathbb{Q}_{p}}\to\phantom{}^{L}\!G(\Lambda)$. Let $\mathfrak{m}_{\phi}\subset\mathcal{O}(X_{\hat{G}})$
be the maximal ideal associated with a given $\phi$.
\begin{definition}
    
Given any $\phi$ as above, $\D(\mathrm{Bun}_{G})_{\phi}\subset \D(\mathrm{Bun}_{G})$
is the full subcategory of sheaves $A\in \D(\mathrm{Bun}_{G})$ such
that for every $f\in\mathcal{O}(X_{\hat{G}})\smallsetminus\mathfrak{m}_{\phi}$,
$A\overset{\cdot f}{\to}A$ is an isomorphism. Here $\cdot f$ is
the endomorphism of $A$ induced by $\Psi_{G}$.
\end{definition}

We will call objects of $\D(\mathrm{Bun}_{G})_{\phi}$ $\phi$-\emph{local
}sheaves.
\\\\
By construction, $\D(\mathrm{Bun}_{G})_{\phi}$ is a full subcategory
of $\D(\mathrm{Bun}_{G})$ stable under arbitrary limits and colimits,
and the tautological inclusion functor $\iota_{\phi}:\D(\mathrm{Bun}_{G})_{\phi}\hookrightarrow \D(\mathrm{Bun}_{G})$
commutes with limits and colimits. Since $\D(\mathrm{Bun}_{G})$ is presentable, $\D(\mathrm{Bun}_{G})_{\phi}$ is then presentable by \cite[Theorem 1.1]{RS}. By the $\infty$-categorical adjoint
functor Theorem \cite[Cor. 5.5.2.9.(2)]{HTT}, the inclusion $\iota_{\phi}$ therefore admits a left adjoint $\mathcal{L}_{\phi}:\D(\mathrm{Bun}_{G})\to \D(\mathrm{Bun}_{G})_{\phi}$.\footnote{To see that $\iota_{\phi}$ is accessible, use \cite[Prop. 5.4.7.7]{HTT} together with the fact that $\iota_\phi$ admits a right adjoint, which follows from \cite[Cor. 5.5.2.9.(1)]{HTT}.}
The unit of the adjunction gives a map $A\to\iota_{\phi}\mathcal{L}_{\phi}A=:A_{\phi}$
functorially in $A$. Since $\iota_{\phi}$ is fully faithful, $\mathcal{L}_{\phi}\iota_{\phi}=\mathrm{id}$,
so $(A_{\phi})_{\phi}=A_{\phi}$, i.e., the endofunctor $A\rightsquigarrow A_{\phi}$
is idempotent. We remark that $\D(\mathrm{Bun}_{G})_{\phi}$ is a
localization of $\D(\mathrm{Bun}_{G})$, and the map $A\to A_{\phi}$
is the initial map from $A$ to a $\phi$-local sheaf.
\begin{proposition}\label{prop:localbasics}
The full subcategory $\D(\mathrm{Bun}_{G})_{\phi}$ is preserved by
the spectral action of $\mathrm{Perf}(\mathfrak{X}_{\hat{G}})$, and $A\rightsquigarrow A_{\phi}$ commutes with
the spectral action. Moreover, $\mathrm{supp}(A_{\phi}) \subseteq \mathrm{supp}(A)$.
\end{proposition}
Here $\mathrm{supp}(A) \subset |\mathrm{Bun}_G|=B(G)$ denotes the set of points $b$ such that $i_{b}^{\ast}A\neq 0$.

\begin{proof}
The first claim is clear, since the spectral action commutes with
the action of $\mathcal{O}(X_{\hat{G}})$. For the remaining claims (and
some later arguments), it is useful to give an explicit formula for
$A_{\phi}$. Let $\mathcal{I}_{\phi}$ be the diagram category whose
objects are elements of $\mathcal{O}(X_{\hat{G}})\smallsetminus\mathfrak{m}_{\phi}$
and where a morphism $f\to g$ is an element $h\in\mathcal{O}(X_{\hat{G}})\smallsetminus\mathfrak{m}_{\phi}$
such that $g=fh$. This is clearly cofiltered. Let $F\in\mathrm{Fun}(\mathcal{I}_{\phi},\D(\mathrm{Bun}_{G}))$
be the functor sending $f$ to $A$ and sending a morphism $h\in\mathrm{Mor}(f,g)$
to $\cdot h\in\mathrm{End}(A)$. Then $A_{\phi}=\mathrm{colim}_{i\in\mathcal{I}_{\phi}}F(i)$.
The remaining claims are now immediate.
\end{proof}
To make sense of the next Proposition, note that for any $A,B\in \D(\mathrm{Bun}_{G})$,
$\mathrm{Hom}(B,A)$ is naturally a $\mathfrak{Z}(\D(\mathrm{Bun}_{G}))$-module,
whence a $\mathcal{O}(X_{\hat{G}})$-module.
\begin{proposition}\label{prop:localizenaive}
If $C\in \D(\mathrm{Bun}_{G})$ is compact, then $\mathrm{Hom}(C,A_{\phi})\cong\mathrm{Hom}(C,A)_{\mathfrak{m}_{\phi}}$
functorially in $A$ and $C$, where the RHS is the usual localization as an $\mathcal{O}(X_{\hat{G}})$-module.
\end{proposition}
\begin{proof}
Notation as in the previous proof, we have
\begin{align*}
\mathrm{Hom}(C,A_{\phi}) & \cong\mathrm{Hom}(C,\mathrm{colim}_{i\in\mathcal{I}_{\phi}}F(i))\\
 & \cong\mathrm{colim}_{i\in\mathcal{I}_{\phi}}\mathrm{Hom}(C,F(i))\\
 & \cong\mathrm{Hom}(C,A)_{\mathfrak{m}_{\phi}}
\end{align*}
where the second isomorphism follows from the compactness of $C$
and the third isomorphism is immediate from the definition of $(-)_{\mathfrak{m}_{\phi}}$.
\end{proof}
\begin{proposition}
If $A$ is ULA, then also $A_{\phi}$ is ULA.
\end{proposition}

\begin{proof}
Recall from \cite[Prop. VII.7.9]{FS} that $B\in \D(\mathrm{Bun}_{G})$
is ULA iff $\RHom(C,B)\in\mathrm{Perf}(\Lambda)$ is a perfect complex for all compact
objects $C\in \D(\mathrm{Bun}_{G})$. Now, if $C$ is compact, $\RHom(C,-)$
commutes with filtered colimits, so
\begin{align*}
\RHom(C,A_{\phi}) & \simeq \RHom(C,\mathrm{colim}_{i\in\mathcal{I}_{\phi}}F(i))\\
 & \simeq\mathrm{colim}_{i\in\mathcal{I}_{\phi}}\RHom(C,F(i))
\end{align*}
with notation as in the proof of Proposition \ref{prop:localbasics}. Since $F(i)\simeq A$
for all $i$, $\mathrm{colim}_{i\in\mathcal{I}_{\phi}}\RHom(C,F(i))$
is a filtered colimit of perfect complexes $P_{i}$ which vanish outside
a finite interval independent of $i$, and with $\mathrm{dim}_{\Lambda}(H^{j}(P_{i}))$
bounded independently of $i$. It then easily follows that $\mathrm{colim}_{i\in\mathcal{I}_{\phi}}\RHom(C,F(i))$
is perfect, whence the claim.
\end{proof}
\begin{proposition}\label{prop:localizeula}
If $A$ is ULA, the natural maps $A\to\prod_{\phi}A_{\phi}\leftarrow\oplus_{\phi}A_{\phi}$
are isomorphisms, where the direct sum and direct product are taken over all semisimple $L$-parameters. In particular, $A_{\phi}$ is functorially a direct
summand of $A$ for ULA sheaves $A$, and the functor $(-)_{\phi}$ on ULA
sheaves is perverse t-exact.
\end{proposition}

\begin{remark}\label{sumproductula}
    
The isomorphism $\oplus_{\phi}A_{\phi}\overset{\sim}{\to}\prod_{\phi}A_{\phi}$
may be surprising at first glance. To put this in context, we remind
the reader that if $(\pi_{i})_{i\in I}$ is a collection of admissible
smooth $\Lambda[G(\mathbb{Q}_{p})]$-modules whose product $\prod_{i}\pi_{i}$ is admissible,
then $\oplus_{i}\pi_{i}\overset{\sim}{\to}\prod_{i}\pi_{i}$ automatically,
because admissibility of $\prod_{i}\pi_{i}$ implies that for any
given compact open subgroup $K\subset G(\mathbb{Q}_{p})$ we have $\pi_{i}^{K}=0$
for all but finitely many $i$. A similar argument occurs in the following
proof, which actually shows that if $(A_{i})_{i\in I}$ is any collection
of ULA sheaves on $\mathrm{Bun}_{G}$ whose product $\prod_{i}A_{i}$
is ULA, then $\oplus_{i}A_{i}\overset{\sim}{\to}\prod_{i}A_{i}$ automatically. 
\end{remark}

\begin{proof}
We first show that $A\to\prod_{\phi}A_{\phi}$ is an isomorphism.
Let $C$ be any compact object. It suffices to prove that the natural
map
\[
\mathrm{Hom}(C,A)\to\prod_{\phi}\mathrm{Hom}(C,A_{\phi})\cong\mathrm{Hom}(C,\prod_{\phi}A_{\phi})
\]
is an isomorphism, since $\D(\Bun_{G})$ is compactly generated \cite[Theorem~I.5.1 (iii)]{FS}. As in the previous proof, $\RHom(C,A)$
is a perfect complex, so $\mathrm{Hom}(C,A)$ is a finite $\Lambda$-vector
space. In particular, it is a finite length $\mathcal{O}(X_{\hat{G}})$-module
supported at a finite set of closed points $S\subset X_{\hat{G}}(\Lambda)$, so
if $\phi\notin S$ then $\mathrm{Hom}(C,A_{\phi})=\mathrm{Hom}(C,A)_{\mathfrak{m}_{\phi}}=0$
using Proposition \ref{prop:localizenaive}. We then conclude that
\begin{align*}
\mathrm{Hom}(C,A) & =\oplus_{\phi\in S}\mathrm{Hom}(C,A)_{\mathfrak{m}_{\phi}}\\
 & =\oplus_{\phi\in S}\mathrm{Hom}(C,A_{\phi})\\
 & =\prod_{\phi}\mathrm{Hom}(C,A_{\phi})
\end{align*}
where the first equality follows from general nonsense about finite
length modules over commutative rings, the second equality follows
from Proposition \ref{prop:localizenaive}, and the third equality follows from the vanishing
of $\mathrm{Hom}(C,A_{\phi})$ for all but finitely many $\phi$.
This also shows that $\mathrm{Hom}(C,\oplus_{\phi}A_{\phi})\cong\oplus_{\phi}\mathrm{Hom}(C,A_{\phi})\to\prod_{\phi}\mathrm{Hom}(C,A_{\phi})$
is an isomorphism (here again the first isomorphism follows from compactness
of $C$), which implies that $\oplus_{\phi}A_{\phi}\overset{\sim}{\to}\prod_{\phi}A_{\phi}$
is an isomorphism.
\end{proof}
Next, recall the Verdier duality functor $\mathbb{D}_{\Bun_{G}}$ on $\D(\mathrm{Bun}_{G})$,
which induces an involutive anti-equivalence on the subcategory of ULA
sheaves. Recall also that, for any $A$, the diagram
\[
\xymatrix{\mathcal{O}(X_{\hat{G}})\ar[rr]^{\Psi_{G}}\ar[d]^{f\mapsto f^{\vee}} & & \mathrm{End}(A)\ar[d]\\
\mathcal{O}(X_{\hat{G}})\ar[rr]^{\Psi_{G}} & & \mathrm{End}(\mathbb{D}_{\Bun_{G}}(A))
}
\]
commutes, where $f\mapsto f^{\vee}$ is the involution of $\mathcal{O}(X_{\hat{G}})$
induced by composition with the Chevalley involution at the level of
$L$-parameters (this follows from a small adaptation of the proof of \cite[Proposition~IX.5.3]{FS}). Since $f\in\mathfrak{m}_{\phi}$ iff $f^{\vee}\in\mathfrak{m}_{\phi^{\vee}},$
we deduce that if $A$ is $\phi$-local then $\mathbb{D}_{\Bun_{G}}(A)$ is $\phi^{\vee}$-local.
Using biduality for ULA sheaves (which follows easily from \cite[Proposition VII.7.7 and Proposition VII.7.9]{FS}), we also get that if $A$ is ULA then $A$ is $\phi$-local
if and only if $\mathbb{D}_{\Bun_{G}}(A)$ is $\phi^{\vee}$-local.
\begin{corollary}{\label{cor: contragradients}}
If $A$ is ULA, then $\mathbb{D}_{\Bun_{G}}(A_{\phi})\cong(\mathbb{D}_{\Bun_{G}}(A))_{\phi^{\vee}}$.
\end{corollary}

\begin{proof}
By Proposition \ref{prop:localizeula} and the remarks preceding its proof, the decomposition
$A=\oplus_{\psi}A_{\psi}$ dualizes to a decomposition
\[
\mathbb{D}_{\Bun_{G}}(A) =\prod_{\psi}\mathbb{D}_{\Bun_{G}}(A_{\psi})\cong\oplus_{\psi}\mathbb{D}_{\Bun_{G}}(A_{\psi})
\]
where the second isomorphism follows from the discussion in Remark \ref{sumproductula}.
On the other hand, applying Proposition \ref{prop:localizeula} directly to $\mathbb{D}_{\Bun_{G}}(A)$
gives a decomposition
\[
\mathbb{D}_{\Bun_{G}}(A)\cong\oplus_{\psi'}(\mathbb{D}_{\Bun_{G}}(A))_{\psi'},
\]
so comparing these we get a natural isomorphism
\[
\oplus_{\psi}\mathbb{D}_{\Bun_{G}}(A_{\psi})\cong\oplus_{\psi'}(\mathbb{D}_{\Bun_{G}}(A))_{\psi'}.
\]
Applying $(-)_{\phi^{\vee}}$ to both sides, we get $\mathbb{D}_{\Bun_{G}}(A_{\phi})$ on the left side (using that $\mathbb{D}_{\Bun_{G}}(A_{\phi})$
is $\phi^{\vee}$-local), and $(\mathbb{D}_{\Bun_{G}}(A))_{\phi^{\vee}}$ on the right side. This gives the claim.
\end{proof}
\printbibliography
\end{document}